\documentclass[12pt]{amsart}

\usepackage{psfrag}
\usepackage{color}
\usepackage{mathrsfs}

\usepackage[applemac]{inputenc}
\usepackage{graphicx,epsfig,amscd,graphics,xypic}
\usepackage{amsmath}
\usepackage{amsthm}
\usepackage{amssymb}
\usepackage[margin=4cm]{geometry}

\usepackage[dvips,all,cmtip]{xy}

\usepackage{fourier}

\usepackage{stackengine}
\usepackage{scalerel}

\usepackage{cellspace}
\usepackage{slashbox}

\newcolumntype{C}{Sc}
\usepackage[dvipsnames]{xcolor}

\usepackage{tikz-cd} 
\usetikzlibrary{matrix,arrows,decorations.pathmorphing}

\usepackage{hyperref}
\hypersetup{
    bookmarks=true,         
    unicode=false,          
    pdftoolbar=true,        
    pdfmenubar=true,        
    pdffitwindow=false,     
    pdfstartview={FitH},    
    pdftitle={My title},    
    pdfauthor={Author},     
    pdfsubject={Subject},   
    pdfcreator={Creator},   
    pdfproducer={Producer}, 
    pdfkeywords={keyword1} {key2} {key3}, 
    pdfnewwindow=true,      
    colorlinks=false,       
    linkcolor=red,          
    citecolor=green,        
    filecolor=magenta,      
    urlcolor=cyan           
}

\usepackage{accents}
\newcommand*{\dt}[1]{%
  \accentset{\mbox{\large\bfseries .}}{#1}}

 \newcommand{\eq}[1][r]
   {\ar@<-3pt>@{-}[#1]
    \ar@<-1pt>@{}[#1]|<{}="gauche"
    \ar@<+0pt>@{}[#1]|-{}="milieu"
    \ar@<+1pt>@{}[#1]|>{}="droite"
    \ar@/^2pt/@{-}"gauche";"milieu"
    \ar@/_2pt/@{-}"milieu";"droite"}

\newcommand{\sk}{\smallskip}
\newcommand{\mk}{\medskip}
\newcommand{\bk}{\bigskip}
\newcommand*{\longhookrightarrow}{\ensuremath{\lhook\joinrel\relbar\joinrel\rightarrow}}

\newcommand{\smxylabel}[1]{{\text{\tiny$#1$}}}

\usepackage[mathscr]{euscript}

\definecolor{Pin}{RGB}{255,0,255}

\definecolor{Griz}{RGB}{160,160,160}

\renewenvironment{proof}{\noindent {\it Dmonstration.}}{$\hfill \square$}

\newtheorem{thm}{Theorem}[section]
\newtheorem{coro}[thm]{Corollary}

\newtheorem*{thm*}{Theorem}   
\newtheorem{prop}[thm]{Proposition}

\newtheorem{lemma}[thm]{Lemma}
\newtheorem{rem}[thm]{Remark}
\newtheorem{ex}[thm]{Example}

\renewenvironment{proof}{\hspace{-0.4cm}{\bfseries Proof.}}{\qed}

\title[{\bf Flat tori and elliptic hypergeometric functions}]{Moduli spaces of flat tori
and    \\  elliptic hypergeometric functions}

\author[{\bf S. Ghazouani}]{{\bf S\'elim GHAZOUANI}}

\address{S. Ghazouani  [DMA - \'ENS  45 rue d'Ulm 
75230 Paris Cedex 05 - France]}
\email{selim.ghazouani@ens.fr}

\author[{\bf L. Pirio}]{{\bf Luc PIRIO}${}^{(\star)}$}
\address{L. Pirio [IRMAR, CNRS (UMR 6625)  -- 
 Universit Rennes 1 
 \&   LMV,  CNRS (UMR 8100) -- Universit Versailles St-Quentin, France]}
\email{luc.pirio@univ-rennes1.fr}
\thanks{${}^{(\star)}$Corresponding author.}

\begin{document}

${}^{}$
\vspace{-4cm}

\maketitle

\vspace{-0.6cm}

\begin{abstract} 
In the genus one case, we make explicit some constructions of Veech \cite{Veech} on flat surfaces and generalize 
some geometric results of Thurston \cite{Thurston} about moduli spaces of flat spheres as well as  some equivalent ones  
but of an analytico-cohomological nature 
of Deligne and Mostow \cite{DeligneMostow}, which concern 
 Appell-Lauricella hypergeometric functions. \sk

In the dizygotic twin paper \cite{GP}, we  follow Thurston's approach and study moduli spaces of flat tori with conical singularities and prescribed holonomy  by 
means of geometrical methods relying on surgeries for flat surfaces.  In the present paper, we study the same objects making use of  analytical and cohomological methods, more in the spirit of Deligne-Mostow's paper. \sk 

Our starting point is an explicit formula for flat metrics with conical singularities on elliptic curves, in terms of theta functions. From this, we deduce an explicit description of Veech's foliation:  at the level of the Torelli space 
of $n$-marked elliptic curves, 
it is given by an explicit affine first integral. From the preceding result,  one determines exactly  the  leaves of Veech's foliation which are closed subvarieties of the moduli space ${\mathscr M}_{1,n}$ 
 of $n$-marked elliptic curves. We  also  give   a local explicit expression,  in terms of hypergeometric elliptic integrals,  for the Veech  map which defines the complex hyperbolic structure of a leaf.    
\sk 

Then we focus on the  $n=2$ case: in this situation,  Veech's foliation does not depend on the values of the conical angles of the flat tori considered. Moreover, a leaf which is a  closed subvariety of ${\mathscr M}_{1,2}$ is  actually algebraic and is isomorphic to a modular curve $Y_1(N)$ for a certain integer $N\geq 2$.
In the considered situation,  the leaves of Veech's foliation are $\mathbb C\mathbb H^1$-curves. By specializing some results of  Mano and Watanabe \cite{ManoWatanabe}, we make explicit the 
Schwarzian differential equation satisfied by the $\mathbb C\mathbb H^1$-developing map 
of any leaf and use this to  prove that the metric completions of the algebraic ones are complex hyperbolic conifolds which are obtained by adding some of its cusps to $Y_1(N)$. Furthermore, we compute explicitly the conifold angle at any cusp $\mathfrak c\in X_1(N)$, 
 the latter being 0 ({\it i.e.}\;$\mathfrak c$ is an usual cusp) exactly when 
  it does not belong to the metric completion of the considered algebraic leaf. \sk 

In the last section of the paper, we discuss various aspects of the objects previously considered, such as:  some particular cases that we make explicit, some links with classical hypergeometric functions in the simplest cases. 
 We explain how to compute explicitly the $\mathbb C\mathbb H^1$-holonomy of any given algebraic leaf,  which is important in order to determine when the image of such a holonomy is a lattice in ${\rm Aut}(\mathbb C\mathbb H^1)
  \simeq {\rm PSL}(2,\mathbb R)$.
 Finally, we compute  the hyperbolic volumes of some algebraic leaves of Veech's foliation and we use this to give an explicit formula for 
 Veech's volume of the moduli space $\mathscr M_{1,2}$. In particular, 
we show that this volume   is finite, as conjectured in \cite{Veech}.    \sk 
 
 The paper ends with  two appendices. The first consists in a  short and easy introduction to the notion of $\mathbb C\mathbb H^1$-conifold. The second appendix is devoted to the  Gau{\ss}-Manin connection associated to our problem: we first give a general and detailed abstract treatment then we consider the specific case of $n$-punctured elliptic curves,  which is made completely explicit when $n=2$. 
\end{abstract}

\newpage

\setcounter{tocdepth}{1}
\tableofcontents


\section{\bf Introduction}

\subsection{\bf Previous works}
\label{S:PreviousWorks}
\subsubsection{}
\label{S:1.1.1}
The classical  {\bf hypergeometric series} 
defined for $\lvert x\lvert<1$ by 
\begin{equation}
\label{HGF}
F(a,b,c;x)=\sum_{n=0}^{+\infty} \frac{(a)_n(b)_n}{(c)_n(1)_n} x^n
\end{equation}
 together with the  {\bf hypergeometric differential equation} 
  it satisfies 
\begin{equation}
\label{HGE}
x(x-1) \cdot F''+\big(c-(1+a+b)x\big)\cdot  F'- ab\cdot F=0
\end{equation}
certainly constitute   
one of the most beautiful and important parts 
 of the  theory of special functions and of complex geometry of 19th century mathematics and has been studied by many generations of mathematicians since its first appearance in the work of Euler (see \cite[Chap. I]{Gray} for an  historical account).\sk

 The link between the solutions of  \eqref{HGE} and complex geometry is particularly well illustrated by the following very famous results obtained by Schwarz in \cite{Schwarz}: he proved that when the parameters $a,b$ and $c$ are real and 
such that the  three values $\lvert 1-c  \lvert$, $\lvert c-a-b  \lvert$ and $\lvert a-b  \lvert$ all are strictly less than 1, if $F_1$ and $F_2$ form a local basis of the space of solutions of \eqref{HGE} at a point distinct from one of the three singularities $0,1$ and $\infty$ of the latter, then after analytic continuation, the associated (multivalued) {\bf Schwarz's map}
\begin{equation*}
S(a,b,c; \cdot)= \big[ F_1: F_2 \big]\, : \, \stackon[-8pt]{$\mathbb P^1 \setminus \{0,1,\infty\}$}{\vstretch{1.5}{\hstretch{10.0}{\widetilde{\phantom{\;}}}}}\longrightarrow \mathbb P^1
\end{equation*}
actually has values into $\mathbb C\mathbb H^1\subset \mathbb P^1$  and 
induces a conformal isomorphism from the upper half-plane 
$\mathbb H\subset \mathbb P^1 \setminus \{0,1,\infty\}$ onto a hyperbolic triangle\footnote{Actually, Schwarz has proved a more general result that covers not only the hyperbolic case but the euclidean and the spherical cases as well.  See 
{\it e.g.}\;\cite[Chap.III\S3.1]{Gray} for a modern and clear exposition of the results of \cite{Schwarz}}.
 Even if  it is multivalued,  $S(a,b,c;\boldsymbol{\cdot})$  can be used to pull-back the standard complex hyperbolic structure of
$\mathbb C\mathbb H^1$  
 and to endow 
$\mathbb P^1\setminus \{0,1,\infty\}$ with a well-defined complete hyperbolic structure with conical singularities of angle 
$2\pi\lvert 1-c  \lvert$, $2\pi\lvert c-a-b  \lvert$ and $2\pi\lvert a-b  \lvert$ at $0, 1$ and $\infty$ respectively. 
\sk 

It has been known very early\footnote{It seems that Legendre was the very first to establish that $F(a,b,c;x)=\frac{\Gamma(c)}{\Gamma(a)\Gamma(c-a)}\int_0^1 t^{a-1}(1-t)^{c-a-1}(1-xt)^{-b}dt$  holds true when  $\lvert x\lvert<1$ if $a$ and $c$ verify $0<a<c$, {\it cf.}\;\cite[p.\,26]{Dutka}, 

.} that   the following 
{\bf hypergeometric integral} $$
F(x)=
\int_{0}^1 t^{a-1}(1-t)^{c-a-1}(1-xt)^{-b}dt
$$
 is a solution of \eqref{HGE}.  More generally, for any $x$ distinct from $0,1$ and $\infty$, any 1-cycle 
 $\gamma$ in $\mathbb P^1\setminus \{0,1,x,\infty\}$ and any 
 determination of the multivalued function $t^{a-1}(1-t)^{c-a-1}(1-xt)^{-b}$  along  $\gamma $,  the  (locally well-defined) 
 map 
  \begin{equation}
  \label{E:HGint-general}
  F_\gamma( x)= \int_\gamma t^{a-1}(1-t)^{c-a-1}(1-xt)^{-b}dt
  \end{equation}
   is a solution of \eqref{HGE} and a basis of the space of solutions can be obtained by taking independent integration cycles ({\it cf.}\;\cite{Yoshida} for a very pleasant modern exposition of these classical results).
 

\subsubsection{}  
\label{S:Appell-LauricellaHypergeometricFunctions}
Formula  \eqref{E:HGint-general}  leads naturally to a multi-variable generalization,  first considered by Pochammer, Appell and Lauricella, then studied  by Picard and his student Levavasseur (among others). 
Let $\alpha=(\alpha_i)_{i=0}^{n+2}$ be a fixed $(n+3)$-uplet of 
 non-integer real parameters  strictly bigger than $-1$ and  such that 
 $\sum_{i=0}^{n+2}
  \alpha_i = -2$. 
  
Given a  $(n+3)$-uplet  $x=(x_i)_{i=0}^{n+2}$ of distinct points on $\mathbb P^1$, 
  one defines a multivalued holomorphic function of the complex variable $t$ by setting 
  $$T^\alpha_x(t)=  \prod_{i=0}^{n+2} (t-x_i)^{\alpha_i}\, .$$  
  
 Then, for any 1-cycle 
 $\gamma$ supported in $\mathbb P^1\setminus  \{ x\}$ with $\{ x \}=\{x_{0},\ldots, x_{n+2}\}$ and any  choice of  a determination of 
 $T^\alpha_x( t)$ along $\gamma$, one defines a 
 {\bf generalized hypergeometric integral} as 
\begin{equation}
\label{E:GenHINT}
 F_\gamma^\alpha (x)=\int_\gamma 
 T^\alpha_x( t) dt  =\int_\gamma   \prod_{i=0}^{n+2} (t-x_i)^{\alpha_i}dt.
 \end{equation}
 
 Since  $T^\alpha_x( t)$ depends holomorphically on  $x$ and since $\gamma$ does not meet any of the $x_i$'s,  $F_\gamma^\alpha$ is holomorphic as well. 
 In fact, it is natural to normalize the integrant by considering only $(n+3)$-uplets $x$'s normalized such that $x_0=0, x_1=1$ and $x_{n+2}=\infty$. 
 This amounts to consider \eqref{E:GenHINT} as a multivalued function defined on the moduli space ${\mathscr M}_{0,n+3}$ of projective equivalence classes of $n+3$ distinct points on $\mathbb P^1$.  As in the 1-dimensional case, it can be proved that the generalized hypergeometric integrals \eqref{E:GenHINT} are solutions of a linear second-order differential system 
 in $n$ variables which has to be seen as a multi-dimensional generalization of Gau{{\ss}}
   hypergeometric equation 
 \eqref{HGE}.    Moreover, one obtains a basis of the space of solutions of this differential system by considering the (germs of) holomorphic functions 
 $F^\alpha_{\gamma_0},\ldots,F^\alpha_{\gamma_n}$ for some  1-cycles $\gamma_0,\ldots,\gamma_{n}$ forming a basis of a certain group of twisted homology.  
 
 \subsubsection{}
 \label{S:MultidimContext}
 In  this multidimensional context, 
 the associated  {\bf generalized Schwarz's map}    is  the multivalued  map 
 $$
 F^\alpha= \big[   F^\alpha_{\gamma_i}\big]_{i=0}^n : \widetilde{{\mathscr M}_{0,n+2}}\longrightarrow \mathbb P^{n}. 
$$ 

It can be proved that the monodromy of 
 this multivalued function on ${\mathscr M}_{0,n+3}$ leaves invariant an hermitian form $H^\alpha$ on $\mathbb C^{n+1}$ whose signature is $(1,n)$ when all the  $\alpha_i$'s are assumed to be negative.  
 In this case: 
\begin{itemize} 
\item[$\bullet$]
 $F^\alpha$ is an \'etale map with values into the image in $\mathbb P^{n}$ of the complex ball $\{  H^\alpha<1\}$ which is a model of the complex hyperbolic space $\mathbb C\mathbb H^{n}$; \sk
\item[$\bullet$]
  the monodromy group $\Gamma^\alpha$ of $F^\alpha$ is the image of the  monodromy  representation 
$\mu^\alpha$ of the fundamental group of ${\mathscr M}_{0,n+3}$
in $${\mathrm{ PU}}\big(\mathbb C^{n+1},H^\alpha\big)\simeq  {\mathrm{ PU}}(1,n)={\rm Aut}\big(\mathbb C\mathbb H^{n}\big).$$ 
\end{itemize}

As in the classical 1-dimensional case, these results imply that  there is a natural  a priori non-complete complex hyperbolic structure on ${\mathscr M}_{0,n+3}$, obtained as the pull-back of the standard one of  
$\mathbb C\mathbb H^{n}$
by the Schwarz map. We will denote by ${\mathscr M}_{0,\alpha}$ the moduli space ${\mathscr M}_{0,n+3}$ endowed with this $\mathbb C\mathbb H^{n}$-structure. \sk 

Several authors (Picard, Levavasseur, Terada, Deligne-Mostow)  have studied the case when the image of the monodromy $\Gamma^\alpha={\rm Im}(\mu^\alpha) $ is a discrete subgroup of ${\mathrm{ PU}}(1,n)$. In this case, the metric completion 
 of  ${{\mathscr M}_{0,\alpha}}$
is an orbifold  isomorphic  to a quotient orbifold $\mathbb C\mathbb H^{n}/\Gamma^\alpha$.  
Deligne  and Mostow  have obtained  very satisfying results on this problem: 
in  \cite{DeligneMostow,MostowIHES}  (completed in \cite{Mostow}) they  gave an arithmetic criterion on the $\alpha_i$'s,   denoted by $\Sigma$INT, which is necessary and sufficient (up to a few known cases) to ensure that the hypergeometric monodromy group $\Gamma^\alpha$ is discrete. 
Moreover, they have determined all the $\alpha$'s satisfying this criterion and have obtained a list of 94 complex hyperbolic orbifolds of dimension $\geq 2$ constructed via the theory of hypergeometric functions.  Finally, they obtain that some of these orbifolds are non-arithmetic.


\subsubsection{} 
\label{S:Bridge-Thurston-DeligneMostow}
In \cite{Thurston}, taking a different approach, Thurston obtains very similar results to Deligne-Mostow's. His approach is more geometric and concerns moduli spaces of flat Euclidean structures on $\mathbb P^1$   with $n+3$  conical singularities.  For $x\in {\mathscr M}_{0,n+3}$, the metric $m^\alpha_x =\lvert T_x^\alpha(t)dt\lvert ^2$ defines a flat structure on $\mathbb P^1$ with conical singularities at the $x_i$'s.   The bridge between the hypergeometric theory and  Thurston' approach is made by the map $x\mapsto 
 m^\alpha_x $.\sk

Using surgeries for flat structures on the sphere as well as Euclidean polygonal representation of such objects, Thurston recovers geometrically Deligne-Mostow's criterion as well as the finite list of  94 complex hyperbolic  orbifold quotients.  More generally,  he proves that 
for any $\alpha=(\alpha_i)_{i=0}^{n+2}
 \in ]-1,0[^{n+3}$ and not only for the (necessarily rational) ones  satisfying $\Sigma$INT,  the metric completion $\overline{{\mathscr M}}_{0,\alpha}$  carries  a complex hyperbolic conifold structure  (see \cite{Thurston,McMullen} or \cite{GP} for this notion)  which extends the 
$\mathbb C\mathbb H^{n}$-structure of the moduli space ${\mathscr M}_{0,\alpha}$.
\sk


 \subsubsection{} 
\label{S:VeechIntro}
In the very interesting (but  long and  hard-reading hence not so well-known) paper \cite{Veech}, 
Veech  generalizes some parts of the preceding constructions to Riemann surfaces of arbitrary genus $g$.
 Veech's starting point is a nice result by Troyanov \cite{Troyanov} asserting that for any 
 $\alpha=(\alpha_{i})_{i=1}^n\in ]-1,\infty[^n$ such that 
 \begin{equation}
 \label{E:GaussBonnetAlphai}
  \sum_{i=1}^n\alpha_i =2g-2
 \end{equation}
 and any Riemann surface $X$ with a $n$-uplet $x=(x_i)_{i=1}^n$ of marked distinct points on it,   there exists a unique flat metric $m_{X,x}^\alpha$  of area 1 on $X$ with  conical singularities of angle $\theta_i=2\pi(1+\alpha_i)$ at  $x_i$ for every $i=1,\ldots,n$,  in the conformal class associated to the complex structure of $X$. 

 From this, Veech obtains a real analytic isomorphism 
\begin{align}
\label{E:TgnEgn}
\mathcal T\!\!\!\mbox{\it eich}_{g,n}\; & \simeq \;  \mathcal E_{g,n}^\alpha\\
\big[ (X, x)\big]&  \mapsto \Big[ \big(X, m_{X,x}^\alpha\big)\Big]
\nonumber
\end{align}
   between the 
Teichmller space $\mathcal T\!\!\!\mbox{\it eich}_{g,n}$ of $n$-marked Riemann surfaces of ge\-nus $g$ and the space  $\mathcal E_{g,n}^\alpha$ of (isotopy classes of) flat Euclidean structures with $n$ conical points of angles $\theta_1,\ldots,\theta_n$ 
 on the marked surfaces of the same type.   

Using \eqref{E:TgnEgn} to 
identify the Teichmller space with $  \mathcal E_{g,n}^\alpha$, Veech constructs a real-analytic map 
\begin{equation}
\label{E:holAlpha}
{H}_{g,n}^\alpha: \mathcal T\!\!\!\mbox{\it eich}_{g,n}\longrightarrow \mathbb U^{2g}
\end{equation}
which associates to 
(the isotopy class of) a marked genus $g$ Riemann surface $(X,x)$ the unitary linear holonomy of the flat structure on $X$ induced by $m_{X,x}^\alpha$.  
 
 This map is a submersion and even though it is just real-analytic, Veech proves that any  level set  
 \begin{equation*}
 \label{E:LeafAccordingToVeech}
 \mathcal F^\alpha_\rho=\big({H}_{g,n}^\alpha\big)^{-1}(\rho)
 \end{equation*}
 is a   complex submanifold of $\mathcal T\!\!\!\mbox{\it eich}_{g,n}$ of complex dimension $2g-3+n$ if $ \rho\in {\rm Im}({H}^\alpha_{g,n})$ is not trivial.  
For such a unitary character $\rho$ and under the assumption that none of the $\alpha_i$'s is an integer, Veech introduces a certain space of 1-cocycles $\mathscr H^1_\rho$. Then considering not only the linear part but the whole Euclidean holonomies of the elements of $\mathcal F^{\alpha}_\rho$ viewed as classes of flat structures, he 
constructs a  `complete holonomy map'
\begin{equation*}
\label{E:VeechFullHolonomy}
 {\rm Hol}^\alpha_\rho: \mathcal F^{\alpha}_\rho \longrightarrow \mathbb P \mathscr H^1_\rho\simeq \mathbb P^{2g-3+n}
 \end{equation*}  
and proves first that 
 this map is a local biholomorphism, then that   there is a hermitian form $H_\rho^\alpha$ on $\mathscr H^1_\rho$ and $ {\rm Hol}^\alpha_\rho$ maps ${\mathcal F}^\alpha_\rho$  into the projectivization $X_\rho^\alpha
\subset  \mathbb P^{2g-3+n}$ of the set $\{   H_\rho^\alpha<0 \}\subset  \mathscr H^1_\rho$ 
 (compare with 
 \S\ref{S:MultidimContext}).

By a long calculation,   Veech determines explicitly the signature $(p,q)$ of $H_{\rho}^\alpha$ and shows that it does depend only on $\alpha$.   The most interesting case is when $(p,q)=(1,2g-3+n)$.  Then $ {\rm Hol}^\alpha_\rho$ takes its values into $X_\rho^\alpha\simeq \mathbb C\mathbb H^{2g-3+n}$.   By pull-back by 
   ${\rm Hol}^\alpha_\rho$, one endows $\mathcal F^{\alpha}_\rho$ with a natural complex hyperbolic structure.  
   
   One occurrence  of this situation is when $g=0$ and all the $\alpha_i$'s are in $]-1,0[$: 
 in this case there is only one leaf which is the whole Teichmller space $\mathcal T\!\!\!\mbox{\it eich}_{0,n}$ and one recovers precisely the case studied by Deligne-Mostow and Thurston. 
 
   \subsubsection{} 
 \label{SS:InAddition}
      In addition to the genus 0 case, Veech shows that the complex hyperbolic situation also occurs in one other case, namely when 
    \begin{equation}
\label{E:g=1ComplexHyperbolicCase}
     g=1 \quad \mbox{\it and } 
     \quad \mbox{\it all the }   \alpha_i\mbox{\it 's are  in }\, ]-1,0[ \;  \mbox{\it  except  one    which  lies  in } ]0,1[.  
    \end{equation}

    In this case,  the level-sets $\mathcal F_\rho^\alpha$'s of the 
    holonomy map   ${H}^\alpha_{g,n}$ 
    form a real-analytic foliation $\mathcal F^\alpha$ of $\mathcal T\!\!\!{\it eich}_{1,n}$ 
   whose leaves carry  natural $\mathbb C\mathbb H^{n-1}$-structures. \mk 
    
    A remarkable fact established by Veech is that the pure mapping class group $\mathrm{PMCG}_{1,n}$  leaves this foliation invariant and induces biholomorphisms between the leaves which preserve their respective complex hyperbolic structure.   Consequently, all the previous constructions  pass to the 
     quotient by $\mathrm{PMCG}_{1,n}$. 
   One finally obtains  a  foliation on the quotient moduli space ${\mathscr M}_{1,n}$, denoted by $\mathscr F^\alpha$, by complex leaves carrying a (possibly orbifold) complex hyperbolic structure.  Furthermore, 
      it comes from  \eqref{E:holAlpha}
  that the foliation $\mathscr F^\alpha$ is transversally symplectic,  hence 
   one can endow ${\mathscr M}_{1,n}$ with a natural real-analytic volume form $\Omega^\alpha$. \sk 
   
  At this point, interesting questions emerge very naturally: 
  \begin{enumerate}
  \item   which are the leaves of $\mathscr F^\alpha$ that are algebraic submanifolds of ${\mathscr M}_{1,n}$?\sk  
    \item   given a leaf of $\mathscr F^\alpha$ which is an algebraic submanifold of ${\mathscr M}_{1,n}$, what is its topology? Considered with its $\mathbb C\mathbb H^{n-1}$-structure, has it finite volume?\sk
  \item  does the $\mathbb C\mathbb H^{n-1}$-structure of an algebraic  leaf extend to its metric completion (possibly as a conifold complex hyperbolic structure)?\sk
    \item  which are the algebraic leaves of $\mathscr F^\alpha$ whose holonomy representation of their $\mathbb C\mathbb H^{n-1}$-structure has a discrete image in ${\mathrm{ PU}}(1,n-1)$?\sk 
    \item is  it possible to construct new non-arithmetic complex hyperbolic lattices this way?
  \sk
       \item  is the $\Omega^\alpha$-volume of ${\mathscr M}_{1,n}$ finite
       as conjectured by Veech in \cite{Veech}?\sk 
    \end{enumerate}
  
  
  In view of what has been done in the genus 0 case, one can 
  distinguish two distinct ways to address  such questions.  The first,  la Thurston, by using geometric arguments relying on surgeries on flat surfaces. The second,  la Deligne-Mostow, through a more analytical and cohomological reasoning.
  
  Our work shows that both approaches are possible, relevant and fruitful.  
  In the twin paper \cite{GP},  we generalize Thurston's approach whereas in the present text, we generalize the one of Deligne and Mostow to the genus 1 case. 
  



\subsection{\bf Results}
\label{S:Results}
We give below a short review of the results contained in the papers.  
All of them are new, even if some of them (namely the first ones) are obtained by rather elementary considerations. 
We present them below in decreasing order of generality, which essentially corresponds to their order of appearance in the text. \sk 

Throughout the text,  $g$ and $n$  will always refer respectively to the genus of the considered surfaces and to the number of 
cone points they carry   and 
it will  always be assumed that $2g-3+n>0$.

\subsubsection{}\hspace{-0.2cm}
\label{S:GeneralRemark}
Our first results just consist in a general remark followed by a natural construction concerning Veech's constructions, whichever 
 $g$ and $n$  are.
  \sk

 Let $N$ be a flat surface with conical singularities whose isotopy class belongs to a moduli space $\mathcal E_{g,n}^\alpha\simeq \mathcal T\!\!\!{\it eich}_{g,n}$ for some $n$-uplet 
  $\alpha$ as in \S\ref{S:VeechIntro}. 
 Since the target space of the associated linear holonomy character $\rho: \pi_1(N)\rightarrow \mathbb U$  
 is abelian, the latter factors through the abelianization of $\pi_1(N)$, namely the first homology group $H_1(N,\mathbb Z)$. From this simple remark, one deduces that the linear holonomy map 
 \eqref{E:holAlpha} actually factors through the quotient map from $\mathcal T\!\!\!{\it eich}_{g,n} $ onto 
 the associated Torelli space 
  and consequently, Veech's foliation $\mathcal F^{\alpha}$ on $\mathcal T\!\!\!{\it or}_{g,n}$ admits a global first integral  $\mathcal T\!\!\!{\it or}_{g,n}\rightarrow \mathbb U^{2g}$, which will be denoted by ${h}^{\alpha}_{g,n}$. \mk
 
 Let ${e}: \mathbb R^{2g}\rightarrow \mathbb U^{2g}$ be the group morphism 
 whose  components all are the map $s\mapsto \exp(2i\pi s)$. 
 Our second point is that, using classical geometric facts about simple closed curves on surfaces,  one can construct a lift  $\widetilde{H}^\alpha_{g,n} : \mathcal T\!\!\!{\it eich}_{g,n} \rightarrow \mathbb R^{2g}$ of 
Veech's first integral \eqref{E:holAlpha}.  These results can be summarized in the following 
  \begin{prop}
  \label{P:SquareDiag}
There are canonical real-analytic maps $\widetilde{H}^\alpha_{g,n}$ 
 and  ${h}^\alpha_{g,n}$ (in blue below) 
making  the following diagram commutative: 
\begin{equation*}
\label{D:Square}
 \xymatrix@R=1.5cm@C=2.3cm{  
    \mathcal T\!\!\!{\it eich}_{g,n}     \ar@{->}@*{[blue]}[r]^{
    \quad \;
    \textcolor{blue}{\widetilde{H}^\alpha_{g,n}}
      }    
      \ar@{->}[rd]
      ^*[@]{\hbox to 0pt{\hss $\quad \smxylabel{{H}^\alpha_{g,n}}$    \hss}}
    \ar@{->}[d]
    &  \;  \mathbb R^{2g}  
    \hspace{-0.5cm} {}^{}
     \ar@{->}[d]^{e}  \\
  \mathcal   T\!\!\!{\it or}_{g,n}   
    \ar@{->}@*{[blue]}[r]_{
    \quad \,
    \textcolor{blue}{{h}^\alpha_{g,n}}
      }        & 
\; \mathbb U^{2g}.\hspace{-0.5cm} {}^{}
    }
    \end{equation*}
\end{prop} 
 
 This result shows that it is more natural to study Veech's foliation on the Torelli space $\mathcal T\!\!\!{\it or}_{g,n}$. Note that the latter is a nice complex variety without orbifold point.  
 Furthermore, the existence of the lift $\widetilde{H}^\alpha_{g,n}$ strongly suggests that the 
level-subsets of  Veech's first integral ${H}_{g,n}^\alpha $ 
  are not connected a priori. 
  \mk 
 
 
 
\subsubsection{}\hspace{-0.2cm}  
\label{SS:SpecializationTog=1}
We now consider only  the case of elliptic curves and specialize everything to the case when $g=1$. 
\mk 

First of  all, by simple geometric considerations specific to this case, one verifies that  the lifted holonomy $\widetilde{H}^\alpha_{1,n}$ descends to the corresponding Torelli space. In other terms:  there exists a real-analytic map $\widetilde{h}_{1,n}^\alpha : \mathcal T\!\!\!{\it or}_{1,n}\rightarrow \mathbb R^{2}$ which fits into the  diagram 
above and makes it commutative.\sk 

From now on, we do no longer make abstract considerations but undertake the opposite approach by  expliciting   everything as much as possible.
\begin{center}
$\Diamond$
\end{center}
\sk 

In the genus 0 case, the link between the `flat surfaces' approach  la Thurston and  the `hypergeometric' one  la Deligne-Mostow comes from the fact that there is an explicit formula for a flat metric with conical singularities on the Riemann sphere (see \S\ref{S:Bridge-Thurston-DeligneMostow} above).  The crucial point of this paper is that something equivalent can be done 
in the $g=1$ case. \sk 

Assume that $\alpha_1,\ldots,\alpha_n$ are fixed real numbers bigger than $-1$ such that $\sum_i \alpha_i=0$.   For $\tau\in \mathbb H$, let  $E_\tau=\mathbb C/(\mathbb Z\oplus \mathbb Z \tau)$ be the associated elliptic curve.  
Then, given $z=(z_i)_{i=1}^n\in \mathbb C^n$    such that  
$[z_1], \ldots,[z_n]$ are $n$ distinct points on $E_\tau$, Troyanov's theorem 
({\it cf.}\;\S\ref{S:VeechIntro}) ensures that, up to normalization, there exists a unique flat metric $m_{\tau,z}^\alpha$ on $E_\tau$ with a singularity of  type $\lvert u^{\alpha_i}du  \lvert^2$ at $[z_i]$ for $i=1,\ldots,n$.  \bk

One can give an explicit formula for this metric by means of theta functions:

\newpage

\begin{prop} 
\label{P:ExplicitFlatMetric} 
Up to normalization, one has 
$$m_{\tau,z}^\alpha=\big\lvert  T_{\tau,z}^\alpha(u)du  \big\lvert^2$$ 
where 
$T_{\tau,z}^\alpha$ is  the following multivalued holomorphic function on $E_\tau$: 
\begin{equation}
\label{E:FonctionT}
T_{\tau,z}^\alpha(u)=  \exp\big(  2i\pi a_0 u    \big)  \prod_{i=1}^n \theta\big(u-z_i, \tau\big)^{\alpha_i}
\end{equation}
where $\theta$ stands for Jacobi's theta function \eqref{E:ThetaFunction} and $a_0$ is given by 
$$
a_0=a_0(\tau,z)= -\frac{\Im{\rm m}\big(\sum_{i=1}^n \alpha_i z_i\big)}{\Im{\rm m}(\tau)}\, .$$
\end{prop}

While the preceding formula is easy to establish\footnote{It essentially amounts to verify that the monodromy 
of $T_{\tau,z}^\alpha$  on the $n$-punctured elliptic curve $E_{\tau,z}=E_\tau\setminus \{ [z_1],\ldots,[z_n] \}$ is multiplicative and unitary.},  it is the key result on which  the rest of the paper relies. 
Indeed, the `explicitness' of the above formulae for $T_{\tau,z}^\alpha$ and $a_0$  will propagate and this will allow us to make  all Veech's 
 constructions completely explicit in the case of elliptic curves. 

\subsubsection{}\hspace{-0.2cm} Another key ingredient is that there exists a nice and explicit description of the Torelli spaces of marked elliptic curves: this result, due to Nag \cite{Nag}, can be summarized by saying that the parameters $\tau\in \mathbb H$ and $z\in \mathbb C^n$ as above
provide global holomorphic coordinates on $ \mathcal T\!\!\!{\it or}_{1,n}$ for any $n\geq 1$,  if we assume the normalization  $z_1=0$.
 
 Using the coordinates $(\tau,z)$ on $\mathcal T\!\!\!{\it or}_{1,n}$, it is then easy to prove the 
\begin{prop}
\label{P:Main}
\label{P:popo}
 For $(\tau,z)\in \mathcal T\!\!\!{\it or}_{1,n}$, one sets 
 $$a_\infty(\tau,z)=a_0(\tau,z)\tau+\sum_{i=1}^n \alpha_i z_i\in \mathbb R.$$  
 \begin{enumerate}
\item The following map 
 is a primitive first integral of Veech's foliation: 
\begin{align*}
\xi^\alpha  :\,   \mathcal T\!\!\!{\it or}_{1,n}  & \longrightarrow \mathbb R^2 
\\ (\tau,z) & \longmapsto \big(a_0(\tau,z),a_\infty(\tau,z)\big)\, . 
 \end{align*}
 \item  One has ${\rm Im}(\xi^\alpha)=\mathbb R^2$ if $n\geq 3$ and ${\rm Im}(\xi^\alpha)=\mathbb R^2\setminus \alpha_1\mathbb Z^2$ if $n=2$. \sk 
 \item For $a=(a_0,a_\infty)\in {\rm Im}(\xi^\alpha)$, the leaf $\mathcal F_a^{\alpha}=(\xi^\alpha)^{-1}(a)$ in $\mathcal T\!\!\!{\it or}_{1,n}$ is cut out by 
 \begin{equation}
 \label{E:EquationLinearieFeuille}
 a_0\tau+\sum_{i=1}^n\alpha_i z_i= a_\infty\,  .
 \end{equation}
 \item Veech's foliation $\mathcal F^\alpha$ on $\mathcal T\!\!\!{\it or}_{1,n}$ only depends on $[\alpha]\in \mathbb P(\mathbb R^n)$.
 \end{enumerate}
 \end{prop}
 
Point (1) above shows that each level-set 
$\mathcal F_\rho^\alpha=({h}^{\alpha}_{1,n})^{-1}(\rho)$ 
of the linear holonomy map ${h}^\alpha_{1,n}:  
 \mathcal T\!\!\!{\it or}_{1,n}   \rightarrow \mathbb U^2$ is a countable disjoint union of leaves 
 $\mathcal F_a^{\alpha}$'s.
%
 %
   Point (2) answers a question of \cite{Veech}. Finally (3) makes 
   the general and abstract result of Veech mentioned in \S\ref{S:VeechIntro} completely explicit in the $g = 1$ case.
%

\subsubsection{}\hspace{-0.2cm} The pure mapping class group ${\rm PMCG}_{1,n}$ does not act effectively on the Torelli space. Indeed, $\mathcal T\!\!\!{\it or}_{1,n}$ can be seen abstractly as the quotient of  $\mathcal T\!\!\!{\it eich}_{1,n}$ 
by the normal subgroup of the pure mapping class group  formed by mapping classes which act trivially on the homology of the model 
$n$-punctured 2-torus. The latter is called the {\bf Torelli group} and is 
denoted by  ${\rm Tor}_{1,n}$.  

Another key ingredient for what comes next is the fact that 
$$
{\rm Sp}_{1,n}(\mathbb Z):={{\rm PMCG}_{1,n}}_{\big/{\rm Tor}_{1,n}}
$$
as well as its action in the coordinates $(\tau,z)$ on $\mathcal T\!\!\!{\it or}_{1,n}$ can be made explicit. 

For instance, there is an isomorphism  
$${\rm Sp}_{1,n}(\mathbb Z)\simeq {\rm SL}_2(\mathbb Z)\ltimes \big(\mathbb Z^2\big)^{n-1}$$
 with the ${\rm SL}_2(\mathbb Z)$-part acting as usual on the 
variable $\tau\in \mathbb H$.

It is then straightforward to determine, first which are the lifted holonomies $a\in {\rm Im}(\xi^{\alpha})$ whose orbits under ${\rm Sp}_{1,n}(\mathbb Z)$ are discrete; then, for such a holonomy $a$,  what is the image $\mathscr F_a^\alpha=\pi(\mathcal F_a^\alpha)\subset {\mathscr M}_{1,n}$
of the leaf $\mathcal F_a^\alpha$ 
by the quotient map 
$$
\pi : \mathcal T\!\!\!{\it or}_{1,n}
\longrightarrow {\mathscr M}_{1,n}= \mathcal T\!\!\!{\it or}_{1,n}/{\rm Sp}_{1,n}(\mathbb Z).
$$

A $m$-uplet $\nu=( \nu_i)_{i=1}^n\in \mathbb R^m$ is said to be {\bf commensurable} if there exists a 
 constant  $\lambda\neq 0$ such that $\lambda\nu=(\lambda \nu_i)_{i=1}^n$ is rational, {\it i.e.}\;belongs to $\mathbb Q^n$.  
\begin{thm} 
\label{T:Main}
 \begin{enumerate}
 \item Veech's foliation $\mathscr F^\alpha$ on ${\mathscr M}_{1,n}$ admits algebraic leaves 
if and only if $\alpha$ is commensurable. 
\sk
\item  
 The leaf $\, \mathscr F_a^\alpha$ is an algebraic subvariety of $\,{\mathscr M}_{1,n}$ if and only if the $(n+2)$-uplet of real numbers $(\alpha,a)$ is commensurable. 
\end{enumerate}
\end{thm}

Actually, under the assumption that $\alpha$ is commensurable, one can give an explicit description of the algebraic leaves of $\mathscr F^\alpha$.    The case when $n=2$ is particular and will be treated very carefully in \S\ref{S:Resultats-n=2} below.  But 
the consideration of the case when $n=3$ already suggests 
what happens more generally  and we will give a general description of an algebraic leaf  $\mathscr F_a^\alpha\subset {\mathscr M}_{1,3}$. \sk

First, we remark that  the ${\rm SL}_2(\mathbb Z)$-part of ${\rm Sp}_{1,3}(\mathbb Z)$ acts in a natural way on $a\in {\rm Im}(\xi^\alpha)$  and that, if  the corresponding leaf $\mathscr F_a^\alpha$ 
is algebraic, the latter has a discrete orbit and the image ${S}_a$ of its stabilizer in $ {\rm SL}_2(\mathbb Z)$ is `big' ({\it i.e.}\;of finite index).  Second we recall that the linear projection $\mathcal T\!\!\!{\it or}_{1,n}\rightarrow \mathbb H, \, (\tau,z)\mapsto \tau$ passes to the quotient and induces the map 
${\mathscr M}_{1,n}\rightarrow {\mathscr M}_{1,1}$,  which corresponds to forgetting the last $n-1$ points of a $n$-marked elliptic curve.
\begin{thm} 
\label{T:Main2}
Assume that $\mathscr F_a^\alpha \subset {\mathscr M}_{1,3}$ is an algebraic leaf of\;\,$\mathscr F^\alpha$.  
\begin{enumerate}
\item There exists an integer $N_a$ such that \,${S}_a\simeq \Gamma_1(N_a)$; 
\sk 
\item The leaf\;$\mathscr F_a^\alpha$ is isomorphic to the total space of the elliptic modular surface $\mathscr E_{1}(N_a)\rightarrow Y_1(N_a)$
 from which the union of a finite number of  torsion
 multi-sections has been removed.
 \end{enumerate}
\end{thm}
Actually,  the way in which  this result is stated here is not completely correct
 since a finite number of pathological cases do occur. Anyway, it  applies to  most of the algebraic leaves and can actually be made more precise and explicit: for instance, there is exactly one algebraic leaf for each integer $N>0$, one can give  $N_a$ in terms of $a$ and it is possible to list which are the torsion multisections to be removed from $\mathscr E_{1}(N_a)$ in order to get the leaf $\mathscr F_a^\alpha$.

 
\subsubsection{}\hspace{-0.2cm} 
\label{S:VeechMapExplicit}
From \eqref{E:EquationLinearieFeuille}, it comes that $(\tau,z')=(\tau,z_3,\ldots,z_n)$ forms a system of global coordinates on any leaf $\mathcal F_a^\alpha$.  
For  $(\tau,z')\in \mathcal F_a^\alpha$, we denote by  $(\tau,z)$ the element of  $\mathcal T\!\!\!{\it or}_{1,n}$ where $z_2$ is obtained from  $(\tau,z')$ by solving the affine equation \eqref{E:EquationLinearieFeuille}. \sk 

 Our next result is about an explicit expression, in these coordinates, of the restriction  to  $\mathcal F_a^\alpha$, denoted by $V_a^\alpha$,   of 
 Veech's full holonomy map ${\rm Hol}^\alpha_a$ of \S\ref{S:VeechIntro}\footnote{As a global holomorphic map,  Veech's map is only defined on the corresponding leaf
 in ${\mathcal T}\!\!\!{\it eich}_{1,n}$.  
  On  $\mathcal F_a^\alpha\subset {\mathcal T}\!\!\!{\it or}_{1,n}$, it has to be considered as a global multivalued holomorphic function, except if  this leaf has no topology (as when $n=2$,  a case such that  
$\mathcal F_a^\alpha\simeq \mathbb H$ for any $a$).}. From Proposition \ref{P:ExplicitFlatMetric}, it comes immediately that for any 
$(\tau,z)\in \mathcal F_a^\alpha$ fixed, 
$$    
z\mapsto \int^z T_{\tau,z}^\alpha(u)du
$$ 
is `the' developing map of the corresponding flat structure on the punctured elliptic curve
$E_{\tau,z}=E_\tau\setminus\{ [z_1],\ldots,[z_n] \}$. Consequently, there is a local analytic expression for $V_a^\alpha$ whose  components  are obtained  by integrating a fixed determination of the multivalued 1-form 
 $T_{\tau,z}^\alpha(u)du$ along certain 1-cycles in $E_{\tau,z}$. 

Using some results of Mano and Watanabe \cite{ManoWatanabe}, one can extend to our situation the analytico-cohomological approach used in the genus 0 case by Deligne and Mostow in \cite{DeligneMostow}. More precisely, for $(\tau,z)\in \mathcal T\!\!\!{\it or}_{1,n}$, let ${L}_{\tau,z}^{\vee}$ be the local system on $E_{\tau,z}$ whose local sections are given by local determinations of $T_{\tau,z}^\alpha$.  
Following 
\cite{ManoWatanabe}, one defines some ${L}_{\tau,z}^{\vee}$-twisted 1-cycles $\boldsymbol{\gamma}_0,
\boldsymbol{\gamma}_2,\ldots,\boldsymbol{\gamma}_n,\boldsymbol{\gamma}_\infty$ 
 by taking regularizations of the relative twisted 1-simplices obtained by considering  certain determinations of $T_{\tau,z}^\alpha$ along the segments $\ell_0,
\ell_2,\ldots,\ell_n,\ell_\infty$ on $E_{\tau,z}$ represented in  Figure \ref{F:VeryNicePosition0} below. 
\begin{center}
\begin{figure}[!h]
\psfrag{B}[][][1]{$\textcolor{red}{B} $}
\psfrag{1}[][][1]{$1 $}
\psfrag{0}[][][1]{$0 $}
\psfrag{4}[][][1]{$z_2 $}
\psfrag{3}[][][1]{$z_3 $}
\psfrag{2}[][][1]{$z_4 $}
\psfrag{11}[][][1]{$z_5 $}
\psfrag{t}[][][1]{$\tau $}
\psfrag{g0}[][][1]{$\ell_{0} $}
\psfrag{l4}[][][1]{$\ell_2\;\; $}
\psfrag{l3}[][][1]{$\ell_{\!3} \, \;\;\; $}
\psfrag{l2}[][][1]{$\; \ell_4 \; $}
\psfrag{l1}[][][1]{$\; \ell_5 \;$}
\psfrag{gi}[][][1]{$\ell_{\!\!\infty} $}
\includegraphics[scale=0.7]{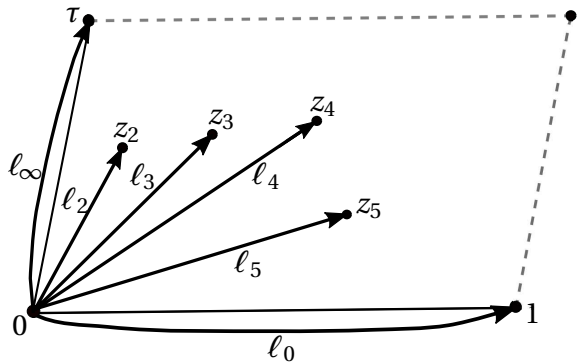}
\caption{For $\bullet=0,2,\ldots,n,\infty$,  $\ell_\bullet$ is the image in the $n$-punctured elliptic curve $E_{\tau,z}$ of the segment $]0,z_\bullet[$ (with $z_0=1$, $z_\infty=\tau$ 
and assuming the normalization $z_1=0$).}
\label{F:VeryNicePosition0}
\end{figure}\vspace{-0.4cm}
\end{center}

By using that the $\boldsymbol{\gamma}_\bullet$'s for $\bullet=0,2,\ldots,n-1,\infty$ induce 
a basis of $H_1(E_{\tau,z}, {L}_{\tau,z}^{\vee})$ ({\it cf.}\;\cite{ManoWatanabe}) and they can be locally continuously extended on the Torelli space,   
  one obtains the following result:
\begin{prop} 
\label{P:VeechExplicit}
\begin{enumerate}
\item The Veech map of $ \mathcal F_{\!\!a}^\alpha$ has a local analytic expression 
	$$V_{\!a}^\alpha: \big(\tau,z'\big)\mapsto \big[F_0(\tau,z):F_3(\tau,z):\cdots: F_n(\tau,z):F_\infty(\tau,z)\big]$$ where for $\bullet=0,3,\ldots,n,\infty$, the  component $F_\bullet$  is  the  (locally defined) {\bf elliptic hypergeomeric integral} depending on  $(\tau,z)\in \mathcal T\!\!\!{\it or}_{1,n}$ defined as 
	$$
F_\bullet: (\tau,z)\longmapsto 	\int_{\boldsymbol{\gamma}_\bullet} T_{\tau,z}^\alpha(u)du\;. 
	$$
\item The matrix of Veech's hermitian form $H_a^\alpha$ on $\mathbb C^{n+1}$ ({\it cf.}\;\S\ref{S:VeechIntro}) in the coordinates associated to  the components $F_\bullet$ of $V_{\!a}^\alpha$ considered in {\rm (1)} can be obtained from  
the twisted intersection products $\boldsymbol{\gamma}_\bullet \,\boldsymbol{\cdot} \, 
\boldsymbol{\gamma}_\circ^{\vee}$ for $\bullet$ and $\circ$ ranging in $\{0,3,\ldots,n,\infty\}$, all of which   can be explicitly computed (see \S\ref{S:IntersectionProduct}). \sk 
\end{enumerate}
\end{prop}


\subsubsection{}\hspace{-0.4cm}
\label{S:Resultats-n=2}
 We now turn to the case of elliptic curves with two  conical points.  
In this case $\alpha=(\alpha_1,\alpha_2)$ is such that $\alpha_2=-\alpha_1$, so one can take $\alpha_1\in ]0,1[$ as the main parameter and consequently replace $\alpha$ by $\alpha_1$ in all the notations.   For instance, Veech's foliations on $\mathcal T\!\!\!{\it or}_{1,2}$ and $\mathscr M_{1,2}$ will be denoted respectively by $\mathcal F^{\alpha_1}$ and $\mathscr F^{\alpha_1}$ from now on.
 It follows from the fourth point of 
Proposition \ref{P:popo}
 that theses foliations do not depend on $\alpha_1$.   In this case, the leaves of $\mathscr F^{\alpha_1}$, and in particular the algebraic ones, can be described very precisely.  
\mk

 It is enlightening  to make more explicit  Proposition \ref{P:Main} in the case under scrutiny. 
In this case,  the rescaled first integral $\Xi=(\alpha_1)^{-1}\xi^{\alpha_1}: 
\mathcal T\!\!\!{\it or}_{1,2} \longrightarrow \mathbb R^2
$ is independent of $\alpha_1$ and ${\rm Im}(\Xi)=\mathbb R^2\setminus \mathbb Z^2$ as image. 
Denoting by 
   $\Pi$ the restriction to $\mathcal T\!\!\!{\it or}_{1,2}$ of the linear projection $\mathbb H\times \mathbb C\rightarrow \mathbb H$ onto the first factor, one has the 
\begin{prop}
\label{P:Main-n=2}
%
\begin{enumerate}
\item 
The following map is a (real analytic) isomorphism
\begin{equation}
\label{E:Pi-Xi}
\Pi \times \Xi \; :\;  \mathcal T\!\!\!{\it or}_{1,2} \stackrel{ \sim }{  \longrightarrow } \mathbb H\times \big( \mathbb R^2\setminus \mathbb Z^2\big)\, .
\end{equation}
\item
  The push-forward of Veech's foliation $\mathcal F^{\alpha_1}$ on 
 $\mathcal T\!\!\!{\it or}_{1,2}$  
  by  this map is  the horizontal foliation  on  the product\;$\mathbb H\times ( \mathbb R^2\setminus \mathbb Z^2) $. 
  \item 
  \sk 
  By restriction, $\Pi$ induces a biholomorphism between any leaf of $\mathcal F^{\alpha_1}$ and Poincar half-plane
   $\mathbb H$.  In particular, the leaves of\;Veech's foliation 
  on the Torelli space $\mathcal T\!\!\!{\it or}_{1,2}$ 
 are  topologically trivial.
\end{enumerate}
\end{prop}

Using this result, the description of the leaves of Veech's foliation $\mathscr F^{\alpha_1}$ on 
the moduli space $\mathscr M_{1,2}$ follows easily. 
For any leaf $\mathcal F^{\alpha_1}_a$ of $\mathcal F^{\alpha_1}$,   we denote by $\pi_a: \mathcal F_a^{\alpha_1}\rightarrow \mathscr F_a^{\alpha_1}$ the 
restriction to $\mathcal F_a^{\alpha_1}$ of the quotient map $\pi: \mathcal T\!\!\!{or}_{1,2}\rightarrow {\mathscr M}_{1,2}$. 
\begin{thm} 
\label{T:Main-n=2-Leaves}
For any leaf  $\;\mathscr F_a^{\alpha_1}$ in ${\mathscr M}_{1,2}$, 
 one of the following situations occurs: 
 \begin{enumerate}
 \item either 
  the quotient mapping $\pi_a$ 
  is trivial,  hence $\mathscr F_a^{\alpha_1}\simeq \mathbb H$; or \sk 
\item  the quotient mapping $\pi_a$ is isomorphic to that of  $\; \mathbb H$ by $\tau \mapsto \tau+1$, hence $\mathscr F_a^{\alpha_1}$ is conformally isomorphic to an infinite cylinder; or \sk
\item the leaf 
$\; \mathscr F_a^{\alpha_1}$ is algebraic. If $N$ stands for the smallest positive integer such that ${N} a\in \alpha_1 \mathbb Z^2$, 
 then $N\geq 2$ and $\, \mathscr F_a^{\alpha_1}$ coincides with  the image of 
\begin{align}
\label{E:AlgebraicLeaf-n=2}
\mathbb H_{\big/ \Gamma_1(N)} & \longrightarrow {\mathscr M}_{1,2}\\
\left(E_\tau,\left[\frac{1}{N}\right]\right) & \longmapsto 
\left(E_\tau,[0], \left[\frac{1}{N}\right]\right)\, ,  \nonumber
\end{align}
hence is  isomorphic to the modular curve $Y_1(N)=\mathbb H/ \Gamma_1(N)$. 
%
 \end{enumerate}
\end{thm}

We thus have described the conformal types of the leaves of $\mathscr F^{\alpha_1}$ which are independent from $\alpha_1$. We now want to go further and describe Veech's 
complex hyperbolic structures  of the leaves and these  depend on $\alpha_1$. Of course, our main interest will be in  the algebraic leaves of Veech's foliation.


\subsubsection{}\hspace{-0.4cm}
\label{S:n=2ExplicitVeechA}
 Let $a=(a_0,a_\infty) \in{\rm Im}(\xi^{\alpha_1})$ be fixed. The leaf 
 $\mathcal F^{\alpha_1}_a$ 
   in ${\mathcal T}\!\!\!{\it or}_{1,2}$ is  cut out by the following affine equation in the variables $\tau$ and $z_2$ ({\it cf.}\;Proposition \ref{P:popo}.(3)) : 
   $$z_2=t_\tau=\frac{1}{\alpha_1}\big(a_0\tau-a_\infty\big)\, . $$

 For $\tau\in \mathbb H$, let $T_a^{\alpha_1}(\cdot , \tau)$ be the multivalued  holomorphic function defined by 
$$
T_a^{\alpha_1}(u, \tau)=\exp\big(2i\pi a_0 u\big) \frac{\theta(u,\tau)^{\alpha_1}}{\theta\left(u-t_\tau, \tau\right)^{\alpha_1}}\, , 
$$
for $u\in \mathbb C$ distinct from 0 and $t(\tau)$ modulo $\mathbb Z\oplus \tau\mathbb Z$.  \sk 

One considers the following two holomorphic functions of $\tau\in \mathbb H$: 
\begin{equation}
\label{E:F0Finfty}
 F_0(\tau)=  \int_{[0,1]}T_a^{\alpha_1}(u, \tau)du \qquad \mbox{ and }
  \qquad 
  F_\infty(\tau) = \int_{[0,\tau]}T_a^{\alpha_1}(u, \tau)du. 
  \end{equation}

Specializing the results of  \S\ref{S:VeechMapExplicit}, one obtains the 
\begin{prop} 
\label{P:InH}
There exists a fractional transformation $z\mapsto ({A z+B})/({Cz+D})$, and examples of such maps can be given explicitly (see \S\ref{S:NormalizationOfVeech'sMap-n=2}) so that   
\begin{equation}
\label{E:V-alpha1-a}
V^{\alpha_1}_a= \frac{A \cdot F_0+B \cdot F_\infty}{C \cdot F_0+D \cdot F_\infty}
 \, : \, \mathbb H
\longrightarrow \mathbb P^1, 
\end{equation}
  is a model of  the Veech map of  the leaf 
 $\mathcal F^{\alpha_1}_a\simeq \mathbb H$  which 
\begin{enumerate}
\item  
 takes values into the upper half-plane $\mathbb H$; 
\sk
\item is such that Veech's  complex hyperbolic structure of $\mathcal F^{\alpha_1}_a$ is the pull-back  by $V$ of the standard one of $\;\mathbb H$. 
\end{enumerate}
\end{prop}

It follows that the Schwarzian  differential equation characterizing Veech's hyperbolic structure of $\mathcal F^{\alpha_1}_a$ can be obtained from the second-order differential equation $(E_a^{\alpha_1})$ on $\mathbb H$ satisfied by  $F_0$ and $F_\infty$.  The definition \eqref{E:F0Finfty} of these two functions  being explicit, one can compute $(E_a^{\alpha_1})$ explicitly. 
\sk 

In order to do so, we specialize some results of \cite{ManoWatanabe} and determine explicitly a certain Gau{\ss}-Manin connection 
 on $ \mathcal F_a^{\alpha_1}$.   Let $\mathcal E_a\rightarrow 
\mathcal F_a^{\alpha_1}\simeq  \mathbb H$ be the universal 2-punctured curve over $
\mathcal F_a^{\alpha_1}$ whose fiber at $\tau\in  \mathbb H$ is the punctured elliptic curve $E_{\tau,t_\tau}=E_\tau\setminus\{[0],[t_\tau]\}$. 
 There is a line bundle $L_a$ on $\mathcal E_a$  the restriction of which on any 
 fiber $E_{\tau,t_\tau}$ coincides with  the line bundle $L_\tau$ on the latter defined by the multivalued function $T_a^{\alpha_1}(\cdot,\tau)$.  The push-forward of $L_a$ onto $\mathcal F_a^{\alpha_1} $ is a local system of rank 2, denoted by $B_a$, whose  fiber  at $\tau$ is nothing else but the 
 first group of twisted cohomology $H^1(E_{\tau,t_\tau},L_\tau)$ considered  above in 
 \S\ref{S:VeechMapExplicit}.   \sk

 One sets $\mathcal B_a=B_a\otimes \mathcal O_{\mathbb H}$.
 We are interested in the Gau{\ss}-Manin connection    $ \nabla_a^{GM} :
 \mathcal B_a \rightarrow    \mathcal B_a\otimes \Omega^1_{\mathbb H}$
 whose flat sections are the sections of $B_a$.  
 Following  \cite{ManoWatanabe}, one defines two trivializing explicit global sections $[\varphi_0]$
 and $[\varphi_1]$ of $\mathcal B_a$.  
 
  \begin{prop} 
 \label{P:InH}
 \begin{enumerate}
  \item  
  In the basis $([\varphi_0], [\varphi_1])$, the action 
of    $\; \nabla_a^{GM} $ is  written 
 $$
 \nabla_{\!\!a}^{GM}\begin{pmatrix}
 [\varphi_0] \\
 [\varphi_1]
 \end{pmatrix} = 
 M_a\cdot \begin{pmatrix}
 [\varphi_0] \\ 
 [\varphi_1]
 \end{pmatrix}
 $$ 
 for a certain explicit matrix $M_a$ of holomorphic 1-forms on $\mathbb H$. \sk 
 \item 
 The differential equation on $\mathbb H$ with  $F_0, F_\infty$ as a basis of  solutions is written
\begin{equation*}
\label{E:???}
\big(E_a^{\alpha_1}\big) \qquad \qquad \qquad \qquad
\stackrel{\bullet\bullet}{F} -\big(2i\pi{a_0^2}/{\alpha_1}\big) \cdot    \stackrel{\bullet}{F}+
\varphi_a 
 \cdot F = 0 \, \qquad \qquad \qquad\qquad 
\end{equation*}
for an explicit global holomorphic function $\varphi_a$  on $\mathbb H$.
 \end{enumerate}
 \end{prop}

 The interest of this result lies in the fact that everything can be explicited. It will be  our main tool to 
study the $\mathbb C\mathbb H^1$-structures of the algebraic leaves of $\mathscr F^{\alpha_1}$ in 
$\mathscr M_{1,2}$.  
 
 
\subsubsection{}\hspace{-0.4cm}
\label{S:n=2ExplicitVeechN}
Let  $N\geq 2$ be fixed. For $(k,l)\in \mathbb Z^2\setminus N\mathbb Z^2$, let $\mathcal F^{\alpha_1}_{k,l}(N)$ be  the leaf of Veech's foliation   on ${\mathcal T}\!\!\!{\it or}_{1,2}$ cut out by $z_2=(k/N)\tau+l/N$.  It is isomorphic to $\mathbb H$ and  
its image  in $\mathscr M_{1,2}$ is precisely the image of the embedding \eqref{E:AlgebraicLeaf-n=2}.   When endowed with Veech's $\mathbb C\mathbb H^1$-structure, we denote the latter by $Y_1(N)^{\alpha_1}$  to emphasize the fact that it is $Y_1(N)$ but with a deformation of its usual hyperbolic structure. \sk

Under the assumption  that $\alpha_1$ is rational, it follows from our main result in \cite{GP}  that for any $N\geq 2$, Veech's hyperbolic structure of $Y_1(N)^{\alpha_1}$ extends as a conifold $\mathbb C\mathbb H^1$-structure to its 
metric completion $\overline{Y_1(N)}^{\alpha_1}$.   The point is that using Proposition \ref{P:InH}, one can recover this result without the  rationality assumption on $\alpha_1$ and precisely  characterize   this conifold 
 structure. \sk 

%

Let $X_1(N) $ be the compactification of $Y_1(N)$ obtained by adding to it its set of cusps $C_1(N)$. For $\mathfrak c\in C_1(N)$, two situations can occur: either $Y_1(N)^{\alpha_1}$ is metrically complete in the vicinity of $\mathfrak c$, either it is not. In the first case, $\mathfrak c$ is a cusp in the classical sense
\footnote{{\it I.e.}\;$(Y_1(N)^{\alpha_1},\mathfrak c)\simeq  (\mathbb H/(z\mapsto z+1), [i\infty])$ 
as germs of punctured hyperbolic surfaces.} and 
 will be called a  {\bf conifold point of angle 0}.
\sk 

To study the geometric structure of $Y_1(N)^{\alpha_1}$ near a  cusp $\mathfrak c\in C_1(N)$, our approach consists in looking at the Schwarzian differential equation associated to Veech's $\mathbb C\mathbb H^1$-structure  on a small punctured neighborhood  $U_{\mathfrak c}$ of $\mathfrak c$ in $Y_1(N)$. 

First, one verifies that there exist  
$k$ and $l$ such that $
  \mathcal F^{\alpha_1}_{k,l}(N)  
\rightarrow Y_1(N)^{\alpha_1}$ is a uniformization which sends $[i\infty]$ onto $\mathfrak c$.  Then, since the functions $F_0,F_\infty$ defined in \eqref{E:F0Finfty} (with the corresponding $a$, namely $a=\alpha_1( k/N, -l/N)$) are  components of the Veech map on $\mathcal F^{\alpha_1}_{k,l}(N) $, 
they can be viewed as the components of the developing map of Veech's hyperbolic structure of $Y_1(N)^{\alpha_1}$. 
So, looking at the asymptotic behavior of $(E_a^\alpha)$ when $\tau$ tends  to $ i\infty$ while belonging to a vertical strip of width equal to that of $\mathfrak c$, 
one obtains that the Schwarzian differential equation of the  $\mathbb C\mathbb H^1$-curve $Y_1(N)^{\alpha_1}$ is Fuchsian at $\mathfrak c$ and one can compute explicitly the  two characteristic exponents at this point. 


Then, since a cusp $\mathfrak c\in C_1(N)$ is a class modulo $\Gamma_1(N)$ of a rational element of the  
 boundary $\mathbb P_{\mathbb R}^1\simeq S^1$ of the 
 closure of $\mathbb H$ in $\mathbb P^1$ hence as such,  is written   $\mathfrak c=[ a/c]$ with $a/c\in \mathbb P^1_{\mathbb Q}$, one eventually gets the following result: 
\begin{thm}
\label{T:Main-n=2}
 For any parameter $\alpha_1\in ]0,1[$: 
\begin{enumerate}
\item Veech's  complex hyperbolic structure of $Y_1(N)^{\alpha_1}$ extends as a conifold structure of the same type to the compactification $X_1(N)$. 
  The latter,  when endowed with this 
   conifold structure,  will be denoted by $X_1(N)^{\alpha_1}$; 
\sk 
\item  the conifold angle of $X_1(N)^{\alpha_1}$ at the cusp $\mathfrak c=[ a/c]
 \in C_1(N) 
 $ is 
 $$
\theta_{\mathfrak c}=\theta(c)=2\pi   \frac{c'(N-c')}{N \cdot {\rm gcd}(c',N)}\cdot \alpha_1
$$
where $c'\in \{0,\ldots,N-1\}$ stands for the residue of $c$ modulo $N$.
\end{enumerate}
\end{thm}
According to a classical result going back to Poincar, a $\mathbb C\mathbb H^1$-conifold structure on a compact Riemann surface is completely characterized by its conifold points and the conifold angles at these points.
 Thus the preceding theorem completely characterizes 
$Y_1(N)^{\alpha_1}$ (or rather $X_1(N)^{\alpha_1}$) as 
a  complex 
 hyperbolic conifold. 
 It can be seen as the generalization,  to the genus 1 case, of the result by Schwarz on the hypergeometric equation, dating of 1873,  evoked in \S\ref{S:1.1.1}.\mk

Defining $N^*$ as the least common multiple of the integers $c'(N-c')/\gcd(c',N)$ when $c'$ ranges in $\{1,\ldots,N-1\}$, one deduces  immediately from 
above 
the
\begin{coro}
\label{C:Main-n=2}
A sufficient condition for $X_1(N)^{\alpha_1}$ to be an orbifold is that 
$$ \qquad  
\alpha_1=\frac{N}{\ell N^*} \qquad \mbox{ for some } \, \ell\in \mathbb N_{>0}\, . 
$$

In this case, the image $\boldsymbol{\Gamma}_1(N)^{\alpha_1}$ of the holonomy representation associated to 
Veech's $\mathbb C\mathbb H^1$-structure on $Y_1(N)^{\alpha_1}$ is 
   a non-cocompact lattice in $ 
     {\rm PSL}_2(\mathbb R)$. 
\end{coro}

 The $\boldsymbol{\Gamma}_1(N)^{\alpha_1}$'s with $\alpha_1\in ]0,1[$ form a real-analytic deformation of $\boldsymbol{\Gamma}_1(N)=\boldsymbol{\Gamma}_1(N)^{0}$ in ${\rm PSL}_2(\mathbb R)$. The problem of determining 
 which of its elements are lattices (or arithmetic lattices, etc.) is quite interesting but  does not seem easy.\sk
 
 An interesting case is when $N$ is equal to a prime number $p$.   It is well-known that $X_1(p)$ is a smooth curve of genus $(p-5)(p-7)/24$ with $p-1$ cusps.

\begin{coro} 
\label{C:Main-n=2-II}
\begin{enumerate}
\item For $k=1,\ldots,(p-1)/2$, the conifold angle of $X_1(p)^{\alpha_1}$ 
at  $[-p/k]$   is $2\pi k(1-k/p)\alpha_1$. The $(p-1)/2$ other conical angles  are 0. 
\mk
\item The hyperbolic volume (area) of $Y_1(p)^{\alpha_1}$ is finite and equal to 
\begin{equation}
\label{E:VolumeY1(p)Alpha1}
\qquad 
{\rm Vol}\big(Y_1(p)^{\alpha_1}\big)= \frac{ \pi  }{6} \big(1-\alpha_1 \big)
\big( p^2-1\big)\, . 
\end{equation}
\end{enumerate}
\end{coro}
\sk

\subsubsection{}
\label{SS:Intro-Volume}
 The preceding corollary can be used
 to give a positive answer, in the case under scrutiny, to a conjecture made by Veech about the volume of $\mathscr M_{1,2}$ (see  Section E. in the introduction of \cite{Veech} or \S\ref{SS:InAddition} above). 
 \mk

 For every leaf $\mathcal F_a^{\alpha_1}$, we denote by $g_a^{\alpha_1}$ the Riemannian metric on $\mathbb H$ given by the pull-back  of the 
 standard hyperbolic metric on Poincar's upper half-plane $\mathbb H$ by the map \eqref{E:V-alpha1-a}.\footnote{Note that if the map  \eqref{E:V-alpha1-a} is not canonically defined, it is up to post-composition by an element of ${\rm PSL}_2(\mathbb R)={\rm Isom}^+(\mathbb H)$,  hence $g_a^{\alpha_1}$ is well-defined for any $a$.} 
 The $g_a^{\alpha_1}$'s depend analytically on $a$ hence  can be glued together to give rise to a smooth partial Riemannian  metric $g^{\alpha_1}$ on the product $\mathbb H\times (\mathbb R^2\setminus \mathbb Z^2)$ which is identified with $\mathcal T\!\!\!{or}_{\!1,2}$ by means of the isomorphism \eqref{E:Pi-Xi}. 

 {\vspace{0.2cm}}
  Let  $d\!s^2_{\rm Euc}$ be the euclidean metric on $\mathbb R^2$. Since $\mathbb R^2\setminus \mathbb Z^2$ can be considered as a global transversal to Veech's foliation 
$ \mathcal F^{\alpha_1}$ (again thanks to Proposition \ref{P:Main-n=2}),   the product $g^{\alpha_1}\otimes d\!s^2_{\rm Euc}$   defines a real analytic Riemmannian metric on the whole  Torelli space $ \mathcal T\!\!\!{or}_{1,2}$. The associated volume form will be denoted by $\Omega^{\alpha_1}$ and will be called {\bf Veech's volume form} on 
$ \mathcal T\!\!\!{or}_{\!1,2}$
  (associated to 
   $\alpha_1$). 
 \sk 
 
 As a particular case of a more general statement proved in \cite{Veech}, one gets  that $\Omega^{\alpha_1}$ is ${\rm Sp}_{1,2}(\mathbb Z)$-invariant hence can be pushed-forward 
 as a volume form\footnote{Strictly speaking, it is an {\it `orbifold volume form'} on 
 $\mathscr M_{1,2}$  but since this subtlety is irrelevant in what concerns volume computations, 
   we will not  dwell on this further in  what follows.} 
 to the moduli space $\mathscr M_{1,2}$, again denoted by $\Omega^{\alpha_1}$.  
 Then in the case under scrutiny, {\bf Veech's volume conjecture}  asserts  that the $\Omega^{\alpha_1}$-volume of   $\mathscr M_{1,2}$ is finite.  \mk 
 
 For any $N\geq 2$, denote by $\nu_{N}^{\alpha_1}$ the  measure on the algebraic leaf $Y_1(N)^{\alpha_1}$  of Veech's foliation on $\mathscr M_{1,2}$
  induced by  the associated hyperbolic structure.  Then one defines a measure $\delta_N^{\alpha_1}$ on $\mathscr M_{1,2}$, supported on $Y_1(N)^{\alpha_1}$, by setting  
  $$
 \delta_N^{\alpha_1}(A)=\nu_{N}^{\alpha_1}\big(A\cap Y_1(N)^{\alpha_1}\big)
 $$
  for any measurable subset $A\subset \mathscr M_{1,2}$. 
\sk 

For any $N>0$, let $\delta_N$ be the measure on $\mathbb R^2$ defined as the sum of the dirac masses at the points of the square lattice $({1}/{N})\mathbb Z^2$. As is well known, for any strictly increasing sequence $(N_k)_{ k\in \mathbb N}$, the measures 
$N_k^{-2}\delta_{N_k}$ strongly converge towards the Lebesgue measure on $\mathbb R^2$. 
 From this, one deduces that the measures ${p^{-2}}\delta_{p}^{\alpha_1}$ tend towards 
 the one associated  to $\Omega^{\alpha_1}$ 
on $\mathscr M_{1,2}$ 
when $p$ tends to infinity among primes.   From this, we deduce that Veech's volume conjecture indeed holds true in the case  under scrutiny. Better,  using 
Corollary \ref{C:Main-n=2-II}, we are able to give a simple closed formula  for the corresponding volume: 
 \begin{thm}
 \label{T:VolumeIsFinite}
For any $\alpha_1\in ]0,1[$, Veech's volume  of\;\,$\mathscr M_{1,2}$ is finite and  equal to 
$$
\qquad 
\int_{\mathscr M_{1,2}}
\!\!\!\!
\Omega^{\alpha_1} = 
 \lim_{
\substack{p\rightarrow +\infty\\ 
 \tiny{p \; \mbox{\it prime }}}}\; 
\frac{1}{p^2} {\rm Vol}\big( Y_1(p)^{\alpha_1}\big)
 =
 \frac{ \pi  }{6} \big(
1-\alpha_1 \big)\, .
$$
\end{thm}

\mk


\subsection{\bf Organization of the paper}
${}^{}$
 Since this text is quite long, we think that stating what we will do and where in the paper could be helpful to the reader. 
 We then make a few general comments  which could also be of help.

\subsubsection{}  
 In the {\bf first section} of this paper (namely the present one), we first take some time
 in \S\ref{S:PreviousWorks} 
 to display some elements about the historical and mathematical background regarding  the problem we are interested in. We then present our results in {\bf \S\ref{S:Results}}. Finally in {\bf \S\ref{S=Notes&References}}, we indicate some of our sources and discuss other works to which the present one is related. 
 \sk 


In {\bf Section 2}, we introduce some classical material and  fix some notations. 
\mk

{\bf Section 3} is about one of the main tools we use in this paper, namely twisted (co-)homology on Riemann surfaces.  After sketching a general theory of what we call `{\it generalized hypergeometric integrals}',  we give a detailed treatment of some results  obtained by  Mano and Watanabe in \cite{ManoWatanabe} concerning the case of punctured elliptic curves. 
The single novelty here is the explicit computation of the twisted intersection product in {\bf \S\ref{S:IntersectionProduct}}.  Note also that what is for us the main hero of this text, namely the multivalued function \eqref{E:FonctionT},  is carefully introduced in {\bf \S\ref{SS:OnTori}} where some of its main properties are established.
\mk

We begin with two simple general remarks about some constructions  of \cite{Veech} in the  first subsection of {\bf Section 4}. One of them leads to the conclusion that Veech's foliation is more naturally defined on the corresponding Torelli space.
 The relevance of this point of view becomes clear when we start focusing on the genus 1 case  in  {\bf \S\ref{S:TheCaseOfEllipticCurves}}.  We then use an explicit description of 
$\mathcal T\!\!\!{\it or}_{1,n}$ obtained by Nag,  as well as an explicit formula (in terms of the function \eqref{E:FonctionT}) 
for a flat metric with conical singularity on an elliptic curve,  to make  Veech's foliation $\mathcal F^\alpha$ on the Torelli space completely explicit.  With that at hand, it is not difficult to obtain some of our main results about $\mathcal F^\alpha$,   such as Proposition \ref{P:Main}, Theorem \ref{T:Main} or 
Theorem \ref{T:Main2}. 
 Finally, in {\bf  \S\ref{S:AnalyticExpressionVeechMapg=1}}, we  turn to the study of  the Veech  map which is used to define the geometric structure (a complex hyperbolic structure under suitable assumptions on the considered conical  angles) 
on the leaves of Veech's foliation. We show that locally, Veech's map admits an analytic description  la Deligne-Mostow in terms of elliptic hypergeometric integrals and we relate Veech's form to the twisted intersection product considered in {\S\ref{S:IntersectionProduct}}. \sk 

We specialize  and  make the previous results more precise in {\bf Section 5} where we restrict ourselves to the case of flat elliptic curves with only two conical points.  In this case, 
one proves that the Torelli space is isomorphic to a product and that, up to this isomorphism,  Veech's foliation identifies with the horizontal foliation.   It is then not difficult to describe the possible conformal types of the leaves of Veech's foliation (Theorem \ref{T:Main-n=2-Leaves} above).

 Using some results of Mano and Watanabe \cite{ManoWatanabe} and of Mano \cite{Mano}, we use the explicit differential  system satisfied by the two  elliptic hypergeometric integrals which are the components of Veech's map 
in this case to look at Veech's $\mathbb C\mathbb H^1$-structure of an algebraic leaf 
$Y_1(N)^{\alpha_1}$
of Veech's foliation on the moduli  space $\mathscr M_{1,2}$ 
 in the vicinity of one of its cusps.   From an easy analysis, one deduces Theorem \ref{T:Main-n=2},  Corollary \ref{C:Main-n=2} and Corollary \ref{C:Main-n=2-II} stated 
  above.  \sk

In {\bf Section 6}, we eventually consider some particular questions or problems to which the results previously obtained naturally lead. In {\bf \S\ref{S:ExplicitDegen}}, one uses a result by Mano  \cite{ManoKumamoto} to give an explicit example of an  analytic degeneration  of some elliptic hypergeometric integrals towards usual hypergeometric functions.  For $N$ small (namely $N\leq 5$), the algebraic leaf $Y_1(N)^{\alpha_1}$ is of genus 0 with 3 or 4 punctures,  hence the associated elliptic hypergeometric integrals can be expressed in terms of classical (hypergeometric or Heun's) functions. This is quickly discussed in {\bf \S\ref{S:LinksWithClassicalHypergeometry}}.   The subsection {\bf \S\ref{S:HyperbolicHolonomy}}, which  constitutes  the main part of the sixth section, is devoted to the determination of the hyperbolic holonomy of the algebraic leaves $Y_1(N)$'s. More precisely, after having explained why we consider this problem as  particularly relevant, we use some connection formulae in twisted homology (due to Mano and presented at the very end of \S\ref{S:Twisted(Co)Homology}) to describe a general method to construct an explicit  representation  
$$\pi_1\big(Y_1(N)\big)=\Gamma_1(N)\rightarrow  {\rm PSL}_2(\mathbb R)$$ corresponding to the $\mathbb C\mathbb H^1$-holonomy of $Y_1(N)^{\alpha_1}$.  \sk

Finally in {\bf \S\ref{S:Volumes}}, we first establish formula \eqref{E:VolumeY1(p)Alpha1} then explain how to deduce from it a proof of Theorem \ref{T:VolumeIsFinite}.
\mk

Two appendices conclude this paper. \sk

{\bf Appendix A} introduces  the notion of {\it `complex hyperbolic conifold curve}'.  In the 1-dimensional case, everything is quite elementary.  Some classical links with the theory of Fuchsian second-order differential equations are recalled as well. \sk

The second appendix, {\bf Appendix B}, is considerably longer.  It offers a very detailed treatment of the Gau{\ss}-Manin connection which is relevant to construct the differential system satisfied by the elliptic hypergeometric integrals which are the components of Veech's map (see Proposition \ref{P:VeechExplicit}).  After recalling some general results about the theory in a twisted relative situation of dimension 1, we treat very explicitly the case of 2-punctured elliptic curves over a leaf 
 of Veech's foliation on $\mathcal T\!\!\!{\it or}_{1,2}$ following \cite{ManoWatanabe}. All 
 the results that we present are justified and made explicit. In the end, we use the Gau{\ss}-Manin connection to construct the second-order differential equation $(E_a^{\alpha_1})$ of Proposition \ref{P:InH}. 
 \vspace{-0.1cm}

\subsubsection{} 
%
We think 
that the length of this text and the originality of the results it offers are worth commenting.
\sk

 From our point of view, the two crucial technical results of this text on which all the others rely are,  first,  the 
explicit global expression \eqref{E:FonctionT}  in Proposition \ref{P:ExplicitFlatMetric} and,  secondly,    some explicit formulae,   first for Veech's map 
 by means of elliptic hypergeometric integrals,  then 
for  the differential equation $(E_a^{\alpha_1})$ satisfied by its components $F_0$ and $F_\infty$ when $n=2$ ({\it cf.}\;Proposition \ref{P:InH}). 

If the  first aforementioned result follows easily from a constructive proof of Troyanov's theorem ({\it cf.}\;the beginning of \S\ref{S:VeechIntro}) described by Kokotov in \cite[\S2.1]{Kokotov}, its use to make Veech's constructions of \cite{Veech} explicit in the genus 1 case is completely original although not difficult. 
 Once one has the explicit formula \eqref{E:FonctionT} at hand, it is rather easy to    obtain  the local expression for the Veech map in terms of elliptic hypergeometric integrals. As for the classical (genus 0) hypergeometric integrals, 
the relevant technology to study such integrals is that of twisted (co-)homology. 

In the case of  punctured elliptic curves, this theory has been worked out by Mano and Watanabe in \cite{ManoWatanabe} where they also give some explicit formulae for  the corresponding Gau{\ss}-Manin connection. It follows that, up to a few exceptions,  the material we present in {Section 3} and in {Appendix B} is not new and should be attributed to them.  So it would have been possible to replace these lengthy parts of the present paper by some references to \cite{ManoWatanabe}. \sk
 
 The reason why we have chosen to do otherwise is twofold. 
First, when we began to work on the subject of this paper, we were not very familiar with the modern twisted (co-)homological way to deal with hypergeometric functions.    In order to understand  
 this theory better, we began to write down detailed notes. Because 
these were helpful for our own understanding, we thought that they could be helpful to some readers as well and 
 decided to incorporate them in the text.  

The second reason  which prompted us to proceed that way is that the context in which  the results of \cite{ManoWatanabe} lie, namely the context of isomonodromic deformations of  linear differential systems with regular singular points on elliptic curves, is more general than ours.
More concretely, the authors in \cite{ManoWatanabe}  deal with a parameter $\lambda$ which corresponds to a certain line bundle $\mathcal O_\lambda$  of degree 0 on the considered elliptic curves.  Our case corresponds to the specialization  $\lambda=0$ which corresponds to $\mathcal O_\lambda$ being trivial.  If the case we are interested in is somehow the simplest one of \cite{ManoWatanabe},  some of the results of the latter,  those about the Gau{\ss}-Manin connection in particular, do not 
apply to the case $\lambda=0$ in a straightforward manner. In order to fill some details which were not explicitly mentioned in \cite{ManoWatanabe}, we worked  out this  case carefully and it leads to Appendix B. \sk 



\subsection{\bf Remarks, notes and references}
\label{S=Notes&References}
This text being already very long, we think  it is not
a problem to add a few lines mentioning  other mathematical works to which the present one is or could be linked.\sk


\subsubsection{}\!\!
As is well-known (or at least, as it must be clear after  reading \S\ref{S:PreviousWorks}),  the distinct approaches of Deligne-Mostow \cite{DeligneMostow} on the  one hand and of Thurston \cite{Thurston} on the other hand, lead to the same results.  
As already said before, Thurs\-ton's approach is more geometric than the hypergeometric one and basically relies on certain surgeries\footnote{By {\it `surgery'} we mean an operation which transforms a flat surface into a new one   which is obtained from the former by cutting along piecewise geodesic segments in it or by removing a part 
of it with  a piecewise geodesic boundary and then identifying certain isometric components of the boundary of the flat surface with geodesic boundary obtained after the cutting operation (see \cite[\S6]{GP} for more formal definitions).}   for flat surfaces (actually flat spheres).

In the present text, we extend the hypergeometric approach of Deligne and Mostow in order to handle the elliptic case.  The point is that Thurston's approach,  in terms of flat surfaces, can be generalized to  the genus 1 case as well. 

In the `non-identical twin' paper  \cite{GP}
\footnote{We use  this terminology since,  if \cite{GP} has the same parents 
and is born at the same time as the present text,  
both papers clearly do not share the same DNA hence are dizygotic twins.}, 
we introduce several surgeries for flat surfaces (some of which are natural generalizations of the one implicitly  used  by Thurston) which we use to generalize some statements of \cite{Thurston} to the case of flat tori with conical singularities. 
  \sk

 We believe that  the important fact highlighted by our work is that both Thur\-stonÕs
geometric approach and Deligne-MostowÕs hypergeometric one can be
generalized to the genus 1 case.  At the moment, we have written two separate papers, one for each of these two approaches.   In the genus 0 case, any one of these approaches suffices, but we believe that this is specific to this case. 
Our credo  is  that  the geometric
approach ( la Thurston) as well as the hypergeometric
one ( la Deligne-Mostow) are truly complementary. Each gives a different
light on the objects under study  and combining these two approaches should be
 powerful and even necessary in order  to better understand  the case $g=1$ with $n\geq 3$.   We plan to illustrate this in forthcoming papers. For the time being,
readers are just strongly encouraged to take a look at \cite{GP} and compare its 
methods and results to the ones of the present text.
 \vspace{-0.2cm}
 
  
\subsubsection{}  
The main mathematical objects studied in \cite{DeligneMostow} are the monodromy groups 
attached to the Appell-Lauricella hypergeometric functions which are the ones 
 admitting  an Eulerian integral representation of the following form 
\begin{equation}
\label{E:Appell-Lauricella-F}
F_{\boldsymbol{\gamma}}  (x)=
\int_{\boldsymbol{\gamma}} \prod_{i=1}^n(t-x_i)^{\alpha_i} dt
\end{equation}
with  $x\in \mathbb C^n$ and where  $\boldsymbol{\gamma}$ is a twisted 1-cycle supported  in $\mathbb P^1\setminus \{ x\}$ ({\it cf.}\;\S\ref{S:Appell-LauricellaHypergeometricFunctions}).

In the present text we are interested in the functions which admit an integral representation  of the following form ({\it cf.}\;\S\eqref{S:VeechMapExplicit} for some explanations) 
\begin{equation}
\label{E:Elliptic-Hypergeometric-Function}
F_{\boldsymbol{\gamma}}  (\tau,z)=
\int_{\boldsymbol{\gamma}} \exp\big(2i\pi a_0 u \big)\prod_{i=1}^n\theta(u-z_i, \tau)^{\alpha_i}
\end{equation}
with  $(\tau,z)\in \mathbb H\times \mathbb C^{n}$ and where  $\boldsymbol{\gamma}$ stands for a twisted 1-cycle supported  in the punctured elliptic curve $E_{\tau,z}$ ({\it cf.}\;\S\ref{S:Appell-LauricellaHypergeometricFunctions}). From our point of view, they are the direct generalization to the genus 1 case of the Appell-Lauricella functions \eqref{E:Appell-Lauricella-F}. For this reason, it seemed to us that the name {\bf elliptic hypergeometric function} 
was quite adequate to describe  them.

Here we have to mention that the Appell-Lauricella hypergeometric   functions \eqref{E:Appell-Lauricella-F} admit developpements in series similar to \eqref{HGF} ({\it cf.}\;\cite[(I')]{DeligneMostow} for instance). Taking this as their main feature and motivated by some questions arising in mathematical physics, several people have developed a theory of {\it `elliptic hypergeometric series}' which have been quickly named {\it `elliptic hypergeometric functions}' as well (see {\it e.g.} the survey paper \cite{Spiridonov}). These share  
several other similarities with the classical hypergeometric functions such as, for instance,  integral representations. We do not know if our elliptic hypergeometric functions are related  to the ones considered by these authors, but we doubt it. \sk 

Anyway, since it sounds very adequate and because we like it  too much, we have  decided to use the  expression {\it `elliptic hypergeometric function'}  in our paper as well. Note that this terminology  has already been used once in a context very similar to the one 
we are dealing with in this text, see \cite{Ito}. 

\subsubsection{}
Note also that in the papers \cite{Mano,ManoWatanabe},  which  we use in a crucial way in \S6, the authors  consider functions defined by integral representations of the form 
\begin{equation}
\label{E:Riemann-Wirtinger-Function}
F_{j,\boldsymbol{\gamma}}  (\tau,z,\lambda)=
\int_{\boldsymbol{\gamma}} \exp\big(2i\pi a_0 u \big)\prod_{i=1}^n\theta(u-z_i, \tau)^{\alpha_i} \mathfrak s(u-z_j,\lambda)
\end{equation}
for $(\tau,z)\in \mathcal T\!\!\!{\it or}_{1,n}$, $\lambda\in \mathbb C\setminus \mathbb Z_\tau$ and $j=1,\ldots,n$,  where $ \mathfrak s(\cdot,\lambda)$ stands for 
the  function $$\mathfrak s(u,\lambda)=\frac{\theta'(0)\theta(u-\lambda)}{\theta(u)\theta(\lambda)}\, .$$

 Such functions were previously  baptized  {\it `Riemann-Wirtinger integrals'} by Mano  in \cite{Mano}.  Since $\lambda \mathfrak s(u,\lambda)\rightarrow -1$  when $\lambda$ goes to 0, our elliptic hypergeometric functions \eqref{E:Elliptic-Hypergeometric-Function} can be seen as natural limits of renormalized  Riemann-Wirtinger integrals.  However, if 
 the functions \eqref{E:Riemann-Wirtinger-Function} for $j\in \{1,\ldots,n\}$ fixed can be seen as translation periods of a certain flat structure on $E_{\tau}$ (namely the one 
 defined by the square of the modulus of the integrand  in \eqref{E:Riemann-Wirtinger-Function}), the latter does not have finite volume hence is not of the type which is of interest to us.

 
\subsubsection{} 
One of the origins of the terminology  `{Riemann-Wirtinger integrals}' (see just above) can be found 
in the little known paper of Wirtinger \cite{Wirtinger1902},  dating of 
1902, in which he gives an explicit expression for the uniformization  of the hypergeometric function \eqref{HGF} to the upper-half plane 
$\mathbb H$.    
This paper has been followed  by a whole series of works  by several authors \cite{Wirtinger1903,Berger1906,Cermak,Pick1907,Graf1907,Graf1908a,Graf1908b,Petr,Graf1910} in which they study particular cases of what we call here `elliptic hypergeometric integrals' (see \cite{Kamp} for an exposition of some of the results obtained by these authors). 
\sk 

The `uniformized approach' 
to the study of the hypergeometric functions 
 initiated by Wirtinger does not seem to have generated much interest from 1910 until very recently. 
 Starting from 2007, Watanabe   begins to work on this subject again. In the series of papers \cite{Watanabe2007,WatanabeWirtinger,Watanabe2009,Watanabe2014}, he applies the modern  approach relying on twisted (co-)homology to the Wirtinger integral (see \cite{WatanabeWirtinger}
  or \S\ref{S:} below for details) and recovers several classical results about Gau{{\ss}} hypergeometric function.    
  It seems to  be  some overlap with several results contained in the former  papers just aforementioned (see Remark \ref{Rem:remrem}) but Watanabe was apparently  
 not aware of them since \cite{Wirtinger1902} is the only paper of that time he refers to.


\subsection{\bf Acknowledgments} 
First, we are very grateful to 
Toshiyuki  Mano
and Humihiko Watanabe  for correspondence and 
for their help to understand their work which proved to be crucial to the study undertaken here.  In particular, T. Mano provided  the second author (L. Pirio) 
 with many detailed explanations which were very helpful. 
 Franois Brunault has kindly answered several elementary questions concerning  modular curves and congruence subgroups, we would like to thank him for this. 
We are grateful to Frank Loray for  sharing with us some references and thoughts about the interplay between complex geometry and `classical hypergeometry'. We would like to thank Adrien Boulanger for the interesting discussions we had about the notion of holonomy.  
We are also thankful to Bertrand Deroin for the constant
support and deep interest he has shown in our work since its very beginning.

  Finally, the second author  thanks Brubru for her careful reading and her numerous corrections.

\section{\bf Notations and preliminary material}
We indicate below some notations for  the objects considered in this paper  as well as a few references.    We have chosen to present this material in telegraphic style: we believe that this presentation is the most useful to the reader.

\subsection{\bf Notations for punctured elliptic curves}
\begin{itemize}
\item $\mathbb H$ stands for Poincar's upper half-plane: 
$\mathbb H=\big\{  u\in \mathbb C \; \big\lvert  \; \Im{\rm m}(u)>0\, \big\}\, ;$
\sk
\item $\mathbb D$ denotes  the unit disk in the complex plane: 
$\mathbb D=\big\{  u \in \mathbb C \; \big\lvert  \; \lvert   u \lvert <1\, \big\}\, ;$
\sk 
\item  $\tau$ stands for an a priori arbitrary element of $\mathbb H$;
\sk 
\item $A_\tau=A+A\tau$ for any $\tau\in \mathbb H$ and any subset $A\subset \mathbb C$;
\sk 
\item $E_\tau=\mathbb C/\mathbb Z_\tau$ is the elliptic curve associated to the lattice $\mathbb Z_\tau$ for $\tau\in \mathbb H$;
\sk 
\item $[0,1[_\tau$ is the fundamental parallelogram of $E_\tau$;
\sk 
\item $z=(z_1,\ldots,z_n)$ denotes a  $n$-uplet of complex numbers: $
(z_i)_{i=1}^n
\in \mathbb C^n$;
\sk  
\item 
$[z_i]\in E_\tau$ stands for  the class of $z_i\in \mathbb C$ modulo $\mathbb Z_\tau$ when $\tau$ is given;
\sk 
\item Most of the time $z=
(z_i)_{i=1}^n
\in \mathbb C^n$ will be assumed to be 
\begin{itemize}
\item such that the $[z_i]$'s are pairwise distinct;
\item  normalized up to a translation, that is $z_1=0$;
\end{itemize}
\sk 
\item  $E_{\tau,z}$ is the $n$-punctured elliptic curve $E_\tau\setminus \big\{ [z_1],\ldots ,[z_n]\big\}$; 
\end{itemize}

\subsection{\bf Notations and formulae for theta functions}
\label{SS:NotationsTheta}
Our main reference concerning theta functions and related material is Chandrasekharan's book \cite{Chandrasekharan}. 
\begin{itemize}
\item  $q=\exp(i\pi\tau)\in \mathbb D$ for $\tau\in \mathbb H$;
\sk 
\item $\theta(\cdot )=\theta(\cdot , \tau)$ for $\tau\in \mathbb H$ (viewed as a fixed parameter) 
stands for Jacobi's theta function defined by, for every $u\in \mathbb C$: 
\begin{equation}
\label{E:ThetaFunction}
\theta(u)=\theta(u, \tau)=-i \sum_{n\in \mathbb Z} (-1)^n\exp\Big(i\pi \big(n+{1}/{2}
\big)^2\tau+2i\pi\big(n+{1}/{2}
\big) u
\Big)\, ; 
\end{equation}
\item   For  
 $\tau \in \mathbb H$, 
the following multiplicative  functional relations hold true: 
\begin{equation}
\label{E:ThetaQuasiPeriodicity}
 \theta(u+1)=-\theta(u) \qquad \mbox{ and }\qquad \theta(u+\tau)=-q^{-1}e^{-2i\pi u}\cdot \theta(u)\, ; 
 \end{equation}
\item $\theta'(u)$ and 
$\dt{\theta}(u)$ 
 stand for the derivative of $\theta$ w.r.t  $u$ and $\tau$ respectively;
\sk 
\item  Heat equation: for every $u\in \mathbb C$, one has: 
$
\dt{\theta}(u)=(4i\pi)^{-1}\theta''(u)
$
\sk 
\item By definition, the four Jabobi's theta functions $\theta_0,\ldots,\theta_3$ are 
\begin{align*}
 \theta_0(u)=& \theta(u,\tau)       &&  \theta_1(u)= -\theta\left(u-\frac{1}{2},\tau\right)
\\  
  \theta_2(u)=&\theta\left(u-\frac{\tau}{2},\tau\right)   i q^{\frac{1}{4}}e^{-i\pi z}   &&  \theta_3(u)= -  \theta\left(u-\frac{1+\tau}{2},\tau\right) q^{\frac{1}{4}}
  e^{-i\pi u}\, ; 
 \end{align*}
\item Functional equations for the $\theta_i$'s:  for every $(u,\tau)\in \mathbb C\times \mathbb H$, one has
\begin{align*}
\theta_1(u+1)= & -\theta_1(u) &&    \theta_1(u+\tau)= q^{-1}e^{-2i\pi u}\theta_1(u) \\
\theta_2(u+1)= & \theta_2(u) &&    \theta_2(u+\tau)= - q^{-1}e^{-2i\pi u}\theta_2(u) \\
\theta_3(u+1)= & \theta_3(u) &&    \theta_3(u+\tau)= q^{-1}e^{-2i\pi u}\theta_3(u) \, ; 
\end{align*}
\item $\rho$ denotes   the logarithmic derivative of $\theta$ w.r.t.\,$u$, {\it i.e.}    $\rho(u)=
 {\theta'(u)}/{\theta(u)}$; 
\sk 
\item  functional equation for $\rho$: for every $\tau\in \mathbb H$ and every $u\in \mathbb C\setminus \mathbb Z_\tau$, one has 
\begin{equation*}
\label{E:RhoFunctionalEquation}
\rho(u+1)= \rho(u)    \qquad \mbox{ and }\qquad   \rho(u+\tau)=  \rho(u)-2i\pi \, ; 
\end{equation*}
\item $\rho'(\cdot)$ is $\mathbb Z_\tau$-invariant hence can be  seen as a rational function on 
$E_\tau$.
\sk 
\end{itemize}

\subsection{\bf Modular curves }
A handy reference for the little  we use on modular curves 
is the nice book \cite{DS} by Diamond and Shurman. 

\begin{itemize}
\item $N$ stands for  a (fixed)  positive integer; 
\sk 
\item classical congruence subgroups of level $N$: 
\begin{align*}
\Gamma(N)=& \, 
\left\{
\begin{bmatrix}
a & b \\
c & d
\end{bmatrix}
\in {\rm SL}_2(\mathbb Z)\, \Big\lvert\, 
a \equiv d \equiv 1 \; \mbox{ and }\; c\equiv d\equiv 0 \; {\rm  mod}\; N\, 
\right\}\, ; 
\\
\Gamma_1(N)= & \, 
\left\{
\begin{bmatrix}
a & b \\
c & d
\end{bmatrix}
\in {\rm SL}_2(\mathbb Z)\, \Big\lvert\, 
a \equiv d \equiv 1 \; \mbox{ and }\; c\equiv 0 \; {\rm  mod}\; N\, 
\right\}\; ; 
\end{align*}
\item  $Y(\Gamma)=\mathbb H/\Gamma$ for $\Gamma$ a congruence subgroup of ${\rm SL}_2(\mathbb Z)$;
\sk
\item 
 $Y(N)=Y(\Gamma(N))$ and $Y_1(N)=Y(\Gamma_1(N))$; 
\sk 
\item  $\mathbb H^\star=\mathbb H\sqcup \mathbb P^1_{\mathbb Q}\subset \mathbb P^1$ is the extended upper-half plane; 
\sk 
\item  $X(\Gamma)=\mathbb H^\star/\Gamma$ is the compactified modular curve associated to $\Gamma$;
\sk 
\item  $X(N)=X(\Gamma(N))$ and $X_1(N)=X(\Gamma_1(N))$.
\sk 
\end{itemize}

\subsection{\bf Teichm\"uller material} 
There are many good books about Teichmller theory. A useful one considering 
what we are doing in this text is \cite{NagBook} by Nag.  
\sk 
\begin{itemize}
\item $g$ and $n$ stand for non-negative integers such that $2g-2+n>0$;
\sk
\item $S_g$ (or just $S$ for short)  is a fixed compact orientable surface of genus $g$; 
\sk 
\item $S_{g,n}$ (or $S^*$ for short) denotes either the $n$-punctured surface $S\setminus \{s_1,\ldots,s_n\}$ or the $n$-marked surface $(S,s)$
where $s=(s_i)_{i=1}^n$ stands for a fixed $n$-uplet of pairwise distinct points on $S$;
\sk 
\item  $\mathcal T\!\!\!{\it eich}_{g,n}$ is a shorthand for $\mathcal T\!\!\!{\it eich}(S_{g},s)$, the Teichm\"uller space of the $n$-marked surface of genus $g$ $S_{g,n}$; 
\sk 
\item ${\rm PMCG}_{g,n}$ denotes  the pure mapping class group; 
\sk 
\item  ${\rm Tor}_{g,n}$ is the Torelli group:  it is the kernel of the epimorphism of groups $
{\rm PMCG}_{g,n}\rightarrow  {\rm Aut}\big(H_1(S_{g,n},\mathbb Z), \cup \big)$ (where $\cup$ stands for the cup product);
\sk 
\item $\mathcal T\!\!\!{\it or}_{\!g,n}=\mathcal T\!\!\!{\it eich}_{g,n}/{\rm Tor}_{g,n}$ is the associated Torelli space; 
 \sk 
\item $\mathscr M_{g,n}=\mathcal T\!\!\!{\it eich}_{g,n}/{\rm PMCG}_{g,n}$ is the associated (Riemann) moduli space. 
\end{itemize}

\subsection{\bf Complex hyperbolic spaces} 
We will make practically no use of
  complex hyperbolic geometry in this text. However, viewed its conceptual importance to understand Veech's constructions, we settle basic definitions and facts below. 
  For a reference, the reader can consult \cite{Goldman}.\sk 
  
\begin{itemize}
\item  $\langle \cdot , \cdot \rangle=\langle \cdot , \cdot \rangle_{1,n}$ is the hermitian form of signature $(1,n)$ on $\mathbb C^{n+1}$: one has 
$$
\langle z , w \rangle= 
\langle z , w \rangle_{1,n}= z_0\overline{w}_0-\sum_{i=1}^n z_i\overline{w}_i
$$
 for $z=(z_0,\ldots,z_n)$ and $w=(w_0,\ldots,w_n)$  in $\mathbb C^{n+1}$;
\sk 
\item  $V_{1,n}^+=\{ z \in \mathbb C^{n+1}\, \lvert \, \langle z,z\rangle_{1,n}>0\, \} \subset \mathbb C^{n+1}$ is the set of $\langle\cdot,\cdot\rangle$-positive vectors;
\sk 
\item
 the complex hyperbolic space 
$\mathbb C\mathbb H^n$ is the projectivization of $V_{1,n}^+$: 
$$
\mathbb C\mathbb H^n = \mathbf P V_{1,n}^+ \subset \mathbb P^{n}\, ; 
$$
\item in the affine coordinates $(z_0=1,z_1,\ldots,z_n)$, $\mathbb C\mathbb H^n$ is the 
complex unit ball: 
  \begin{equation}
  \label{E:CHn-Ball}
  \mathbb C\mathbb H^n=\left\{ (z_i)_{i=1}^n\in \mathbb C^n\, \Big \lvert \, \sum_{i=1}^n \lvert z_i\lvert^2<1\, \right\}\, ;
  \end{equation}
\sk 
\item the complex hyperbolic metric is the Bergman metric of the unit complex ball
\eqref{E:CHn-Ball}. For $[z]\in\mathbb C\mathbb H^n$ with $z\in V_+$,    it is given explicitly  by  
$$
{g}^{hyp}_{[z]}=-\frac{4}{\langle z ,z \rangle^2}\det \begin{bmatrix}  \langle z ,z \rangle & \langle dz ,z \rangle \\
\langle z ,dz \rangle  & \langle dz ,dz \rangle    \end{bmatrix}\, ; 
$$
\sk
\item 
${\rm PU}(1,n) ={\rm PAut}(\mathbb C^{n+1},\langle\cdot , \cdot \rangle_{1,n})
< {\rm PGL}_{n+1}(\mathbb C)$ acts transitively on $\mathbb C\mathbb H^n$ and coincides with its group of biholomorphisms ${\rm Aut}(\mathbb C\mathbb H^n)$;
\sk
\item  being a Bergman metric,  $
{g}^{hyp}$ is invariant by ${\rm Aut}(\mathbb C\mathbb H^n)={\rm PU}(1,n)$; 
\sk
\item 
$(\mathbb C\mathbb H^n, {g}^{hyp})$ is a non-compact complete hermitian symmetric space of rank 1 with constant holomorphic sectional curvature; 
 \sk
\item  for $n=1$ and $(1,z)\in V_+$, one has 
$g^{hyp}_{[z]}={4  \big(1-\lvert z\lvert^2\big)^{-2}
    \lvert dz \lvert^2}
$,  
therefore $\mathbb C\mathbb H^1$ coincides with Poincar's hyperbolic disk $\mathbb D^{hyp}$ 
hence with the real hyperbolic plane  $\mathbb R\mathbb H^2$. In other terms, 
there are some identifications
$$
\mathbb C\mathbb H^1 \simeq \mathbb D^{hyp} \simeq \mathbb H \simeq \mathbb R\mathbb H^2
\quad 
\mbox{ and }
\quad 
{\rm Aut}\big(\mathbb C\mathbb H^1\big)={\rm PU}(1,1)\simeq {\rm PSL}_2(\mathbb R)\, .
$$

\end{itemize}

\section{\bf Twisted (co-)homology and  integrals of hypergeometric type}
\label{S:Twisted(Co)Homology}
It is now well-known that a rigorous and relevant framework to deal with (generalized) hypergeometric functions is the one of twisted (co-)homology. 

For the sake of completeness, we give below a short review of this theory in the simplest 1-dimensional case.  All this material and its link with the theory of hypergeometric functions is exposed 
 in many  modern references,  such as  \cite{DeligneMostow, KitaYoshida,Yoshida,AomotoKita},  to  which we refer for proofs and details. \mk 

After recalling some generalities, we focus on the case we will be interested in, namely the one of punctured elliptic curves. This case has been studied extensively by Mano and Watanabe. 
 Almost all the material presented below has been taken from 
\cite{ManoWatanabe}. The unique exception is Proposition \ref{P:GammaEll} in subsection \S\ref{S:IntersectionProduct}, where we compute explicitly the twisted intersection product.  If this result relies on simple computations, it is of importance for us since it will allow us to give an explicit expression of the Veech form 
({\it cf.}\;Proposition \ref{P:VeechFormExplicit}).


\subsection{\bf The case of Riemann surfaces: generalities}
 Interesting general referen\-ces in arbitrary dimension are \cite{AomotoKita,Vassiliev}.  The case corresponding  to the classical theory of hypergeometric functions is 
the one where the ambient variety is a punctured projective line. 
 It is treated in a very nice but  informal way in the fourth chapter of 
 Yoshida's love book \cite{Yoshida}. 
 A more detailed treatment is given in the second section of 
Deligne-Mostow's paper \cite{DeligneMostow}.    
 For arbitrary Riemann surfaces, the reader can consult \cite{Ito}.  

\subsubsection{}\hspace{-0.3cm}
\label{S:rho&T}
 Let $\mu$ be a {\bf multiplicative complex character} on the fundamental group of a (possibly non-closed) Riemann surface $X$, {\it i.e.}\;a group homomorphism $\mu: \pi_1(X)\rightarrow \mathbb C^*$.
  Since the target group is abelian, it factorizes through a homomorphism of abelian groups $\pi_1(X)^{\rm ab}=H_1(X,\mathbb Z)\rightarrow \mathbb C^*$, denoted abusively by the same notation $\mu$. 
 We will denote by  $ L_{\mu}$  or just $L$ `the' {\bf local system associated to} $\boldsymbol{\mu}$.  
  We  use the notation ${ L}_\mu^\vee$,  ${L}^\vee$ for short, 
   to designate the dual local system which is  the local system $  L_{\mu^{-1}}$ associated to the dual character $\mu^{-1}$.\sk 

Assume that $T$ is a  multivalued holomorphic function on  $X$  whose monodromy 
  is multiplicative and equal to $\mu^{-1}$: the analytic continuation along  any loop $\gamma: S^1\rightarrow X$ of a determination of $T$ at $\gamma(1)$  is $\mu(\gamma)^{-1}$ times this initial determination. Then $T$ can be seen as a global section of ${L}^\vee$.    
Conversely,  assuming that  $T$ does not vanish on $X$,  one can define ${ L}$ as the line bundle,  the stalk of which at any point $x$  of $X$ is the 1-dimensional complex vector space spanned by a (or equivalently by all the) determination(s) of $T^{-1}$ at $x$.
\sk 

Assuming that $T$ satisfies all the preceding assumptions, the logarithmic derivative $\omega=(d\log T)=
T^{-1}dT$ of $T$ 
 is a global  (univalued) holomorphic 1-form on $X$.  
 Then one can define $L$ more formally 
as  the local system formed by the solutions of the 
global differential equation $dh+\omega h=0$ on $X$.

\subsubsection{}\hspace{-0.3cm} For $k=0,1,2$,  a {\bf (${ L^\vee}$-)twisted $k$-simplex} is the data of a $k$-simplex $\alpha$ in $X$ together with a determination 
$T_\alpha$ 
of $T$ along $\alpha$. We will denote this object by  $\alpha\otimes T_\alpha$ or,  more succinctly,  by $\boldsymbol{\alpha}$.
 A {\bf twisted $k$-chain} is  a finite linear combination with complex coefficients of twisted $k$-simplices.  
 By taking the restriction of $T_\alpha$ to the corresponding facet of $\partial \alpha$, one defines a boundary operator $\partial$ on twisted $k$-simplices which extends to twisted  $k$-chains by linearity.  It satisfies $\partial ^2=0$,  which allows to define the {\bf twisted homology group} $H_k(X,{ L^\vee})$.\sk  
 
More generally, one defines a {\bf locally finite  twisted $k$-chain} as a  possibly infinite linear combination 
$
\sum_{i\in I}
c_i\cdot \boldsymbol{\alpha}_i $
  with complex coefficients of ${ L^\vee}$-twisted $k$-simplices $\boldsymbol{\alpha}_i$, 
but such that  there are only finitely many indices $i\in I$ such that $\alpha_i$ intersects  non-trivially any previously given compact subset  of $X$.  The boundary operator previously defined extends to such chains and allows to define  the {\bf groups of locally finite twisted homology} $H_k^{\rm lf}(X,{ L^\vee})$ for $k=0,1,2$.  

\subsubsection{}\hspace{-0.3cm} A {\bf (${ L}$-)twisted $k$-cochain} is a map which associates a  section of ${ L}$ over $\alpha$  to any $k$-simplex $\alpha$ (or equivalently,  it is a map which associates a complex number to any 
${L^\vee}$-twisted $k$-simplex $\alpha\otimes T_\alpha$). 
 The fact that such a section extends in a unique way to any 
 $(k+1)$-simplex admitting $\alpha$ as a face allows to define a coboundary operator. The latter will satisfy all the expected properties in order to construct the {\bf twisted cohomology groups} $H^k(X,{ L})$ for $k=0,1,2$.  Similarly, by considering twisted $k$-cochains with compact support, one constructs the   groups of 
{\bf twisted cohomology with compact support}  $H^k_c(X,{ L})$. 
 
The   (co)homology spaces considered above are dual to each other: for any  $k=0,1,2$, there are natural dualities 
\begin{equation}
\label{E:dualities}
H_k\big(X,{ L^\vee}\big)^{\!\lor}\simeq H^k\big(X,{ L}\big)
\qquad  
\mbox{and } \qquad H_k^{\rm lf}\big(X,{ L^\vee}\big)^{\!\lor}\simeq H^k_c\big(X,{ L}\big).\end{equation}

\subsubsection{}\hspace{-0.3cm}
A {twisted $k$-chain} being locally finite, there are natural maps 
$H_k (X,{ L^\vee})\rightarrow H_k^{\rm lf}(X,{ L^\vee})$.  We focus on the case when $k=1$ which is the only one to be of  interest for our purpose.  In many situations, when $\mu$ is sufficiently generic, it turns out  that the natural map $H_1 (X,{ L^\vee})\rightarrow H_1^{\rm lf}(X,{ L^\vee})$  actually is an isomorphism. In this case, one denotes the inverse map 
by 
\begin{equation}
\label{E:Reg}
{\rm Reg}: 
H_1 ^{\rm lf}\big(X,{ L^\vee}\big)\longrightarrow H_1\big(X,{ L^\vee}\big)
\end{equation}
 and call it  the {\bf regularization morphism}. Note that it is canonical. \sk

Assume that $\boldsymbol{\alpha}_1,\ldots,\boldsymbol{\alpha}_N$ are locally finite twisted 1-chains (or even better, twisted 1-simplices) in $X$ whose homology classes  generate $H_1 ^{\rm lf}(X,{ L^\vee})$. 

A {\bf regularization map} is a map ${\rm reg}: \boldsymbol{\alpha}_i\mapsto {\rm reg}(\boldsymbol{\alpha}_i) $ such that: 
\begin{enumerate}
\item[(1)] for every $i=1,\ldots,n$,  ${\rm reg}(\boldsymbol{\alpha}_i)$ is a ${ L^\vee}$-twisted  1-chain 
which is no longer locally finite but finite on $X$; 
\item[(2)] ${\rm reg}$ factors through the quotient and the induced map $H_1 ^{\rm lf}(X,{ L^\vee})\rightarrow H_1(X,{ L^\vee})$ coincides with  the regularization morphism \eqref{E:Reg}.
\end{enumerate}

\subsubsection{}\hspace{-0.4cm} 
\label{S:TwistedPoincarDuality}
{\bf Poincar duality} holds true for twisted (co)homology: for $i=0,1,2$, there are natural isomorphisms $H^i(X,{ L})\simeq 
H_{2-i}^{\rm lf}(X,{ L^\vee})$ ({\it cf.}\;\cite[Theorem 1.1\,p.\,218]{Vassiliev} or \cite[\S2.2.11]{AomotoKita} for instance). 
Combining the latter isomorphism with \eqref{E:dualities}, one obtains a non-degenerate bilinear pairing $H_1(X, { L^\vee})\otimes H_1^{\rm lf}(X, { L})\rightarrow \mathbb C$. When the regularization morphism \eqref{E:Reg} exists, it matches the induced pairing 
\begin{equation}
\label{E:pairing}
H_1\big(X, { L^\vee} \big)\otimes H_1\big(X, { L}\big)\longrightarrow \mathbb C
\end{equation}
which, in particular, is non-degenerate.

\subsubsection{}\hspace{-0.4cm} 
\label{S:GeneralTwistedIntersectionProduct}
Assume now that $\mu$ is unitary, {\it i.e.}\;${\rm Im}(\mu)\subset \mathbb U$. Then  $\mu^{-1}$ 
coincides with the conjugate morphism $\overline{\mu}$,  thus the twisted homology groups 
$H_1(X,{ L})$ and $H_1(X,{L}_{\overline{\mu}})$ are equal. 
On the other hand, the map $\alpha\otimes T_\alpha \rightarrow \alpha \otimes \overline{T_\alpha}$ defined on ${ L^\vee}$-twisted 1-simplices induces a complex conjugation  $\boldsymbol{\alpha}\mapsto \boldsymbol{\overline{\alpha}}$ from $ H_1(X,{ L^\vee})$ onto $ H_1(X, L_{\overline{\mu}})$. 

Using \S\ref{S:TwistedPoincarDuality}, one gets the following non-degenerate hermitian pairing
\begin{align}
\label{E:TwistedIntProduct}
H_1\big(X, { L^\vee}\big)^{ 2} & \longrightarrow \mathbb C\\
\big( \boldsymbol{\alpha}, \boldsymbol{\beta}   \big) & \longmapsto 
 \boldsymbol{\alpha}\cdot \overline{\boldsymbol{\beta}}
 \nonumber
\end{align}
which  in this situation is called the {\bf twisted intersection product}.

  \subsubsection{}\hspace{-0.4cm}
 \label{S:}
 Let $\eta$ be a holomorphic 1-form on $X$. 
By setting   
\begin{equation}
\label{E:IntTeta}
 \int_{\boldsymbol{\alpha}} T\cdot \eta= \int_{\alpha} T_\alpha\cdot \eta
 \end{equation}
  for any twisted 1-simplex $\boldsymbol{\alpha}=\alpha\otimes T_\alpha$, and by extending this map by linearity, one defines a complex linear map 
 $\int T\cdot \eta$ on the spaces of twisted 1-cycles.    The value \eqref{E:IntTeta} does not depend on $\boldsymbol{\alpha}$ but only on its (twisted) homology class.  Consequently, the preceding map factorizes and gives rise to a linear map
 \begin{align*}
 \int_{}\; \,  T\cdot \eta  \;  : \; 
 H_1\big(X,{ L^\vee}\big)
 & \longrightarrow \mathbb C\\
 [\boldsymbol{\alpha}] & \longmapsto  \int_{\boldsymbol{\alpha}} T\cdot \eta.
 \end{align*}
 
On the other hand, there is an  exact sequence of sheaves $0\rightarrow  L\rightarrow \Omega_X
  ^0( L)\stackrel{d
  }{\rightarrow}
 \Omega_X^1( L){\rightarrow}0$ on $X$ 
  whose hypercohomology groups are proved to be isomorphic to the simplicial ones $H^k(X, L)$.  Then for any  $\eta$ as above, it can be verified that \eqref{E:IntTeta}
  depends only on the associated class $[T\cdot \eta]$ in $H^1(X, L)$ and 
 its value on  $\boldsymbol{\alpha}$ is given by means of the pairing \eqref{E:pairing}:
 for $\eta$ and $\boldsymbol{\alpha}$ as above,  one has  
 \begin{equation*}
  \int_{\boldsymbol{\alpha}} T\cdot \eta=\left\langle \big[T\cdot \eta\big], [\boldsymbol{\alpha}]\right\rangle\, .
 \end{equation*}
  
From this, one deduces a precise cohomological definition for what we call a  {\bf generalized elliptic integral}, that is an integral of the form 
\begin{equation}
\label{E:HypergeomIntegral}
\int_{\boldsymbol{\gamma}} T\cdot \eta
\end{equation}
where $\eta$ is a 1-form on $X$ and $\boldsymbol{\gamma}$ a twisted 1-cycle (or a twisted homology class).     
 
 \subsubsection{}\hspace{-0.4cm}
 \label{S:TwistedDeRham}
  Assume that $T$ is a non-vanishing function as in \S\ref{S:rho&T}. 
 Then using $\omega=d\log(T)$, one defines a twisted covariant differential operator $\nabla_\omega$ on $\Omega^\bullet_X$ by setting $\nabla_\omega(\eta)=d\eta+\omega\wedge \eta$ for any holomorphic form $\eta$ on $X$.   In this way one gets a complex $(\Omega^\bullet_X,\nabla_\omega)$ called {\bf the twisted De Rham complex} of $X$.  
 
  There is an  exact sequence of sheaves on $X$
  $$0\longrightarrow  L\longrightarrow \Omega_X^0\stackrel{\nabla_\omega
  }{\longrightarrow}
 \Omega_X^1{\longrightarrow}0\, , $$ 
 from which it comes that $(\Omega^\bullet_X,\nabla_\omega)$ is a resolution of $ L$. Consequently (see {\it e.g.}\;\cite[\S2.4.3 and \S2.4.6]{AomotoKita}), the twisted simplicial cohomology groups of $X$ are naturally isomorphic to the twisted hypercohomology groups $H^k(\Omega^\bullet_X,\nabla_\omega)$ for $k=0,1,2$.  The main conceptual interest of using this twisted de Rham formalism is that it allows to construct what is called the associated {\bf Gau{{\ss}}-Manin connection} which in  turn can be used to construct (and  actually is essentially equivalent to) the linear differential system  satisfied by the hypergeometric integrals  
\eqref{E:HypergeomIntegral}.  We will return to this in Appendix B, where we will treat  the case of 2-punctured elliptic curves very explicitly.
\mk 

When  $X$ is affine  (a punctured compact Riemann surface for instance), 
the hypercohomology groups $H^k(\Omega^\bullet_X,\nabla_\omega)$ can be shown to be   isomorphic to some particular cohomology groups built from global holomorphic objects on $X$. 

For instance, in the affine case,  there are  natural  isomorphisms  
\begin{equation}
\label{E:popol}
H^1\big(X, L\big)\simeq 
H^1\big(\Omega^\bullet_X,\nabla_\omega\big)\simeq \frac{H^0\big(X,\Omega_X^1\big)}{
\nabla_\omega\big(H^0(X,{\mathcal O}_X)\big)}\, . 
\end{equation}

 \subsubsection{}\hspace{-0.4cm}
 \label{S:AlgebraicdeRhamComparisonTHM}
 Assume that $X$ is a punctured compact Riemann surface, {\it i.e.}\;$X=\overline{X}
 \setminus \Sigma$ where  $\Sigma$ is a non-empty finite subset of a compact Riemann surface $\overline{X}$.  If $\omega$ extends to a rational 1-form on $\overline{X}$ (with poles on $\Sigma$), then  one can consider the {\bf algebraic twisted de Rham complex}
  $(\Omega^\bullet_X(*\Sigma),\nabla_\omega)$.  It is the subcomplex of $(\Omega^\bullet_X,\nabla_\omega)$ formed by the restrictions to $X$  of the rational forms on $\overline{X}$ with poles supported exclusively on $\Sigma$.  The {\bf (twisted) algebraic de Rham comparison theorem}  ({\it cf.}\;\cite[\S2.4.7]{AomotoKita})   
  asserts  that these two resolutions of $ L$ are quasi-isomorphic, 
  {\it i.e.}\;their associated  hypercohomology groups $H^k(\Omega^\bullet_X(*\Sigma),\nabla_\omega)$ and 
  $H^k(\Omega^\bullet_X,\nabla_\omega)$ are isomorphic.   \sk 
 %
  \sk

Taking one step further, one gets that the singular $ L$-twisted cohomology of $X$ can be described by means of  global holomorphic objects on $X$ which actually are restrictions to $X$ of some rational forms on the compact Riemann surface $\overline{X}$. More precisely, there is a isomorphism
\begin{equation}
\label{E:IsomTutu}
H^1\big(X, L\big)\simeq 
 \frac{H^0\big(X,\Omega_X^1(*\Sigma)\big)}{
\nabla_\omega\big(H^0\big(X,{\mathcal O}_X(*\Sigma)\big)\big)}\, . 
\end{equation}

The interest of this isomorphism lies in the fact  that it allows to describe the twisted cohomology group $
H^1\big(X, L\big)$ by means of rational 1-forms on $\overline{X}$.  
For instance, this  is quite useful to simplify the computations involved in making  
the Gau{{\ss}}-Manin connection mentioned in \ref{S:TwistedDeRham} explicit. 
 Usually (for instance for classical hypergeometric functions, see \cite[\S2.5]{AomotoKita}), one even uses a (stricly proper) subcomplex of $(\Omega^\bullet_X(*\Sigma),\nabla_\omega)$ by considering  rational forms on $\overline{X}$ with logarithmic poles on $\Sigma$.  However, such a simplification is not always possible. An example is precisely the case of punctured elliptic curves we are interested in, for which it is necessary to consider rational 1-forms with poles of order 2 at (at least one of) the punctures 
 in order to get an isomorphism similar to  \eqref{E:IsomTutu}, see \S\ref{SS:TwistedCohomologyH1} below.


\subsection{\bf On punctured elliptic curves} 
\label{SS:OnTori}
We now specialize and make the theory described above  explicit  in the case of punctured elliptic curves. This  case has been treated very carefully in  \cite{ManoWatanabe}  
 to  which we refer for proofs and details.  For some particular cases with few punctures, the interested reader can consult \cite{WatanabeWirtinger,Ito2009,ManoKumamoto,ManoKyushu}.  
\sk 

In this subsection, $n\geq 2$ is an integer and $\alpha_0,\alpha_1,\alpha_2,\ldots,\alpha_n$ are fixed real parameters such that 
\begin{equation}
\label{E:alphai}
\alpha_i\in ]-1,\infty[ \;\;  \mbox{ for }\,  i=1,\ldots,n \qquad \mbox{ and  }\qquad 
\sum_{i=1}^n\alpha_i=0\, .
\end{equation}
 
Note that unlike the others $\alpha_i$'s, $\alpha_0$ is arbitrary. 

\subsubsection{\bf }\hspace{-0.4cm} For $\tau \in \mathbb H$ and $z=(z_1,z_2,\ldots,z_n)\in \mathbb C^n$, one denotes by $E_{\tau,z}$ the punctured elliptic curve $E_\tau\setminus \{[z_i]\, \lvert \, i=1,\ldots,n\}$, where $[z_i]$ stands for the class of $z_i$ in $E_\tau$.  We will always assume that the $[z_i]$'s are pairwise distinct and that 
$z$ has been normalized, meaning that $z_1=0$. 
\sk

For $\tau$ and $z$ as above, one considers the holomorphic multivalued function
\begin{equation}
\label{E:functionT}
T^\alpha(\cdot ,  \tau,z) \; : \; u\longmapsto 
T^\alpha(u; \tau,z)= 
\exp\big(2i\pi\alpha_0 u \big)\prod_{k=1}^n\theta\big(u-z_k\big)^{\alpha_k}\, , 
\end{equation}
of a complex variable $u$,  where $\theta$  stands for the theta function $ \theta(\cdot, \tau)$, {\it cf.}\;\eqref{E:ThetaFunction}. 

Since $\tau$, $z$ and the $\alpha_i$'s will stay  fixed in this section,   we will write $T(\cdot)$ instead of $T^\alpha(\cdot, \tau,z)$  to make the notations simpler.
\sk 

A straightforward computation gives 
$$\omega:=d\log T= 
\big( \partial \log T/{\partial u}\big) du
= 2i\pi\alpha_0 du+\sum_{k=1}^n  {\alpha_k}  \rho(u-z_k)du$$ where $\rho(\cdot)$ stands for the logarithmic derivative of $\theta(\cdot)$, see again \eqref{E:ThetaFunction}.  Using \eqref{E:alphai},  this can be rewritten  as 
\begin{equation}
\label{E:dLogT}
\omega 
=  2i\pi\alpha_0 du+\sum_{k=2}^n  {\alpha_k}  
\big( \rho(u-z_k)-\rho(u)\big) du\, .
\end{equation}
Starting from 2 instead of 1 in the summation above forces to subtract $\rho(u)$ at each step. The advantage is that the functions 
$( \rho(u-z_k)-\rho(u))$, $k=2,\ldots,n$ which appear in \eqref{E:dLogT} are all  
rational on $E_{\tau,z}$. This shows that $\omega$  is a logarithmic rational 1-form on $E_\tau$ with poles exactly at the $[z_i]$'s.  

Clearly,  $T$ is nothing else but  the pull back by the universal covering map $\mathbb C \rightarrow E_\tau$ of a solution of 
the differential operator 
\begin{align}
\label{E:nabla-omega}
\nabla_{-\omega} : 
{\mathcal O}_{E_{\tau,z}} & \longrightarrow {\Omega}^1_{E_{\tau,z}} \\ 
h & \longmapsto 
dh-\omega\cdot  h\, ,  \nonumber 
\end{align}
hence can be considered as a multivalued holomorphic function 
on 
 $E_{\tau,z}$.  \sk
 
 Since $\omega=d\log T$ is a rational 1-form,  the monodromy of $\log T$ is additive, hence that of $T$ is multiplicative. 
For this reason,  it is not necessary to refer to a base point  to specify the monodromy of $T$. Thus the latter can be encoded by means of a 
morphism 
\begin{equation}
\label{E:Rho}
\rho: H_1(E_{\tau,z},\mathbb Z)\longrightarrow \mathbb C^*
\end{equation}
that we are going to give explicitely below. \sk 

Let us define ${ L^\vee}$ as the kernel of the differential operator \eqref{E:nabla-omega}. It is the local system of $E_{\tau,z}$ the local sections of which are local determinations of $T$.


\subsubsection{\bf }\hspace{-0.4cm} 
\label{E:WhereRhoIsDefined}
Let $\epsilon>0$ be very small and set $\star=-\epsilon(1+i)\in \mathbb C$. Denote by $\beta_0$ (resp. by $\beta_\infty$) the image in $E_{\tau,z}$ of the rectilinear segment in $\mathbb C\setminus \cup_{i=1}^n(z_i+\mathbb Z_\tau)$ linking $\star$ to $\star+1$ (resp. $\star+\tau$).  For $i=1,\ldots,n$, let $\beta_i$ stand  for the image in $E_{\tau,z}$ of a circle centered at $z_i$ of radius $\epsilon/2$ and positively oriented (see Figure \ref{F:??}).

\begin{center}
\begin{figure}[!h]
\psfrag{E}[][][1.2]{$\;\;  {E_\tau}$}
\includegraphics[scale=0.7]{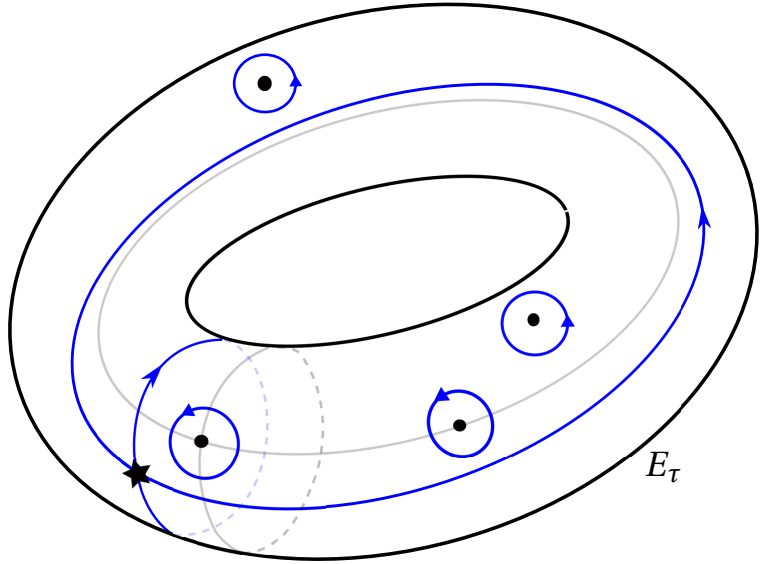}
\caption{In blue, the 1-cycles $\beta_\bullet$, $\bullet=0,1,\ldots,n,\infty$ 
(the two cycles in grey are the images in $E_\tau$ of the segments $[0,1]$ and $[0,\tau]$).}
\label{F:??}
\end{figure}
\end{center}

For $\bullet\in \{0,\infty,1,\ldots,n\}$, the analytic continuation of any determination $T_\star$ of $T$ at $\star$ along $\beta_\bullet$ is equal to  $T_\star$ times a complex number $\rho_\bullet=\rho(\beta_\bullet)$ which does not depend on $\epsilon$ or on the initially chosen determination $T_\star$.  Moreover, since $H_1(E_{\tau,z},\mathbb Z)$ is spanned by the homology classes of the 1-cycles $\beta_\bullet$ (which
do not depend on $\epsilon$ if the latter is sufficiently small), the $n+2$ values $\rho_\bullet$ completely characterize the monodromy morphism \eqref{E:Rho}. \sk 

For any $k=1,\ldots, n$, up to multiplication by a non-vanishing constant, one has $T(u)\sim (u-z_k)^{\alpha_k}$ for $u$ in the vicinity of $z_k$. It follows immediately  that 
$$\rho_k=\exp\big(2i\pi \alpha_k\big).$$

  It is necessary to use different kinds of arguments in order to determine the 
 values  $\rho_0$ and $\rho_\infty$ which account for the monodromy coming from the global topology of $E_\tau$.  We will deal only with the monodromy along $\beta_\infty$ since the determination of the monodromy along $\beta_0$  relies on similar (and actually simpler) computations. 
 For $u$ close to $\star$, using the functional equation \eqref{E:ThetaQuasiPeriodicity} satisfied by  $\theta$  and because $\sum_{i=1}^n \alpha_i=0$, the following equalities hold true:
\begin{align*} 
T(u+\tau)
= & e^{2i\pi \alpha_0( u+\tau)}\prod_{k=1}^n \Big(-q^\frac{1}{4} e^{-2i\pi  (u-z_k)   }\theta(u-z_k,\tau)\Big)^{\alpha_k}\\
=  & \, \big(-q^\frac{1}{4}\big)^{\sum_k \alpha_k} e^{2i\pi ( \alpha_0 \tau-\sum_{k } \alpha_k (u-z_k)) } T(u)\\
=  & \, e^{2i\pi ( \alpha_0 \tau+\sum_{k } \alpha_k z_k) }  T(u)\, . 
\end{align*}

Setting 
\begin{equation}
\label{E:alphaInfty}
\alpha_\infty= \alpha_0 \tau+\sum_{k =1}^n \alpha_k z_k\,  , 
\end{equation} the preceding
computation  shows that 
$$\rho_\infty= \exp\big(2i\pi \alpha_\infty\big)\, .$$
  
By similar computations, one proves that 
$\rho_0= \exp\big(2i\pi \alpha_0\big)$. \sk

All the above computations can be summed up  in the following
\begin{lemma} 
\label{L:rhoexplicit}
The monodromy of $T$ is multiplicative and the 
values $\rho_\bullet$ characterizing the monodromy morphism 
\eqref{E:Rho} are given by 
$$
\rho_{\!\bullet}=\exp\big( 2i\pi \alpha_\bullet   \big)
$$
for $\bullet \in \{ 0,1,\ldots,n,\infty\}$, where $\alpha_\infty$ is given by \eqref{E:alphaInfty}. 
\end{lemma}


\subsubsection{\bf }\hspace{-0.4cm}
\label{SS:ellbullet}
 Let  $\tau \in \mathbb H$ and $z
 \in \mathbb C^n$ be as  above.  For $i=2,\ldots,n$,  let $\tilde z_i$ be the element of $z_i+\mathbb Z_\tau$ lying in the fundamental parallelogram $[0,1[_\tau\subset \mathbb C$ and denote by $\tilde \ell_i$ the image of 
  $]0,\tilde z_i[$ in $E_{\tau ,z}$. Then let us 
 modify the $\tilde \ell_i$'s, each  in its respective relative homotopy class,  in order to get locally finite 1-cycles $\ell_i$ in $E_{\tau,z}$ which do not intersect nor have non-trivial  self-intersection ({\it cf.}\;Figure \ref{F:lilo} below where, to simplify the notation, we have assumed that $z_i=\tilde z_i$ for $i=2,\ldots,n$).\footnote{
 \label{FTN:NotFormally}
 If the ${\ell}_{i}$'s are not formally defined as a locally finite linear combinations of twisted 1-simplices, 	a  natural way to see them like this  
is by  subdividing each segment $]0,\tilde z_i[$ into a countable union of 1-simplices overlapping only at their extremities. There is no canonical way to do this, but two locally finite twisted 1-chains obtained by this construction are clearly homotopically equivalent.}
\begin{center}
\begin{figure}[!h]
\psfrag{B}[][][1]{$\textcolor{red}{B} $}
\psfrag{1}[][][1]{$1 $}
\psfrag{0}[][][1]{$0 $}
\psfrag{z4}[][][1]{$z_2 $}
\psfrag{z3}[][][1]{$z_3 $}
\psfrag{z2}[][][1]{$z_4 $}
\psfrag{z1}[][][1]{$z_5 $}
\psfrag{t}[][][1]{$\tau $}
\psfrag{g0}[][][1]{$\ell_0 $}
\psfrag{g4}[][][1]{$\ell_{\!2} \, $}
\psfrag{g3}[][][1]{$\ell_3 $}
\psfrag{g2}[][][1]{$\ell_4 $}
\psfrag{g1}[][][1]{$\!\!\ell_{\!5} \; $}
\psfrag{gi}[][][1]{$\ell_{\!\!\infty} \, $}
\includegraphics[scale=0.85]{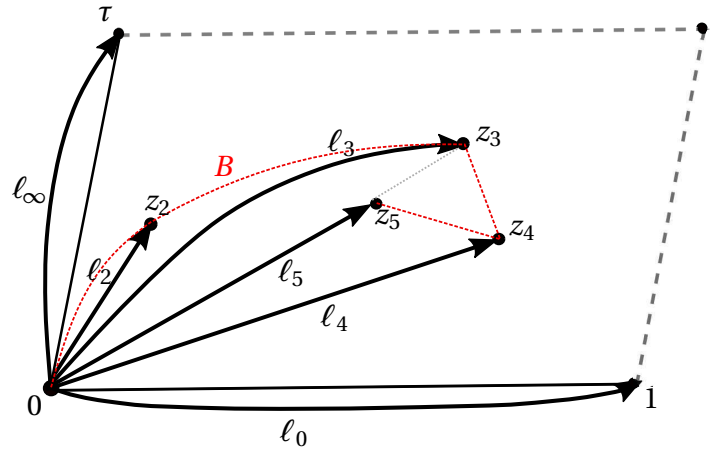}
\caption{The locally finite 1-cycles $\ell_0, \ell_2,\ldots,\ell_n $ and $\ell_{\!\!\infty}$. 
}
\label{F:lilo}
\end{figure}
\end{center}

Let $\varphi: [0,1]\rightarrow \mathbb R$ be a non-negative smooth function, such that $\varphi(0)=\varphi(1)=0$ and such that  $\varphi>0$ on $ ]0,1[$. Define $\tilde \ell_0$ as  the image of $]0,1[$ in $E_\tau$. Let $\ell_0$ be  the image in the latter tori of $f: t\mapsto t-i\epsilon \varphi(t)$   with  $\epsilon$ positive and sufficiently  small so that the bounded area delimited by the segment $[0,1]$ from above  and by the image of $f$ from below does not contain any element $\mathbb Z_\tau$-congruent to one of the $z_i$'s.  By a similar construction but starting from  the segment $]0,\tau[$, one constructs  a locally finite 1-chain $\ell_{\!\!\ \infty} $ in $E_{\tau,z}$ (see Figure \ref{F:lilo} above).   
 We prefer to consider small deformations of  the segments $]0,1[$ and $]0,\tau[$  to define $\ell_0$ and $\ell_{\!\infty}$ in order  to  avoid any ambiguity if  some of the $\tilde z_i$'s happen to be  located on one (or on both) of these segments.
\sk 

Let  $B$ be the branch cut in $E_\tau$ defined as the image of an embedding $[0,1]\rightarrow [0,1[_\tau$ sending 0 to $0$ and $1$ to $\tilde z_n$ which does not meet the $\ell_i$'s except at their extremities 
$\tilde z_i$'s which all belong to $B$ ({\it cf.}\;the curve in red in Figure \ref{F:lilo}).

Denote by $U$ the complement in $E_\tau$ of  
 the topological closure of the union of ${\ell}_{0}, \ell_{\infty}$ and $B$. Then $U$ is a simply connected open set which is naturally  identified to the  bounded open subset of $\mathbb C$, denoted by $V$,  the boundary of which is the union of $B$ with 
 the 1-chains 
  ${\ell}_{0}, \ell_{\infty}$ and their respective  horizontal and  
  `vertical'
  translations  $1+\ell_{\!\!\infty}$ and $\tau+\ell_{0}$.  
   
Thus it makes sense to speak of a (global) determination of the function $T$ defined in 
\eqref{E:functionT}
on $U$. Let $T_U$ stand for such  a global determination  on $U$. 
It extends continuously to the topological boundary  $\partial U$ of $U$ in $\mathbb C$ minus the $n-1+4$ points of 
$
\cup_{i=1}^n (z_i+\mathbb Z_\tau)$ lying on $\partial U$. 
 Then for any 
$\bullet\in  \{0,2,\ldots,n,\infty\}$, the restriction $T_{\!\bullet}$ of this continuous extension  to ${\ell_{\!\!\bullet}}$ is well defined and one defines 
 a locally finite ${ L^\vee}$-twisted 
1-chain ({\it cf.}\;the footnote of the preceding page) by setting 
$$\boldsymbol{\ell}_{\!\!\bullet}= {\ell}_{\!\!\bullet}\otimes T_{\!\bullet}\, .$$ 

The continuous extension of $T_U$ to $\overline{U}\setminus \cup_{i=1}^n (z_i+\mathbb Z_\tau)$ does not vanish. Hence for any $\bullet$ as above, one defines a locally ${ L}$-twisted 1-chain by setting: 
$$\boldsymbol{l}^{}_{\!\bullet}= {\ell}_{\!\!\bullet}\otimes (T_{\!\bullet})^{-1}.$$

We let the reader verify  that the $\boldsymbol{\ell}_{\!\!\bullet}$'s as well as the 
$\boldsymbol{l}^{}_{\!\bullet}$'s  actually are (twisted) 1-cycles. Therefore they  induce twisted locally finite homology classes, respectively  in $H_1^{\rm lf}(E_{\tau,z},{ L^\vee})$ and $H_1^{\rm lf}(E_{\tau,z},{ L})$. A bit abusively, we will denote these homology classes   by the same notation $\boldsymbol{\ell}_{\!\!\bullet}$ and 
$\boldsymbol{l}^{}_{\!\bullet}$. This will not cause any problem whatsoever.


\subsubsection{\bf }\hspace{-0.4cm}
\label{SS:VeryNice}
 Such as they are defined above, the locally finite twisted 1-cycles $ \boldsymbol{\ell}_0, \boldsymbol{\ell}_2,\ldots ,$ $ \boldsymbol{\ell}_n$ and $\boldsymbol{\ell}_\infty$ 
depend on some choices. 
 Indeed, except for $\boldsymbol{\ell}_0$ and $\boldsymbol{\ell}_\infty$, the way the supports $\ell_i$'s are chosen is anything but constructive. Less important issues are 
 the choices of a  branch cut $B$ and of a 
 determination of $T$ on $U$, which are  not specified.  \sk

There is a way to remedy to this lack of determination by considering specific $z_i$'s.  Let us say that these are in {\bf (very) nice position} if 
\begin{quote}
{\it for every $i=1,\ldots,n-1$, the principal argument  of $\tilde z_i$  is (stricly) bigger than that of 
$\tilde z_{i+1}$.
 } 
\end{quote}
\vspace{0.15cm}

Remark that when $n=2$, the $z_i$'s are always in very nice position.
\sk

When the $z_i$'s are in very nice position,  there is no need to modify the $\tilde \ell_i$'s considered above since  they already satisfy all the required properties.  For the branch cut $B$, we take the union of a small deformation of $[0,\tilde z_1]$ with 
the segments $[\tilde z_i,\tilde z_{i+1}]$ for $i=2,\ldots,n-1$ (see 
Figure \ref{F:VeryNicePosition} just below). 
\begin{center}
\begin{figure}[!h]
\psfrag{B}[][][1]{$\textcolor{red}{B} $}
\psfrag{1}[][][1]{$1 $}
\psfrag{0}[][][1]{$0 $}
\psfrag{4}[][][1]{$z_2 $}
\psfrag{3}[][][1]{$z_3 $}
\psfrag{2}[][][1]{$z_4 $}
\psfrag{11}[][][1]{$z_5 $}
\psfrag{t}[][][1]{$\tau $}
\psfrag{g0}[][][1]{$\ell_{0} $}
\psfrag{l4}[][][1]{$\ell_2\;\; $}
\psfrag{l3}[][][1]{$\ell_{\!3} \, \;\;\; $}
\psfrag{l2}[][][1]{$\; \ell_4 \; $}
\psfrag{l1}[][][1]{$\; \ell_5 \;$}
\psfrag{gi}[][][1]{$\ell_{\!\!\infty} $}
\includegraphics[scale=0.8]{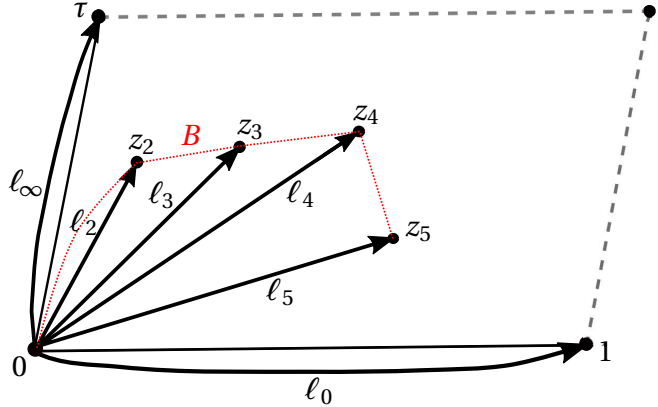}
\caption{The 1-cycles $\ell_{\!\!\bullet}$ for $\bullet=0,2,\ldots,n,\infty$ and the branch cut $B$,  for points $z_i$'s in very nice position.}
\label{F:VeryNicePosition}
\end{figure}\vspace{-0.3cm}
\end{center}

As to the choice of a determination of $T$ on $U$, let us remark that $\theta(\cdot,\tau)$ takes positive real values on $]0,1[$ for any purely imaginary modular parameter $\tau$. If   ${\rm Log}$ stands for the principal determination of the logarithm, one can define $\theta(u-z_i, \tau)^{\alpha_i}$ as $\exp(\alpha_i {\rm Log}\, \theta(u-z_i,\tau))$ on the intersection of the suitable translate of $V$ with a  disk centered at $z_i$ and of very small radius, for any $i=1,\ldots,n$ (remember the normalization $z_1=0$).  By analytic continuation, one gets a global determination of this function on $V$. Now,  
since $\tau$ varies in the upper half-plane which is simply connected, there is no problem to perform analytic continuation with respect to this parameter in order to obtain 
a determination of the $\theta(u-z_i,\tau)^{\alpha_i}$'s,  hence of $T$ on $U$  for any $\tau$ in 
 $\mathbb H$. 
\sk

The $\ell_{\!\bullet}$'s as well as the chosen determination $T_U$ on $U$ being perfectly well determined, the same holds true for the twisted 1-cycles $\boldsymbol{\ell}_{\!\!\bullet}$'s and,  by extension, for the $\boldsymbol{l}_{\!\bullet}$'s (hence for the corresponding   twisted homology classes as well).  \sk 

Finally,  by continuous deformation of the $\ell_i$'s ({\it cf.}\;\cite[Remark\,({\bf 3.6})]{DeligneMostow}), one constructs canonical  twisted 1-cycles (and associated homology classes) $\boldsymbol{\ell}_{\!\bullet}$ and $\boldsymbol{l}_{\!\bullet}$  for points $z_i$'s only  supposed to be in  nice position (see the two pictures below).
\mk 

\begin{tabular}{lr}
\hspace{-0.4cm}\psfrag{B}[][][1]{$\textcolor{red}{B} $}
\psfrag{1}[][][1]{$1 $}
\psfrag{0}[][][1]{$0 $}
\psfrag{4}[][][1]{$z_2 $}
\psfrag{3}[][][1]{$z_3 $}
\psfrag{2}[][][1]{$z_4 $}
\psfrag{11}[][][1]{$\; z_5 $}
\psfrag{t}[][][1]{$\tau $}
\psfrag{g0}[][][1]{$\ell_{0} $}
\psfrag{l4}[][][1]{$\ell_2\;\; $}
\psfrag{l3}[][][1]{$\ell_{\!3} \, \;\;\; $}
\psfrag{l2}[][][1]{$\; \ell_4 \; $}
\psfrag{l1}[][][1]{$\; \ell_5 \;$}
\psfrag{gi}[][][1]{$\ell_{\!\!\infty}\; $}
\includegraphics[scale=0.6]{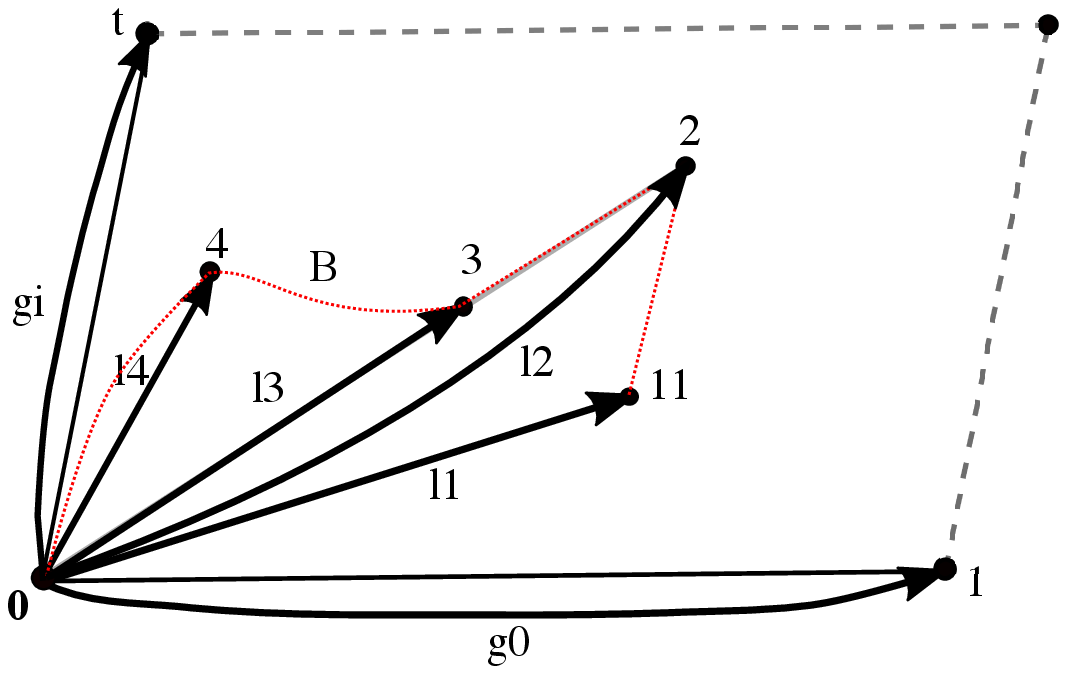}
& 
\psfrag{B}[][][1]{$\textcolor{red}{B} $}
\psfrag{1}[][][1]{$1 $}
\psfrag{0}[][][1]{$0 $}
\psfrag{4}[][][1]{$z_2 $}
\psfrag{3}[][][1]{$z_4 $}
\psfrag{2}[][][1]{$z_3 $}
\psfrag{11}[][][1]{$\; z_5 $}
\psfrag{t}[][][1]{$\tau $}
\psfrag{g0}[][][1]{$\ell_{0} $}
\psfrag{l4}[][][1]{$\ell_2\;\; $}
\psfrag{l3}[][][1]{$\ell_{\!3} \, \;\;\; $}
\psfrag{l2}[][][1]{$\; \ell_4 \; $}
\psfrag{l1}[][][1]{$\; \ell_5 \;$}
\psfrag{gi}[][][1]{$\ell_{\!\!\infty}\; $}
\includegraphics[scale=0.6]{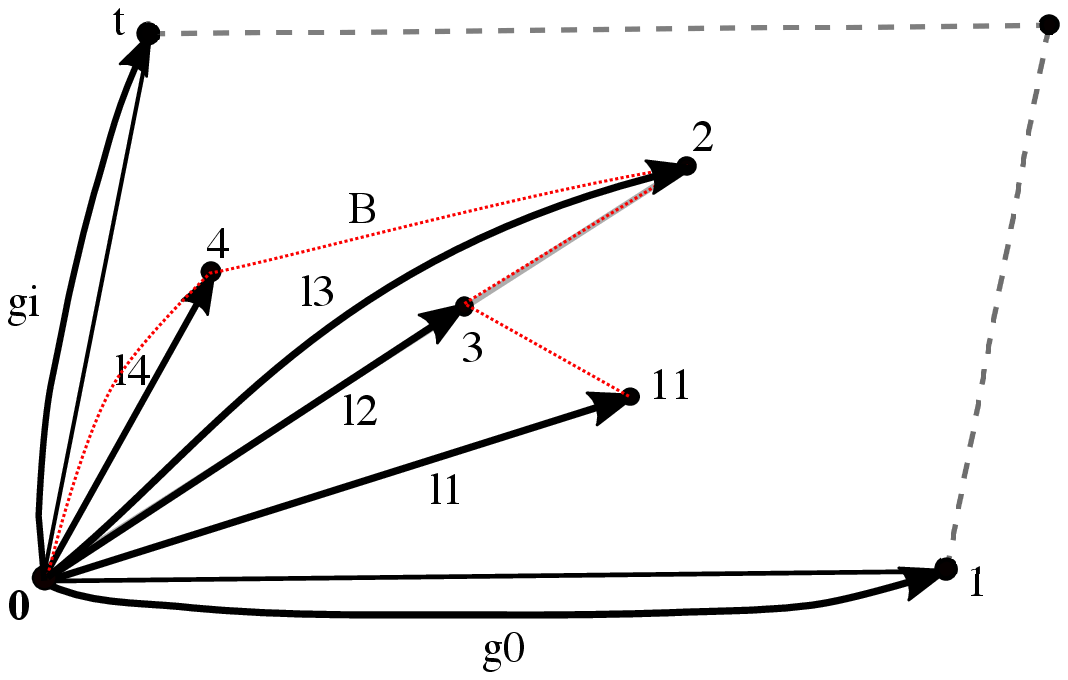}
\end{tabular}


\subsection{\bf Description of the first twisted (co)homology groups} 
In this subsection,   we follow \cite{ManoWatanabe} very closely and give  explicit descriptions of   the (co)homology groups ${H_1(E_{\tau,z},{ L}_\rho)}$ and ${H^1(E_{\tau,z},{ L}_\rho)}$.

In what follows, we assume that the  points $z_i$'s are in nice position.
 \mk

Recall that $\rho_{\!\bullet}=\exp(2i\pi\alpha_\bullet)$ for any $\bullet \in \{0,2,\ldots,n,\infty  \}$, with $\alpha_{\infty}$ given by \eqref{E:alphaInfty}.\footnote{It could be useful for  the reader to indicate here the relation between our $\alpha_\bullet$'s and the corresponding notations $c_\bullet $ used in \cite{Mano,ManoWatanabe}: one has 
$\alpha_j=c_j$ for $j=0,1,\ldots,n$ but $\alpha_\infty=-c_\infty$.} Given $m$ elements $\bullet_1,\ldots,\bullet_m$ of the set of indices   $\{0,2,\ldots,n,\infty  \}$, one sets: 
$$ 
\rho_{\bullet_1\ldots \bullet_m}= \rho_{\bullet_1}  \cdots \cdot \rho_{\bullet_m} \qquad \mbox{ and }
\qquad d_{\bullet_1\ldots \bullet_m}=\rho_{\bullet_1\ldots\bullet_m}-1\, .\sk 
$$


\subsubsection{\bf The first twisted homology group $\boldsymbol{H_1(E_{\tau,z},{ L}_{\rho})}$}

\paragraph{\bf }
Denote by $V$ the bounded simply connected open subset of $\mathbb C$ whose boundary is the topological closure of the union  of the $\ell_\bullet$'s for $\bullet$ in $\{0, 2,\ldots,n,\infty\}$ with the two translated cycles $1+\ell_0$ and $\tau+\ell_\infty$.   By analytic extension of the restriction 
of the determination $T_U$ of $T$ on $U$ 
in the vicinity of $1+\tau$, one gets a determination $T_V$ of $T$ on $V$. Considering now $V$ as 
an open subset of $E_{\tau ,z}$, one defines a locally-finite ${ L}_\rho$-twisted 2-chain\footnote{Strictly speaking, we do not define $\boldsymbol{\overline{V}}$ as a locally-finite 2-chain but there is a natural way to see it as such (by using similar arguments to the ones in footnote  \ref{FTN:NotFormally}  above).}
by setting
$$
\boldsymbol{\overline{V}}
={\overline{V}}\otimes T_V . 
$$

This is not a cycle: one verifies easily that the following relation holds true: 
\begin{align*}
\partial \boldsymbol{\overline{V}} = &\, {\boldsymbol{\ell}}_{0} +   \rho_0 {\boldsymbol{\ell}}_{\infty}-\rho_{\!\infty }{\boldsymbol{\ell}}_0-{\boldsymbol{\ell}}_{\infty}\\
 & + \, (\rho_n-1){\boldsymbol{\ell}}_n+\rho_n(\rho_{n-1}-1){\boldsymbol{\ell}}_{n-1}+\cdots+ 
 \rho_{3\cdots n}(\rho_2-1){\boldsymbol{\ell}}_2\, . 
\end{align*}

It follows that,  in $H_1^{\rm lf}(E_{\tau,z},{ L}_\rho)$, one has 
$$
-d_\infty\cdot  {\boldsymbol{\ell}}_{0}+d_0\cdot  {\boldsymbol{\ell}}_{\infty}+\sum_{k=2}^n    \frac{  d_k }{\rho_{1\cdots k}}\cdot  {\boldsymbol{\ell}}_k=0\, . 
$$

\paragraph{\bf }
\label{SS:NiceBasisTwistedHomologyClasses}
In order to construct a regularization map, we fix $\epsilon>0$. The constructions given below  are all  independent of $\epsilon$ (at the level of homology classes)    if the latter is supposed sufficiently small. Of course, we assume that it is the case in what follows. \sk 

For any $k=2,\ldots, n$,  let $\sigma_k: S^1\rightarrow \mathbb C$ be  a positively oriented parametrization of the circle  centered at $\tilde z_k$ and of radius $\epsilon$ such that the point $p_k=\sigma_k(1)$ is on the branch locus $B$.   The image of $]0,1[$ by $s_k= 
\sigma_k(\exp(2i\pi \cdot )) : [0,1]\rightarrow \mathbb C
$ is included in $U$,  hence $s_k^*(T_U)$ is well defined and extends continuously to the closure $[0,1]$.  Denoting  this extension by $T_k$, one defines a twisted 1-simplex in $E_{\tau,z}$ by setting 
$$
\boldsymbol{s}_k=[0,1]\otimes T_k\, .
$$

\paragraph{\bf }\hspace{-0.4cm}
Let $\varphi\in ]0,\pi[$ be the principal argument of $\tau$, set 
$I^0=[0,\varphi]$, 
$I^1=[\varphi,\pi]$, 
$I^2=[\pi,\pi+\varphi]$ and $I^3=[\varphi+\pi,2\pi]$ 
and  denote by $\sigma_0^\nu$ the restriction of 
$[0,2\pi]\rightarrow S^1, \, t\mapsto \epsilon \exp(it)$ to $I^\nu$ for $\nu=0,1,2,3$. 
We denote by $m^\nu$ the image of $\sigma_0^\nu$ viewed as a subset of $E_{\tau,z}$. 
These are  circular arcs the union of which  is a circle of radius $\epsilon$ centered at $0$ in $E_{\tau,z}$. 

In order to specify a determination of $T$ on each of the $m^\nu$, we are going to use the fact that each of them is also  the image in $E_\tau$ of a suitable translation of $\sigma_0^\nu$, the (interior of) the image of which is included in $U$.  More precisely, one sets $\widehat \sigma_0^\nu(\cdot)=\sigma_0^\nu(\cdot)+x^\nu$ for $\nu=1,\ldots,3$, with $x^1=1$, $x^2=1+\tau$ and $x^3=\tau$. 

For $\nu=1,2,3$, the image of the interior of $I^\nu$ by $\widehat \sigma_0^\nu$ is included in $U$. The restriction of $T_U$ to this image extends continuously to  $I^\nu$. Denoting these extensions by $T^\nu_U$, one defines twisted 1-simplices in $E_{\tau,z}$ by setting 
$$
\boldsymbol{m}^1=I^1 \otimes \big(\rho_0^{-1} T^1_U\big)\, , 
\qquad 
\boldsymbol{m}^2=I^2 \otimes \big(\rho_{0\infty}^{-1} T^2_U\big)\qquad 
\mbox{ and }
\qquad 
\boldsymbol{m}^3=I^3 \otimes \big(\rho_{\infty}^{-1} T^3_U\big). $$

Since the image of $\sigma_0^0$ meets the branch cut $B$, one cannot proceed as above in this particular case.  We use the fact that $p_0=\sigma_0^0(0)=\epsilon$ belongs to $U$. 
 Since $m^0\subset E_{\tau,z}$, the germ of $(\sigma_0^0)^*T_U$ at $0\in I^0$ extends to the  
  whole simplex $I^0$. Denoting  this extension by $T_U^0$, one defines a twisted 1-simplex in $E_{\tau,z}$ by setting 
$$
\boldsymbol{m}^0=I^0 \otimes T^0_U.  
$$

\paragraph{\bf }\hspace{-0.4cm}
For 
$k=2,\ldots,n$, let $\ell_{\!k}^\epsilon$ be the rectilinear segment  
linking $p_0$ to $p_k$ in $\mathbb C$: 
$\ell_{\!k}^\epsilon= [p_0, p_k]$. Setting $p_\infty=\sigma_0^0(\varphi)=\epsilon \tau$ and deforming the two segments $[p_0,1-p_0]=[\epsilon,1-\epsilon]$ and  $[p_\infty,\tau-p_\infty]=[\epsilon\tau,(1-\epsilon)\tau]$
by means of a function $\varphi$ as in \ref{SS:ellbullet}, one constructs two 1-simplices in $U$, denoted by $\ell_{\!0}^\epsilon$ and $\ell_{\!\infty}^\epsilon$ respectively.  
 
 For $\epsilon$ small enough,  the $\ell_{\!\!\bullet}$'s, $\bullet=0,2,\ldots,n,\infty$,  are pairwise disjoint and included in $U$,  hence  one defines  twisted 1-simplices in $E_{\tau,z}$ by setting 
$$
\boldsymbol{\ell}_{\!\bullet}^\epsilon=\ell_{\!\!\bullet}^\epsilon  \otimes \Big(T_U\lvert_{\ell_{\!\bullet}^\epsilon}\Big)
\,. 
$$

The 1-simplices $\ell_{\!\!\bullet}^\epsilon$ for $\bullet =0,2,\ldots,n,\infty$, $s_k$ for $k=2,\ldots,n$ and $m^\nu$ for $\nu=0,1,2,3$ are pictured in blue 
in Figure \ref{F:VeryNicePositionW} below (in the case when $n=3$).
\begin{center}
\begin{figure}[h]
\psfrag{U}[][][1]{$
\quad {U} $}
\psfrag{B}[][][1]{$\textcolor{red}{B} $}
\psfrag{1}[][][1]{$1 $}
\psfrag{0}[][][1]{$0 $}
\psfrag{3}[][][1]{$\textcolor{blue}{s_3} $}
\psfrag{2}[][][1]{$\textcolor{blue}{s_2} $}
\psfrag{t}[][][1]{$\tau $}
\psfrag{t1}[][][1]{$1+\tau $}
\psfrag{g0}[][][1]{$\textcolor{blue}{\ell_{\!0}^\epsilon} $}
\psfrag{K}[][][1]{$\textcolor{blue}{\ell_{\!2}^\epsilon }$}
\psfrag{L}[][][1]{$\textcolor{blue}{\ell_{\!3}^\epsilon} $}
\psfrag{gi}[][][1]{$\textcolor{blue}{\ell_{\!\infty}^\epsilon}\;\;  $}
\psfrag{pi}[][][1]{$\textcolor{blue}{p_\infty} \;\;\; $}
\psfrag{p0}[][][1]{$\textcolor{blue}{p_0}  $}
\psfrag{p2}[][][1]{$\textcolor{blue}{p_2} \;\;\;  $}
\psfrag{p3}[][][1]{$\textcolor{blue}{p_3}  \;\;  $}
\includegraphics[scale=0.77]{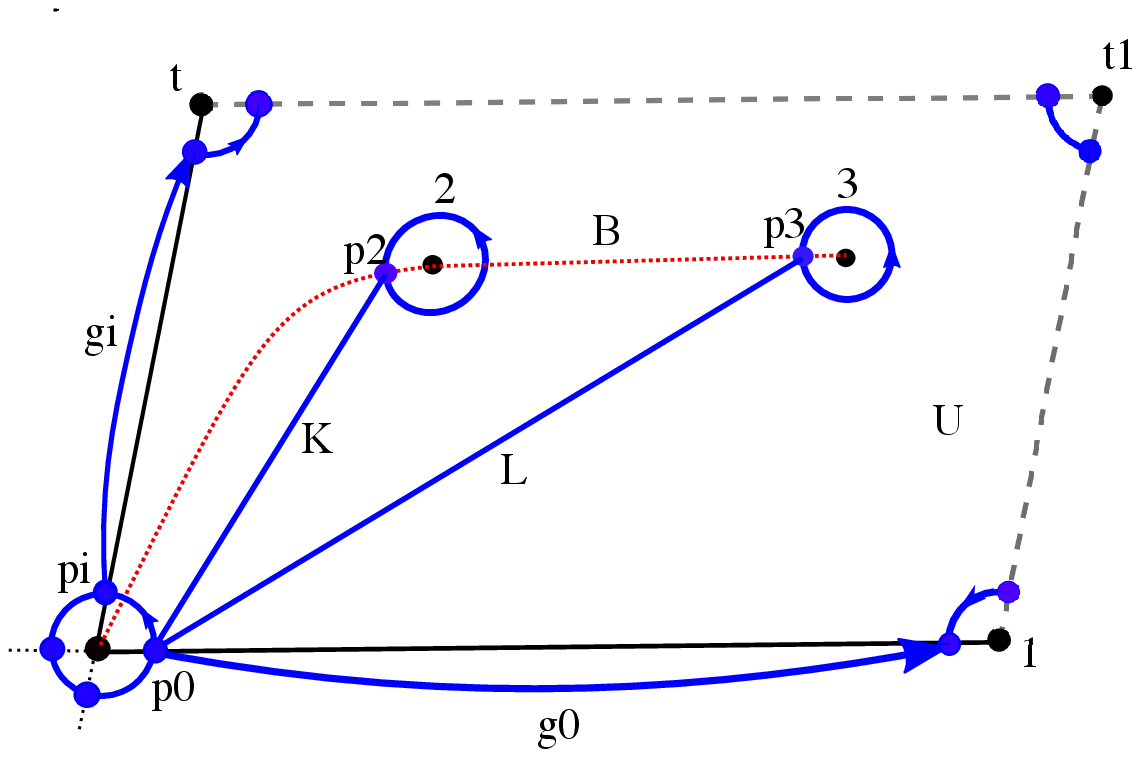}
\caption{}
\label{F:VeryNicePositionW}
\end{figure}
\end{center}

\paragraph{\bf }
Using the twisted 1-simplices constructed above, one defines ${ L}_\rho$-twisted 1-chains in $E_{\tau,z}$ by setting 
\begin{align}
\label{E:GammaBullet}
\boldsymbol{\gamma}_{\!\!\infty}=  & \,   \frac{1}{d_1}\Big[
\boldsymbol{m}^0+\boldsymbol{m}^1+\boldsymbol{m}^2+\boldsymbol{m}^3\Big]
 \hspace{-0.2cm}&& +  \boldsymbol{\ell}_{\!\infty}^\epsilon  &&  - \frac{\rho_\infty}{d_1}\Big[ \boldsymbol{m}^3+\boldsymbol{m}^0+\rho_1\big(\boldsymbol{m}^1+ \boldsymbol{m}^2\big)
\Big]\, , \nonumber 
\\
\boldsymbol{\gamma}_{\!\!0}=  & \, 
\frac{1}{d_1}\Big[
\boldsymbol{m}^0+\rho_1\big(\boldsymbol{m}^1+\boldsymbol{m}^2+\boldsymbol{m}^3\big)\Big]
\hspace{-0.2cm} && + \boldsymbol{\ell}_{\!0}^\epsilon && - \frac{\rho_0}{d_1}\Big[ \boldsymbol{m}^2+\boldsymbol{m}^3+\boldsymbol{m}^0+\rho_1 \boldsymbol{m}^1
\Big]\,\footnotemark \\
\mbox{and }\quad 
\boldsymbol{\gamma}_{\!\!k } = & \, \frac{1}{d_1}\Big[
\boldsymbol{m}^0+\rho_1\big(\boldsymbol{m}^1+\boldsymbol{m}^2+\boldsymbol{m}^3\big)\Big]
\hspace{-0.2cm} && +  \boldsymbol{\ell}_{\!k}^\epsilon  && - \frac{1}{d_k} \boldsymbol{s}_k
 \qquad \mbox{for }\,  k=2,\ldots,n. \nonumber
 \end{align}
\footnotetext{Note that here is a typo in the formula for $\boldsymbol{\gamma}_0$ in \cite{ManoWatanabe}. With the notation of the latter, 
the numerator of the coefficient of the term $(m_0+e^{2\pi\sqrt{-1}c_1}m_1)$ appearing in the definition of $\gamma_0$ page 3877 should be ${1-e^{2\pi\sqrt{-1}c_0}}$ and not 
${1-e^{-2\pi\sqrt{-1}c_0}}$.}
\sk 

By straightforward computations (similar to the one in \cite[Example 2.1]{AomotoKita} for instance), one verifies that the $\boldsymbol{\gamma}_{\!\!\bullet}$'s  actually are 1-cycles hence define twisted homology classes in $H_1(E_{\tau,z}, 
{ L}_\rho)$.   We will again use $\boldsymbol{\gamma}_{\!\!\bullet}$ 
to designate the corresponding twisted homology classes. It is quite clear that these 
 do not depend on $\epsilon$.   In particular,  this justifies that $\epsilon$ does not appear in the notation $\boldsymbol{\gamma}_{\!\!\bullet}$. 
\sk 

When the $z_i$'s are in nice position, the following proposition holds true: 
\begin{prop}[Mano-Watanabe \cite{ManoWatanabe}]
\label{P:BasisTwistedHomology} ${}^{}$
\begin{enumerate} 
\item  The map ${\rm reg}: \boldsymbol{\ell}_{\!\bullet}\mapsto  {\rm reg}(\boldsymbol{\ell}_{\!\bullet})=\boldsymbol{\gamma}_{\!\!\bullet}$ is a regularization map: at the homological level, it induces the 
isomorphism $H_1^{\rm lf}(E_{\tau,z},{ L}_\rho)\simeq H_1(E_{\tau,z},{L}_\rho)$.
\mk 
\item 
The homology classes $\boldsymbol{\gamma}_\infty,\boldsymbol{\gamma}_0,\boldsymbol{\gamma}_2,\ldots,\boldsymbol{\gamma}_n$  satisfy the following 
 relation: 
\begin{equation}
\label{E:LinearRelationGammai}
-d_\infty\, \boldsymbol{\gamma}_0+d_0\boldsymbol{\gamma} _\infty
+ \sum_{k=2}^n \frac{d_k}{\rho_{1\ldots k}}\, \boldsymbol{\gamma}_k=0\, .
\end{equation}
\item 
The twisted homology group 
$H_1(E_{\tau,z}, { L}_\rho)$ is of dimension $n$ and admits $$\boldsymbol{\gamma}=\big(\boldsymbol{\gamma}_\infty, \boldsymbol{\gamma}_0, \,  \boldsymbol{\gamma}_3, \ldots, \boldsymbol{\gamma}_n \big) $$ as a basis.
\end{enumerate} 
\end{prop} 

%

Since $T$ does not vanish on $E_{\tau,z}$, all the preceding constructions can be done with replacing $T$ by its inverse $T^{-1}$.  The regularizations 
$\boldsymbol{\gamma}_\bullet^{\vee}={\rm reg}(\boldsymbol{l}_{\!\bullet})$ 
of  the ${ L}^\vee_{\rho}$-twisted 1-cycles $\boldsymbol{l}_{\!\bullet}$ defined at the end of \S\ref{SS:ellbullet}  are defined by the same formulae than \eqref{E:GammaBullet} but  with 
 replacing $\rho_\bullet$ by $\rho_\bullet^{-1}$ for  $\bullet=0,2,\ldots,n,\infty $. 
 Then 
 $$
 \boldsymbol{\gamma}^{\vee}=\big(\boldsymbol{\gamma}_\infty^{\vee}, \boldsymbol{\gamma}_0^{\vee}, \,  \boldsymbol{\gamma}_3^{\vee}, \ldots, \boldsymbol{\gamma}_n^{\vee} \big)
$$ is a basis of $H_1(E_{\tau,z},{ L}_{\rho}^\vee)$.


\subsubsection{\bf The first cohomology group $\boldsymbol{H^1(E_{\tau,z},{ L}_{\rho})}$}
\label{SS:TwistedCohomologyH1}
In \cite{ManoWatanabe}, the authors give a  very detailed treatment  of the material described above in \S\ref{S:AlgebraicdeRhamComparisonTHM} in the case of a punctured elliptic curve. In particular, they show that in this case, it is not possible to use only logarithmic differential forms to describe $H^1(E_{\tau,z},{ L}_{\rho})$. 
\mk 

We continue to use the previous notation. In \cite[Proposition 2.4]{ManoWatanabe}, the authors prove that 
$H^1(E_{\tau,z},{ L}_{\rho})$ is isomorphic to the quotient of $H^0(E_{\tau,z},\Omega^1_{E_{\tau,z}})$ by the image by $\nabla_{\omega}$ of the space of  holomorphic functions on $E_{\tau,z}$ ({\it cf.}\;\eqref{E:popol}).  Then they give a direct proof of the twisted algebraic de Rham comparison theorem  
({\it cf.}\;\cite[Proposition 2.5]{ManoWatanabe}, see also \S\ref{S:AlgebraicdeRhamComparisonTHM} above) which asserts that one can 
 consider only rational objects on $E_\tau$ (but with poles at the $[z_i]$'s).  
\sk 

Viewing   $Z=\sum_{i=1}^n[z_i]$  as a divisor on  $E_\tau$, one has 
\begin{equation}
\label{E:gogo}
H^1\big(E_{\tau,z},{L}_{\rho}\big)\simeq 
\frac{H^0\big(E_{\tau},\Omega^1_{E_{\tau}}(*Z)\big)}
{\nabla_\omega \big(H^0\big(E_{\tau},\mathcal O_{E_{\tau}}(*Z)\big)\big)}
\end{equation}
(recall that,  with our notations, $H^0(E_{\tau},\mathcal O_{E_{\tau}}(*Z))$ 
 (resp.\;$H^0(E_{\tau},\Omega^1_{E_{\tau}}(*Z))$) stands for the space of rational functions (resp.\;1-forms) on $E_\tau$ with poles  only at the $z_i$'s.\sk 

We now consider the non-reduced  divisor $Z'=Z+[0]=2[0]+\sum_{k=2}^n[z_k]$.  There is a natural map 
from  the space 
$H^0(E_{\tau},\Omega^1_{E_{\tau}}(Z'))$ 
of rational 1-forms on $E_\tau$ with poles at most $Z'$ to the right hand  quotient space of \eqref{E:gogo}:
\begin{equation}
\label{E:gogoo}
H^0\big(E_{\tau},\Omega^1_{E_{\tau}}\big(Z'\big) \big)\longrightarrow 
\frac{H^0\big(E_{\tau},\Omega^1_{E_{\tau}}(*Z)\big)}
{\nabla_\omega \big(H^0\big(E_{\tau},\mathcal O_{E_{\tau}}(*Z)\big)\big)}\; . 
\end{equation}

One of the main results of \cite{ManoWatanabe} is Theorem 2.7 which says that the preceding map is surjective with a kernel of dimension 1.  \sk 

It is not difficult to see that, as a 
 vector space, $H^0\big(E_{\tau},\Omega^1_{E_{\tau}}\big(Z'\big) \big)$ is spanned  by 
\begin{align*}
\varphi_0= & \, du\, ,  \\
\varphi_1= & \, \rho'(u) du \\
\mbox{and }\quad \varphi_j= & \, 
\big(   \rho(u-z_j)-\rho(u)  \big)
\, du\quad \quad \mbox{for }\, j=2,\ldots,n. 
\end{align*}

Remark that all these forms are logarithmic on $E_\tau$ ({\it i.e.}\;have at most poles of the first order), at the exception of  $\varphi_1$ which has a pole of order 2 at the origin. \sk 

On the other hand, $1$ is holomorphic  on $E_\tau$ and,  according to \eqref{E:dLogT}, one has: 
\begin{equation}
\label{E:RELATION}
\nabla_\omega(1)=\omega=2i\pi\alpha_0\cdot \varphi_0+\sum_{j=2}^n\alpha_j\cdot \varphi_j\, .
\footnotemark
\end{equation}
\footnotetext{Even if we are interested only in the case when $\lambda=0$, we  mention here that the general formula given in \cite[Remark 2.8]{ManoWatanabe} is not correct.  Setting, as in \cite{ManoWatanabe}, $\mathfrak s(u)=\mathfrak s(u;\lambda)=\theta'\theta(u-\lambda)/(\theta(u)\theta(\lambda))$ 
for $u, \lambda\in \mathbb C\setminus \mathbb Z_\tau$, 
 the correct formula when $\lambda\neq 0$ is \vspace{-0.2cm}
$$ \nabla_\omega(1)=
\Big[2i\pi \alpha_0  -\alpha_1\rho(\lambda)+ \sum_{j=2}^n \alpha_j\big(    
\mathfrak s(z_j)-\rho(z_j)\big)
 \Big] \cdot \varphi_0+\lambda (\alpha_1-1)\cdot \varphi_1-\lambda\sum_{j=2}^n \alpha_j \, \mathfrak s(z_j)\cdot \varphi_j\, .$$}
 
One then deduces the following description of the cohomology group we are interested in: 
\begin{thm} 
\label{T:DescriptionTwistedH1}
\begin{enumerate}
\item Up to the isomorphisms \eqref{E:gogo} and \eqref{E:gogoo}, the space 
$H^1(E_{\tau,z},{ L}_{\rho})$ is identified with the space  spanned by the rational 1-forms  $\varphi_m $ for  $m=0,1,2,\ldots,n$, 
modulo the relation 
$
0= 2i\pi \alpha_0 \, {\varphi}_0+\sum_{k=2}\alpha_k \,{\varphi}_k
$, {\it i.e.}: 
$$
H^1\big(E_{\tau,z},{ L}_{\rho}\big)\simeq \frac{  \bigoplus_{m=0}^n  \mathbb C \cdot   {\varphi}_m}{\big\langle 
2i\pi\alpha_0  \, {\varphi}_0+\sum_{k=2}\alpha_k \, {\varphi}_k
\big\rangle}\, . 
$$

\item In particular when $n=2$, up to the preceding isomorphism, the respective classes $[{\varphi}_0]$ and $[{\varphi}_1 ]  $ of $d\!u$ and $\rho'(u)d\!u   $ form a basis of 
$
H^1(E_{\tau,z},{ L}_{\rho})$. 
\end{enumerate}
\end{thm}


\subsection{\bf The twisted intersection product}
\label{S:IntersectionProduct}
It follows from  Lemma \ref{L:rhoexplicit} that 
the monodromy character $\rho$ is unitary if and only if the quantity $\alpha_\infty$ defined in \eqref{E:alphaInfty} is real.  Starting from now on, we assume that it is indeed the case.


Since $\rho$ is unitary, 
 the constructions of \S\ref{S:GeneralTwistedIntersectionProduct} apply. We want to make them completely explicit.  More precisely, we want to express the intersection 
 product \eqref{E:TwistedIntProduct} in the basis $\boldsymbol{\gamma}$, {\it i.e.}\;we want to compute the coefficients of the following intersection matrix: 
 $$
\qquad I\!\!I_\rho=\big(  \boldsymbol{\gamma}_{\!\!\circ}\cdot  \overline{\boldsymbol{\gamma}}_{\!\!\bullet}\big)_{\circ,\bullet=\infty,0,3,\ldots,n}\, . 
$$

 Since $\rho$ is unitary,  $\rho^{-1}=\overline{\rho}$,  hence for any $\bullet$,  
 $\overline{\boldsymbol{\gamma}}_{\!\!\bullet}$ is the regularization of the locally finite $ L$-twisted 1-cycle $\boldsymbol{l}_\bullet$ and consequently 
$\boldsymbol{\gamma}_{\!\!\circ}\cdot  \overline{\boldsymbol{\gamma}}_{\!\!\bullet}=
\boldsymbol{\gamma}_{\!\!\circ}\cdot  {\boldsymbol{l}}_{\!\bullet}$ for every $
\circ,\bullet$ in $\{\infty,0,2,3,\ldots,n\}$.   Using the method explained in \cite{KitaYoshida}, it is just a computational task to determine these twisted  intersection  numbers. 
\sk 

Assuming that the $z_i$'s are in nice position, one has the 
\begin{prop} 
\label{P:GammaEll}
For $i=2,\ldots,n$, $j=2,\ldots,j-1$ and $k=i+1,\ldots,n$,  one has: 
\begin{align*}
\boldsymbol{\gamma}_\infty \cdot \boldsymbol{l}_\infty=   \, & 
\frac{d_\infty d_{1\infty}}{d_1\rho_\infty}
        && 
        \boldsymbol{\gamma}_i\cdot \boldsymbol{l}_\infty =
-\frac{\rho_1d_\infty}{\rho_\infty d_1}
         \\ \mk
\boldsymbol{\gamma}_\infty \cdot \boldsymbol{l}_0= \, &   
\frac{1-\rho_0+\rho_{0\infty}-\rho_{1\infty}}{ \rho_0 d_1}  && 
\boldsymbol{\gamma}_i\cdot \boldsymbol{l}_0 = 
 -\frac{\rho_1 d_0}{\rho_0 d_1}
\\ \mk
\boldsymbol{\gamma}_\infty   \cdot \boldsymbol{ l}_i= & \ 
\frac{d_\infty}{d_1}
 &&  
 \boldsymbol{\gamma}_j\cdot \boldsymbol{l}_i=  - \frac{\rho_{1}}{d_1}    
 \\  \mk 
\boldsymbol{\gamma}_0\cdot \boldsymbol{l}_\infty=& \,
\frac{\rho_1-\rho_{1\infty}     -\rho_0+\rho_{01\infty} }{\rho_\infty d_1} && \boldsymbol{\gamma}_i\cdot \boldsymbol{l}_i=
- \frac{d_{1i}}{d_1d_i}
 \\ \mk
\boldsymbol{\gamma}_0\cdot \boldsymbol{l}_0=& \,
\frac{d_0}{d_1}\left(1-\frac{\rho_1}{\rho_0}\right)
 &&  \boldsymbol{\gamma}_k\cdot \boldsymbol{l}_i=   - \frac{1}{d_1}\\ \mk
\boldsymbol{\gamma}_0\cdot \boldsymbol{l}_i=& \,    \frac{d_0}{d_1} \, . 
\end{align*}
\end{prop}
\begin{proof} Let  ${\boldsymbol{\alpha}}$ and ${\boldsymbol{a}}$ 
stand for the classes,  in $H_1(E_{\tau,z},{ L^\vee})$ and $H_1^{\rm lf}(E_{\tau,z},{ L})$ respectively, of two  twisted $1$-simplices denoted somewhat abusively by the same notation.  Denote respectively  by $\alpha$ and $a$ the supports of these twisted cycles and let 
$T_\alpha$ and $T_a$  be  the two determinations of $T$  along $\alpha$ and  $a$ respectively, such that 
 $${\boldsymbol{\alpha}}=\alpha\otimes T_\alpha
\qquad \mbox{ and }\qquad  
 {\boldsymbol{a}}=a\otimes T_a^{-1}.$$ 

Since the intersection number ${\boldsymbol{\alpha}}\cdot {\boldsymbol{a}}$ depends only on the associated twisted homotopy classes and because $\alpha$ is a compact subset of $E_{\tau,z}$,  one can assume that the topological 1-cycles  $\alpha$  and $a$ are smooth and  intersect transversally in a finite number of points. 
As explained in \cite{KitaYoshida}, ${\boldsymbol{\alpha}}
\cdot  
{\boldsymbol{a}} $ is equal to the sum of the local intersection numbers at the intersection points of the supports $\alpha$ and $a$ of the two considered twisted 1-simplices. In other terms, 
one has 
$${\boldsymbol{\alpha}}
\cdot  
{\boldsymbol{a}} =\sum_{x\in \alpha\cap a} \langle {\boldsymbol{\alpha}}
\cdot  
{\boldsymbol{a}}\rangle_x$$
 where for any intersection point $x$ of $\alpha$ and $a$, the twisted local intersection number  
 $\langle 
{\boldsymbol{\alpha}}, {\boldsymbol{a}} \rangle_x$ 
is defined as the  product of the usual topological local intersection number $\langle \alpha,a\rangle_x \in \mathbb Z$ with the complex ratio $T_\alpha(x)/T_a(x)$, {\it i.e.}  
 $$
 \left\langle 
{\boldsymbol{\alpha}}, {\boldsymbol{a}} \right\rangle_x=
 \left\langle 
\alpha\otimes T_\alpha, a\otimes T_a^{-1}
 \right\rangle_x
 =\big\langle \alpha,a\big\rangle_x\cdot 
T_\alpha(x)T_a(x)^{-1} 
 \in \mathbb C.
 $$

With the preceding result at hand, 
determining all the intersection numbers of the proposition is just a computational task.  We will detail only one case below. The others can be computed in a similar way.\footnote{Some of these computations can be considered as classical since they  already appear  in the existing literature, such as the one of $\boldsymbol{\gamma}_i\cdot \boldsymbol{l}_i$ for $i=2,\ldots,n $, which follows immediately from the computations given p. 294 of \cite{KitaYoshida} (see also \cite[\S2.3.3]{AomotoKita}).}. 
\mk 

As an example, let us detail the  computation of  $\boldsymbol{\gamma}_\infty \cdot \boldsymbol{l}_\infty$. The picture below is helpful for this. On it, the 1-cycle 
$\boldsymbol{\gamma}_{\!\!\infty}$ has been drawn in blue whereas the locally finite 1-cycle 
$\boldsymbol{l}_{\infty}$ is pictured in green. 

\begin{center}
\begin{figure}[!h]
\psfrag{U}[][][1]{$
\quad {U} $}
\psfrag{B}[][][1]{$\textcolor{red}{B} $}
\psfrag{Bt}[][][1]{$\;\, \quad \textcolor{red}{B+\tau} $}
\psfrag{1}[][][1]{$1 $}
\psfrag{t}[][][1]{$\tau $}
\psfrag{x1}[][][1]{$\textcolor{green}{x_1} \; $}
\psfrag{x2}[][][1]{$\textcolor{green}{x_2}\; \,  $}
\psfrag{y1}[][][1]{$\textcolor{green}{y_1} \;\;  $}
\psfrag{y2}[][][1]{$\textcolor{green}{y_2}\;  $}
\psfrag{linf}[][][1]{$\textcolor{green}{l'_{\infty}}\;\;\;\;  $}
\psfrag{GI}[][][1]{$\;\;\;\textcolor{blue}{\gamma_{\infty}^i} $}
\psfrag{GE}[][][1]{$\textcolor{blue}{\gamma_{\infty}^\epsilon}  $}
\psfrag{GF}[][][1]{$\textcolor{blue}{\gamma_{\infty}^f} $}
\includegraphics[scale=0.77]{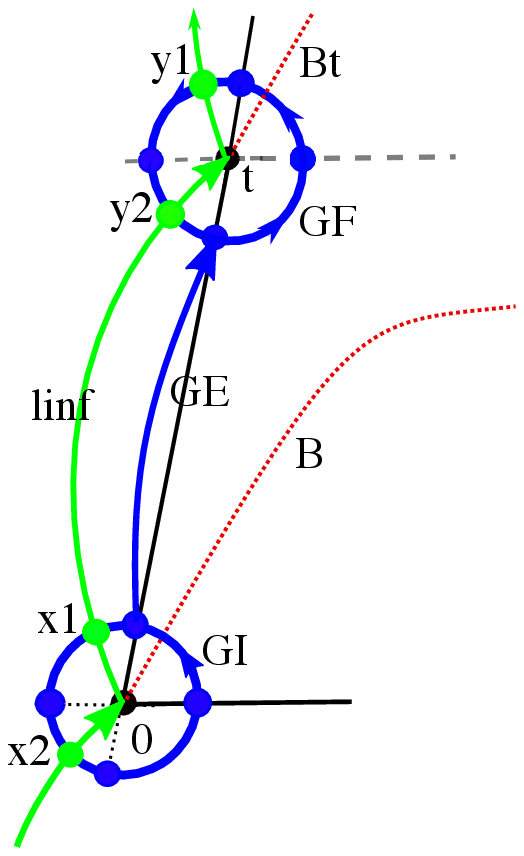}
\end{figure}
\end{center}
The picture shows that ${l}_{\infty}$ does not meet $\gamma_\infty^\epsilon$ and 
intersects $\gamma_\infty^i$ and $\gamma_\infty^i$ at the points $x_1,x_2$ and $y_1,y_2$ respectively.  

It follows that 
\begin{align}
\label{E:powpowpow}
\boldsymbol{\gamma}_\infty \cdot \boldsymbol{l}_\infty
=  &  \boldsymbol{\gamma}_\infty^i \cdot \boldsymbol{l}_\infty+
\boldsymbol{\gamma}_\infty^f \cdot \boldsymbol{l}_\infty
\nonumber
\\
= & \sum_{k=1}^2 \big\langle \boldsymbol{\gamma}_\infty^i \cdot \boldsymbol{l}_\infty
\big\rangle_{x_k}+ \sum_{k=1}^2 \big\langle \boldsymbol{\gamma}_\infty^f \cdot \boldsymbol{l}_\infty
\big\rangle_{y_k}
\nonumber
\\
= & \frac{1}{d_1} 
\sum_{k=1}^2
 \big\langle \boldsymbol{m}^k \cdot \boldsymbol{l}_\infty
\big\rangle_{x_k} -
\frac{\rho_{1\infty}}{d_1} 
\sum_{k=1}^2
 \big\langle \boldsymbol{m}^k \cdot \boldsymbol{l}_\infty
\big\rangle_{y_k}\, , 
\end{align}
the last equality coming from the formula for $\boldsymbol{\gamma}_\infty^i$ and $\boldsymbol{\gamma}_\infty^f$ and from 
the fact that $x_k,y_k\in m^k$ for $k=1,2$. 

The topological intersection numbers are the following: 
$$
\big\langle {m}^1 \cdot {l}_\infty
\big\rangle_{x_1}=\big\langle {m}^1 \cdot {l}_\infty
\big\rangle_{y_1}=-1
\quad \mbox{ and }\quad 
\big\langle {m}^2 \cdot {l}_\infty
\big\rangle_{x_2}=\big\langle {m}^2 \cdot {l}_\infty
\big\rangle_{y_2}=1\; . 
$$

It is then easy to compute the four  intersection numbers appearing in 
\eqref{E:powpowpow}: 
\begin{itemize}
\item let $\zeta_1$ stand for $x_1$ or $y_1$. The determination of $T$ associated to $\boldsymbol{l}_\infty$ at $\zeta_1$ is the same as the one associated to $\boldsymbol{m}^1$ at this point. It follows that 
$$\langle \boldsymbol{m}^1 \cdot \boldsymbol{l}_\infty \rangle_{\zeta_1}=
\langle {m}^1 \cdot {l}_\infty
\big\rangle_{\zeta_1}=-1\; ;
$$
\item 
let $\zeta_2$ stand for $x_2$ or $y_2$.
The determination of $T$ associated to $\boldsymbol{l}_\infty$ at $\zeta_2$ is 
$\rho_{\infty}$ times  the one associated to $\boldsymbol{m}^2$ at this point. It follows that 
$$\langle \boldsymbol{m}^2 \cdot \boldsymbol{l}_\infty \rangle_{\zeta_2}=
\langle {m}^2 \cdot {l}_\infty
\big\rangle_{\zeta_2} \cdot \rho_{\infty}^{-1}=  \rho_{\infty}^{-1}\, . 
$$
\end{itemize}

From all the preceding considerations, it comes that 
\begin{equation*}
\label{E:powpowpow}
\boldsymbol{\gamma}_\infty \cdot \boldsymbol{l}_\infty
=  \frac{1}{d_1} \bigg[  -1+\rho_\infty^{-1}  \bigg]
 -
\frac{\rho_{1\infty}}{d_1} 
\bigg[  -1+\rho_\infty^{-1}  \bigg]
 =\frac{d_\infty d_{1\infty}}{\rho_\infty d_1}\, . 
 \qedhere
 \end{equation*}
\end{proof}



\subsection{\bf The particular case $n=2$}
In this case, the complete intersection matrix is
\begin{equation*}
I_\rho= \big( 
\boldsymbol{\gamma}_\circ \cdot {\boldsymbol{l}}_\bullet
\big)_{\circ,\bullet=0,\infty,2}
=
\begin{bmatrix} 
\frac{d_{\infty}d_{1\infty}}{d_1\rho_\infty }&  
\boldsymbol{\gamma}_{\!\!\infty} \cdot {\boldsymbol{l}}_0     &   \frac{d_\infty}{d_1} \\
\boldsymbol{\gamma}_0 \cdot {\boldsymbol{l}}_\infty  &  \frac{d_0}{d_1}\big( 1- \frac{\rho_1}{\rho_0}  \big)   & \frac{d_0}{d_1} \\
 -\frac{\rho_1 d_\infty }{\rho_\infty d_1} &-\frac{\rho_1 d_0 }{\rho_0 d_1}      &  0
\end{bmatrix}
\end{equation*}
with 
\begin{align*}
\boldsymbol{\gamma}_\infty \cdot {\boldsymbol{l}}_0  = \, &   
-\frac{1}{d_1}+\frac{\rho_0^{-1}}{d_1} +\frac{\rho_\infty}{d_1}-\frac{\rho_1\rho_\infty\rho_0^{-1}}{d_1}\\ 
\mbox{ and }\quad 
\boldsymbol{\gamma}_0 \cdot {\boldsymbol{l}}_\infty  =& \,
-\frac{\rho_1}{d_1}+\frac{\rho_0\rho_1}{d_1}
+ \frac{\rho_1\rho_\infty^{-1}}{d_1}
-\frac{\rho_0\rho_\infty^{-1}}{d_1}\; .
\end{align*}

The linear relation between the  twisted 1-cycles ${\boldsymbol{\gamma}}_0, {\boldsymbol{\gamma}}_\infty $ and ${\boldsymbol{\gamma}}_2$ is 
$$
(1-\rho_\infty){\boldsymbol{\gamma}}_0-(1-\rho_0){\boldsymbol{\gamma}}_\infty= (1-\rho_2){\boldsymbol{\gamma}}_2\; . 
$$
Thus, since $\rho_2\neq 1$, one can express ${\boldsymbol{\gamma}}_2$ in function 
of $ {\boldsymbol{\gamma}}_0$ and ${\boldsymbol{\gamma}}_\infty$
as follows: 
\begin{equation}
\label{E:gamma2}
{\boldsymbol{\gamma}}_2=    -\rho_1\frac{d_\infty}{d_1}\, {\boldsymbol{\gamma}}_0
+\rho_1\frac{d_0}{d_1}\, {\boldsymbol{\gamma}}_\infty\, . 
\end{equation}

The intersection matrix relative to the basis $(\boldsymbol{\gamma}_0, \boldsymbol{\gamma}_\infty)$ and
$({\boldsymbol{l}\, }_0, {\boldsymbol{l}}_\infty)$ is 
\begin{equation*}
I\!\!I_\rho=
\begin{bmatrix}  
\boldsymbol{\gamma}_\infty \\
  \boldsymbol{\gamma}_0
\end{bmatrix}\cdot 
\begin{bmatrix}
{\boldsymbol{l}}_\infty & 
{\boldsymbol{l}}_0
\end{bmatrix}
=
\begin{bmatrix}    \frac{d_\infty d_{1\infty}}{d_1\rho_\infty}
& 
 \frac{-1+\rho_0^{-1} +\rho_\infty-\rho_1\rho_\infty\rho_0^{-1}}{d_1}
  \\ 
  \frac{-\rho_1+\rho_0\rho_1+\rho_1\rho_\infty^{-1}
-\rho_0\rho_\infty^{-1}}{d_1}  
  &   
   \frac{d_0}{d_1}\Big(
   1-\frac{\rho_1}{\rho_0}
  \Big)   
\end{bmatrix}.\sk
\end{equation*}

By a direct computation, one verifies that the determinant of this  anti-hermitian matrix is always  equal to 1,  hence this matrix is  invertible. 
 Then one can consider 
 \begin{equation}
  \label{E:Hrho}
I\!\!H_\rho=\big({2i}I\!\!I_\rho\big)^{-1}=\frac{1}{2i}  
\begin{bmatrix}   \frac{d_0}{d_1}\Big(
   1-\frac{\rho_1}{\rho_0}
  \Big)
  &  
 \frac{\rho_0-1 -\rho_{0}\rho_{\infty}+\rho_{1}\rho_{\infty}}{\rho_0 d_1}
  \sk  \\ 
    \frac{\rho_0-\rho_1-\rho_0\rho_1\rho_\infty
+\rho_1\rho_\infty}{\rho_\infty d_1}  
   &  
    \frac{d_\infty d_{1\infty}}{d_1\rho_\infty} 
    \end{bmatrix}\, . 
\end{equation}

This matrix is hermitian and its determinant is $-1/4<0$. It follows that the signature of the  hermitian form associated to  $I\!\!H_\rho$ is $(1,1)$, as expected.

\subsubsection{\bf Some connection formulae}
\label{SS:ConnectionFormulae}
We now let the parameters $\tau$ and $z$ vary. More precisely, let 
$f: [0,1]\rightarrow  \mathbb H\times \mathbb C^n, \,  s\mapsto (\tau(s),z(s))$ be a smooth path in the corresponding parameter space:  for every $s\in [0,1]$,  $z_1(s)=0$ and  $z_1(s),\ldots,z_n(s)$ are pairwise distinct modulo $\mathbb Z_{\tau(s)}$.  
For $s\in [0,1]$, let  $\rho_s$ be the  corresponding monodromy morphism (namely, the one corresponding to the monodromy of  $T(\cdot,\tau(s),z(s))$) 
and denote by $ L_s=L_{\rho_s}$ the 
associated local system on $E_{\tau(s),z(s)}$.\sk 

Since the $E_{\tau(s),z(s)}$'s form a topologically trivial family of $n$-punctured elliptic curves over $[0,1]$, the corresponding twisted homology groups 
$H_1(E_{\tau(s),z(s)},L_s)$ organize themselves into a local system over $[0,1]$, which is necessarily trivial. 
If in addition
 the $z_i(0)$'s are in very nice position, then the  twisted 1-cycles $\boldsymbol{\gamma}_{\!\!\bullet}^0$ (for $\bullet=\infty,0,3,\ldots,n$) are 
well-defined and can be smoothly  deformed along $f$.  One obtains a  deformation 
parametrized by $s\in [0,1]$
  $$
 \qquad 
 \boldsymbol{\gamma}^s=
  \big(\boldsymbol{\gamma}_{\!\! \bullet}^s\big)_{\bullet=\infty,0,3,\ldots,n}
$$
  of the initial $\boldsymbol{\gamma}^0$'s,  such that the map $\boldsymbol{\gamma}_{\!\!\bullet}^0\mapsto \boldsymbol{\gamma}_{\!\!\bullet}^1$ 
 induces an isomorphism denoted by $f_*$ between the corresponding twisted homology spaces  $H_1(E_{\tau(0),z(0)},L_0)$ and $H_1(E_{\tau(1),z(1)},L_1)$.  
 It only   
   depends on the homotopy class of $f$. 
    Similarly, one constructs an analytic deformation $
  \boldsymbol{l}^s=
 (\boldsymbol{l}_{\infty}^s, \ldots, \boldsymbol{l}_{n}^s\big)
 $, $s\in [0,1]$.\sk 
 
  Let us suppose furthermore that the $z_i(1)$'s also are in very nice position. 
 Then 
let $\boldsymbol{\gamma}'=(\boldsymbol{\gamma}'_{\!\!\bullet})_{\bullet=\infty,0,3,\ldots,n}$ be the nice basis of $H_1(E_{\tau(1),z(1)}, L_{1})$ constructed in \S\ref{SS:NiceBasisTwistedHomologyClasses}.
The matrix of  $f_*: H_1(E_{\tau(0),z(0)}, L_{0})\simeq 
H_1(E_{\tau(1),z(1)}, L_{1})$ expressed in the nice bases $\boldsymbol{\gamma}^0$ and $\boldsymbol{\gamma}'$  is nothing else but  the $ n \times n$ matrix  $M_f$ such that 
\begin{equation}
\label{E:ConnectionFormula}
{}^{t}\! \boldsymbol{\gamma}'=M_f \cdot {}^{t}\! \boldsymbol{\gamma}^1\, . 
\end{equation}
 
 Such a relation is called a {\bf connection formula}.  In the particular case when $f$ is a loop, one has $(\tau',z')=(\tau,z)$ and such a formula appears as nothing else but a monodromy formula.\mk   
 
One verifies easily that the twisted intersection product is constant up to small deformations.  In particular, for any $\bullet, \circ$ in $\{\infty,0,2,\ldots,n\}$, the twisted intersection number $\boldsymbol{\gamma}_{\!\!\bullet}^s\cdot \boldsymbol{l}_{\!\circ}^s$ does not depend on 
$s\in [0,1]$, thus ${}^{t}\! \boldsymbol{\gamma}^1\cdot \boldsymbol{l}^1= I\!\!I_{\rho_0}$. Combining this  with  
 \eqref{E:ConnectionFormula}, it comes that the following matricial relation holds true:
 $$
 I\!\!I_{\rho_1}
 =M_f \cdot  
 I\!\!I_{\rho_0} 
 \cdot {}^{t} \overline{M}_{\!f}\, .\sk
 $$
 
 In what follows, we give several natural connection formulae in the case when $n=2$.  
  All these are particular cases  of the formulae given in  \cite[\S6]{Mano} 
 for $n\geq 2$ arbitrary.  Note that the reader will not find rigorous proofs of these formulae in \cite{Mano} but rather some pictures explaining what is going on.   However, with the help of these pictures and using similar arguments than  those of \cite[Proposition (9.2)]{DeligneMostow},  it is not too difficult to give rigorous proofs of the formulae below. Since it is rather long, it is left to the motivated reader.
 \mk 
 
 In what follows, the modular parameter $\tau \in \mathbb H$ is fixed as well as  $z=(z_1,z_2)$ which is a pair of points of $\mathbb C$ which are not congruent modulo $\mathbb Z_\tau$. As remarked before, 
 $z_1,z_2$ are in very nice position, hence  the twisted 1-cycles $\boldsymbol{\gamma}_{\!\!\bullet}$, $\bullet=\infty,0,2$ are well-defined. 
 To remain close to \cite{Mano}, we will not write that  $z_1$ is $0$ in the lines below, even if 
one can suppose that $z_1$ as been normalized in this way. 
\mk

\paragraph{\bf Half-Twist formula ``$\boldsymbol{z_1\longleftrightarrow
 z_2}$''} 
 We first deal with the connection formula associated to the (homotopy class of a) half-twist   exchanging $z_1$ and $z_2$ with $z_2$ passing above $z_1$ as pictured in red in Figure \ref{F:HTwist-Figure} below. This case is the one treated at the bottom of p. 15 of \cite{Mano}.\footnote{Note that there is a typo in the formula for the half-twist in page 15 of \cite{Mano}. With the formula given there, relation 
\eqref{E:VerifHalfTwist} does not hold true.} 
\sk

Setting  $z'=(z_2,z_1)$, there is a  linear isomorphism ${\rm HTwist}_\rho$ from $H_1(E_{\tau,z}, L_\rho)$ onto 
$H_1(E_{\tau,z'}, L_{\rho'})$ with $\rho'=R_{\rm HTwist}(\rho)$ where 
\begin{align*}
R_{\rm HTwist}: (\rho_\infty,\rho_0,\rho_1)\longmapsto \big(\rho_\infty, \rho_0, \rho_1^{-1}\big)\, . \end{align*}

Setting ${}^t\!\boldsymbol{\gamma}={}^t\!(\boldsymbol{\gamma}_\infty,\boldsymbol{\gamma}_0,\boldsymbol{\gamma}_2)$ and 
${\boldsymbol{l}}=({\boldsymbol{l}}_\infty, {\boldsymbol{l}}_0,{\boldsymbol{l}}_2)$
 with analogous notations 
for $\boldsymbol{\gamma}'$ and ${\boldsymbol{l}}'$
one has  
${}^t\!\boldsymbol{\gamma}'={\rm HTwist}_\rho \cdot {}^t\!\boldsymbol{\gamma}
$ and $
{\boldsymbol{l}}\, '=
{\boldsymbol{l}} \cdot {}^t \overline{{\rm HTwist}_\rho} $
 with 
\begin{equation}
\label{E:HalfTwist}
{\rm HTwist}_\rho= \begin{bmatrix} 
1  &    0   &  \frac{d_\infty}{\rho_1} \\
0  &  1     &  \frac{d_0}{\rho_1} \\
0 &  0  &    -\frac{1}{\rho_1}
\end{bmatrix}.
\end{equation}
\sk

\begin{center}
\begin{figure}[!h]
\psfrag{0}[][][1]{$z_2'=z_1 \qquad $}
\psfrag{z}[][][1]{$\qquad \quad z_2=z_1'  $}
\includegraphics[scale=0.77]{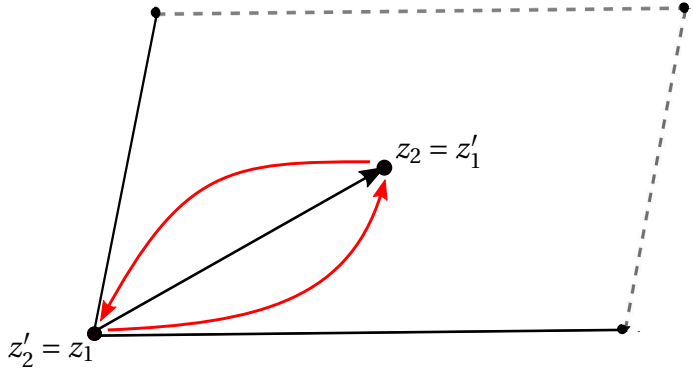}
\caption{Half-twist in the direct sense exchanging $z_1$ and $z_2$.}
\label{F:HTwist-Figure}
\end{figure}
\end{center}

Verification: one should have $I_{\rho'}=\boldsymbol{\gamma}'\cdot{\boldsymbol{l}}\,'= 
{\rm HTwist}_\rho\cdot \boldsymbol{\gamma} \cdot 
\boldsymbol{l} \cdot {}^t 
\overline{{\rm HTwist}_\rho}
= {\rm HTwist}_\rho\cdot I_\rho\cdot {}^t 
\overline{{\rm HTwist}_\rho}
 $ and, indeed,  
one verifies that the following relation  holds true: 
\begin{equation}
\label{E:VerifHalfTwist}
I_{\rho'}= {\rm HTwist}_\rho\cdot I_\rho\cdot {}^t 
\overline{{\rm HTwist}_\rho}
\, .
\end{equation}


\paragraph{\bf First horizontal translation  formula ``$\boldsymbol{z_1\longrightarrow z_1+1}$''} 
We now consider the connection formula associated to the path 
$$ 
f_{\rm HTrans1} : 
[0,1] \longrightarrow \mathbb H\times \mathbb C^2, \; 
 s\longmapsto \big(\tau, z_1+s,z_2\big)\, . 
$$

We  define $\widetilde \rho=R_{{\rm HTrans1}}(\rho)$ with 
\begin{align*}
R_{\rm HTrans1}: (\rho_\infty,\rho_0,\rho_1)\longmapsto (\rho_\infty\rho_1^{-1}, \rho_0,\rho_1)\; .
\end{align*}

We set  $\widetilde z=(z_1+1,z_2)=f_{\rm HTrans1}(1)$. The 
 path $f_{\rm HTrans1} $ gives us a  linear isomorphism  from $H_1(E_{\tau,z},  L_\rho)$ onto 
$H_1(E_{\tau,\widetilde z}, L_{\widetilde \rho})$ which will be denoted by ${\rm HTrans1}_\rho$. 

The corresponding connection matrix is
\begin{equation*}
{\rm HTrans1}_\rho= \begin{bmatrix} 
\frac{1}{\rho_1}  &    -\frac{d_\infty}{\rho_0\rho_1}   &  0 \\
0  &  \frac{1}{\rho_0}     &  0 \\
0 & \frac{1}{\rho_0}  &   \frac{1}{\rho_1}
\end{bmatrix}.
\end{equation*}

One verifies that  the following relation is satisfied: 
\begin{equation}
\label{E:HTwistTransfo}
I_{\widetilde\rho}={\rm HTrans1}_\rho\cdot I_\rho\cdot {}^t\overline{{\rm HTrans1}_\rho}.
\end{equation}



\paragraph{\bf Second horizontal translation  formula ``$\boldsymbol{z_2\longrightarrow z_2+1}$''} 
We now consider the connection formula associated to the path 
$$ 
f_{\rm HTrans2} : 
[0,1] \longrightarrow \mathbb H\times \mathbb C^2, \; 
 s\longmapsto \big(\tau, z_1,z_2+s\big)\, . 
$$

We  define 
$ \rho''=     R_{\rm HTwist} \circ R_{\rm HTrans}   \circ R_{\rm HTwist}    (\rho)$, 
 that is 
\begin{align*}
\rho''= \big(\rho_\infty'',\rho_0'',\rho_1''\big)= \Big(\rho_\infty\rho_1, \rho_0,\rho_1\Big)\, .
\end{align*}

We set  $z''=(z_1,z_2+1)=f_{\rm HTrans2}(1)$. The  map $f_{\rm HTrans2} $ gives us a  linear isomorphism  from $H_1(E_{\tau,z}, L_\rho)$ onto 
$H_1(E_{\tau,z''}, L_{\rho''})$, which will be denoted by ${\rm HTrans2}_\rho$. 

The corresponding connection matrix is
\begin{equation*}
{\rm HTrans2}_\rho= 
{\rm HTwist}_{\tilde \rho'}
\cdot
{\rm HTrans1}_{\rho'}
\cdot 
{\rm HTwist}_{\rho}
\end{equation*}
with $\tilde{\rho}'= R_{\rm HTrans}\circ R_{\rm HTwist}(\rho)$. 
 Explicitly,  one has 
\begin{equation*} 
{\rm HTrans2}_\rho = 
\begin{bmatrix} 
 \rho_1 &  \frac{\rho_{1\infty}d_1}{\rho_0}
       &  
       - \frac{d_1(\rho_{01\infty}+\rho_\infty-\rho_0)}{\rho_0}
        \\
 0 &   \frac{1+d_0\rho_1}{}       &         -\frac{d_0d_1(\rho_{01}+1)}{\rho_0\rho_1} \\
    0  &   -\frac{\rho_1}{\rho_0}    &     \frac{\rho_0d_1+1}{\rho_0} \\
\end{bmatrix} . 
\end{equation*}

We verify that  the following relation is satisfied: 
\begin{equation*}
\label{E:HTwistTransfo}
I_{\rho''}= {\rm HTrans2}_\rho\cdot I_\rho\cdot {}^t\overline{{\rm HTrans2}_\rho}.
\end{equation*}

The matrix ${\rm HT2}$ of the isomorphism ${\rm HTrans2}_\rho$ expressed in the basis $(\boldsymbol{\gamma}_\infty,\boldsymbol{\gamma}_0)$ and $(\boldsymbol{\gamma}''_\infty,\boldsymbol{\gamma}''_0)$ is more involved. 
 But since this formula will be used later (in Lemma \ref{L:ConnectionFormulaeTranslation}),  we give it explicitly below: 
\begin{equation}
\label{E:HTrans2rho22}
 {\rm HT2}_\rho= 
\begin{bmatrix} {\frac { \left( \rho_{{01\infty}}-\rho_0\rho_{{01\infty}}+\rho_{{\infty}}-\rho_{{0\infty}}+{\rho_{{0}}}^{2} \right) \rho_{{1}}}{\rho_{{0}}}}&
{\frac { \left( \rho_{{10\infty}}{\rho_{{\infty}}}-\rho_{{01\infty}}+\rho_{{1\infty}}-2\,\rho_{{\infty}}+\rho_{{0}}+{\rho_{{\infty}}}^{2}-\rho_{{0\infty}} \right) \rho_{{1}}}{\rho_{{0}}}}\mk
\\ 
-{\frac { (\rho_0-1) ^{2} \left( \rho_{{01}}+1 \right) }{\rho_{{0}}}}
&
{\frac {-\rho_{{01\infty}}+2\,\rho_{{01}}-\rho_{{1}}+
\rho_{{0}}
\rho_{{01\infty}}-{\rho_{{0}}}\rho_{{01}}-\rho_{{\infty}}+\rho_{{0\infty}}+2-\rho_{{0}}}{\rho_{{0}}}} \end{bmatrix} .
\end{equation}

This matrix satisfies the following relation: $I\!\!I_{\rho''}={\rm HT2}_\rho\cdot I\!\!I_\rho\cdot {}^t\overline{{\rm HT2}}_\rho$.



\sk 

\paragraph{{\bf First vertical translation  formula ``$\boldsymbol{z_1\longrightarrow z_1+\tau}$''}}
We now consider the connection formula associated to the path 
$$ 
f_{\rm VTrans1} : 
[0,1] \longrightarrow \mathbb H\times \mathbb C^2, \; 
 s\longmapsto \big(\tau, z_1+s\tau ,z_2\big)\, . 
$$

We define $\widehat \rho=R_{{\rm VTrans1}}(\rho)$ where $R_{{\rm VTrans1}}$ stands for the following  map: 
\begin{align*}
R_{\rm VTrans1}: (\rho_\infty,\rho_0,\rho_1)\longmapsto (\rho_\infty, \rho_0\rho_1,\rho_1)\; .
\end{align*}

We set  $\widehat z=(z_1+\tau,z_2)=f_{\rm VTrans1}(1)$. The map   $f_{\rm VTrans1} $ gives us a  linear isomorphism from $H_1(E_{\tau,z}, L_\rho)$ onto 
$H_1(E_{\tau,\widehat z}, L_{\widehat \rho})$ which will be denoted by  ${\rm VTrans1}_\rho$. 

The corresponding connection matrix is
\begin{equation*}
{\rm VTrans1}_\rho= \begin{bmatrix} 
\frac{1}{\rho_\infty}  &   0   &  0 \\
-\frac{d_0\rho_1}{\rho_\infty}  & {\rho_1}     &  0 \\
\frac{\rho_1}{\rho_\infty} &0  &   {\rho_1}
\end{bmatrix}.
\end{equation*}
\mk 

One verifies that  the following relation is satisfied: 
\begin{equation}
\label{E:HTwistTransfo}
I_{\widehat \rho}={\rm VTrans1}_\rho\cdot I_\rho\cdot {}^t\overline{{\rm VTrans1}_\rho}\, .
\end{equation}
\sk 


\paragraph{\bf Second vertical translation formula ``$\boldsymbol{z_2\rightarrow z_2+\tau}$''}
We  finally consider the connection formula associated to the path 
$$ 
f_{\rm VTrans2} : 
[0,1] \longrightarrow \mathbb H\times \mathbb C^2, \; 
 s\longmapsto \big(\tau, z_1, z_2+s\tau\big)\, . 
$$

We  define $ \rho^*=R_{{\rm HTwist}}\circ  R_{{\rm VTrans}}\circ R_{{\rm HTwist}}\circ    (\rho)$, that is 
$$
 \rho^*=\big(  \rho^*_\infty,  \rho^*_0,  \rho^*_1
 \big) = \Big(
 \rho_\infty , \rho_0\rho_1^{-1} , \rho_1 \Big)\, .
$$

We set  $z^*=(z_1, z_2+\tau)=f_{\rm VTrans2}(1)$.  The map  $ 
f_{\rm VTrans2} $ gives us a linear isomorphism ${\rm VTrans2}_\rho$ from $H_1(E_{\tau,z}, L_\rho)$ onto 
$H_1(E_{\tau,z^*}, L_{\rho^*})$. 

The corresponding connection matrix is
\begin{equation*}
{\rm VTrans2}_\rho= \begin{bmatrix} 
1  &   0   &  0 \\
-\frac{d_1}{\rho_1 \rho_\infty}  & \frac{1}{\rho_1}     &   \frac{d_1}{\rho_1^2\rho_\infty}  \\
- \frac{1}{\rho_\infty} &0  &   \frac{1}{\rho_1\rho_\infty}
\end{bmatrix}\, . 
\end{equation*}
\mk 

One verifies that  the following relation is satisfied: 
\begin{equation}
\label{E:VTrans2Verif}
I_{\rho^*}= {\rm VTrans2}_\rho\cdot I_\rho\cdot {}^t\overline{{\rm VTrans2}_\rho}.\sk 
\end{equation}

The matrix  $ {\rm VT2}_\rho$ of the isomorphism ${\rm VTrans2}_\rho$ expressed in the basis $(\boldsymbol{\gamma}_\infty,\boldsymbol{\gamma}_0)$ and $(\boldsymbol{\gamma}^*_\infty,\boldsymbol{\gamma}^*_0)$ is quite simple compared to \eqref{E:HTrans2rho22}. 
It is 
\begin{equation}
\label{E:VTrans2rho22}
{\rm VT2}_\rho=\begin{bmatrix} 
 1  &  0   
\\
 \frac{\rho_0-\rho_1}{\rho_1\rho_\infty}   &   \frac{1}{\rho_1\rho_\infty}
\end{bmatrix} .
\end{equation}

This matrix satisfies the following relation:  $I\!\!I_{\rho^*}= {\rm VT2}_\rho\cdot I\!\!I_\rho\cdot {}^t\overline{{\rm VT2}}_\rho$. 


\subsubsection{\bf Normalization in the case when $\boldsymbol{\rho_0=1}$}
\label{SS:normalization}
If we assume that   $\rho_0=1$,  then   \eqref{E:Hrho} simplifies and one has: 
\begin{equation*}
  I\!\!H_\rho=
(2i)^{-1}
\begin{bmatrix}
 0   &  
 \rho_\infty \mk   \\ 
   -\frac{1}{\rho_\infty}
   &   \frac{d_\infty d_{1\infty}}{d_1\rho_\infty}  \end{bmatrix}\, . 
\end{equation*}

%
%
Then setting 
$$
Z_\rho=\sqrt{2}\begin{bmatrix}
\rho_\infty  &  - \frac{d_{1\infty}}{d_1}   
  \\ 
  0   &   1  \end{bmatrix}, 
$$
one verifies that 
$$
\boldsymbol{H} =
\begin{bmatrix}
0  &   -i 
  \\ 
  i &   0  \end{bmatrix}=
{}^t  \overline{Z_\rho}\cdot I\!\!\!H_\rho 
\cdot 
Z_\rho\, . 
$$

The interest of considering $\boldsymbol{H}$ instead of 
$ I\!\!H_\rho$  is clear:  the automorphism group of the former is 
${\rm SL}_2(\mathbb R)$,  hence, in particular, does not depend on  $\rho$. \sk

This  normalization will be used later in \S\ref{S:HolonomyY1(N)alpha1}.



\section{\bf An explicit expression for Veech's map and some consequences}

\subsection{\bf Some general considerations about Veech's foliation} 
\label{S:GeneralConsiderationsAboutVeech'sFoliation}
In this subsection, we make general remarks about  Veech's foliation in the general non-resonant case. 
Hence
$g$ and $n$ are arbitrary positive integers such that $2g-2+n>0$ and  $\alpha=(\alpha_k)_{k=1}^n\in ]-1,\infty[^n$ is supposed to be non-resonant, {\it i.e.} none of the $\alpha_k$'s is an integer.  
\sk 

In \cite{Veech}, Veech defines the isoholonomic foliation $\mathcal F^\alpha$ on the Teichmller space by means of a real analytic
map ${H}_{g,n}^\alpha: \mathcal T\!\!\!{\it eich}_{g,n}\rightarrow \mathbb U^{2g}$. The point is that this map descends to the Torelli space $ \mathcal T\!\!\!{\it or}_{g,n}$ and  even on this quotient, it is probably not a primitive integral of the foliations formed by its level-sets (this is proved below only when $g=1$).  It is what we explain below. 

\subsubsection{} \!\!\! Let
$(S,(s_k)_{k=1}^n)$ be 
 a reference model  for $n$-marked surfaces of genus $g$.  We fix a base point $s_0\in S^*=S\setminus \{s_k \}_{k=1}^n$. One can find a natural `symplectic basis' $(A_i, B_i, C_k)$, $i=1,\ldots,g$, $k=1,\ldots,n$ of $\pi_{1}(S^*,s_0) $ such that the latter group is isomorphic to  
$$
 \pi_{1}(g,n)=\left\langle 
 A_1,\ldots,A_g, B_1,\ldots,B_g,C_1,\ldots , C_n \; \Big\lvert \prod_{i=1}^g [A_i,B_i]=C_n\cdots C_1
 \right\rangle,  $$
see the picture just below (case $g=n=2$):  
\begin{center}
\begin{figure}[!h]
\psfrag{S}[][][1]{$S $}
\psfrag{A1}[][][0.8]{$\textcolor{blue}{A_1} $}
\psfrag{A2}[][][0.8]{$\textcolor{blue}{A_2 }$}
\psfrag{B1}[][][0.8]{$\textcolor{green}{B_1} $}
\psfrag{B2}[][][0.8]{$\textcolor{green}{B_2} $}
\psfrag{C1}[][][0.8]{$\;\,  \textcolor{red}{C_1} $}
\psfrag{C2}[][][0.8]{$\textcolor{red}{C_2} \;\;\,   $}
\includegraphics[scale=1.2]{{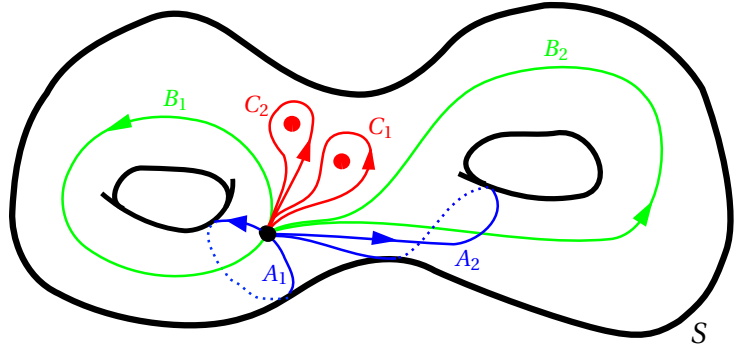}}
\caption{
The base-point $s_0$ is the black dot, 
 $s_1,s_2$ are  the red ones.}
\end{figure}
\end{center}
\subsubsection{} {\hspace{-0.6cm}} We recall Veech's definition of the space $\mathcal E_{g,n}^\alpha$: it is the space of isotopy classes of flat structures on $S$ with conical singularity of type $\lvert u^{\alpha_k}du\lvert^2$ (or equivalently, 
with cone angle $2\pi(1+\alpha_k)$) 
at $s_k$ for $k=1,\ldots,n$.  Since a flat structure of this type  induces a natural conformal structure on $S$, there is a natural map 
\begin{equation}
\label{E:Egnalpha->Teichgn}
\mathcal E_{g,n}^\alpha \longrightarrow \mathcal T\!\!\!{\it eich}_{g,n}
\end{equation}
which turns out to be a real analytic\footnote{The space $\mathcal E_{g,n}^\alpha$ carries a natural intrinsic  real analytic structure, {\it cf.}\,\cite{Veech}.} diffeomorphism.
We need to describe the inverse map of \eqref{E:Egnalpha->Teichgn}.
 To this end we are going to use a somehow old-fashioned definition of the Teichm\"uller space that will be useful for our purpose.  \sk 

Let $(X,x)=(X,(x_1,\ldots,x_n))$ be a $n$-marked Riemann surface of genus $g$.  Considering a point over it in 
 $\mathcal T\!\!\!{\it eich}_{g,n} $ amounts  to specify  a marking of its fundamental group, that is a class, up to inner automorphisms, of isomorphisms $\psi: \pi_1(g,n) \simeq \pi_1(X^*,x_0) $ for any  
 $x_0\in X^*=X\setminus \{x_k \}_{k=1}^n$ (see\;{\it e.g.}\;\cite[\S2]{TeichmullerTranslation} or
 \cite{Ahlfors,Weil}\footnote{The definition of a point of 
 the Teichmller space (of a closed surface  and without marked points) by means of a marking of the fundamental group follows from a result attributed to Dehn by Weil in \cite{Weil}, whereas in  \cite{TeichmullerTranslation},   Teichmller refers for this to the paper \cite{Mangler} by Mangler.}).   Finally, we denote by $ m_{X,x}^\alpha$ the unique 
  flat metric of area 1 in the conformal class corresponding to the complex structure of $X$, with a conical singularity of exponent $\alpha_k$ at $x_k$ for every $k$ ({\it cf.}\;Troyanov's theorem mentioned in {\S\ref{S:VeechIntro}}).
 
With these notations,  the inverse  of \eqref{E:Egnalpha->Teichgn} is written  
 $$
 \big(X,x, \psi \big) \longmapsto  \big(X,x, \psi , m_{X,x}^\alpha \big) \, . 
 $$
  
\subsubsection{} {\hspace{-0.2cm}} 
\label{SS:VeechLinearHolonomyMap}
Since $m_{X,x}^\alpha$ induces a smooth flat structure on $X^*$, 
its linear holonomy along any loop 
$\gamma$ in $X^*$ is a complex number of modulus 1, noted by ${\rm hol}^\alpha_{X,x}(\gamma)$.  Of course, this  number actually  only depends on the homotopy class of $\gamma$ in $X^*$.  With this formalism at hand, it is   easy to describe the map constructed  in \cite{Veech} to define the foliation $\mathcal F^\alpha$ on the Teichmller space: 
it is the map which associates to 
$(X,x,\psi)$ the holonomy character induced by $m_{X,x}^\alpha$. 

Note that,  since the conical angles are fixed, for every $k=1,\ldots,n$, one has 
 $$\qquad {\rm hol}^\alpha_{X,x}\big(\psi(C_k)\big)=\exp\big(2i\pi \alpha_k\big)
 \in \mathbb U\, . $$
  Consequently,  there is  a  well-defined map
\begin{align}
\label{E:Chi-g-n-alpha}
{\chi}^\alpha_{g,n}:
\mathcal T\!\!\!{\it eich}_{g,n} & \longrightarrow {\rm Hom}^\alpha\left( \pi_1(g,n), \mathbb U\right)\\ 
\big(X,x,\psi\big) & \longmapsto   {\rm hol}^\alpha_{X,x}\circ \psi\, , 
\nonumber
\end{align}
the exponent $\alpha$ in the formula of the target space meaning  that one considers only unitary characters on $ \pi_1(g,n)$ which map $C_k$ to $\exp(2i\pi \alpha_k)$ for every $k$. 

\subsubsection{}{\hspace{-0.4cm}}  
\label{SS:H1(g,n)}
Let 
$H_1(g,n)$ be the abelianization of $\pi_1(g,n)$: it is the $\mathbb Z$-module generated by the $A_i$'s, the  $B_j$'s and the $C_k$'s up to the relation $\sum_{k}  C_k=0$. 
We denote by $a_i,b_i$ and $c_k$ the corresponding homology classes.
 We take $H_1(g,n)$ as  a model for the first homology  group of $n$-punctured genus $g$ Riemann surfaces. 
\begin{center}
\begin{figure}[!h]
\psfrag{S}[][][1]{$S $}
\psfrag{A1}[][][0.8]{$\textcolor{blue}{a_1} $}
\psfrag{A2}[][][0.8]{$\textcolor{blue}{a_2 }$}
\psfrag{B1}[][][0.8]{$\textcolor{green}{b_1} $}
\psfrag{B2}[][][0.8]{$\textcolor{green}{b_2} $}
\psfrag{C1}[][][0.8]{$\;\,  \textcolor{red}{c_1} $}
\psfrag{C2}[][][0.8]{$\textcolor{red}{c_2} \;\;\,   $}
\includegraphics[scale=1.2]{{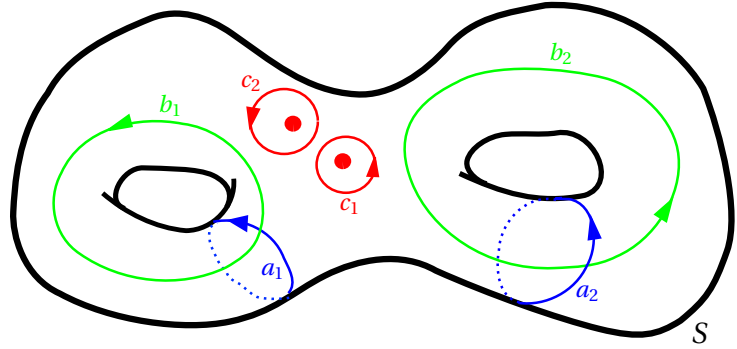}}
\caption{A model for the homology of the punctured surface $S^*$.}
\end{figure}
\end{center}

The Torelli space $\mathcal T\!\!\!{\it or}_{g,n}$ can be defined as the set of triples $(X,x,\phi)$ where $(X,x)$ is a marked Riemann surface as above and $\phi$ an isomorphism from $H_1(g,n)$ onto $H_1(X^*,\mathbb Z)$. Moreover, the projection from the Teichmller space onto the Torelli space is given by 
\begin{align*}
p_{g,n}:  \mathcal T\!\!\!{\it eich}_{g,n} & \longrightarrow  \mathcal T\!\!\!{\it or}_{g,n}
\\ 
\big(X,x,\psi\big) & \longmapsto  \big(X,x,[\psi]\big)
\end{align*}
where $[\psi]$ stands for the isomorphism in   homology induced by 
$\psi$. 

\subsubsection{} {\hspace{-0.6cm}} Now, the key (but obvious) point  is that the holonomy ${\rm hol}^\alpha_{(X,x)}(\gamma)$ for $\gamma\in \pi_1(X^*)$ does not depend on the base point but only on the (base-point) free homology class $[\gamma]\in H_1(X^*,\mathbb Z)$.  Since $\mathbb U$ is commutative, 
any unitary representation of $\pi_1(g,n)$ factors trough $\pi_1(g,n)^{\rm ab}=H_1(g,n)$,  thus 
there is a natural map 
$ {\rm Hom}^\alpha\left( \pi_1(g,n), \mathbb U\right)\rightarrow 
 {\rm Hom}^\alpha\left( H_1(g,n), \mathbb U\right)$.   Since  $\mu\in    {\rm Hom}^\alpha( \pi_1(g,n), \mathbb U)$  is completely determined by its values on the $A_i$'s and the $B_i$'s (it verifies $\mu(C_k)=\exp(2i\pi\alpha_k)$ for $k=1,\ldots,n$), 
  the space $  {\rm Hom}^\alpha( \pi_1(g,n), \mathbb U)$ is naturally  isomorphic to $\mathbb U^{2g}$.  This applies verbatim to 
  $  {\rm Hom}^\alpha( H_1(g,n), \mathbb U)$ as well.  It follows that these two spaces of unitary characters are both   naturally isomorphic to $\mathbb U^{2g}$. \sk 
  
  From the preceding discussion, it comes that one can define a map 
$
\mathcal T\!\!\!{\it or}_{g,n}  \rightarrow {\rm Hom}^\alpha( H_1(g,n), \mathbb U)$ 
 which makes the following  square 
  diagram commutative: 
    \begin{equation}
  \label{D:SquareCommutative}
    \xymatrix@R=1cm@C=0.15cm{  
    \mathcal T\!\!\!{\it eich}_{g,n}    
     \ar@{->}[rrrr]^{\chi^\alpha_{g,n}\qquad } 
      \ar@{->}[d]_{p_{g,n}}  &&&&  {\rm Hom}^\alpha\left( \pi_1(g,n), \mathbb U\right)  \eq[r] \ar@{->}[d]^{}& \; {\mathbb U^{2g}}   \ar@{=}[d]^{} \\
  \mathcal   T\!\!\!{\it or}_{g,n} \ar@{->}[rrrr]
  &&&&
    {\rm Hom}^\alpha\left( H_1(g,n), \mathbb U\right)    \eq[r]    &  \;  \mathbb U^{2g} . 
    }
        \end{equation}
        
   Both maps with values into $\mathbb U^{2g}$ given by the two lines   of the preceding diagram will be called the {\bf linear holonomy maps}. 
   We will use the (a bit abusive) notation
   $H^\alpha_{g,n}:   \mathcal T\!\!\!{\it eich}_{g,n} \rightarrow \mathbb U^{2g}$ for the first and the second will be denoted by 
    \begin{equation}
\label{E:VeechMapTorelli}
{h}_{g,n}^\alpha : \, 
\mathcal T\!\!\!{\it or}_{g,n}  \longrightarrow  \mathbb U^{2g}\, .
\end{equation}
  
  Since $p_{g,n}$ is the universal covering map of the Torelli space which is a complex manifold,   the maps  $H^\alpha_{g,n}$ 
  and $h_{g,n}^\alpha $  enjoy the same local analytic properties. 
  Then from  \cite[Theorem 0.3]{Veech}, one deduces immediately the 
\begin{coro} If  $\alpha$ is non-resonant,  the linear holonomy map
$ h_{g,n}^\alpha$
  is a real analytic submersion. Its level sets are  complex submanifolds of the Torelli space $\mathcal T\!\!\!{\it or}_{g,n}$, 
 of complex dimension $2g-3+n$.\footnote{Actually, the statement is valid for any $\alpha$ but on the complement of the preimage by 
the linear holonomy map   of the trivial character on $H_1(g,n)$.  Note that the latter does not belong to ${\rm Im}({h}_{g,n}^\alpha)$ (so its preimage is empty) as soon as at least one of the $\alpha_k$'s is not an integer.}  
\end{coro} 
     This implies that the foliation constructed by Veech on $ \mathcal T\!\!\!{\it eich}_{g,n} $ in \cite{Veech} actually is the pull-back of a foliation defined on $\mathcal T\!\!\!{\it or}_{g,n}$. We will also  call the latter Veech's foliation and will denote it the same way, that is by  $\mathcal F^\alpha$.

\subsubsection{}{\hspace{-0.4cm}} We now explain that Veech's  linear holonomy map $\chi^\alpha_{g,n}:   \mathcal T\!\!\!{\it eich}_{g,n} \rightarrow \mathbb U^{2g}$ actually  admits a canonical lift to $\mathbb R^{2g}$.  To this end, we use  elementary arguments (which can be found in  \cite{MasurSmillie}, p. 488-489).

Let $(X,x)$ be as above and consider  $\gamma$, a smooth simple curve in $X^*$. 
If $\ell $ stands for its length for the flat metric $m_{X,x}^\alpha$, there exists a 
$\ell$-periodic smooth map $g: \mathbb R \rightarrow  X^*$ which induces a isomorphism of flat circles $\mathbb R/\ell \mathbb Z\simeq \gamma$ ({\it i.e.}\;the pull-back of $m_{X,x}^\alpha$ by $g$ coincides with the Euclidean metric on $\mathbb R$).  For any $t$, $g'(t)$ is a unit tangent vector at $g(t)\in \gamma$,  thus there exists a unique 
other tangent vector at this point, noted by 
 $g(t)^\bot$, such that $(g'(t),g(t)^\bot)$ form a direct orthonormal basis of $T_{g(t)} X^*$. 
Then there exists a smooth function $w: [0,\ell]\rightarrow \mathbb R$ such that $g''(t)=w(t)\cdot g(t)^\bot$ for any $t\in [0,\ell]$ and 
one defines the {\bf total angular curvature of the loop $\gamma$ in the flat surface $(X,m_{X,x}^\alpha) $} as the real number
$$
\kappa(\gamma)=
\kappa_{X,x}^\alpha(\gamma)=\int_0^\ell w(t)dt\, .
$$

There is a nice geometric interpretation of this number as a sum 
of the oriented interior angles of the triangles of a given Delaunay triangulation of $X$ which meet $\gamma$ (see \cite[\S6]{MasurSmillie}). In particular, one obtains that  $\kappa(\gamma)$ only depends on the free isotopy class of $\gamma$ and that $\exp(2i\pi \kappa(\gamma))$ is nothing  else but the linear holonomy of $(X,m_{X,x}^\alpha)$ along $\gamma$, that is: 
\begin{equation}
\label{E:Exp2iPiKappaGamma}
\exp\big(2i\pi \kappa(\gamma)\big)={\rm hol}^{\alpha}_{X,x}(\gamma)\, . 
\end{equation}

Let $\widetilde \gamma$ be another simple curve in the free homotopy class $\langle \gamma\rangle$ of $\gamma$. According to a classical result of the theory of surfaces (see \cite{Epstein}), $\widetilde \gamma$ and $\gamma$ actually are isotopic,  hence $\kappa(\gamma)=\kappa(\widetilde \gamma)$. Consequently,  the following definition makes sense: 
\begin{equation}
\label{E:KappaHomologyClass}
\kappa_{X,x}^\alpha\big(\langle \gamma\rangle \big)=
\kappa_{X,x}^\alpha(\gamma)\, .
\end{equation}
 
\subsubsection{}{\hspace{-0.4cm}}  
\label{SS:LiftedHolonomyMap}
 Now assume that a base point $x_0\in X^*$ has been fixed. By the preceding construction, one can attached a real number 
 $\kappa_{X,x}(\langle \gamma\rangle )$ 
 to each element $[\gamma]\in \pi_1(X^*,x_0)$ which is representable by a simple loop $\gamma$.  If $\eta$ stands for an inner automorphism of $\pi_1(X^*,x_0)$, a classical result of the topology of surfaces ensures that $ \gamma$ and $\eta(\gamma)$ are freely homotopic, {\it i.e.}\;$\langle \gamma\rangle=\langle \eta(\gamma) \rangle$. 
 \sk

 We now have explicited everything needed to construct a lift of Veech's linear holonomy map. Let $(X,x,\psi)\simeq (X,x,m_{X,x}^\alpha,\psi)$ be a point of $\mathcal T\!\!\!{\it eich}_{g,n} \simeq \mathcal E_{g,n}^\alpha$. Then for any 
 element $D$ of  $\{A_k, B_k\, \lvert \, k=1,\ldots,g\}\subset \pi_1(g,n)$, its `image' $D^\psi$ by $\psi$ can be seen as the conjugacy class of the homotopy class of a simple curve  in $X^*$. By the preceding discussion, it comes that the  map  
 \begin{align}
\label{E:TildeHolAlpha}
\widetilde{H}^{\alpha}_{g,n} \, : \; \mathcal T\!\!\!{\it eich}_{g,n} & \longrightarrow \mathbb R^{2g}\\
\big(X,x,\psi\big)  & \longmapsto \Big(  \kappa\big(A_1^\psi\big), \ldots,\kappa\big(A_g^\psi \big),\kappa\big(B_1^\psi\big), \ldots,\kappa\big(B_g^\psi\big)\Big) \nonumber
\end{align}
is well-defined.   This map is named the {\bf lifted holonomy map}.
\sk

 It is easy (left to the reader) to verify that it enjoys the following properties: 
\begin{enumerate}
\item  it is a lift of ${H}^{\alpha}_{g,n}$  to $\mathbb R^{2g}$:
if $e: \mathbb R^{2g}
\rightarrow \mathbb U^{2g}$ is  the universal covering,  then 
 $$ {H}^{\alpha}_{g,n}= e\circ \widetilde{H}^{\alpha}_{g,n}\, ; $$
\item  it is real analytic. 
\end{enumerate}
The first point follows at once from \eqref{E:Exp2iPiKappaGamma} and the second 
is an immediate consequence of the first combined with   the obvious fact that $\widetilde{H}^{\alpha}_{g,n}$ is continuous.

\subsubsection{} \hspace{-0.4cm} 
\label{SS:VeechLeavesAreNotCoonected}
From the lines above, it comes that $\widetilde{H}^{\alpha}_{g,n}$ is a real analytic first integral for Veech's foliation on $\mathcal T\!\!\!{\it or}_{g,n}$ which could  enjoy better properties than ${H}^{\alpha}_{g,n}$.  Note that,  among the lifts of the latter which are continuous, it is unique up to translation by an element of $2\pi\mathbb Z^{2g}$. \sk 

For $a\in \mathbb R^{2g}$, one defines   $\mathcal F_a^\alpha$ as the inverse image   of $a$ by 
the lifted holonomy map in 
the Teichmller space. 
In particular, for 
$\rho=e(a) \in \mathbb U^{2g}$, one has 
\begin{equation}
\label{EE}
\mathcal F^\alpha_{\rho}=\big({H}^{\alpha}_{g,n}\big)^{-1}(\rho)= 
\bigcup_{{\boldsymbol{m}} \in \mathbb Z^{2g}} \mathcal F^\alpha_{a+2\pi \boldsymbol{m}}
\end{equation}
hence it is natural to expect that any  level-set $\mathcal F^\alpha_{\rho}$ has a countable set of connected components.  
We will prove below that it is indeed the case when $g=1$, by completely expliciting the lifted holonomy map (see Remark \ref{R:ImageHolonomy&VeechFoliationCC}.(2)).  We conjecture that it is also true when $g\geq 2$ but it is not proved yet. 
\mk 

To conclude this subsection, we would like to warn the reader that the  terminology 
`{\it lifted holonomy}'  that 
we use to designate $\widetilde{H}^{\alpha}_{g,n}$ can be  misleading.   Indeed, the latter map is not a holonomy in a natural sense.  This will make be clear below when considering the genus 1 case, see Example \ref{Ex:KiteStuf} below.

\subsection{\bf An explicit description of Veech's foliation when $\boldsymbol{g=1}$}
\label{S:TheCaseOfEllipticCurves}
We now focus on the case when  $g=1$,  with $n\geq 2$ arbitrary. \sk 

The first particular and interesting feature of this special case is that it is possible to define a `lifted holonomy map' on the Torelli space $\mathcal T\!\!\!{\it or}_{1,n}$. \mk 

Then, when dealing with elliptic curves, all the rather abstract considerations of the preceding subsection can be made completely explicit.  The reason for this is twofold:  first, there is a nice explicit description of the Torelli space 
$\mathcal T\!\!\!{\it or}_{1,n}$ due to  Nag; second, on tori, one can give an explicit formula for a metric inducing a flat structure with conical singularities in terms of theta functions. 

\subsubsection{} \hspace{-0.4cm} 
Our goal here is to construct abstractly a lift to $\mathbb R^2$ of the map  
$h_{1,n}^{\alpha}: \, {\mathcal T}\!\!\!{\it or}_{1,n}\rightarrow \mathbb U^2$.  Our construction  is based on the following crucial result: 

\begin{lemma}
 On a punctured torus, two simple closed 
curves which are homologous are actually  isotopic.
\end{lemma}
\begin{proof} 
Let $\Sigma$ stand for a finite subset of $T=\mathbb R^2/\mathbb Z^2$. 
  We consider 
 two simple closed curves $a$ and $b$ in $T^*=T\setminus \Sigma $,  assumed to be  homologous.\sk
 
 We need first to treat the case when  $\Sigma$ is empty. Since 
 $\pi_1(T)=\mathbb Z^2$ is abelian, it coincides with its abelianization, namely $H_1(T,\mathbb Z)$. From this, it follows that  $a$ and $b$ are homotopic. 
 Then a classical result from the theory of surfaces \cite{Epstein} allows to conclude that these two curves are isotopic. \sk 
 
We now consider the case when $\Sigma $ is not empty which is the one of interest for us.   
The hypothesis implies that $a$ and $b$ are  {\it a fortiori} homologous in $T$.  From above, it follows that  they are isotopic through an isotopy $I : [0,1] \times S^1 \rightarrow T$. We can assume that this isotopy is minimal in the sense that the number $m$ of couples $(t,\theta) \in [0,1] \times S^1$ such that $I(t,\theta)  \in \Sigma$ is minimal. \sk 

We denote by  $(t_1,\theta_1),\ldots, (t_m,\theta_m)$ the elements of  $I^{-1}(\Sigma)$.  For any $i=1,\ldots,m$,  we set 
$s_i=I(t_i,\theta_i)\in \Sigma$ and define $\epsilon(i)\in \{\pm 1\}$ as follows: $\epsilon(i)=1$ if  $({\partial I}/{\partial \theta}, {\partial I}/{\partial t})$ form an direct basis of the tangent space of $T$ at $s_i$;   otherwise, we set $\epsilon(i)=-1$.

 Therefore we have 
$$ [b] = [a] + \sum_{i=1}^m {\epsilon(i) [\delta_{s_i}]} $$ 
 in $H_1(T^*, \mathbb Z)$,  where $\delta_{s}$ stands for a small circle turning around $s$ counterclockwise for any $s\in \Sigma$.   If $i$ and $i'$ are two indices such that $s_{i} = s_{i'}$  then $\epsilon(i)$ and $\epsilon(i')$ must be equal. Indeed otherwise these two crossings could be 
  cancelled what would contradict the minimality of $m$. Since  
 $ [b] = [a] $ by assumption, we have $\sum_{i=1}^m{\epsilon(i) [\delta_{s_i}]} = 0$.  This relation is necessarily an integer multiple of  $\sum_{s\in \Sigma} [\delta_{\!s}]$.  Remark that the latter 
 can be  be cancelled by modifying $I$: a simple closed curve on $T^*$ cuts $T$ into a cylinder and we can find an isotopy consisting of going along this cylinder crossing every puncture once in the same direction. Post-composing $I$ by such an isotopy the appropriate number of time allows to cancel all the remaining crossings. The lemma follows. 
 \end{proof}
 \mk 

With this lemma at hand, one can proceed as in 
\S\ref{SS:LiftedHolonomyMap} and construct a real analytic map $\widetilde{h}_{1,n}^{\alpha}: \, {\mathcal T}\!\!\!{\it or}_{1,n}\rightarrow \mathbb R^2$ 
making  the following diagram commutative: 
$$
    \xymatrix@R=0.5cm@C=3cm{  
  \mathcal   T\!\!\!{\it or}_{1,n} \ar@{->}[r]^{\quad \widetilde{h}^\alpha_{1,n} }    
  \ar@{->}[rd]_{\quad {h}^\alpha_{1,n} }  
  &    \;  \mathbb R^2  \ar@{->}[d]^{e(\cdot) }  \\
  &     \;\mathbb U^2  \, . 
    }
$$

Note that the lift  of ${h}^\alpha_{1,n}$ to $\mathbb R^2$ 
is unique,  up translation by an element of $2\pi\mathbb Z^2$, as soon as one demands that it is continuous. We have proved above that such  a {\it `lifted holonomy'} exists on the Torelli space of punctured elliptic curves. We will give  an explicit and particularly simple expression for it  in \S\ref{SS:FlatMetricsOnEllipticCurves} below. 
\sk

To conclude these generalities, we would like to warn the reader that the  terminology 
{\it `lifted holonomy}'  we use to designate $\widetilde{h}^{\alpha}_{1,n}$ is misleading.   Let $E_\tau^*$ be  a $n$-punctured elliptic curve.  With the notations of 
\S\ref{SS:H1(g,n)}, the homology classes  $a_1,b_1$ and $c_1,\ldots,c_{n-1}$ are representable by closed simple curves and  span freely $H_1(E_\tau^*,\mathbb Z)$. Using \eqref{E:KappaHomologyClass}, one can  construct a lift $
 \mathcal T\!\!\!{\it or}_{1,n}\rightarrow {\rm Hom}(H_1(1,n),\mathbb R)$ of 
 the map $\mathcal T\!\!\!{\it or}_{1,n}\rightarrow {\rm Hom}^\alpha(H_1(1,n),\mathbb U) $  in  \eqref{D:SquareCommutative}. If the latter is  a genuine  (linear)  holonomy map, 
 this is not the case for the former.  Indeed,  this additive character  on $H_1(1,n)$ 
  is not geometric in a meaningful  sense:  as the example below shows, its value on 
$c_1+\cdots+c_{n-1}$ a priori differs from  $-\kappa(c_n)$.  

\begin{ex} 
\label{Ex:KiteStuf}
{\rm Let $\tau\in \mathbb H$  be arbitrary. Consider a disk $D $ in $]0,1[_\tau\subset E_\tau$ and a kite $K \subset D$, whose exterior angles are $\vartheta_1, \theta_2,\theta_3\in ]0,2\pi[$ (see the picture below). 
\begin{center}
\begin{figure}[!h]
\psfrag{0}[][][1]{$0$}
\psfrag{1}[][][1]{$1$}
\psfrag{t}[][][1.1]{$\tau$}
\psfrag{D}[][][1.1]{$D$}
\psfrag{K}[][][1.1]{$\!\!\!K$}
\psfrag{t1}[][][0.8]{$\vartheta_1 $}
\psfrag{t2}[][][0.8]{$\theta_2 $}
\psfrag{t3}[][][0.8]{$\theta_3 $}
\includegraphics[scale=1]{{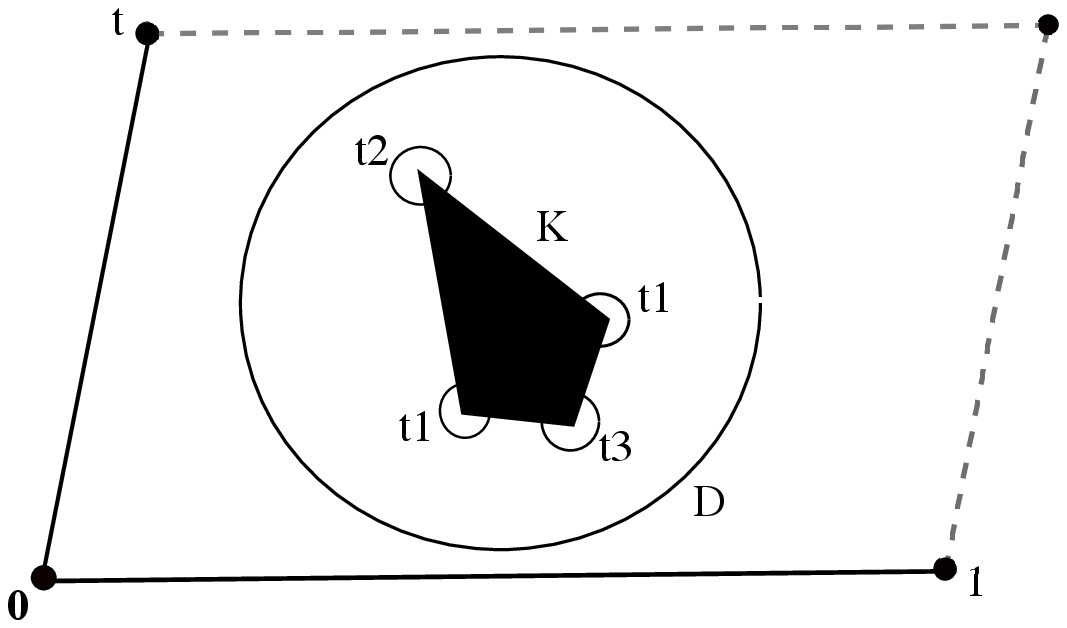}}
\end{figure}
\end{center}

Consider $E_\tau$ with its non-singular  canonical flat structure.  
Removing the interior of $K$ and gluing pairwise the edges of its boundary which are 
of the same length, one ends up  with a flat tori $E_{\tau,K}$ with three conical singularities, of cone angles $\theta_1=2\vartheta_1\in ]0,4\pi[$ and $\theta_2,\theta_3\in ]0,2\pi[$ (in the language of \cite[\S6]{GP}, we have performed a `Kite surgery' on the flat torus 
 $E_\tau$ in order to construct $E_{\tau,K}$).
\sk 

Let $a_1$ and $b_1$ be the loops in $E_{\tau,K}$ which correspond   to the images in $E_\tau$ of the two segments $[0,1]$ and   $[0,\tau]$ respectively.  With the same meaning for $c_1,c_2$ and $c_3$  as  above, one can see $(E_{\tau,K}, a_1,b_1,c_1,c_2,c_3)$ as a point in $\mathcal T\!\!\!{\it or}_{1,3}$.  If $c$ stands for the loop given by the boundary of $D$ oriented in the direct order, then $c=c_1+c_2+c_3$,  hence $c=0$  in $H_1(E_{\tau,K}^*,\mathbb Z)$. But clearly, computing the total angular curvature of $c$ depends only on  the flat geometry  along $\partial D$,  hence can be performed in the flat tori $E_\tau$.  One gets $\kappa(c)=2\pi\neq 0$ although $c$ is trivial in homology.  
This shows that $\kappa$ does not induce a real character on $H_1(1,3)=\pi_1(1,3)^{ab}$ in a natural way.
 }
\end{ex}

To summarize the discussion above: 
 what we have constructed is a natural lift to $\mathbb R^{2}$ of the map ${h}^{\alpha}_{1,n}: \mathcal T\!\!\!{\it or}_{1,n}\rightarrow \mathbb U^{2}$  but not a lift to $ {\rm Hom}(H_1(1,n),\mathbb R)$
of the genuine holonomy map 
 $\mathcal T\!\!\!{\it or}_{1,n}\rightarrow {\rm Hom}^\alpha(H_1(1,n),\mathbb U)$ in diagram \eqref{D:SquareCommutative}. \mk 
 

\subsubsection{\bf The Torelli space of punctured elliptic curves} 
\label{S:TorelliNag}

For $(g,n)$ arbitrary, the Torelli group ${\rm Tor}_{g,n}$  is defined as the subgroup of the pure mapping class group ${\rm PMCG}_{g,n}$ which acts trivially on the first homology 
group of fixed   $n$-punctured model surface $S_{g,n}$.\footnote{Beware that several kinds of Torelli groups have been considered in the literature, especially in geometric topology 
(see {\it e.g.}\;\cite{Putman} where this is carefully explained). 
 The Torelli group we are considering in this text is known as the {\it `small Torelli group'}.} It is known that it 
 acts holomorphically, properly, discontinuously and without any fixed point on the Teichm\"uller space 
  ({\it cf.}\;\S2.8.3 in \cite{NagBook}).      Consequently, the Torelli space  $\mathcal{T}\!\!\!{\it or}_{1,n}=\mathcal{T}\!\!\!{\it eich}_{1,n}/{\rm Tor}_{1,n}$
 is a smooth complex variety (in particular, it has no orbifold point). 
 \sk 
 
 In  \cite{Nag}, the author shows that, setting $z_1=0$,  one has an identification 
 \begin{align*}
 \mathcal{T}\!\!\!{\it or}_{1,n}=& \, \left\{ (\tau,z_2,\ldots,z_n)\in \mathbb H\times \mathbb C^{n-1}\, \Big\lvert \;  z_i-z_j\not \in \mathbb Z_\tau\; \mbox{ for } i,j=1,\ldots,n, \, i\neq j
 \right\}. 
 \end{align*}
 
 Moreover there is a universal curve  
 $$
 \mathcal E_{1,n}\longrightarrow  \mathcal{T}\!\!\!{\it or}_{1,n}
 $$
 whose  fiber over $(\tau,z)=(\tau,(z_2,\ldots,z_n))$ is the elliptic curve $E_\tau=\mathbb C/\mathbb Z_\tau$.  It comes with $n$ global sections 
 $\sigma_i$
     defined by 
 $\sigma_i(\tau,z)=[z_i]\in E_\tau$ for $i=1,\ldots,n$. 
 \sk
 
The action of the pure mapping class group ${\rm PMCG}_{1,n}$ on the Torelli space is not effective and its kernel is precisely the Torelli group.  We denote by  ${\rm Sp}_{1,n}(\mathbb Z)$ the quotient ${\rm PMCG}_{1,n}/{\rm Tor}_{1,n}$. It is isomorphic to the group of automorphisms of the first homology group of a $n$-punctured genus $g$ surface which leaves  the cup-product invariant.\footnote{We are not aware of any other proof of this result than the one given in the unpublished thesis \cite{VdB}.} 

From \cite{Nag}, one deduces that there is  an isomorphism
$$
{\rm Sp}_{1,n}(\mathbb Z) \simeq {\rm SL}_2(\mathbb Z)\ltimes \big(\mathbb Z^2\big)^{n-1}
$$
where the semi-direct product is given by   
$$
\Big(M',\big(\boldsymbol{k}', \boldsymbol{l}'\big)\Big)\cdot 
\Big(M,\big(\boldsymbol{k}, \boldsymbol{l}\big)\Big)
= \Big(M'M, \rho(M)\cdot \big(\boldsymbol{k}', \boldsymbol{l}'\big)+\big(\boldsymbol{k}, \boldsymbol{l}\big)\Big)
$$
for 
$M,M'\in {\rm SL}_2(\mathbb Z)$ and $(\boldsymbol{k}, \boldsymbol{l})=\big((k_i,l_i)\big)_{i=2}^n , (\boldsymbol{k}', \boldsymbol{l}')=\big((k'_i,l'_i)\big)_{i=2}^n
\in \big(\mathbb Z^2\big)^{n-1}$, with
\begin{equation}
\label{E:RhoM}
\rho\left(M \right)= 
\begin{bmatrix} d & b \\
c & a \end{bmatrix}
\quad \mbox{ and }
\quad 
M\cdot  
\big(\boldsymbol{k}, \boldsymbol{l}\big)
=
\Big(\big(ak_i+bl_i,ck_i+dl_i \big)\Big)_{i=2}^n
\end{equation}
if  
$M= \tiny{\big[\!\!
\begin{tabular}{cc}
$a$ \!\!&\!\! \!\!\!\!$b$\vspace{-0.1cm}\\
$c$ \!\!&\!\! \!\!\!\! $d$
\end{tabular}\!\!
\big]}\in {\rm SL}_2(\mathbb Z)
$. 

Moreover, one has 
$$\big(M, (\boldsymbol{k}, \boldsymbol{l})\big)^{-1}= 
\bigg( M^{-1}, \Big(   \big(a k_i-bl_i, -ck_i+dl_i\big)_{i=2}^n\Big)\bigg)
.$$

The action of 
$\big(M,(\boldsymbol{k}, \boldsymbol{l})\big)
\in {\rm SL}_2(\mathbb Z)\ltimes \big(\mathbb Z^2\big)^{n-1}$ 
on the Torelli space is given by 
\begin{equation}
\label{E:SP1nACTIONonTor1n}
  \Big( M,   \big(\boldsymbol{k}, \boldsymbol{l}\big)\Big)
\cdot (\tau , z) =\left(
\frac{a\tau+b}{c\tau+d},  \frac{z_2+k_2+l_2\tau}{c\tau+d}, \ldots, 
\frac{z_n+k_n+l_n\tau}{c\tau+d}
\right)
\end{equation}
for  $(\tau,z)=(\tau , z_2,\ldots,z_n)\in \mathcal T\!\!\!{\it or}_{1,n}$.
\sk

The epimorphism of groups ${\rm SL}_2(\mathbb Z)\ltimes \big(\mathbb Z^2\big)^{n-1}\rightarrow {\rm SL}_2(\mathbb Z)$ is compatible with the 
natural projection 
$$\mu=\mu_{1,n} \, : \;  {\mathcal T\!\!\!{\it or}_{1,n}}\longrightarrow {\mathcal T\!\!\!{\it or}_{1,1}}=\mathbb H\, .$$

  In other terms:    there is a surjective morphism in the category of analytic $G$-spaces:  
\begin{equation}
\label{E:MorphismGspace}
    \xymatrix@R=0.4cm@C=1cm{  
  {\mathcal T\!\!\!{\it or}_{1,n}}
  \ar@(dl,dr) \ar[rr]^{\mu}    &      &       \mathbb H    \ar@(dl,dr)   \\
  {\tiny{  {\rm SL}_2(\mathbb Z)\ltimes\big(\mathbb Z^2\big)^{n-1} } } \ar[rr]  &   &  {\rm SL}_2(\mathbb Z)\, . 
}
\end{equation}

\subsubsection{\bf Flat metrics with conical singularities on elliptic curves} 
\label{SS:FlatMetricsOnEllipticCurves}
As in the genus 0 case, there is a general explicit formula for the flat metrics on elliptic curves we are interested in.  We fix $n>1$ and  $\alpha=(\alpha_i)_{i=1}^n\in ]-1,\infty[$ such that $\sum_{i} \alpha_i=0$.
\sk 

For any elliptic curve $E_\tau=\mathbb C/\mathbb Z_\tau$, we will denote by $u$ the usual complex coordinates on $\mathbb C=\widetilde{E_\tau}$.  
 Let $\tau$ be fixed in $ \mathbb H$ and assume that $z=(z_1,\ldots  z_n)$ is a $n$-uplet of  pairwise distinct points on $\mathbb C$. If one defines $a_0$  as the real number
 \begin{equation}
 \label{E:alpha0}
 a_0=- \frac{{\Im}{\rm m}\big( \sum_{i=1}^n \alpha_i z_i  \big)}
 {{\Im}{\rm m}(   \tau)},  
 \end{equation}
 then   ${\Im}{\rm m}(a_0\tau+\sum_i \alpha_i z_i )=0$,  hence the constant  
\begin{equation}
 \label{E:alphaInfinity}
a_\infty =a_0\tau+\sum_{i=1}^n \alpha_i z_i
\end{equation} is a real number as well.\sk

Considering $(\tau,z) \in \mathcal T\!\!\!{\it or}_{1,n}$ as a fixed parameter,  we recall the definition of  the function $T^\alpha$ introduced in \S\ref{SS:OnTori} above: it is the function of the variable 
 $u$ defined by 
$$
T^\alpha_{\tau,z}(u)=
T^\alpha(u, \tau ,z)=e^{2i\pi a_0 u}\prod_{i=1}^n \theta(u-z_i,\tau)^{\alpha_i}. 
$$
We see this  function as a holomorphic multivalued function on 
the $n$-punctured elliptic curves $E_{\tau,z}$. 
From Lemma \ref{L:rhoexplicit}, we know that the monodromy of $T^\alpha_{\tau,z}$ is multiplicative and is given by the character $\rho$ whose characteristic values are 
$$
\rho_0=e^{2i\pi a_0}, \qquad  \rho_k=e^{2i\pi \alpha_k}\quad \mbox{for} \, k=1,\ldots,n
\qquad \mbox{ and }\qquad 
\rho_\infty=e^{2i\pi a_\infty}\, .
$$

Since  $a_0$, $a_\infty$ and the $\alpha_k$'s are real, $\rho$ is unitary. Thus for 
any $(\tau,z)\in \mathcal T\!\!\!{\it or}_{1,n}$, 
$$
m^\alpha_{\tau,z}=\big\lvert T^\alpha_{\tau,z}(u) du\big\lvert ^2
$$
defines a flat metric on $E_{\tau,z}$.  Moreover, the theta function $\theta(\cdot)$, viewed as a section of a line bundle on $E_\tau$,  has a single zero 
at the origin,  which is simple. 

This implies that  up to multiplication by a positive constant, one has 
$$m^\alpha_{\tau,z} \sim \big\lvert (u-z_k)^{\alpha_k}du\big\lvert^2 $$ on a neighborhood of $z_k$,  for  $k=1,\ldots,n$.  This shows that the flat structure induced by $
m^\alpha_{\tau,z}$ on $E_{\tau,z}$ has conical singularities with exponent $\alpha_k$ at $[z_k]$ for every $k=1,\ldots,n$.  We recall the reader that  assuming that $(\tau,z)\in \mathcal T\!\!\!{\it or}$ implies that $z=(z_1,\ldots,z_n)\in \mathbb C^n$ has been normalized such that $z_1=0$. 
\sk 

It follows from Troyanov's theorem that 
considering $\mathcal T\!\!\! {\it or}_{1,n}$ as the quotient of  
the space 
$\mathcal E_{1,n}^\alpha$ of isotopy classes of flat tori  
by the Torelli group ${\rm Tor}_{1,n}$ amounts to associate the triplet $(\tau,z,m^\alpha_{\tau,z})$ to any $(\tau,z)\in \mathcal T\!\!\! {\it or}_{1,n}$.

 Any element of $(\tau,z)\in \mathcal T\!\!\! {\it or}_{1,n}$ comes with a well-defined system of generators of $H_1(E_{\tau,z},\mathbb Z)$. Let $\epsilon>0$ be very small. For $k=1,\ldots,n$, let $\gamma_k^\epsilon$ be a positively oriented small circle centered at $[z_k]$ in $E_\tau$, of radius $\epsilon$.  Let $\gamma_0^\epsilon$ (resp.\;$\gamma_\infty^\epsilon$) be the loop  in $E_{\tau,z}$ defined as the image  of $[0,1]\ni t\mapsto \epsilon(1+i)+t$ 
(resp.\;of $[0,1]\ni t\mapsto \epsilon(1+i)+t\cdot \tau$) by the canonical projection.  For $\epsilon$ sufficiently small, the  homology  classes of 
the  $\gamma_\bullet^\epsilon$'s for $\bullet=0,1,\ldots,n,\infty$  do not depend on $\epsilon$.  We just denote by $\gamma_\bullet$ the associated homology classes. These generate  $H_1(E_{\tau,z},\mathbb Z)$ and $\sum_{k=1}^n \gamma_n=0$ is 
 the unique linear relation they satisfy  (see Figure \ref{Fi:totuti} below). 
\sk

At this point, it is quite obvious 
 that 
 the linear holonomies of the flat surface $(E_{\tau,z},m^\alpha_{\tau,z})$ along $\gamma_0$ and $\gamma_\infty$ are respectively  
$$\rho_0=\rho_0(\tau,z)=\exp\big(2i\pi a_0\big)
\qquad \mbox{ and }\qquad 
\rho_\infty=\rho_\infty(\tau,z)=\exp\big(2i\pi a_\infty\big)\, .$$

 \begin{center}
\begin{figure}[!h]
\psfrag{1}[][][1]{$1 $}
\psfrag{t}[][][1]{$\tau $}
\psfrag{0}[][][1]{$\;  \gamma_1$}
\psfrag{t1}[][][1]{$\;\;  \gamma_4 $}
\psfrag{t2}[][][1]{$\gamma_2 \; $}
\psfrag{t3}[][][1]{$\;\; \gamma_3 $}
\psfrag{G0}[][][1]{$\gamma_0   $}
\psfrag{GI}[][][1]{$\gamma_\infty \; \;  $}
\includegraphics[scale=0.8]{{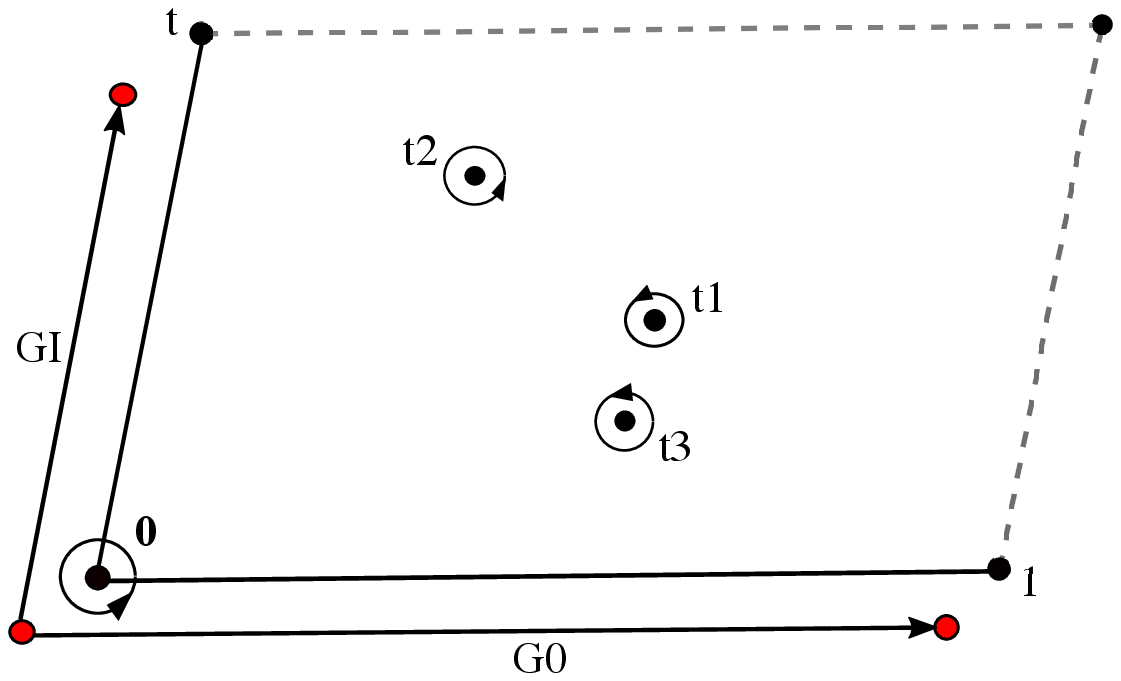}}
\caption{}\label{Fi:totuti}
\end{figure}
\end{center}

 It follows that one has the following explicit expression for the 
 linear holonomy map 
  \eqref{E:VeechMapTorelli} (to make the notations simpler, we do not specify the subscripts $1,n$ starting from now): 
 \begin{align*}
{h}
^\alpha : \,  \mathcal T\!\!\! {\it or}_{1,n} & \longrightarrow \mathbb U^2\\
 (\tau,z) & \longmapsto \big(\rho_0(\tau,z), \rho_\infty(\tau,z)\big) \, .
 \end{align*}

This map is the composition with the exponential map  $e(\cdot)=\exp(2i\pi\cdot)$ of  \begin{align}
\label{E:xiPrim}
\xi
^\alpha : \,  \mathcal T\!\!\! {\it or}_{1,n} & \longrightarrow \mathbb R^2\\
 (\tau,z) & \longmapsto \big(a_0(\tau,z), a_\infty(\tau,z)\big) \, .\nonumber 
 \end{align}
where $a_0(\tau,z)$ and $a_\infty(\tau,z)$ are 
 defined in 
  \eqref{E:alpha0} and \eqref{E:alphaInfinity} respectively. 
\sk 

\begin{rem}
\label{E:XiAlphaTildeHolAlpha}
{\rm  The map $\xi^\alpha$ defined above is then a real-analytic lift of ${h}^\alpha$ to $\mathbb R^2$.  We do not know if it coincides with the lifted holonomy map $\widetilde{h}^\alpha$ constructed in \S\ref{SS:LiftedHolonomyMap}. But since the former differs from the latter 
up to  translation by an element of $2\pi\mathbb Z^2$, this will be irrelevant for our purpose.}
\end{rem}

With the material introduced so far, one can make   some of the general results obtained by Veech in \cite{Veech} in the case when $g=1$ completely  explicit.

Since we are mainly interested in the case when the leaves of Veech's foliation carry a complex hyperbolic structure, we will assume from now on that  \eqref{E:g=1ComplexHyperbolicCase} holds true. Moreover, there is no loss of generality by assuming that the$\alpha_i$'s are presented in decreasing order. 
Hence,  from now on, one assumes that
\begin{equation}
\label{E:AssumptionAlpha}
-1 < \alpha_n\leq \alpha_{n-1}\leq \cdots \leq \alpha_2 <0 < \alpha_1 < 1. 
\end{equation}

\begin{prop}[Explicit description of  $\mathcal F^\alpha$
on the Torelli space
]${}^{}$
\label{P:FalphaGeneralities}

\begin{enumerate}
\item[$(i)$] The map 
$\xi^\alpha$ 
 is a primitive first integral of Veech's foliation on 
$ \mathcal T\!\!\!{\it or}_{1,n}$. \mk 
\item[$(i\!i)$]  For any $a=(a_0,a_\infty)\in {\rm Im}(\xi^\alpha)$, the leaf $\mathcal F^\alpha_a=(\xi^\alpha)^{-1}(a)$ is the complex subvariety of 
$ \mathcal T\!\!\!{\it or}_{1,n}$
cut out by the affine equation
\begin{equation}
\label{E:EqLeaf}
a_0\tau +\sum_{k=2}^n \alpha_k z_k=a_\infty\, .
\end{equation}
\sk
\item[$(i\!i\!i)$] The image of $\xi
^\alpha$ is $\mathbb R^2$ if $n\geq 3$ and $\mathbb R^2\setminus \alpha_1 \mathbb Z^2$ if $n=2$: 
\begin{equation*}
{\rm Im}\big(\xi^\alpha\big)= \begin{cases}
\mathbb R^2\setminus \alpha_1 \mathbb Z^2 
\quad \mbox{ if }\, n=2; \\
\mathbb R^2
\hspace{1.55cm} \mbox{ if }\, n\geq 3\, .
\end{cases}
\end{equation*}
\mk 
\item[$(i\!v)$] Veech's foliation $\mathcal F^\alpha$ is invariant by ${\rm Sp}_{1,n}(\mathbb Z)\simeq 
{\rm SL}_2(\mathbb Z)\ltimes \big(\mathbb Z^2\big)^{n-1}$. \sk 

More precisely, one has
 $$g^{-1}\big(\mathcal F^\alpha_a\big)=\mathcal F^\alpha_{g\bullet a}$$ 
 for any 
 $a=(a_0,a_\infty)\in {\rm Im}(\xi^\alpha)$
 and any $g=\big(M,  (\boldsymbol{k}, \boldsymbol{l})\big)\in 
{\rm SL}_2(\mathbb Z)\ltimes (\mathbb Z^2)^{n-1}$, for a certain action  $\bullet$ of this group on $\mathbb R^2$ given explicitly by 
\begin{equation}
\label{E:ActionOnHolonomy}
\Big(M , \big(\boldsymbol{k}, \boldsymbol{l}\big)
\Big) \bullet \big(a_0,a_\infty\big)= 
\left(a_0 a-a_\infty c  +\sum_{i=2}^n \alpha_i l_i    \, , \, -a_0 b+ a_\infty
d -\sum_{i=2}^n \alpha_i k_i 
\right)
\end{equation}
if $M= \tiny{\big[\!\!
\begin{tabular}{cc}
$a$ \!\!&\!\! \!\!\!\!$b$\vspace{-0.1cm}\\
$c$ \!\!&\!\! \!\!\!\! $d$
\end{tabular}\!\!
\big]}\in {\rm SL}_2(\mathbb Z)
$ and 
$(\boldsymbol{k}, \boldsymbol{l})=\big((k_i,l_i)\big)_{i=2}^n\in \big(\mathbb Z^2\big)^{n-2}$.
\end{enumerate}
\end{prop}
\begin{proof}
The fact that $\xi^\alpha$ is a first integral for $\mathcal F^\alpha$ has been established in the discussion preceding the proposition. The fact that it is primitive follows from $(i\!i)$ since any equation 
of the form \eqref{E:EqLeaf} cuts out a connected subset of $\mathcal T\!\!\!{\it or}_{1,n}$. 
\sk 

Considering the formulae \eqref{E:alpha0} and    \eqref{E:alphaInfinity}, the proof of $(i\!i)$ is straightforward. 
\sk 

To prove $(i\!i\!i)$,  remark that 
$\mathcal T\!\!\!{\it or}_{1,n}$
 is nothing else  but $\mathbb H\times \mathbb C^{n-1}$  minus the union of the complex hypersurfaces $\Sigma_{i,j}^{p,q}$
 cut out by 
\begin{equation}
\label{E:Sigmaijpq}
z_i-z_j+p+q\tau=0\, , 
\end{equation}
for $(p,q)\in \mathbb Z^2$ and $i,j$ such that $1\leq i<j\leq  n$ (remember that $z_1=0$ according to our normalization).   For $(a_0,a_\infty)\in \mathbb R^2$,  \eqref{E:EqLeaf} has no solution in $\mathcal T\!\!\!{\it or}_{1,n}$ if and only if it cuts out one of the hypersurfaces  $\Sigma_{i,j}^{p,q}$. As a consequence,  the linear parts (in $(\tau,z)$) of the affine equations  \eqref{E:EqLeaf} and 
 \eqref{E:Sigmaijpq} should be proportional. Since all the $\alpha_i$'s are negative for $i\geq 2 $ according to our assumption \eqref{E:AssumptionAlpha}, this is clearly impossible if $n\geq 3$. Consequently, one has ${\rm Im}(\xi^\alpha)=\mathbb R^2$ when $n\geq 3$. 

When $n=2$, the equations \eqref{E:AssumptionAlpha} reduce to the following one: 
$q\tau+z_2+p=0$ with $(p,q)\in \mathbb Z^2$. Such an equation is proportional to an equation of the form $a_0\tau+\alpha_2z_2-a_\infty=0$ if and only if $(a_0,a_\infty)\in \alpha_2 \mathbb Z^2$. Since $\alpha_1=-\alpha_2$ when $n=2$, one obtains that 
${\rm Im}(\xi^\alpha)=\mathbb R^2\setminus \alpha_1\mathbb Z^2$ in this case. 
\sk 

Finally, the fact that $\mathcal F^\alpha$ is invariant by the suitable quotient of the pure mapping class group has been proved in greater generality by Veech.  In the particular case we are considering, this can be verified by direct and explicit computations  by using the material of \S\ref{S:TorelliNag}.  In particular, formula \eqref{E:ActionOnHolonomy} for the action of ${\rm SL}_2(\mathbb Z)\ltimes (\mathbb Z^2)^{n-1}$ on $\mathbb R^2$ follows easily from \eqref{E:SP1nACTIONonTor1n}.
\end{proof}

\begin{rem}
\label{R:ImageHolonomy&VeechFoliationCC}
{\rm  (1). The description of the image of the 
 map $\xi^\alpha$ given in $(i\!i\!i)$ allows to answer (in the particular case when $g=1$) a question raised implicitly  by Veech (see the sentence just after Proposition 7.10 in \cite{Veech}).}\sk
 
\noindent {\rm  (2). From 
Remark \ref{E:XiAlphaTildeHolAlpha} and from the proposition above, it follows that the 
phenomenon evoked at the end of \S\ref{SS:VeechLeavesAreNotCoonected}  occurs indeed when $g=1$: in this case, any level subset $\mathcal F_\rho^\alpha$ 
of the linear holonomy map  \eqref{E:VeechMapTorelli}, which is called a `leaf' by Veech in \cite{Veech}, actually has a countable set of connected components, {\it cf.} \eqref{EE}.} 
\end{rem}

From the explicit and elementary description of Veech's foliation on $\mathcal T\!\!\!{\it or}_{1,n}$ given above, one deduces easily the following results. 
\begin{coro}
\label{C:FAlphaDoesNotDependOnAlpha}
Veech's foliation $\mathcal F^\alpha$ depends only on  $[\alpha_1:\alpha_2:\cdots :\alpha_n]\in \mathbb P(\mathbb R^n)$. \sk 

In particular, when $n=2$, Veech's foliation does not depend on $\alpha$.\end{coro}

The preceding statement concerns only $\mathcal F^\alpha$ viewed as a real-analytic foliation of $\mathcal T\!\!\!{\it or}_{1,n}$. If its leaves do depend only on $\alpha$ up to a scaling factor,  it does not apply 
to the complex hyperbolic structures they carry: they do not  depend only on $[\alpha]$ but on $\alpha$ as well (see Theorem \ref{T:Main-n=2} when  $n=2$  for instance). 

\begin{prop} 
\label{P:Pi1ofaLeafInjectsInT1n}
For any leaf $\mathcal F_a^\alpha$ of Veech's foliation on the Torelli space: 
\begin{enumerate}
\item   the inclusion $\mathcal F_a^\alpha\subset 
\mathcal T\!\!\!{\it or}_{1,n}$ induces an injective morphism of the corresponding fundamental groups;  
\sk 
\item any connected component  of the preimage of $\mathcal F_a^\alpha$ in 
$\mathcal T\!\!\!{\it eich}_{1,n}$ is simply connected.
\end{enumerate}
\end{prop}
\begin{proof}     
 In the case when $n=2$, any leaf $\mathcal F_r^\alpha$ is isomorphic to $\mathbb H$ (see \S \ref{SS:VeechFoliationg=1n=2} below for some details) thus is simply connected hence there is nothing to prove. \sk 
 
We sketch a proof of the proposition in the case when $n=3$. The proof in the general case is similar and left to the reader.  Let $r\in \mathbb R^2$ be fixed. One verifies that the linear projection 
 $\mathcal T\!\!{\it or}_{1,n}\rightarrow \mathbb H$ induces  a trivial topological bundle $\mathcal F_r^\alpha\rightarrow \mathbb H$ by restriction. Let $\tau\in \mathbb H$ be fixed and denote by $\mathcal F_r^\alpha(\tau)$ the fiber over this point.  Since $\mathbb H$ is simply connected, the inclusion of $\mathcal F_r^\alpha(\tau)$ into $\mathcal F_{r}^\alpha$ induces an identification of the corresponding fundamental groups. 

Let $\eta=(\tau,z_2^*,z_3^*)$ be an arbitrary element of $\mathcal F_r^\alpha(\tau)$ and consider the map from $\mathbb C$ into $\{\tau\}\times \mathbb C^2$ that associates  the 3-uplet 
$ (\tau, z_2^*+\alpha_3 \xi ,z_3^*-\alpha_2\xi)
$ to any $\xi\in \mathbb C$.  One promptly verifies that there exists a discrete  countable subset $C_\eta=C_{r,\eta}^\alpha\subset \mathbb C$ such that, by  restriction, the preceding injective affine map induces an isomorphism: $i: \mathbb C\setminus C_\eta\simeq \mathcal F_r^\alpha(\tau)$.  Moreover, for $\eta$ sufficiently generic in $\mathcal F_r^\alpha(\tau)$, the segment $]0, c  [$ does not meet $C_\eta$, this  for any $c\in C_\eta$.  
Then for any such $c$, let $\gamma_c $ be the homotopy class of a path 
in $(\mathbb C\setminus C_\eta, 0)$
consisting of the concatenation of $]0, (1-\epsilon_c) c[$ with $0<\epsilon_c<<1$, then of a circular loop in the direct send,  of center $c$, with radius 
$\epsilon_c$ starting and finishing at $(1-\epsilon_c)c$ then going back to $0$ along the segment 
$] (1-\epsilon_c) c,0[$.  The classes $\gamma_c$ for $c\in C_\eta$ freely generate $\pi_1(\mathbb C\setminus C_\eta, 0)\simeq \pi_1(\mathcal F_r^\alpha,\eta)$,  hence to prove the proposition, it suffices to prove that the class of $i_*(\gamma_c)$ is not trivial in $\pi_1(\mathcal T\!\!\! or_{1,n},\eta)$ for every $c\in C_\eta$.   

For $c\in C_\eta$ arbitrary, there is a  divisor $D_c(p,q)$ cut out by an equation of the form 
$z_2-(p+q\tau)=0$, $z_3-(p+q\tau)=0$ or $z_2-z_3-(p+q\tau)=0$ for some integers $p,q\in \mathbb Z$ such that $c$ is the intersection point of $\mathcal F_r^\alpha$ with $D_c(p,q)$.
\sk

Fact: for any $c\in C_\eta$, the homology class of $\gamma_c$ in $\mathcal T\!\!\! or_{1,n}$ is not trivial. In particular, this implies that $\gamma_c$, viewed as an element  of $\pi_1(\mathcal T\!\!\! or_{1,n})$,  is not trivial. Then the natural map $\pi_1(\mathcal F_r^\alpha)\rightarrow \pi_1(
\mathcal T\!\!\! or_{1,n})$ is injective which proves (1).
\sk 

Finally, the second point of the proposition follows  at once from the first since the projection $\mathcal T\!\!\!{\it eich}_{1,n}\rightarrow \mathcal T\!\!\! {\it or}_{1,n}$ is nothing else but the universal covering map of the Torelli space. 
 \end{proof}

\subsubsection{\bf Algebraic leaves of Veech's foliation}
\label{SS:AlgebraicLeavesg=1n>1} 
Using Proposition \ref{P:FalphaGeneralities}, it is easy to determine and to describe the algebraic leaves of Veech's foliation on $\mathscr M_{1,n}$. 
\sk 

Since the case $n=2$ is particular and because we are going to focus on it in the sequel, we left it aside until 
\S\ref{SS:VeechFoliationg=1n=2} 
 and assume $n\geq 3$ in the lines below.\mk

\paragraph{}\hspace{-0.3cm} For   $a=(a_0,a_\infty)\in \mathbb R^2$, let  
$\mathscr F_a^\alpha$ be the corresponding leaf of Veech's foliation 
$\mathscr F^\alpha$ 
on $\mathscr M_{1,n}$: it is the image of $\mathcal F_a^\alpha\subset \mathcal T\!\!\!{\it or}_{1,n}$ by the action of ${\rm SL}_2(\mathbb Z)\ltimes (\mathbb Z^2)^{n-1}$.  The question we are interested in it twofold:  first, we want  to determine the lifted holonomies $a$'s such that $\mathscr F_a^\alpha$ is an algebraic subvariety of the moduli space; secondly, we  would like to give a 
description of such leaves.\sk 

A preliminary remark is in order: 
 on  the moduli space $\mathscr M_{1,n}$,  Veech's foliation is not truly a foliation but  an orbi-foliation. Consequently, from a rigorous point of view,  the algebraic leaves of $\mathscr F^\alpha$, if any,  are a priori algebraic sub-orbifolds of $\mathscr M_{1,n}$. However, this subtlety, if important for what concerns the complex hyperbolic structure on the algebraic leaves, is not really relevant for what interests us  here, namely their topological/geometric description.  For this reason, we will not consider this point further  
 and will abusively  speak only of subvarieties and not of  sub-orbifolds 
 in the lines below. \mk 
 
\paragraph{}\hspace{-0.3cm}
For $a=(a_0,a_\infty)\in \mathbb R^2$, 
we denote its orbit
 under  ${\rm SL}_2(\mathbb Z)\ltimes (\mathbb Z^2)^{n-1}$ by: 
$$
	[  a ]= \big[a_0,a_\infty\big]=
	\left(
	{\rm SL}_2(\mathbb Z)\ltimes \big(\mathbb Z^2\big)^{n-1} \right) \bullet a\subset \mathbb R^2\, .
$$

According to a classical result of the theory of foliations (see \cite[p.51]{CamachoLinsNeto} for instance) a necessary  and sufficient condition for the leaf $\mathscr F^\alpha_a$ to be an  (analytic) subvariety of $\mathscr M_{1,n}$ is that 
$[a]$  be a discrete subset of $\mathbb R^2$.  \sk 

From \eqref{E:ActionOnHolonomy}, one gets easily $ {\rm Id}\ltimes(\mathbb Z^2)^{n-1}\bullet a=a+ \mathbb Z(\alpha )^2$ where $\mathbb Z(\alpha )$ stands for the $\mathbb Z$-submodule of $\mathbb R$ spanned by the $\alpha_i$'s, {\it i.e.} $\mathbb Z(\alpha )=\sum_{i=1}^n \alpha_i \mathbb Z$.   Thus a necessary condition for  $[a]$ to be discrete  is that   the $\alpha_i$'s  all are commensurable, {\it i.e.}\;there exists a non-zero real constant $\lambda $ such that $\lambda \alpha_i\in \mathbb Q$ for $i=1,\ldots,n$. \sk 

Assuming that $\alpha$ is commensurable, let $\lambda$ be the  positive real number such that $\mathbb Z(\alpha )=\lambda \mathbb Z$.  Thus one has 
  \begin{equation}
  \label{E:Z2n-1ActingOnHolonomy}
\left(\small{\bigg[\!\!
\begin{tabular}{cc}
$1$ \!\!&\!\! \!\!\!\!$0$\vspace{-0.1cm}\\
$0$ \!\!&\!\! \!\!\!\! $1$
\end{tabular}\!\!
\bigg]}  \ltimes   \Big(\mathbb Z^2\Big)^{n-1}\right)\bullet a=a+ \lambda  \mathbb Z^2
\end{equation}

We denote by $\mathbb Z^{n-1}_{\boldsymbol{l}}$  the subgroup of 
$(\mathbb Z^{n-1})^2$ formed by pairs $(\boldsymbol{k},\boldsymbol{l})\in (\mathbb Z^{n-1})^2$ with $\boldsymbol{k}=0$.  Setting 
$a=(a_0,a_\infty)$, it follows immediately from \eqref{E:ActionOnHolonomy} that 
  \begin{equation*}
 { \left(\small{\bigg[\!\!
\begin{tabular}{cc}
$1$ \!\!&\!\! \!\!\!\!$1$\vspace{-0.1cm}\\
$\mathbb Z$ \!\!&\!\! \!\!\!\! $0$
\end{tabular}\!\!
\bigg]},  \mathbb Z^{n-1}_{\boldsymbol{l}} \right)}\bullet a= \, \Big(a_0+a_\infty \mathbb Z+ \mathbb Z(\alpha) 
, a_\infty\Big).
  \end{equation*}

It comes that if $[a] $ is discrete then $a_\infty \mathbb Z+\mathbb Z(\alpha)=a_\infty \mathbb Z+\lambda \mathbb Z$ is discrete in $\mathbb R$ 
which implies that $a_\infty \in \lambda \mathbb Q$.  Using a similar argument, one obtains that  $a_0 \in \lambda \mathbb Q$ is also a necessary  condition for the orbit $[a]$ to be discrete in $\mathbb R^2$. 
 \sk 

At this point, we have proved that  in order for  the leaf $\mathscr F_a^\alpha$ to be a closed analytic subvariety of $\mathscr M_{1,n}$, it is necessary that $ (\alpha,a)=(\alpha_1,\ldots,\alpha_n,a_0,a_\infty)$ be commensurable.
We are going to see that this condition  is also sufficient and  actually implies the algebraicity of the considered leaf. 
\mk 

\paragraph{}\hspace{-0.3cm} 
We assume that $ (\alpha,a)$ is commensurable. Our goal now is to prove that the leaf $\mathscr F_a^\alpha$ is an  algebraic subvariety of $\mathscr M_{1,n}$.  We will give a detailed proof of this fact only in the case when $n=3$. We claim that  the general case when $n \geq  3$ can   be treated 
 in the exact same way 
but  let the verification of that to the reader. 
\sk

As above,  let $\lambda>0$ be such that $\lambda \mathbb Z=\mathbb Z(\alpha)$ (note that $\lambda$ is uniquely characterized by this  equality).   Since 
the two foliations 
$\mathcal F^\alpha$ and $\mathcal F^{\alpha/\lambda}$ 
coincide (more precisely, from \eqref{E:EqLeaf},  it comes that 
$\mathcal F^\alpha_b=\mathcal F^{\alpha/\lambda}_{b/\lambda}
$ for every $b\in \mathbb R^2$), there is no loss in generality by assuming that $\lambda=1$ or equivalently, that 
\begin{equation}
\label{E:HypothesisLambda=1}
\begin{tabular}{l}
$\bullet$ 
one has $\alpha_i=-p_i$  for $i=1,\ldots,n$,  
for some positive integers \\ {\hspace{0.3cm}}$p_1,\ldots,p_n$
 such that $p_1+\cdots+p_n=0$ and  $\gcd(p_2,\ldots,p_n)=1$;\mk  \\
$\bullet$  $a$ is rational, {\it i.e.} $ a \in \mathbb Q^2$.
\end{tabular}
\end{equation}

In what follows, we assume that these 	assumptions hold true.\mk

\paragraph{}\hspace{-0.3cm} 
To show that $[a]$ is discrete when $a$ is rational, we first determine  a normal form for a representative of such an orbit.  
\begin{prop}
\label{P:NormalFormForAinQ2}
${}^{}$

 \begin{enumerate}
 \item  For $a\in \mathbb Q^2$,  let $N$ be the smallest positive integer such that $N a \in \mathbb Z^2$.  
\begin{enumerate}
\item One has  $[a]=\big[0,  -1 /N\big]$. \sk
\item 
If $N=1$ (that is, if $a\in \mathbb Z^2$), then $[a]=\big[0,0\big]$.
\end{enumerate}
\sk 
\item  The orbit $[a]$ is discrete in \,$\mathbb R^2$ if and only if $(\alpha,a)$ is commensurable.
\sk 

\end{enumerate}
\end{prop}

\begin{proof}
 For $a\in \mathbb Q^2$, one can write 
 $a_0=p_0/q$ and $a_\infty=p_\infty/q$ for some integers 
 $p_0,p_\infty$ and $q>0$ such that $\gcd(p_0,p_\infty,q)=1$. 
 Let $p$ be the greatest common divisor of $p_0$ and $p_\infty$: $p=\gcd(p_0,p_\infty)$.  
 
 From the proof of Lemma 3 in \cite{Mazzocco}, it comes that 
  $\Gamma(2)\bullet a$,  hence $[a]$ contains one of the three following elements: 
  $( {p}/{q},0),   ( {p}/{q}, {p}/{q})$ or $(0, {p}/{q})$. 
  
  Since 
  $$
    \begin{bmatrix}
0 & -1 \\
1  &  0 
\end{bmatrix}
\bullet    \left( \frac{p}{q}, 0\right) =
  \begin{bmatrix}
1 & 0 \\
-1  & 1
\end{bmatrix}
\bullet    \left( \frac{p}{q}, \frac{p}{q}\right) = \left( 0, \frac{p}{q}\right)\, ,   $$
  it comes that $(0, {p}/{q}) \in {\rm SL}_2(\mathbb Z)\bullet a \subset [a]$.
  \sk 
  
   Because $\gcd(p_0,p_\infty,q)=\gcd(p,q)=1$, there exist two integers $d $ and $k$ such that   $dp-kq=1$.    From the relation 
     $$
   \Bigg(  \tiny{\begin{bmatrix}
 1 &  1-d \\
 -1 &   d\end{bmatrix}}
, \big(  k ,0\big) 
\Bigg)
\bullet    \left( 0, \frac{p}{q}\right) =  
\left( \frac{p}{q}, \frac{dp-kq}{q}\right)=\left( \frac{p}{q}, \frac{1}{q}\right)\, , 
$$
   one deduces that $(p/q,1/q)\in [a]$.  Since $(p,q,1)=1$, it follows from the arguments above that 
   $(0,1/q)\in {\rm SL}_2(\mathbb Z)\bullet (p/q,1/q)$. This implies  that $(0,1/q)$  belongs to $[a]$,  hence the same holds true for $(0,-1/q)=(-{\rm Id})\bullet (0,1/q)$.  \sk
   
   Since $qa=(p_0,p_\infty)\in \mathbb Z^2$ one has $N\leq q$ where $N$ stands for the integer defined in the statement of the lemma.  On the other hand, since $\gcd(p_0,p_\infty,q)=1$, there exists $u_0,u_\infty,v\in \mathbb Z$ such that $u_0p_0+u_\infty p_\infty +vq=1$. 
  Since $Na\in \mathbb Z^2$, one has $u_0 Na_0+u_\infty N a_\infty=
   N(1-vq)/q\in \mathbb Z$, which  implies that $q$ divides $N$. 
   This shows that $q=N$,  thus 
   that the  first point of (1) holds true.\sk

   When $a\in \mathbb Z^2$,
 the fact that $(0,0)\in [a]$ follows immediately from 
\eqref{E:Z2n-1ActingOnHolonomy} (recall that we have assumed that $\lambda=1$), which proves (b).
\mk 

Finally, using \eqref{E:ActionOnHolonomy}, it is easy to verify that all the orbits $[0,0]$ and  $[0,-1/N]$ with $N\geq 2$  are discrete subsets of $\mathbb R^2$.   Assertion (2) follows immediately.\end{proof}
\mk 

 From the preceding proposition, it follows that the leaves of Veech's foliation 
$\mathscr F^\alpha$ which are closed analytic subvarieties of 
$\mathscr M_{1,3}$ are exactly the one associated to the following `lifted holonomies'
 \begin{equation}
 \label{E:NormalFormLiftedHolonomy}
\qquad (0,0) 
\qquad \mbox{ and }
\qquad  (0,-1/N) \quad 
 \mbox{with }\, N
 \geq 2
 \, .  
\end{equation}

 We will use the following notations for the corresponding leaves:
\begin{equation}
\label{E:LeavesFNinM1n}
\mathscr F^\alpha_0=\mathscr F^\alpha_{(0,0)} 
\qquad \mbox{ and }
\qquad  \mathscr F^\alpha_N=\mathscr F^\alpha_{(0,-1/N)}
 \, .  
\end{equation}

Let $p_1,p_2$ and $p_3$ be the positive integers such that $\alpha_i=-p_i$ for $i=1,2,3$ (remember our simplifying assumption \eqref{E:HypothesisLambda=1}).  The leaves in the Torelli space 
 which correspond to the  `lifted holonomies' \eqref{E:NormalFormLiftedHolonomy} 
 are  the following: 
\begin{align}
\mathcal F_{0}^\alpha =\mathcal F_{(0,0)}^\alpha= & \, \Big\{ \big(\tau,z_2,z_3\big)\in \mathcal T\!\!\!{\it or}_{1,3}\; \big\lvert \;  p_2z_2+p_3z_3=0\, 
\Big\} \nonumber
\\ 
\label{E:EquationF(N)}
\mbox{and}\quad 
\mathcal F_{N}^\alpha=  \mathcal F_{(0,1/N)}^\alpha= & \,  \, \left\{ \big(\tau,z_2,z_3\big)\in \mathcal T\!\!\!{\it or}_{1,3}\; \big\lvert \;  p_2z_2+p_3z_3=\frac{1}{N}\, 
\right\}\, . 
\end{align}

Note that the preceding leaves are (possibly orbifold) coverings of the leaves 
\eqref{E:LeavesFNinM1n}: for any $N\neq 1$, the image of 
$\mathcal F_{N}^\alpha$ by the quotient map $\mathcal T\!\!\!{\it or}_{1,3}\rightarrow \mathscr M_{1,3}$ is $\mathscr F_{N}^\alpha$. 
\mk 

\paragraph{}\hspace{-0.3cm}
\label{S:LEAF F0Alpha}
 We are going to consider carefully the case of the leaf $\mathscr F_{0}^\alpha$. We will deal with the case of $\mathscr F_{N}^\alpha$ with $N\geq 2$  more succinctly in the next subsection. 
\sk 

In what follows, we set $p=(p_1,p_2,p_3)$. 
Remember that $p_2$ and $p_3$ determine $p_1$ since the latter is the sum of the two former: $p_1=p_2+p_3$. Note that according to 
\eqref{E:HypothesisLambda=1}, one has  $\gcd(p_1,p_2,p_3)=\gcd(p_2,p_3)=1$. 
\mk 
%

Consider the affine map from $\mathbb H\times  \mathbb C$ to $\mathbb H\times \mathbb C^2$ defined 
 for any $(\tau,\xi) \in \mathbb H\times \mathbb C$ 
by
\begin{align*}
U_0(\tau,\xi)=\big(\tau, p_3 \xi,-p_2\xi\big)\, . 
\end{align*}  

 Let $\mathcal U_p$ be  the inverse image
 of $\mathcal T\!\!\!{\it or}_{1,3}\subset \mathbb H\times \mathbb C^2$
  by  $U_0$. 
 One verifies easily that 
\begin{equation}
\label{E:Up2p3}
\mathcal U_p= 
\left\{ (\tau,\xi)\in \mathbb H\times \mathbb C \; \Big\lvert \; 
\xi \not \in \bigg(\frac{1}{p_1}\mathbb Z_\tau  \cup  \frac{1}{p_2}\mathbb Z_\tau
\cup  \frac{1}{p_3}\mathbb Z_\tau
\bigg)
\right\}
\end{equation}
and,  by restriction, $U_0$ induces a global holomorphic isomorphism
\begin{equation}
\label{E:Up2p3SimeqF0alpha}
U_0 : \mathcal U_p 
\simeq 
\mathcal F_0^\alpha \subset \mathcal T\!\!\!{\it or}_{1,3}\, . 
\end{equation}

Let ${\rm Fix}(0)$ be the subgroup of ${\rm Sp}_{1,3}(\mathbb Z)$ which leaves $\mathcal F_{0}^\alpha$ globally invariant. It is nothing else but   the subgroup of $g\in {\rm SL}_2(\mathbb Z)\ltimes (\mathbb Z^2)^2$ such that $g\bullet (0,0)=(0,0)$. From \eqref{E:ActionOnHolonomy}, it is clear that   $g=(M, (k_2,l_2),(k_3,l_3))$ is of this kind if and only if 
 $p_2l_2+p_3l_3=p_2k_2+p_3k_3=0$. 
  It follows that 
$$
 {\rm SL}_2(\mathbb Z)\ltimes \mathbb Z^2
 \simeq 
{\rm Fix}(0)
\; ,  
$$
where the injection $\mathbb Z^2\hookrightarrow \big(\mathbb Z^2\big)^2$ is given by $(k,l)\mapsto  \big( p_3(k,l), -p_2(k,l)\big)$. \sk 

By pull-back by $U_0$, one obtains immediately that the corresponding action of $ {\rm SL}_2(\mathbb Z)\ltimes \mathbb Z^2$ on $\mathcal U_p$ is given by 
\begin{equation}
   \label{E:gogigu}
\left(
\small{\bigg[\!\!
\begin{tabular}{cc}
$a$ \!\!&\!\! \!\!\!\!$b$\vspace{-0.1cm}\\
$c$ \!\!&\!\! \!\!\!\! $d$
\end{tabular}\!\!
\bigg]} , \big( k,l\big) \right) \cdot \big(\tau,\xi\big)=\left(
\frac{a\tau+b}{c\tau+d} , \frac{\xi+k+l\tau }{c\tau+d}
\right)\, .
\end{equation}

For any subgroup $\Gamma$ in $\mathrm{SL}_{2}(\mathbb Z)$, one sets 
$$
\mathscr M_{1,3}(\Gamma):= \mathcal T\!\!\!{\it or}_{1,3}\big/ \Gamma\ltimes \big( \mathbb Z^2\big)^2\, .
$$
It is an orbifold covering of $\mathscr M_{1,3}$ which is finite hence algebraic if $\Gamma$ has finite index in ${\rm SL}_2(\mathbb Z)$.  In this 
 case, the image  $\mathscr F_0^{\alpha}(\Gamma)$  of $\mathcal F_0^\alpha$ in $\mathscr M_{1,3}(\Gamma)$ is algebraic if  and only if $\mathscr  F_0^\alpha$ is an algebraic subvariety of $
\mathscr M_{1,3}$. 
We are going to use this equivalence for  
	a group  $\Gamma_{\!\!p}$ 
	which satisfies the following properties:
	\begin{enumerate}
	\item[(P1).] 
it 	is a subgroup of finite index of $\Gamma\big({\rm lcm}(p_1,p_2,p_3)\big) $; 
	and \sk
	\item[(P2).]	
  it acts without fixed point on $\mathbb H$.
	\end{enumerate}
	
For instance, setting  $$
M_p=
 	\begin{cases}
\, 4 
 \hspace{3cm} \mbox{ if } p_2=p_3;
\\ 
\, {\rm lcm}(p_1,p_2,p_3)
\hspace{0.64cm}  \mbox{ otherwise,  }
	\end{cases}
	$$
	 one can take  for $\Gamma_{\!\!p}$ 
the congruence subgroup of level $M_p$: 	 
$$ \Gamma_{\!\!p}	 =\Gamma\big(M_p\big)$$ 
(the case when $p_2=p_3$ is particular: this equality  implies that  $p_2=p_3=1$ since $\gcd(p_2,p_3)=1$. 
	%
	   Consequently
	   $\Gamma({\rm lcm}(p_1,p_2,p_3))=\Gamma(2)$ and this group   contains $-{\rm Id}$ hence  does not act effectively on $\mathbb H$). 
\sk

	Since $M=M_p\geq 3$ in every case, $\Gamma_{\!\!p}$ satisfies the properties (P1) and (P2) stated above. Consequently,  
	  the quotient of $\mathbb H\times \mathbb C$ by  the action  \eqref{E:gogigu} is the total space of the (non-compact) 
	 modular elliptic surface of level $M$:\footnote{See  \cite{Shioda} for a reference. Note that  we do not use the most basic  construction of the theory 
	 of modular elliptic surfaces, namely  that 	 \eqref{E:Ep} can be compactified over $X(M_p)$ by adding 
	 as fibers over the cusps some generalized elliptic curves  ({\it cf.} also \cite[\S8]{Kodaira}).}
\begin{equation}
\label{E:Ep}
	 \mathcal E_p:=\mathcal E( M_p)\longrightarrow Y(M_p)\, . 
	 \end{equation}

	According to \cite[\S5]{Shioda}, $\mathcal E_p$ comes with $M^2$ sections of $M$-torsion forming an abelian  group $S(\mathcal E_p)$ 
	isomorphic to $(\mathbb Z/ M \mathbb Z)^2$. For any divisor $m$ of $ M $, one denotes by $\mathcal E_p[m]$ the union of the images of the  elements of order $m$ of $S(\mathcal E_p)$: 
	$$
	\mathcal E_p[m] = \bigcup_{\substack{ \sigma\in S(\mathcal E_p)    \\   m\cdot \sigma=0  }} \sigma\big( Y(M_p)\big)\subset \mathcal E_p\, . 
	$$

	We are almost ready to state our result about the leaf $\mathscr F_0^\alpha$.
	To simplify the notations, we denote respectively by $
\mathscr M_{1,3}(p)$ and $\mathscr F_0^{\alpha}(p)$ the 
intermediary moduli space $
\mathscr M_{1,3}(\Gamma_{\!\!p})$ and the image of the leaf $\mathcal F_0^\alpha$ in it. \sk

The map  $U_0$ induces an isomorphism 
$$
 \mathcal U_p\big/\big( 
\Gamma_{\!\! p }
\ltimes \mathbb Z^2\big)
\simeq
\mathscr F_0^\alpha (p)\, .
$$

Using \eqref{E:Up2p3} and  \eqref{E:gogigu}, it is then 
easy  to deduce the 
\begin{prop} 
\label{P:JustHere}
The map \eqref{E:Up2p3SimeqF0alpha} induces an embedding 
$$
\mathcal E_{p}\setminus \Big( \mathcal E_{p}[p_1] \cup 
\mathcal E_{p}[p_2]\cup \mathcal E_{p}[p_3]\Big) \longhookrightarrow 
\mathscr M_{1,3}(p)
$$ which is algebraic and whose image is 
the leaf\;\,$\mathscr F_0^\alpha (p)$.
\end{prop}
In short, this result says that the inverse image of $\mathscr  F_0^\alpha$ in  a certain finite covering of $\mathscr M_{1,3}$  is an elliptic modular surface 
from which the images of some torsion sections have been removed. 
There is no difficulty to deduce from it a description of  $\mathscr  F_0^\alpha$ itself.  
But since ${\rm SL}_2(\mathbb Z)$ has elliptic points and contains $-{\rm Id}$, the latter is not as nice as the description of $\mathscr F_0^\alpha (p)$ given above. 
\begin{coro} 
\begin{enumerate}
\item The  projection $\mathcal T\!\!\!{\it or}_{1,3}\rightarrow \mathbb H$ induces a dominant rational map $ \mathscr F_0^\alpha\rightarrow Y(1)\simeq \mathbb C$  whose fibers are punctured projective lines. 
\sk 
\item  The fiber over  any $j(\tau)$ distinct from $0$ and $1728$ (the two elliptic points of $Y(1)=\mathbb C$) is the  quotient of the punctured elliptic 
curve $E_\tau\setminus \big(  E_\tau[p_1]\cup E_\tau[p_2] \cup E_\tau[p_3]\big)$  by the elliptic involution. 
\sk 
\item Both the description of the fibers over 0 and $1728$ are similar but more involved. 
\end{enumerate}
\end{coro}

To end this subsection, we would like to make  the particular case when $p_2=p_3=1$ more explicit (note that this condition is equivalent to $\alpha_2=\alpha_3$). 
 It is more convenient to describe the inverse image 
$\mathscr F_0^\alpha(\Gamma(2))$
 of 
$\mathscr F_0^\alpha$ in $\mathscr M_{1,3}(\Gamma(2))$:  it is the bundle over $Y(2)=\mathbb P^1\setminus\{0, 1,\infty\}$,  the fiber of which over $\lambda\in \mathbb C\setminus \{0,1\}$ is the 4-punctured projective line $\mathbb P^1\setminus \{0,1,\lambda,\infty\}$.   As an algebraic variety, 
$\mathscr F_0^\alpha(\Gamma(2))$ 
 is then isomorphic to the moduli space $\mathscr M_{0,5}$. Actually, there is more: in \S\ref{S:TheCaseN=2LinkedWithClassicalHypergeometry}, we will see  that,   endowed with 
 Veech's complex hyperbolic structure,  $\mathscr F_0^\alpha(\Gamma(2))$
can be identified with  a Picard/Deligne-Mostow/Thurston moduli space. 
\mk

\paragraph{}\hspace{-0.3cm}
\label{P:FNalpha}
 We now consider  succinctly the case of the leaf $\mathscr F_{N}^\alpha$ when $N$ is a fixed integer bigger than or equal to 2. One  proceeds as for $\mathcal F_0^\alpha$. 
\sk 

Since $\mathscr F_{N}^\alpha$ is cut  out  by $p_2z_2+p_3z_3=1/N$  in the Torelli space (see \eqref{E:EquationF(N)}), one gets that, by restriction,  
the affine map 
$$\xi\mapsto \left(p_3\xi+\frac{1}{Np_2},-p_2\xi\right)$$ 
induces a global holomorphic  parametrization of $\mathscr F_{N}^\alpha$ by an open subset $\mathcal U_{p,N}$ of $\mathbb H\times \mathbb C$ which is not difficult to describe explicitly.  
\sk

The stabilizer ${\rm Fix}(N)$ of $(0,-1/N)$ for the action $\bullet$ is easily seen to be  the image of the injective morphism of groups 
\begin{align*}
\Gamma_1(N)\ltimes \mathbb Z^2 & \longmapsto {\rm SL}_2(\mathbb Z)\ltimes \big(   \mathbb Z^2 \big)^2\\
\left(
{\tiny{\begin{bmatrix}
a & b\\
c & d\end{bmatrix}}} , \big(k, l\big)
\right) & \longmapsto 
\left(
{\tiny{\begin{bmatrix}
a & b\\
c & d\end{bmatrix}}} , \big(k_2, l_2\big),  \big(k_3, l_3\big)
\right)
\end{align*}
with 
$$
\big(k_2, l_2\big)=q_2\left(\frac{d-1}{N},\frac{c}{N}\right)+p_3\big(k,l\big)
\quad \mbox{ and }\quad 
  \big(k_3, l_3\big)=q_3\left(\frac{d-1}{N},\frac{c}{N}\right)-p_2\big(k,l\big)
$$
where $(q_2,q_3)$ stands for a (fixed) pair of integers such that $q_2p_2+q_3p_3=1$.
\sk

Embedding $\Gamma_1(Np_2)\ltimes \mathbb Z^2$ into 
  ${\rm Fix}(N)$   by  setting
  $$
  \big(k_2, l_2\big)=\left(\frac{d-1}{Np_2},\frac{c}{Np_2}\right)+p_3\big(k,l\big)
\quad \mbox{ and }\quad 
  \big(k_3, l_3\big)=-p_2\big(k,l\big)\, 
  $$
  one verifies that the induced action of  $\Gamma_1(Np_2)\ltimes \mathbb Z^2$ on $(\tau,\xi)\in \mathcal  U_{p,N}$ is  the usual one ({\it i.e.} is given by 
 \eqref{E:gogigu}).  Consequently, when $Np_2\geq 3$\footnote{The case when $p_2=p_3=1$ and $N=2$ is particular and must be treated separately.}, the inclusion $\mathcal U_{p,N}\subset \mathbb H\times \mathbb C$ induces an algebraic embedding 
$$
\mathscr F^\alpha_N\big(  \Gamma_1(Np_2)  \big)\simeq 
\mathcal  U_{p,N}\big/_{\Gamma_1(Np_2)\ltimes \mathbb Z^2}
\subset  \mathcal E_1(Np_2)\rightarrow Y_1(Np_2)
$$
where $\mathcal E_1(Np_2)$ stands for the total space of the elliptic modular surface associated to $\Gamma_1(Np_2)$.  Moreover, it can be easily seen that the complement of $\mathscr F^\alpha_N\big(  \Gamma_1(Np_2)  \big)$ in $\mathcal E_1(Np_2)$ is the union of certain torsion sections.

\begin{prop}
\begin{enumerate}
\item  For a certain congruence group $\Gamma_{\!\! p,N}$ (which can be explicited), the inverse image of\hspace{0.17cm}$\mathscr F_N^\alpha$ in the  intermediary moduli space\hspace{0.17cm}$\mathscr M_{1,3}(\Gamma_{\!\! p,N})$ is algebraic and isomorphic to the total space of  the modular elliptic surface $\mathcal E(\Gamma_{\!\! p,N})$ 
 from which   the union of some  torsion sections have been removed.\sk 
\item  For $N\geq 3$, the leaf\hspace{0.17cm}$\mathscr F_N^\alpha$ is an algebraic subvariety of\hspace{0.17cm}$\mathscr M_{1,3}$  isomorphic to the 
total space of  the modular elliptic surface $\mathcal E_1(N)\rightarrow Y_1(N)$ from which the union of certain torsion multi-sections have been removed.\sk  
\item  The leaf\hspace{0.17cm}$\mathscr F_2^\alpha$ is a bundle in punctured projective lines over $Y_1(2)$.\end{enumerate}
\end{prop} 

Here again, the dichotomy between the cases when $N=2$ and $N\geq 3$ comes from the fact that $-{\rm Id}$ belongs to $\Gamma_1(2)$ whereas it is not the case for $N\geq 3$. 
\mk

\paragraph{}\hspace{-0.3cm} We think that considering an explicit example will be  quite enlightening. \sk 

We assume that $p_2=p_3=1$ (which is equivalent to $\alpha_2=\alpha_3$) and we fix $N\geq 2$.  The preimage 
$\mathscr F_{N}(2N)$
of $\mathscr F_{N}$ in $\mathscr M_{1,2}(\Gamma(2N))$ admits a  nice description. 

Let $\mathcal E(2N)\rightarrow Y(2N)$ be the modular elliptic curve associated to $\Gamma(2N)$.  We fix a `base point'  $\tau_0\in \mathbb H$. 
Then for any integers $k,l$, $(k+l\tau_0)/2N$ defines a point of $2N$-torsion of $E_{\tau_0}$ which belongs to exactly one of the $(2N)^2$ $2N$-torsion  sections of $\mathcal E(2N)\rightarrow Y(2N)$. We denote the latter section 
 by $[(k+l\tau)/2N]$.  
 \sk
 
 Then $\mathscr F_{N}(2N)$ is isomorphic to the complement in $\mathcal E(2N)$ of the union of 
  $[0]$ and $[1/N]$ with the translation by $[1/2N]$ of the four sections of 2-torsion: 
 $$
 \mathscr F_{N}(2N)\simeq \mathcal E(2N)\setminus \left(
 \big[0\big]\cup \Big[\frac{1}{N}
 \Big] 
\cup \Bigg( \bigcup_{k,l=0,1}
\Big[\frac{1}{2N}+\frac{k+l\tau}{2}
 \Big]
\Bigg)
 \right)\, . 
 $$
(See also Figure \ref{F:FN(2N)} below). 
\mk

\paragraph{}\hspace{-0.3cm} 
To finish our uncomplete study of the algebraic leaves of Veech's foliation on $\mathscr M_{1,n}$ when $n\geq 3$, we state the following result which follows easily from all what has been said before (we do not assume that 
\eqref{E:HypothesisLambda=1} 
 holds true anymore): 
\begin{coro} Let $a\in \mathbb R^2$.  The three  following assertions are equivalent: 
\begin{enumerate}
\item $ (\alpha,a)$ is commensurable, {\it i.e.} $[\alpha_1: \cdots:\alpha_n:a_0:a_\infty]\in \mathbb P(\mathbb Q^{n+2})$;
\sk 
\item  the leaf\hspace{0.17cm}$\mathscr F_a^\alpha$ is a closed analytic subvariety of\hspace{0.17cm}$\mathscr  M_{1,n}$; 
\sk 
\item the leaf\hspace{0.17cm}$\mathscr F_a^\alpha$ is a closed algebraic subvariety of\hspace{0.17cm}$\mathscr M_{1,n}$. 
\end{enumerate}
\end{coro}
\begin{center}
\begin{figure}[!h]
\psfrag{1}[][][1]{$1 $}
\psfrag{t}[][][1]{$\tau $}
\psfrag{T}[][][1]{$[\tau] $}
\psfrag{Y}[][][1]{$Y(2N) $}
\psfrag{E}[][][1]{$E_\tau $}
\psfrag{0}[][][0.9]{$\;\;  \big[ 0\big] $}
\psfrag{2N}[][][0.9]{$\quad \;  \textcolor{Pin}{\big[ \frac{1}{2N}\big] }$}
\psfrag{122N}[][][0.9]{$\quad \textcolor{Pin}{\big[ \frac{1+N}{2N}\big] }$}
\psfrag{t22N}[][][0.9]{$\quad \textcolor{Pin}{\big[ \frac{1+N\tau}{2N}\big] }$}
\psfrag{1t22N}[][][0.9]{$\quad \textcolor{Pin}{\qquad \big[ \frac{1+N(1+\tau)}{2N}\big] }$}
\psfrag{N}[][][0.9]{$ \textcolor{blue}{\qquad \big[ \frac{1}{N}\big] }\quad $}
\includegraphics[scale=0.6]{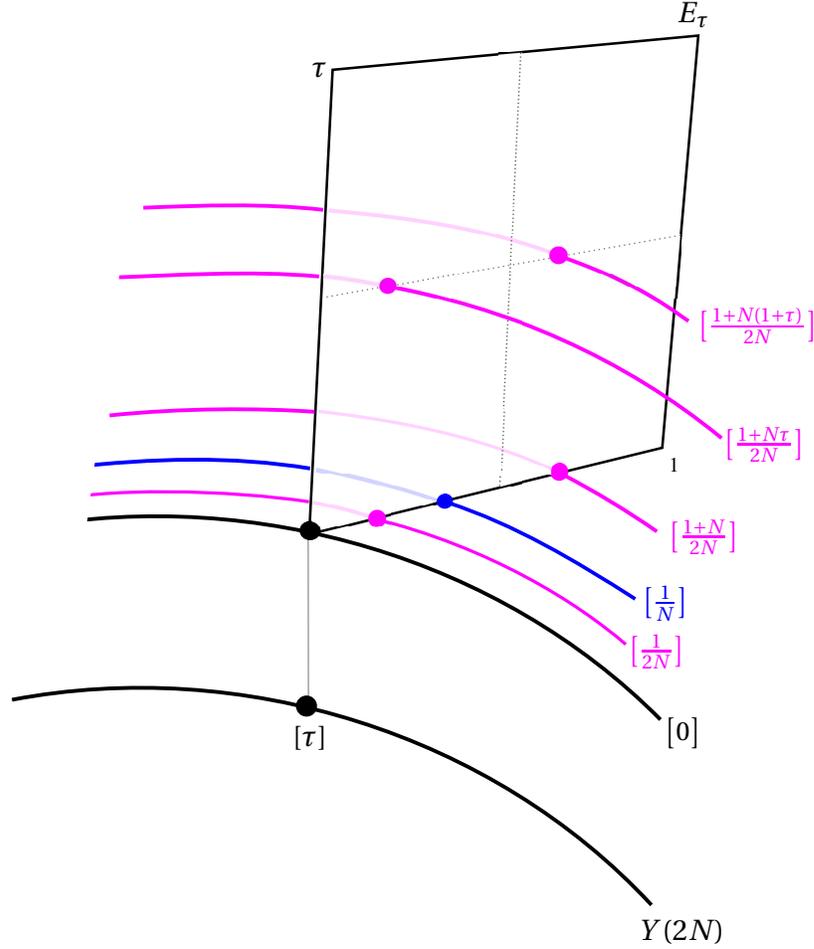}
\caption{The covering $ \mathscr F_{N}(2N)$ of the leaf $ \mathscr F_{N}$ is the total space of the modular surface 
$\mathcal E(2N)\rightarrow Y(2N)$ 
with the six sections $[0]$, $[1/N]$ and $[(1+N(k+l\tau))/2N]$ for $k,l=0,1$ removed.}
\label{F:FN(2N)}
\end{figure}
\end{center}

\subsubsection{\bf Some algebraic leaves in $\mathscr M_{1,3}$ related with some  Picard/Deligne-Mos\-tow's  orbifolds} 
\label{SS:nolose}
In the $n=3$ case,  assume that $\alpha=(\alpha_1,\alpha_2,\alpha_3)$ is such that $\alpha_2=\alpha_3=-\alpha_1/2$.  Then the leaf $\mathscr F_0^\alpha$ formed of flat tori with 3 conical singularities 
 is related to a moduli space of flat spheres with five cone points. 
\sk 

Indeed, the equation of $\mathcal F_{0}^\alpha$ in the Torelli space is 
$ z_2+z_3=0$. It follows that an element $E_{\tau,z}$ of this leaf is 
a flat structure on $E_\tau$ with a conical point of angle $\theta_1=2\pi(\alpha_1+1)$ at the origin  and two equal conical angles $
\theta_2=\theta_3=\pi(2-\alpha_1)$ at $[z_2]$ and $[z_3]=[-z_2]$.  This 
flat structure is invariant by the elliptic involution $i: [z]\mapsto [-z]=-[z]$ on $E_\tau$ hence can be pushed-forward by 
$\wp: E_\tau\rightarrow E_\tau/\langle i \rangle\simeq \mathbb P^1$.  
 The flat structure one obtains on $\mathbb P^1$ has three cone points of angle $\pi$ at the image of the three 2-torsion points of $E_\tau$ by $\wp$, one cone point of angle $\theta_1/2=\pi(\alpha_1+1)$ at $\wp(0)=\infty$ and one cone point of angle $\pi(2-\alpha_1)$ at $\wp(z_2)=\wp(z_3)$. \sk 
 
 At a more global level, this shows that when $\alpha_2=\alpha_3$, the leaf $\mathscr F_0^\alpha\subset \mathscr M_{1,3}$ admits a special automorphism of order 2 which induces a biholomorphism
   $$
 \mathscr F_0^\alpha\longrightarrow \mathscr M_{0,\theta(\alpha)}
 $$ onto the moduli space of flat spheres with five conical points $\mathscr M_{0,\theta(\alpha)}$ with associated angle datum  $$
\theta(\alpha)=\big(\pi,\pi,\pi, \pi(1+\alpha_1), \pi(2-\alpha_1) \big)  \, .
$$  

Moreover, it is easy to verify that the preceding map is compatible with the $\mathbb C\mathbb H^2$-structures carried by each of these two moduli spaces of flat surfaces. \sk 

The  5-uplet $\mu(\alpha)=(\mu_1(\alpha),\ldots,\mu_5(\alpha))\in ]0,1[^5$ 
corresponding to the angle datum $\theta(\alpha)$ 
in Deligne-Mostow's notation of \cite{DeligneMostow} is given by 
$$
\mu(\alpha)=\Big(\frac{1}{2},\frac{1}{2},\frac{1}{2}, 
\frac{1-\alpha_1}{2}, \frac{\alpha_1}{2} \Big)  \, . 
$$

Then looking at the table page 86 in \cite{DeligneMostow}, it comes easily that the metric completion of $\mathscr M_{0,\theta(\alpha)}$ is a complex hyperbolic orbifold exactly for two values of $\alpha_1$, namely $\alpha_1=1/3$ and $\alpha_1=2/3$.   It follows that the image of the holonomy of Veech's $\mathbb C\mathbb H^2$-structure of the leaf $\mathscr F_0^\alpha$ is a lattice in ${\rm PU}(1,2)$ exactly when 
$\alpha_1$ is equal to one of these two values.  Note that the two corresponding lattices are isomorphic,   arithmetic and not cocompact.

\subsubsection{\bf Towards a description of the metric completion of an algebraic leaf of Veech's foliation}  
\label{SS:Towards}
We consider here how to describe the metric completion of an algebraic leaf of Veech's foliation when it is endowed with the metric associated to the complex hyperbolic structure it carries. This is a natural question in view of Thurston's results \cite{Thurston} in the genus 0 case.   \sk  

\paragraph{}\hspace{-0.3cm} In \cite{GP}, we have  generalized   the approach initiated by Thurston which relies on surgeries on flat surfaces to the genus 1 case. 
 From our main result in this paper, it comes that, when $\alpha$ is assumed to be rational, then  
  the metric completion $\overline{\mathscr F_N}$ of an algebraic leaf 
${\mathscr F_N}$ of Veech's  foliation 
on $\mathscr M_{1,n}$: 
\begin{enumerate}
\item  carries a complex hyperbolic conifold structure of finite volume which extends Veech's hyperbolic structure of ${\mathscr F_N}$;
\sk 
\item  $\overline{\mathscr F_N}$ is obtained by adding to ${\mathscr F_N}$
some (covering of some) strata of flat tori and of flat spheres obtained as particular degenerations of flat tori whose moduli space is ${\mathscr F_N}$. 
\end{enumerate}
The strata mentioned in (2) which parametrize flat tori with $n'<n$ cone points are obtained by making several cone points collide hence are called  {\bf $\boldsymbol{C}$-strata} ($C$ stands here for {\it `colliding'}). 
The others  parametrizing  flat spheres with $n''\geq n+1$ cone points are obtained by pinching an essential simple curve on some element of $\mathscr F_N$  hence are called  {\bf $\boldsymbol{P}$-strata} ($P$ stands here for {\it `pinching'}).\footnote{See \cite[\S10.1]{GP} where this terminology is introduced.} \sk 

Actually, the main result of \cite{GP} concerns the algebraic `leaves' in $\mathscr M_{1,n}$ in the sense  of  Veech \cite{Veech}. In our notation, such a leaf 
$\mathscr F_{\rho}^\alpha$ is the image in $\mathscr M_{1,n}$ of a level-subset 
$\mathcal F_\rho^\alpha=({\chi}_{1,n}^{\alpha})^{-1}(\rho)
\subset \mathcal T\!\!\!{\it eich}_{1,n}$ by Veech's linear holonomy map ${\chi}_{1,n}^{\alpha}$ of   some element $\rho$ of $ {\rm Hom}^\alpha(\pi_1(1,n),\mathbb U)\simeq \mathbb U^2$ 
 (see \S\ref{SS:VeechLeavesAreNotCoonected}).  The point is that it is not completely clear yet which are the connected components of such a `leaf' 
$\mathscr F_{\rho}^\alpha$ in terms of the irreducible leaves $\mathscr F_{\!N}^\alpha$ considered in the present paper (for instance, the answer depends on $\alpha$ already in the $n=2$ case, see \S\ref{SS:FThetaNNotConnected} below). 

 It follows that the methods developed in \cite{GP} to list the strata which must be added to $\mathscr F_{\rho}^\alpha$ in order to get its metric completion do not apply in an effective way to any of the irreducible leaves $\mathscr F_{\!N}$'s considered    here.  An interesting feature of the analytic approach to Veech's foliation developed above is that it suggests an explicit and effective way to describe 
$\overline{\mathscr F_{\!N}}$ for any $N$ given.  
\mk 

\paragraph{}\hspace{-0.3cm} In the $n=2$ case, one can give a complete and explicit description of the metric completion 
of any leaf $\mathscr F_N\subset \mathscr M_{1,2}$, see  \S\ref{SS:MetricCompletionOfY1(N)} further. In particular, using  the results of \cite{GP}, it comes that the metric completion of a leaf $\mathscr F_N$ is obtained by adjoining to it a finite number of $P$-strata
 which, in this case,  are moduli spaces of flat spheres with three cone points $\mathscr M_{0,\theta}$ for some angle data $\theta=(\theta_1,\theta_2,\theta_3)\in ]0,2\pi[^3$ which can be explicitly given.
\mk 

\paragraph{}\hspace{-0.3cm}  We now say a few words about the $n=3$ case. We take $N\geq 4$ in order to avoid any pathological case. Let $\mathcal E_1(N)\rightarrow Y_1(N)$ be the modular elliptic surface associated to $\Gamma_1(N)$.  Then, as proven above in \S\ref{P:FNalpha},   there exists a finite number of torsion multi-sections $\Sigma(1),\ldots,\Sigma(s)\subset  \mathcal E_1(N)$ such that  $\mathscr F_N$ is isomorphic to $\mathcal E_1(N)\setminus \Sigma$ with 
 $\Sigma=
 \Sigma(1)\cup \ldots\cup \Sigma(s)$. 
   By restriction, one gets a surjective map 
 $$\nu_N: \mathscr F_N= \mathcal E_1(N)\setminus \Sigma\rightarrow  Y_1(N)$$
  which is relevant to describe the first strata (namely the ones of complex codimension 1) which must be attached  to $\mathscr F_N$ along the  inductive process described in \cite[\S7.1]{GP} giving $\overline{\mathscr F_N}$ at the end. \sk

Indeed, by elementary analytic considerations, it is not difficult to see that 
the $C$-strata of codimension 1 which must be added to $\mathscr F_N$ are precisely the multi-sections $\Sigma(i)$ for $i=1,\ldots,s$, which then appear as being horizontal for  $\nu_N$. It is then rather easy  to see that each $\Sigma(i)$ is a non-ramified cover of a certain leaf algebraic leaf $\mathscr F(i)=\mathscr F_{N(i)}^{\alpha(i)}$ of Veech's foliation 
on $\mathscr M_{1,2}$, for a certain integer $N(i)\geq 0$ and a certain 2-uplet $\alpha(i)=(\alpha_1(i),-\alpha_1(i))$ with $\alpha_1(i)\in ]0,1[$.  Moreover, all these objects (namely $N(i)$, $\alpha(i)$ as well as the cover $\Sigma(i)\rightarrow \mathscr F(i)$) can be determined explicitly. 
\sk 

{At the moment,  we do not have as nice and precise descriptions of the 
 $P$-strata of codimension 1 which must be added to $\mathscr F_N$ as the one we have for the $C$-strata}. What seems  likely is  that these $P$-strata, which are (coverings of) moduli spaces of flat spheres with 4  cone  points,  are vertical with respect to $\nu_N$. More precisely, we believe that they are vertical fibers 
 at  some cusps  of a certain extension  of $\nu_N$ over 
 a partial completion of $Y_1(N)$ contained in $X_1(N)$.  
 \sk 
 
Thanks to some classical works by Kodaira and Shioda \cite{Kodaira,Shioda}, it is known that $\nu_N: \mathcal E_1(N)\rightarrow Y_1(N)$ admits a compactification 
 $\overline{\nu}_N: \overline{\mathcal E_1(N)}\rightarrow X_1(N)$ obtained by gluing some generalized elliptic curves as vertical fibers over the cusps of $Y_1(N)$. 
Note that  such compactified modular surfaces (but for the 
level $N$ congruence group $\Gamma(N)$) 
have been used by Livn\'e in his thesis \cite{Livne} (see also \cite[\S16]{DeligneMostowBook}) to construct some non-arithmetic lattices in ${\rm PU}(1,2)$. This fact prompts us to believe that  it might be possible to construct the metric completion of $\mathscr F_N$ from $\overline{\mathcal E_1(N)}$ by means of geometric operations similar to the ones used by Livn\'e. 
 This could provide a nice way to investigate further the topology and the complex analytic  geometry of the $\mathbb C\mathbb H^2$-conifold 
 $ \overline{\mathscr F_N}$. \sk
 
  We hope to return on this in some future works.

\subsection{\bf Veech's foliation for flat tori with two conical singularities}
\label{SS:VeechFoliationg=1n=2}
We now specialize in the special case when $g=1$ and $n=2$.  

In this case, the 2-uplet   
  $\alpha=(\alpha_1,\alpha_2)\in \mathbb R^2$ is necessarily such that 
\begin{equation}
\label{E:alpha n=2}
\alpha_1=-\alpha_2\in ]0,1[.
\end{equation}

Since Veech's foliation does depend only on $[\alpha_1:\alpha_2]$ and in view of  
our hypothesis \eqref{E:alpha n=2}, one obtains that $\mathcal F^\alpha$ does not depend on $\alpha$ (Corollary \ref{C:FAlphaDoesNotDependOnAlpha}).  
 \mk 

\paragraph{}\hspace{-0.3cm}  In  the case under study,   it is relevant   to consider the rescaled first integral 
$$\Xi=\big(\alpha_1\big)^{-1} \xi^\alpha \, : \, 
\mathcal T\!\!\!{\it or}_{1,2} \longrightarrow \mathbb R^2
$$ which is independent of $\alpha$. Indeed,   for $(\tau,z_2)\in \mathcal T\!\!\!{\it or}_{1,2}$, one has
$$
\Xi(\tau,z_2)=\left( \frac{\Im{\rm m} (z_2)}{\Im{\rm m}(\tau)} ,   \frac{\Im{\rm m} (z_2)}{\Im{\rm m}(\tau)}\cdot \tau-z_2
\right).$$

Moreover, it follows immediately from the third point of  Proposition \ref{P:FalphaGeneralities}
that  the image of $\Xi$ is exactly $\mathbb R^2\setminus \mathbb Z^2$.  
We denote by  $\Pi_1$ the restriction to $\mathcal T\!\!\!{\it or}_{1,2}$ of the linear projection $\mathbb H\times \mathbb C\rightarrow \mathbb H$ onto the first factor.
\begin{prop} 
\begin{enumerate}
\item The following map   is a real analytic diffeomorphism 
\begin{equation}
\label{E:Tor12IsomProduct}
\Pi_1\times \Xi : \mathcal T\!\!\!{\it or}_{1,2} \longrightarrow \mathbb H\times \big( \mathbb R^2\setminus \mathbb Z^2\big) .
\end{equation}
\item 
 The push-forward of Veech's foliation $\mathcal F^\alpha$ 
by this map is  the horizontal foliation 
on  the product $\mathbb H\times ( \mathbb R^2\setminus \mathbb Z^2) $. 
\sk 
\item By restriction, the projection $\Pi_1$ induces a biholomorphism between any leaf of $\mathcal F^\alpha$ and Poincar upper half-plane $\mathbb H$.  In particular, the leaves of Veech's foliation on 
$ \mathcal T\!\!\!{\it or}_{1,2}$
are  topologically trivial.
\end{enumerate}
\end{prop}

Since Veech's foliation does not depend on $\alpha$, we will drop the subscript  $\alpha$ in the notations of the rest of this section. We will also identify $\mathcal T\!\!\!{\it or}_{1,2}$ with $
\mathbb H\times ( \mathbb R^2\setminus \mathbb Z^2) $. The 
corresponding action of  
${\rm SL}_2(\mathbb Z)\ltimes \mathbb Z^2$   is given by 
\begin{equation*}
\Big(
M, \big(k,l\big)\Big)
 \boldsymbol{\circ}\Big( \tau,  \big(r_0,r_\infty\big)\Big) =
\Big( M\cdot \tau,  \big(r_0+l , r_\infty-k\big)\cdot 
\rho(M) \Big) 
\end{equation*}
for any $\big(M,(k,l)\big)\in {\rm SL}_2(\mathbb Z)
\ltimes \mathbb Z^2$  and any $\big(\tau,(r_0,r_\infty)\big)\in \mathbb H\times (\mathbb R^2\setminus \mathbb Z^2)$.\footnote{We remind the reader that $\rho(M)$ is the matrix obtained from $M$ by exchanging the antidiagonal coefficients, see \eqref{E:RhoM}.}
\mk 
%
%
%
%


The preceding  proposition shows that, on the Torelli space,  Veech's foliation  is trivial from a topological point of view. It is really  more interesting to look at $\mathscr F$, which is by definition the push-forward  of $\mathcal F$ 
by the natural quotient map
\begin{equation}
\label{E:QuotientMu}
\mu: \mathcal T\!\!\!{\it or}_{1,2} \longrightarrow {\mathscr M}_{1,2}. 
\end{equation}

  For a `rescaled lifted holonomy' $r=(r_0,r_\infty) \in \mathbb R^2\setminus \mathbb Z^2$, one sets
\begin{align*}
 [r]= & \, \big({\rm SL}_2(\mathbb Z)\ltimes \mathbb Z^2\big)\bullet r\subset \mathbb R^2\, , \\
\mathcal F_{\!\!r}=& \, \Xi^{-1}(r)= \Big\{ \big(\tau,z_2\big)\in \mathcal T\!\!\!{\it or}_{1,2}\; \big\lvert \; r_0\tau-
 z_2=r_\infty\Big\}
\subset \mathcal T\!\!\!{\it or}_{1,2}
 \\
\mbox{ and }\quad \mathscr F_{\!\!r}= & \,  \mu\big( \mathcal F_{r}\big)\subset \mathscr M_{1,2}\, .
\end{align*}
 (Note that the correspondence  with the notations of 
  \S\ref{SS:AlgebraicLeavesg=1n>1}  are obtained via $ r\leftrightarrow a$ with $a=\alpha_1 r$, {\it i.e.} $	a_0=\alpha_1 r_0$ and $a_\infty=\alpha_1 r_\infty$). \mk

\paragraph{}\hspace{-0.3cm}
It turns out that when $g=1$ and $n=2$, 
one can give a complete and explicit description of all the leaves of 
 Veech's foliation on ${\mathscr M}_{1,2}$ and in particular of the algebraic ones. \mk

For $r\in \mathbb R^2$,  one denotes by ${\rm Stab}(r)$  its stabilizer for the action $\bullet$: 
$${\rm Stab}(r)= \Big\{ g\in {\rm SL}_2(\mathbb Z)\ltimes \mathbb Z^2 \, \big\lvert \, g\bullet r=r\, \Big\}$$
and one sets: 
$$
\delta(r)=\dim_{\mathbb Q} \big\langle 
r_0\, , \, r_\infty\, ,\, 1
\big\rangle  \in \big\{1 ,2, 3\big\}. \mk
$$

The following facts are easy to establish:  
\begin{itemize}
\item   $\delta(r)$  does depend only on $[r]$, {\it i.e.} $\delta(r)=\delta(r')$ if $r'\in [r]$; in fact
$$
\delta(r)=\dim_{\mathbb Q} \big\langle [r]\big\rangle; 
$$
\item one has $\delta(r)=1$ if and only if $r\in \mathbb Q^2$; 
\sk 
\item one has $\delta(r)=2$ if and only if there exists a triplet  $(u_0,u_\infty,u_1)\in \mathbb Z^3$, unique up to multiplication by $-1$, such that 
\begin{equation}
\label{E:u(r)}
u_0r_0+u_\infty r_\infty=u_1 \qquad \mbox{ and }\qquad 
\gcd(u_0,u_\infty, u_1)=1. 
\end{equation}
\end{itemize}

\begin{prop}${}^{}$ 
\label{P:Stab(r)}
Let $r$ be an element of\;\,$ \mathbb R^2$. 
\begin{enumerate}
\item If \,$\delta(r)=3$   then ${\rm Stab}(r)$ is trivial.
\sk 
\item If \,$\delta(r)=2$ then  ${\rm Stab}(r)$ is isomorphic to $\mathbb Z$. 
\sk 
\item If \,$\delta(r)=1$   then  $r\in \mathbb Q^2$ and ${\rm Stab}(r)$ is isomorphic to the congruence subgroup 
$\Gamma_1(N)$ where $N$ is the smallest integer such that $Nr\in \mathbb Z^2$.
\sk 
\item For any positive integer $N$, the stabilizer  of $(0,1/N)$  in 
${\rm SL}_2(\mathbb Z)\ltimes \mathbb Z^2$ 
is the image of the following injective morphism of groups
\begin{align*}
\Gamma_1(N)  & \longrightarrow {\rm SL}_2(\mathbb Z)\ltimes \mathbb Z^2 \\
\small{\begin{bmatrix}
a & b\\
c & d
\end{bmatrix} }& \longmapsto \left(  \begin{bmatrix}
a & b\\
c & d
\end{bmatrix} ,  \Big(  \frac{d-1}{N} \, , \, \frac{c}{N}   \Big)  \right).
\end{align*}
\end{enumerate}
\end{prop}
\begin{proof} For $r=(r_0,r_\infty)\in \mathbb R^2$ and 
$g=\big(\tiny{\big[\!\!
\begin{tabular}{cc}
$a$ \!\!&\!\! \!\!\!\!$b$\vspace{-0.1cm}\\
$c$ \!\!&\!\! \!\!\!\! $d$
\end{tabular}\!\!
\big]}, (k,l)\big)\in {\rm SL}_2(\mathbb Z)\ltimes \mathbb Z^2$, 
one has 
 \begin{align}
 \label{E:equiv}
g
\bullet r=r
\quad \Longleftrightarrow \quad
 \begin{bmatrix}
a-1 & -c\\
-b& d-1
\end{bmatrix}\cdot 
 \begin{bmatrix}
r_0\\
r_\infty 
\end{bmatrix}= \begin{bmatrix}
-l \\k
\end{bmatrix}. 
\end{align}

We first consider the case when $r\not \in \mathbb Q^2$. 
If $\delta(r)=3$ it is easy to deduce from the preceding equivalence that ${\rm Stab}(r)$ is trivial. 
Assume that  $\delta(r)=2$ and  let 
$u(r)=(u_0,u_\infty ,u_1)\in \mathbb Z^3$  be such that 
\eqref{E:u(r)} holds true.
  We  denote by $u$ the greatest common divisor of $u_0$ and $u_\infty$: $u=\gcd(u_0,u_\infty)\in \mathbb N_{>0}$. \sk 

The equality  $g\bullet r=r$ is equivalent to the fact that $(a-1,-c,-l)$ and $(-b,d-1,k)$ are integer multiples of $u(r)$. From this remark, one deduces easily that in this case, any  $g\in  {\rm Stab}(r)$ is written $g={\boldsymbol{1}}+\frac{\lambda}{u} X(u)$  for some $\lambda\in \mathbb Z$, where $X(u)$ stands for the following element of ${\rm SL}_2(\mathbb Z)\ltimes \mathbb Z^2$: 
\begin{equation}
\label{E:X(u)}
X(u)=\Bigg( \begin{bmatrix}
{u_0u_\infty} & {u_0^2}\\
- {u_\infty^2} &  -{u_0u_\infty}
\end{bmatrix}, \big( {u_0 u_1}, {u_\infty u_1}\big)\Bigg)\, . 
\end{equation}

Then a short (but a bit laborious hence left to the reader) computation shows that, 
if ${\boldsymbol{1}}$ stands for the unity  $( {\rm Id}, (0,0))$ of ${\rm SL}_2(\mathbb Z)\ltimes \mathbb Z^2$, then  the map 
\begin{align*}
\mathbb Z & \longrightarrow {\rm Stab}(r)\\
\lambda & \longmapsto {\boldsymbol{1}}+\frac{\lambda}{u} X(u). 
\end{align*}
is an isomorphism of group.
 \sk 

Finally, we consider the case when $r\in \mathbb Q^2$. Let $N$ be the integer as in the statement of the proposition.  Then $[r]=[0,1/N]$ according to 
Proposition \ref{P:NormalFormForAinQ2},  hence (3) follows from (4). 
 For $r=(0,1/N)$,  \eqref{E:equiv} is equivalent to the fact that the integers $c,d,k$ and $l$ verify $c=lN$ and $d-1=kN$. The fourth point of the proposition follows then easily. 
\end{proof}\mk

\paragraph{}\hspace{-0.3cm}  With the preceding proposition at hand, it is then not  difficult to determine the conformal types of the leaves of Veech's foliation\,\;$\mathscr F$ on\,\;${\mathscr M}_{1,2}$ (as abstract complex orbifolds of dimension 1, not as embedded subsets of ${\mathscr M}_{1,2}$). 
\begin{coro}
\label{C:LeafFr-n=2}
 Let $r$ be an element of    \hspace{0.01cm} ${\rm Im}(\Xi)= \mathbb R^2\setminus 
\mathbb Z^2$. 
\begin{enumerate}
\item If  \,$\delta(r)=3$ then $\mu$ induces  an isomorphism 
$\mathcal F_{\!r}=\mathbb H\simeq \mathscr F_{\!r}$.\sk 
\item If  \,$\delta(r)=2$ then the leaf \,$ \mathscr F_{\!r}$ is isomorphic to the infinite cylinder 
$\mathbb C/\mathbb Z$.\sk 
\item If  \, $\delta(r)=1$ then the leaf \,$ \mathscr F_{\!r}$ is isomorphic 
(as a complex orbicurve) to the modular curve $Y_1(N)=\mathbb H/{\Gamma_1(N)}$ where $N$ is  the smallest positive integer such that $Nr\in \mathbb Z^2$. 
\end{enumerate}
\end{coro}
\begin{proof}
Since any leaf $\mathcal F_r$ is simply connected, one has 
$\mathscr F_{\!r}=\mu(\mathcal F_r)\simeq \mathcal F_r/_{{\rm Stab}(r)}$.
(the latter being a isomorphism of  orbifolds if ${\rm Stab}(r)$ has fixed points on $\mathcal F_r$). \sk 

Let $S(r)$ be the image of ${\rm Stab}(r)$ by the epimorphism ${\rm SL}_2(\mathbb Z)\ltimes \mathbb Z^2\rightarrow {\rm SL}_2(\mathbb Z)$.   The restriction of \eqref{E:MorphismGspace}
 to $\mathcal F_r$ gives an isomorphism of $G$-analytic spaces
\begin{equation}
\label{E:FrF[r]}
    \xymatrix@R=0.4cm@C=1.3cm{  
 {\mathcal F}_r
  \ar@(dl,dr) \ar[r]         &       \mathbb H    \ar@(dl,dr)   \\
 {\rm Stab}(r)  \ar[r]     &  S(r)\, .
}
\end{equation}
This implies   that (as complex orbifolds if $S(r)$ has fixed points on $\mathbb H$) one has 
$$\mathscr F_{\!r}\simeq \mathbb H/S(r).$$ 

If $\delta(r)=3$ then $S(r)$ is trivial by Proposition \ref{P:Stab(r)},  hence (1) follows.\sk  

If $\delta(r)=2$,  it comes from the second point 
of Proposition \ref{P:Stab(r)} that $S(r)$ coincides with the infinite cyclic group spanned by ${\rm Id}+M(u)\in {\rm SL}_2(\mathbb Z)$ where $M(u)$ stands for the matrix component of the 
element $X(u)$ defined in \eqref{E:X(u)}. Since ${\rm Tr}({\rm Id}+M(u))=2+ {\rm Tr}(M(u))=2$, this generator is parabolic,  hence the action of $S(r)$ on the upper half-plane is conjugated with the one spanned by the translation $\tau\mapsto \tau+1$.  It follows that $\mathscr F_{\!r}$ is isomorphic to the infinite cylinder. \sk

The fact that $\delta(r)=1$ means that $r\in \mathbb Q^2\setminus \mathbb Z^2$. In this case, let $N$ be  the integer defined in the third point of the proposition. Then $(0,1/N)\in [r]$ according to Proposition \ref{P:NormalFormForAinQ2},   hence $\mathscr F_{\!r}=\mathscr  F_{(0,1/N)}$.  From  the fourth point of Proposition \ref{P:Stab(r)},  it follows that 
$S(0,1/N)=\Gamma_1(N)$. Consequently, one has 
$\mathscr F_{\!r}\simeq \mathbb H/\Gamma_1(N)=Y_1(N)$. 
\end{proof}\mk 

For any integer $ N$ greater than or equal to 2, one sets 
\begin{equation}
\label{E:LeafFN-g=1-n=2}
\mathscr F_N=\mathscr F_{(0,1/N)} 
\subset {\mathscr M}_{1,2}\, . 
\end{equation}

From the preceding results, we deduce the following very precise description 
of the algebraic leaves of  Veech's foliation on\,\;${\mathscr M}_{1,2}$: 
{\begin{coro}  
\label{C:AlgebraicLeavesg=1n=2}
${}^{}$
\begin{enumerate}
\item 
The  leaves  of\,\;$\mathscr F$ which are closed analytic sub-orbifolds of\,\;${\mathscr M}_{1,2}$ are  exactly  the\,\;$\mathscr F_{N}$'s for $N\in \mathbb N_{\geq 2}$;\sk 
\item 
For any integer $N\geq 2$, the leaf\,\;$\mathscr F_{N}$ is algebraic,  isomorphic to the modular curve $Y_1(N)$ and is the image of the following embedding: 
\begin{align*}
Y_1(N) & \longhookrightarrow {\mathscr M}_{1,2}  \\
\big[ \big(E_\tau,[1/N]\big)\big] & \longmapsto \big[ \big(E_\tau,[0],[1/N]\big)\big]\, .
\end{align*}
\end{enumerate}
\end{coro}}

Note that the modular curve $Y_1(N)$ (hence the leaf \,\;$\mathscr F_{N}$)  has orbifold points only for $N=2,3$ ({\it cf.}\,\cite[Figure 3.3]{DS}). \mk

The (orbi-)leaves $\mathscr F_{N}$ of Veech's foliation on  ${\mathscr M}_{1,2}$ are clearly the most interesting ones. Thus is true at    the topological level already. In the next sections we will go further by taking into account the parameter $\alpha$. Our goal will be to study Veech's hyperbolic structures of the algebraic leaf  $\mathscr F_{N}$ for any $N\geq 2$  
 as far  as possible. 

\subsubsection{\bf A remark about the `leaves' $\mathscr F_{\theta}(M)$ considered in \cite{GP}}
\label{SS:FThetaNNotConnected}
From the description  of the leaves of Veech's foliation given above, it is possible to give an explicit example of the non-connectedness phenomenon mentioned in  \cite{GP}. \sk

\paragraph{} We first recall  some notations of \cite{GP} for ($g=1$ and) $n\geq 1$ arbitrary. 
We consider a fixed $\alpha=(\alpha_1,\ldots,\alpha_n)$ satisfying  \eqref{E:AssumptionAlpha}  and we denote by $\theta=(\theta_1,\ldots,\theta_n)$ the associated angles datum.
For $\rho\in {\rm Hom}(\pi_1(1,n),\mathbb U)$, let   $\mathcal F_{\!\rho}\subset \mathcal E_{1,n}^\alpha \simeq \mathcal T\!\!\!{\it eich}_{1,n}$ be the preimage of $\rho$ by Veech's linear holonomy map $\chi_{1,n}^\alpha$ (see \S\ref{SS:VeechLinearHolonomyMap} above). 

If $[\rho]$ stands for the orbit of $\rho$ under the action of ${\rm PMCG}_{1,n}$, then $\mathscr F_{[\rho]}$ is the notation used in \cite{GP} for the image of $\mathcal F_{\!\rho}$ into $\mathscr M_{1,n}$. 
\sk

We now assume that the image ${\rm Im}(\rho)$ of $\rho$ in $\mathbb U$ is finite. From our results above, it comes that $\mathscr F_{[\rho]}$ is an algebraic subvariety of 
$\mathscr M_{1,n}$.  Moreover,  
$\alpha$ is necessarily rational,  hence $G_\theta=\langle e^{i\theta_1},\ldots, e^{i\theta_n}\rangle$ is a finite subgroup of the group $\mathbb U_\infty$ of roots of unity. 
 If $\omega_\rho$ stands for a  generator of ${\rm Im}(\rho)\subset 
\mathbb U_\infty$, one denotes by $M=M_\theta(\rho)$ the smallest positive integer such that $G_\theta=\langle \omega_\rho^M\rangle$.  Now in \cite[\S10]{GP}, it is proved that the leaf $\mathscr F_{[\rho]}$ is uniquely determined by this integer $M$ (remember that $\theta$ is fixed) and  the corresponding notation for it is $\mathscr F_\theta(M)$.  
\mk 

In \cite{GP}, using certain surgery for flat surfaces, one proves several results 
about   the geometric structure of an arbitrary algebraic leaf $\mathscr F_\theta(M)$.
 However, the geometrical methods used to establish these results, if quite relevant 
to answer some questions,  
  do not allow to answer some basic other questions such as the 
  connectedness   of a given   leaf $\mathscr F_\theta(M)$. \sk  

  Using our results presented in the preceding subsections, it is easy to 
   answer this question
   when  $n=2$. We stick to this case  in what follows. 
\mk 


\paragraph{} 
Assume that  $\alpha_1\in ]0,1[$  is rational and let $N$ be an integer bigger or equal to 2. 
The following statement follows from the very definition of the algebraic leaf $\mathscr F_N^{\alpha_1}\simeq Y_1(N)$ given in the preceding subsection: 

\begin{lemma} The image of the linear holonomy of any flat surface belonging to the  leaf\;\,$\mathscr F_N^{\alpha_1}$ is the finite subgroup 
of\;\,$\mathbb U$ generated by $\exp(2i\pi\alpha_1 / N)$.
\end{lemma}

Assume that $\alpha_1=p/q$ where $p,q$ are two  relatively prime positive integers.   In this case, the corresponding angles datum is 
$$\theta=\big(2\pi(1-\alpha_1), 2\pi(1+\alpha_1)\big)
= \left( 2\pi\Big(\frac{q-p}{q}\Big)
, 2\pi\Big(\frac{q+p}{q}\Big)\right).
$$ 

From the preceding  lemma, one immediately deduces that 
for any $M>0$, the connected components of the leaf\;$\mathscr F_\theta(M)$ considered in \cite{GP} are exactly the leaves $\mathscr F_N^{\alpha_1}$'s  
considered above in \eqref{E:LeafFN-g=1-n=2}, 
 for any integer  $N\geq  2$ such that $$M=\frac{N}{\gcd(N,p)}\, .$$ 
 
One first deduces that $\mathscr F_\theta(M)$ is actually empty  if and only if $p=M=1$. 
For convenience, we set $\mathscr F_0=\mathscr F_0^{\alpha_1}=Y_1(0)=\emptyset$ in the lines below. 

Then from the discussion above, one  immediately deduces the : 
%
%

\begin{coro} 
\label{C:CComp-Ftheta(M)}
Let  $M$ be a fixed positive integer.   
\begin{enumerate}
\item  In full generality, one has : 
\begin{equation*} 
\label{E:CComp-Ftheta(M)}
\mathscr F_\theta(M)=\bigsqcup_{\substack{k \in \mathbb N^*, \; k \lvert  p   \\ \gcd(p/k,M)=1}} \mathscr F_{kM}^{\alpha_1}\, . 
\end{equation*}
\item  In the generic case when   $\gcd(p,M)=1$, one has  $\mathscr F_\theta(M)
=\sqcup_{ k \lvert  p} \mathscr F_{kM}^{\alpha_1}$.
\sk 
\item On the contrary, if $p$ divides $M$ then $\mathscr F_\theta(M)
=\mathscr F_{pM}^{\alpha_1}\simeq Y_1(pM)$. 
\end{enumerate}
\end{coro}
This result gives an explicit description of the connected components of the algebraic leaves $\mathscr F_\theta(M)$ for any positive integer $M$. Note that it depends in a subtle way of the arithmetic properties of the parameters 
$\alpha_1$ and $M$.  We illustrate this fact with two concrete examples below.
\bk

\paragraph{} 
We first deal with the case when $\alpha_1=1/q$ for an integer  $q{\geq 2}$. 
In this case, one has $p=1$, $\theta=(2\pi(q-1)/q, 2\pi(q+1)/q)$,  hence $\mathscr F_\theta(1)=\emptyset$ 
and the third point of the  preceding corollary gives us that for any $M\geq 2$, one has
$$
\mathscr F_\theta(M)=\mathscr F_M^{\alpha_1}\simeq Y_1(M)\,.
$$

\paragraph{} We then consider the case when $\alpha_1=2/q$ where $q\geq 3$ stands for an odd integer. In this case $p=2$ and  $\theta=2\pi((q-2)/q,(q+2)/q)$.  

From Corollary \ref{C:CComp-Ftheta(M)}, one gets that for any $M\geq 1$, one has 
$$
\mathscr F_\theta(M)=\begin{cases}\; 
\mathscr F_{2}^{\alpha_1}\simeq Y_1(2)  \hspace{3.5cm} \mbox{ if }\; M=1;\sk \\
\; \mathscr F_{2M}^{\alpha_1}\simeq Y_1(2M) \hspace{3cm} \mbox{ if }\; M\; \mbox{ is even};\sk \\
\; \mathscr F_{M}^{\alpha_1}\sqcup 
\mathscr F_{2M}^{\alpha_1} \simeq Y_1(M) \sqcup Y_1(2M) 
 \hspace{0.5cm} \mbox{ if }\; M>1\, \mbox{ is odd}\, .
\end{cases}
$$
\sk 


To conclude this subsection, we mention that 
we find the question of determining  g{eometricall}y  the which $\mathscr F_N$'s are the connected components of a given `leaf' $\mathscr F_\theta(M)$ quite interesting. 
As already said above, the answer  depends on the arithmetic of $\theta$ in a rather subtle way. 
 Hence  it could be difficult to discriminate these connected components solely   by means of geometrical methods.

  \subsubsection{\bf An aside:  a connection with Painlev theory}
The leaves of Veech's foliation on the Torelli space  are cut out by the
affine equations \eqref{E:EqLeaf}. The fact that  the latter   have real coefficients reflects the fact that 
Veech's foliation is transversely  real analytic (but not holomorphic) on $\mathcal T\!\!\!{\it or}_{1,n}$. A natural way to get a holomorphic object is by allowing the coefficients $a_0$ and $a_\infty$ to take any complex value (the $\alpha_i$'s staying fixed).  Performing this complexification, 
 one obtains a 2-dimensional complex family of hypersurfaces in $\mathcal T\!\!\!{\it or}_{1,n}$ that are nothing else but the solutions of the second-order linear differential system 
\begin{equation}
\label{E:DiffSystem}
\frac{\partial ^2 \tau}{\partial z_i^2}=0 \quad \mbox{ and }\quad 
\alpha_i\frac{\partial  \tau}{\partial z_j}
-\alpha_j\frac{\partial  \tau}{\partial z_i}
=0\qquad \mbox{for }\, i,j=2,\ldots,n, \, i\neq j.
\end{equation}

Using the explicit formula given above, it is not difficult to verify that \eqref{E:DiffSystem} is invariant by the action of the mapping class group. This implies that it can be pushed-down onto 
${\mathscr M}_{1,n}$ and give rise to a (no more linear) second-order holomorphic differential system on this moduli space, denoted by $\mathscr D^\alpha$.   The integral varieties of $\mathscr D^\alpha$ form a complex 2-dimensional family of   1-codimensional locally analytic subsets of ${\mathscr M}_{1,n}$ which can be seen as a  kind of complexification of Veech's foliation. 
\sk 

If the preceding construction seems natural, one could have some doubt concerning its interest. 
What shows that it is actually relevant is the consideration of the simplest case when $n=2$.  
In this situation, \eqref{E:DiffSystem} reduces to the second order differential equation $d^2\tau/dz_2^2=0$.  In order to avoid considering orbifold points, it is more convenient to look at the push-forward modulo  the action of  $\Gamma(2)\ltimes \mathbb Z^2< {\rm Sp}_{1,n}(\mathbb Z)$.  
It is known that 
$$
{\mathcal T\!\!\!{\it or}_{1,2}}/_{\Gamma(2)\ltimes \mathbb Z^2}\simeq \big(\mathbb P^1\setminus \{0,1,\infty\}\big)\times \mathbb C$$ 
and that the quotient map is given
by 
\begin{align*}
\nu \, : \, \mathcal T\!\!\!{\it or}_{1,2} & \longrightarrow \big(\mathbb P^1\setminus \{0,1,\infty\}\big)\times \mathbb C\\
(\tau,z_2) & \longmapsto \left(\lambda(\tau), \frac{\wp(z_2)-e_1}{e_2-e_1}\right).
\end{align*}
Here $\wp: E_\tau\rightarrow \mathbb P^1$ is the Weirerstrass $\wp$-function associated to the lattice $\mathbb Z_\tau$,  one has 
$e_i=\wp(\omega_i)$ for $i=1,2,3$  where   $\omega_1=1/2$, $\omega_2=\tau/2$ and $\omega_3=\omega_1+\omega_2=(1+\tau)/2$ and where 
$\lambda:\tau\mapsto (e_3-e_1)/(e_2-e_1)$
 stands for the classical elliptic modular lambda function (the usual Hauptmodul\footnote{We recall here  the definition of what is a {\bf Hauptmodul} for a genus 0 congruence group $\Gamma\subset {\rm SL}_2(\mathbb Z)$:  it is a $\Gamma$-modular function on $\mathbb H$ which induces a generator of the field of rational functions on $X_\Gamma=\overline{\mathbb H/\Gamma}\simeq \mathbb P^1$  with a pole of order 1 with  residue $1$ at the cusp $[i\infty]\in X_\Gamma$.} for $\Gamma(2)$).\sk
 
As already known by Picard  \cite[Chap.\,V, \S17]{Picard} (see  \cite{Manin} for a recent proof), the push-forward of $d^2\tau/dz_2^2=0$ by $\nu$ is the following non-linear second order differential equation
\begin{align*}
\hspace{-1.5cm}
{\rm (PPVI)}\qquad \quad 
\frac{d^2X}{d\lambda^2}= & 
\frac{1}{2}\left(
\frac{1}{X}  +\frac{1}{X-1}+\frac{1}{X-\lambda}
\right)\left(\frac{dX}{d\lambda}\right)^2 \\  & -
\left(
\frac{1}{\lambda}+\frac{1}{\lambda-1}+\frac{1}{X-\lambda}
\right)\cdot \frac{dX}{d\lambda}+
\frac{1}{2}\frac{X(X-1)}{\lambda(\lambda-1)(X-\lambda)}.
\end{align*}

This equation, now known as {\bf Painlev-Picard equation},  was 
first considered by Picard in \cite{Picard}. 
The name of Painlev is associated to it since it is a particular case (actually the simplest case) of the sixth-Painlev equation.\footnote{Actually, the sixth Painlev equation has not been discovered by Painlev himself (due to some mistakes in some of its computations) but by his student R. Fuchs in 1905.} 
\sk

For $r=(r_0,r_\infty)\in \mathbb R^2\setminus \mathbb Z^2$,  
the image of the leaf $\mathcal F_r$ in $(\mathbb P^1\setminus \{0,1,\infty\})\times \mathbb C$ 
is parametrized by $
\tau \mapsto ( \lambda(\tau), (\wp(r_0\tau-r_\infty)-e_1)/(e_2-e_1))
$,  hence the leaves of Veech's foliation can be considered as particular solutions of (PPVI).  The leaves $\mathscr F_{r}$ with  $r\in \mathbb Q^2 $ are precisely the algebraic solutions of (PPVI), a fact already known to  Picard (see \cite[p. 300]{Picard}).
\mk 
 
The existence of this  link between the theory of Veech's foliations and the theory of Painlev equations is not so surprising. Indeed, both domains are related to the notion of isomonodromic (or isoholonomic) deformation.   But this leads to interesting questions such as the following ones: 
\begin{enumerate}
\item is there a geometric characterization of the leaves $\mathscr F_{N}$, $N\in \mathbb N_{\geq 2}$ of Veech's foliation among the solutions of (PPVI)?
\sk
\item given $\alpha=(\alpha_1,\alpha_2)$ as in \eqref{E:alpha n=2}, is it possible to obtain  the hyperbolic structures constructed by Veech on the leaves of $\mathscr F^\alpha$ within the framework of (PPVI)? Moreover, does the general solution of (PPVI) carry a hyperbolic structure which specializes to Veech' one on a leaf $\mathscr F^\alpha_{N}$?
\sk
\item for $n\geq 2$ arbitrary, is there a nice formula for the push-forward $\mathscr D^\alpha$ of the differential system \eqref{E:DiffSystem} onto a suitable quotient of $\mathcal T\!\!\!{\it or}_{1,n}$? If yes, does such a  push-forward enjoy a generalization for differential systems in several variables  of the Painlev property?
\end{enumerate}




\subsection{\bf An analytic expression for the Veech map when $\boldsymbol{g=1}$}
\label{S:AnalyticExpressionVeechMapg=1}
In this section, we begin with dealing with the general case when ($g=1$ and) $n\geq 2$. 
Our goal here is to get an explicit local analytic expression for the Veech map. 
\sk

After having recalled the definition of this map, we define another map by adapting/generalizing to our context the approach developed in the genus 0 case by Deligne and Mostow.  We show that,  after some identifications, these two maps are identical. 
Although all this is not really necessary to our purpose (which is to study 
 Veech's hyperbolic structure on a leaf $\mathcal F_{\!a}^\alpha\subset {\mathscr M}_{1,n}$), we believe that it is worth considering, since it shows how  the constructions of the famous papers \cite{Veech} and \cite{DeligneMostow} are related in the genus 1 case. 
\sk

Our aim here is to study Veech's hyperbolic structure on a leaf 
$\mathcal F^\alpha_{\!a} $  
of Veech's foliation   in the Torelli space $\mathcal T\!\!{\it or}_{1,n}$
as  extensively as possible.   We denote by  $\widetilde{\mathcal F}_{\!a}^\alpha$ its  preimage in the Teichmller space $\mathcal T\!\!\!{\it eich}_{1,n}$. Then Veech constructs a holomorphic map 
\begin{equation}
\label{E:VeechMapOnTheLEafInTeichmullerSpace}
\widetilde{V}_r^\alpha : \widetilde{\mathcal F}_{\!a}^\alpha\longrightarrow \mathbb C\mathbb H^{n-1}
\end{equation}
and the hyperbolic structure he considers on $\widetilde{\mathcal F}_{\!a}^\alpha$ is just the one obtained by pull-back by this map (which is tale  according to \cite[Section 10]{Veech}).   To study this 
hyperbolic structure,  we are going to give an explicit analytic expression for  $\widetilde{V}_{\!a}^\alpha $, or more precisely,  for its push-forward by the (restriction to $\widetilde{\mathcal F}_{\!a}^\alpha$ of the) 
projection 
$
\mathcal T\!\!\!{\it eich}_{1,n}
\rightarrow
\mathcal T\!\!\!{\it or}_{1,n}
$ which is a multivalued holomorphic function on $\mathcal F_{\!a}^\alpha$ that  will be denoted by $V_{\!a}^\alpha$.  \mk 

We recall and fix some notations that will be used in the sequel. 

In what follows, $a=(a_0,a_\infty)$ stands for a fixed element of  $\mathbb R^2\setminus \alpha_1 \mathbb Z^2$ if $n=2$ or of $\mathbb R^2$ if $n\geq 3$. 
We denote by $\rho_a$ the linear holonomy map shared by the elements of 
$\mathcal F_{\!a}^\alpha \subset \mathcal T\!\!\!{\it or}_{1,n}$: we recall that  
it is the element of $H^1(E_{\tau,z},\mathbb U)$ defined for any $n$-punctured torus  $E_{\tau,z}$ by the following relations: 
\begin{equation}
\label{E:RhoA}
\rho_{a,0}=e^{2i\pi a_0}\, , \quad \rho_{a,k}=e^{2i\pi \alpha_k}\; \mbox{ for }\, 
k=1,\ldots,n\quad \mbox{and }\quad \rho_{a,\infty}=e^{2i\pi a_\infty}
\end{equation}
(see \S\ref{E:WhereRhoIsDefined} where all the corresponding notations are fixed).\sk 

Note that $\rho_a$ can also be seen as an element of ${\rm Hom}^\alpha(\pi_1(E_{\tau,z}),\mathbb U)$.  Considering the latter  as a  $\mathbb C$-valued morphism, one gets a  1-dimensional  $\pi_1(E_{\tau,z})$-module  that will be denoted by  $\mathbb C_a^\alpha$.  Finally, for any $(\tau,z)\in \mathcal F_{\! a}^\alpha$, the function $T^\alpha_{\tau,z}(\cdot)=T^\alpha(\cdot, \tau,z)$ defined in   \eqref{E:functionT}
 is a multivalued function on $E_{\tau,z}$ whose monodromy is multiplicative and given by \eqref{E:RhoA}.  One denotes by $L^\alpha_{\tau,z}$ the local system on $E_{\tau,z}$ defined as the kernel of 
the connexion  \eqref{E:nabla-omega} on $\mathcal O_{E_{\tau,z}}$.  Note that the representation of $\pi_1(E_{\tau,z})$ associated to $L^\alpha_{\tau,z}$ is precisely $\mathbb C_{a}^\alpha$.

\subsubsection{\bf The original Veech's map}
We first recall Veech's abstract definition of 
\eqref{E:VeechMapOnTheLEafInTeichmullerSpace}.  We refer to 
the ninth and tenth sections 
of  \cite{Veech} for proofs and details. \sk 

\paragraph{}
\hspace{-0.2cm}
Let $(E_{\tau,z},m_{\tau,z}^\alpha,\psi)$ be a point of $\mathcal E_{1,n}^\alpha\simeq \mathcal T\!\!\!{\it eich}_{1,n}$ (see \S\ref{S:GeneralConsiderationsAboutVeech'sFoliation}). We fix $x\in E_{\tau,z}$ as well as a determination $D$ at $x$ of the developing map on $E_{\tau,z}$ associated with the flat structure induced by $m_{\tau,z}^\alpha$.  For any loop $c$ centered at $x$ in $E_{\tau,z}$, one denotes by $ M^c(D)$ the germ at $x$ obtained after the analytic continuation of $D$ along $c$. Then one has 
$
M^c(D)=\rho(c)\, D+\mu_{D}(c)
$
with $\rho(c)\in \mathbb U$ and $\mu_{D}(c)\in \mathbb C$.  

One verifies that the affine map $m^c:z\mapsto \rho(c) z+\mu_{D}(c)$ only depends on the pointed homotopy class of $c$ and one gets this way the {\bf (complete) holonomy representation} of the flat surface $(E_{\tau,z},m_{\tau,z}^\alpha)$: 
\begin{align*}
\pi_1\big(E_{\tau,z},x\big) &  \longrightarrow {\rm Isom}^+(\mathbb C)\simeq \mathbb U\ltimes \mathbb C\\
[c] & \longmapsto m^c: z\mapsto \rho(c)z+\mu_{D}(c)\, .
\end{align*}

Now  assume that $(E_{\tau,z},m_{\tau,z}^\alpha,\psi)$ belongs to $\widetilde{\mathcal F}_{\! a}^\alpha$. Then the map $[c]\mapsto \rho(c)$ is nothing else than the linear holonomy map $\chi_{1,n}^\alpha(E_{\tau,z},m_{\tau,z}^\alpha,\psi)$ ({\it cf.}\;\eqref{E:Chi-g-n-alpha}) and will be denoted by $\rho_{\!a}$ in what follows. 
 Let $\mu_a$ be  the complex-valued map on $\pi_1(E_{\tau,z},x)$ which 
associates  $1-\rho_{\!a}(c)$ to any homology class $[c]$. The translation part of the holonomy $[c]\mapsto \mu_{D}(c)$ 
can be seen as a complex  map  on the fundamental group of the punctured torus $E_{\tau,z}$ at $x$  which satisfies the following properties: 
\begin{enumerate}
\item for any two homotopy classes $[c_1],[c_2]\in \pi_1(E_{\tau,z},x)$, one has: 
\begin{align*}
\mu_{D}(c_1c_2)= & \rho(c_1)\mu_{D}(c_2)+\mu_{D}(c_1)\\
\mbox{and}\quad  \mu_{D}(c_1c_2c_1^{-1})= & 
\rho_{\!a}(c_1)\mu_{D}(c_2)+\mu_D(c_1)\mu_{\!a}(c_2)\, ; 
\end{align*}
\item if $\widetilde D=\kappa D+\ell$ is another germ of the developing map at $x$, with $u\in \mathbb C^*$ and $v\in \mathbb C$, then one has the following relation between $\mu_D$ and $\mu_{\widetilde D}$: 
$$
\mu_{\widetilde D}=\kappa \mu_D+\ell \mu_{\!a}\; .
$$
\end{enumerate}

Now remember that considering $(E_{\tau,z},m_{\tau,z}^\alpha, \psi)$  as a point of $\mathcal T\!\!\!{\it eich}_{1,n}$ means that  the third component $\psi$ stands for an isomorphism $\pi_1(1,n)\simeq \pi_1(E_{\tau,z},x)$ well defined up to inner  automorphisms.   Consider then the composition $\mu_{\tau,z}=\mu_{D}\circ \psi: \pi_1(1,n)\rightarrow \mathbb C$. It is an element of the following space of 1-cocycles for the 
 $\pi_1(1,n)$-module, denoted by $\mathbb C_{a}^\alpha$,   associated to the unitary character  $\rho_a$: 
$$
Z^1\big(\pi_1(1,n),\mathbb C_{\rho_a}\big)= \Big\{
\mu: \pi_1(1,n)\rightarrow \mathbb C\; \big\lvert 
\; \mu(\gamma\gamma ')= \rho_{\!a}(\gamma)\mu(\gamma')+\mu(\gamma) \; \forall \gamma,\gamma'
\Big\}\, . 
$$

One denotes again by $\mu_a$ the map $\gamma \mapsto 1-\rho_a(\gamma)$ on $\pi_1(1,n)$. From the properties (1) and (2) satisfied by $\mu_D$ and from  the fact that $\psi$ is canonically defined up to inner isomorphisms, it follows that the class of $\mu_{\tau,z}$ in the projectivization of  the first group of group cohomology 
$$
H^1\big( \pi_1(1,n), \mathbb C_{a}^\alpha \big)= Z^1\big(\pi_1(1,n),\mathbb C_{a}^\alpha\big)\big/ \mathbb C \mu_a\, , 
$$
is well defined. Then one defines the {\bf Veech map} $ \boldsymbol{\widetilde V_{\! a}^\alpha} $ as the map 
\begin{equation}
\label{E:VeechMapAbstract}
 \xymatrix@R=0.1cm@C=1cm{  
 \widetilde{\mathcal F}_{\! a}^\alpha 
  \ar@{->}[r]^{} &  {\mathbf P} H^1\big( \pi_1(1,n), \mathbb C_{a}^\alpha\big)\\ 
 \hspace{-1cm}\big(E_{\tau,z},m_{\tau,z}^\alpha,\psi\big) 
\ar@{|->}[r]
 & \big[ \mu_{\tau,z}     \big]\, .  \qquad \qquad\quad 
 }
 \end{equation}
 
 In \cite[\S10]{Veech}, Veech proves the following result: 
 \begin{thm}
 The map $ \widetilde V_{\! a}^\alpha $ 
  is a local biholomorphism.
 \end{thm}
 
 Actually, Veech proves more. Under the assumption that 
 at least one of the $\alpha_i$'s is not an integer, there is a projective bundle $\mathbf P\mathscr H^1$ over ${\rm Hom}^\alpha(\pi_1(1,n),\mathbb U)$,  the 
 fiber of which at  $\rho$ is  ${\mathbf P} H^1( \pi_1(1,n), \mathbb C_{\rho})$. Then Veech proves that the $ \widetilde V_{\! a}^\alpha $'s considered above are just the restrictions of a global real-analytic immersion 
 $$
\widetilde{V}^\alpha \; : \; \mathcal T\!\!\!{\it eich}_{1,n} \longrightarrow 
\mathbf P\mathscr H^1\, . 
 $$
 
 An algebraic-geometry inclined reader may see the preceding result as a kind of local Torelli theorem for flat surfaces: {\it once the conical angles 
have been fixed, a flat surfaces is locally determined by its complete holonomy representation}.\sk 
 
 An differential geometer may rather this preceding result as 
 a particular occurrence  of the Ehresmann-Thurston's theorem which asserts essentially the same thing but in the more general context of 
 geometric structures on manifolds (see \cite{GoldmanICM} for a nice general account of this point of view).


\subsubsection{\bf Deligne-Mostow's version} We now  adapt and apply  the constructions and results of the third section of 
 \cite{DeligneMostow}
   to the genus 1 case we are considering here. \mk 

Let $$\pi: \mathcal E_{1,n}\longrightarrow \mathcal T\!\!\!{\it or}_{1,n}$$ be the universal elliptic curve: 
for $(\tau,z)=(\tau,z_2,\ldots,z_n)\in \mathcal T\!\!\!{\it or}_{1,n}$, the fiber of $\pi$ over $(\tau,z)$ is nothing else but the 
 $n$-punctured elliptic curve $E_{\tau,z}$.
\sk

For any  $(\tau,z)\in \mathcal F_{\! a}^\alpha$, 
recall the local system $L^\alpha_{\tau,z}$ on $E_{\tau,z}$ which admits 
 the multivalued holomorphic function 
$T^\alpha_{\tau,z}(u)=T^\alpha(u,\tau,z)$ defined in  \eqref{E:functionT} as   a section.
 All these local systems can be glued together over the leaf $\mathcal F_a^\alpha$: there exists a local system $L_a^\alpha$ over  $\mathcal E_a^\alpha=\pi^{-1}(\mathcal F_{\!a}^\alpha)\subset \mathcal E_{1,n}$ whose restriction along $E_{\tau,z}$ is  $L^\alpha_{\tau,z}$  for any $(\tau,z)\in \mathcal F_{\!a}^\alpha$ (see \S B.2 in Appendix B for a detailed proof). 
 
 Since the restriction  of $\pi$ to $\mathcal E_a^\alpha$ is a topologically locally trivial fibration, the  spaces of twisted cohomology $H^1(E_{\tau,z}, L^\alpha_{\tau,z})$'s organize themselves into a local system $R^1 \pi_* (L_a^\alpha)$ on $\mathcal F^\alpha_{\! a}$.  
  We will be interested in  its projectivization: 
 $$
 B_a^\alpha={\mathbf P}R^1 \pi_*\big( L_a^\alpha\big)\,. 
 $$
 
 It is a flat projective bundle whose fiber 
 $B^\alpha_{\tau,z}$ 
 at any point
    $(\tau,z) $ of $ \mathcal F_{\! a}^\alpha$ is just ${\mathbf P} H^1(E_{\tau,z},L_{\tau,z}^\alpha)$. Its flat structure is of course the one induced by the local system $R^1\pi_* (L_a^\alpha)$. 
%
For $(\tau,z)\in \mathcal F^{\alpha}_a$, let $\omega^\alpha_{\tau,z}$ be the (projectivization of the) twisted cohomology class 
defined by    $T_{\tau,z}^\alpha(u)du$ in cohomology: 
$$
\omega^\alpha_{\tau,z}= \big[ T_{\tau,z}^\alpha(u)du  \big] \in
B^\alpha_{\tau,z}
\, . 
$$
  As in the genus 0 case (see \cite[Lemma {\bf (3.5)}]{DeligneMostow}, 
  it can be proved that 
  the class $
\omega^\alpha_{\tau,z}$ is never trivial hence induces  a global holomorphic  section 
$\omega_a^\alpha$ of $B_a^\alpha$ over the leaf $\mathcal F_a^\alpha$. 
 Since the inclusion  of the latter into 
$
 \mathcal T\!\!\!{\it or}_{1,n}$  induces an injection 
of the corresponding fundamental groups (see Proposition \ref{P:Pi1ofaLeafInjectsInT1n}), any connected component of $\widetilde{\mathcal F}_{\! a }^\alpha$  (quite abusively denoted by the same notation in what follows) is simply connected hence can be seen as the universal cover of the leaf $\mathcal F_{\! a}^\alpha$. It follows that the pull-back $\widetilde{B}_a^\alpha$ of $B_a^\alpha$ by  
$
  \widetilde{\mathcal F}_{\! a }^\alpha \rightarrow {\mathcal F}_{\! a}^\alpha$ 
 can be  
  trivialized:
   the choice of any base-point in $\widetilde{\mathcal F}_{\!a}^\alpha$ over 
   $(\tau_0,z_0)\in {\mathcal F}_{\!\! a}^\alpha$ gives rise to  an isomorphism 
    \begin{equation}
   \label{E:IsomTildeBaalpha}
   \widetilde{B}_a^\alpha \simeq  \widetilde{\mathcal F}_{\!a}^\alpha\times B^\alpha_{\tau_0,z_0}\, . 
   \end{equation}
   
It follows that the section $\omega_a^\alpha$ of $B_a^\alpha$ on $\mathcal F_{\!a}^\alpha$ gives rise to a holomorphic map 
\begin{equation}
\label{E:VDM}
\widetilde{V}_a^{\alpha,DM} : \widetilde{\mathcal F}_{\! a}^\alpha\longrightarrow B^\alpha_{\tau_0,z_0}
={\mathbf P} H^1\big(E_{\tau_0,z_0}, L^\alpha_{\tau_0,z_0}\big)\, .
\end{equation}

We remark now that the results of \cite[p. 23]{DeligneMostow} generalize verbatim to the genus 1 case which we are considering here.  In particular, for any (local) horizontal basis 
$(\boldsymbol{C}_i)_{i=1}^n$
of  the twisted homology with coefficients in $L_a^\alpha$ on $\mathcal F_{\! a}^\alpha$, $(\int_{\boldsymbol{C}_i} \cdot )_{i=1}^n$ forms a local horizontal system of projective coordinates on $B_a^\alpha$.  
Generalizing  the first paragraph of the proof of \cite[Lemma (3.5)]{DeligneMostow} to our case, it comes that 
$$
\big(\boldsymbol{\gamma}_0,\boldsymbol{\gamma}_3,\ldots,\boldsymbol{\gamma}_n,\boldsymbol{\gamma}_\infty\big)$$
 is  such a basis, where the $\boldsymbol{\gamma}_{\!\bullet}$'s are the twisted cycles 
defined in \eqref{E:GammaBullet}. \sk 

As a direct consequence, it follows that the  push-forward  ${V}_a^{\alpha,DM}$ of $\widetilde{V}_a^{\alpha,DM}$ onto the leaf $\mathcal F_{\! a }^\alpha$ in the Torelli space  $ \mathcal T\!\!\!{\it or}_{1,n}$ 
(which is a multivalued holomorphic function) 
 admits the following local analytic expression whose components are expressed in  terms of elliptic hypergeometric  integrals: 
%
\begin{align}
 V_r^{\alpha,DM}: \mathcal F_{\! a}^\alpha  & \longrightarrow \mathbb P^{n-1} 
\nonumber \\
\label{E:VeechMapDMExplicit}
(\tau,z) & \longmapsto \left[
\int_{\boldsymbol{\gamma}_\bullet} T^\alpha(u ,\tau,z) du \right]_{\bullet=0,3,\ldots,n,\infty}.  
\end{align}

\subsubsection{\bf Comparison of $\widetilde{V}_a^\alpha$ and $\widetilde{V}_a^{\alpha,DM}$}
\label{SS:ComparisonVeechDM}
We intend here  to prove that Veech's and Deligne-Mostow's maps coincide, up to some natural identifications.\sk

\paragraph{}
\hspace{-0.2cm}
 At first sight, the two abstractly defined maps \eqref{E:VeechMapAbstract} and 
\eqref{E:VDM} do not seem to have the same target space.  It turns out that 
they actually do, but up to some natural isomorphisms. 

Indeed, for any  $(\tau,z)\in {\mathcal F}_{\!a}^\alpha$, 
since $L^{\alpha}_{\tau,z}$ is `the' local system on $E_{\tau,z}$ associated to the 
$\pi_1(E_{\tau,z})$-module $\mathbb C_a^\alpha$,  there is a natural morphism 
\begin{equation}
\label{E:pupuc}
H^1\big(\pi_1(E_{\tau,z}), \mathbb C_a^\alpha\big)\longrightarrow 
H^1\big(E_{\tau,z}, L_{\tau,z}^{\alpha}\big)\, . 
\end{equation}
Since $E_{\tau,z}$ is uniformized by the unit disk $\mathbb D$ which is contractible, it  follows that the preceding map is an isomorphism (see \cite[\S2.1]{Goldman} for instance). 
\sk 

On the other hand,  for any  $(E_{\tau,z},m^\alpha_{\tau,z},\psi) \in \widetilde{\mathcal F}_{\!a}^\alpha$,  the (class of) map(s) $\psi: \pi_1(1,n)\simeq   \pi_1(E_{\tau,z})$ induces a well defined isomorphism 
$[\psi^*]: H^1(\pi_1(E_{\tau,z}) , \mathbb C_a^\alpha)\simeq H^1(\pi_1(1,n), \mathbb C_a^\alpha)$. Then, up to the isomorphism \eqref{E:pupuc}, one can see 
the lift
of $ (\tau,z)\mapsto  [\mu_{\tau,z}]\circ  [\psi^*]$  as a global section of 
$\widetilde{B}_{\!a}^\alpha$ over  $\widetilde{\mathcal F}_{\!a}^\alpha$.   Then using \eqref{E:IsomTildeBaalpha}, one eventually obtains that Veech's map $\widetilde{V}_a^\alpha$ can be seen as a map with the same source and target spaces than Deligne-Mostow's map $\widetilde{V}_a^{\alpha,DM}$.\mk 

\paragraph{}
\hspace{-0.2cm} Comparing the two maps \eqref{E:VeechMapAbstract} and \eqref{E:VDM} is not difficult and relies on some arguments elaborated by Veech.  
In \cite{Veech},  to prove that \eqref{E:VeechMapAbstract} is indeed a holomorphic immersion, 
 he explains how to get a local analytic expression for this map. It is then easy to relate this expression to \eqref{E:VeechMapDMExplicit} and  eventually get  the 
\begin{prop}
\label{P:VDM=VVeech}
The two maps $\widetilde{V}_a^\alpha$ and $\widetilde{V}_a^{\alpha,DM}$  coincide.  In particular, 
\eqref{E:VeechMapDMExplicit} is also a local analytic expression for the push-forward $V_a^\alpha$ of Veech's map on $\mathcal F_{\! a}^\alpha$. 
\end{prop}
\begin{proof} We first review some material from the tenth and eleventh sections of 
\cite{Veech} to which the reader may refer for some details and proofs. 
\sk 

Remember that for $(\tau,z)$ in $\mathcal F_{\! a}^\alpha$, one sees $E_{\tau,z}$ 
 as a flat torus with $n$ conical singularities, the flat structure being the one induced by 
  the singular metric 
 $
 m_a^\alpha(\tau,z)=
 \lvert T^\alpha(u,\tau,z)du\lvert^2$.  
 Given such an element $(\tau',z')$,  there exists a geodesic polygonation $\mathcal T=\mathcal T_{\tau',z'}$ of $E_{\tau',z'}$ whose set of  vertices is exactly the set of conical singularities of this flat surface (for instance, one can consider  its Delaunay decomposition\footnote{The {\it `Delaunay decomposition}' of a compact flat surface is a canonical polygonation of it   by Euclidean polygons inscribed in circles (see \cite[\S4]{MasurSmillie} or \cite{Bowditch} for some details).}). Moreover, the set of points $(\tau,z)\in   \mathcal F_{\! a}^\alpha$ such that the associated flat surface admits a geodesic triangulation $\mathcal T_{\tau,z}$ combinatorially equivalent to $\mathcal T$ is open (according to \cite[\S5]{Veech}),  hence contains an open ball  $U_\mathcal T\subset \mathcal F_{\! a}^\alpha$ to which  the considered base-point $(\tau',z')$ belongs.   \sk 
 
As explained in \cite[\S10]{Veech}, by removing the interior of some edges (the same edges 
 for every point $(\tau,z)$ in $U_{\mathcal T}$), one obtains a piecewise geodesic graph $\Gamma_{\!\tau,z}\subset E_{\tau,z}$ formed by $n+1$ edges of $\mathcal T_{\tau,z}$   such that $Q_{\tau,z}=E_{\tau}\setminus \Gamma_{\!\tau,z}$ is homeomorphic to the open disk $\mathbb D\subset \mathbb C$.  
  Then one considers   the length  metric on  $Q_{\tau,z}$ associated to the restriction of the flat structure of $E_{\tau,z}$.  The metric completion $\overline{Q}_{\tau,z}$ for this  intrinsic metric is isomorphic to the closed disk $\overline{\mathbb D}$. Moreover, the latter 
carries a flat structure with (geodesic) boundary, whose singularities are  $2n+2$ conical points $v_1,\ldots,v_{2n+2}$ located on the boundary circle $\partial \mathbb D$.  One can and will assume that the $v_i$'s are cyclically  enumerated in the trigonometric order, $v_1$   being chosen arbitrarily. For $i=1,\ldots,2n+2$, let $I_i$ be the circular arc on $\partial \mathbb D$ whose endpoints are $v_i$ and $v_{i+1}$ (with $v_{2n+3}=v_1$ by convention). \sk 

The developing map $D_{\tau,z}$ of the flat structure on $Q_{\tau,z}\simeq \mathbb D$ extends continuously to 
$\overline{Q}_{\tau,z}\simeq \overline{\mathbb D}$. For every $i$, this  extension maps $I_i$ onto the segment $[\zeta_i,\zeta_{i+1}]$ in the Euclidean plane $\mathbb E^2\simeq \mathbb C$, where we have set 
for $i=1,\ldots,2n+1$: 
$$\qquad
\qquad
 \zeta_i=\zeta_i(\tau,z)=D_{\tau,z}(v_i)\, . 
$$ 

There exists an involution $\theta$ without fixed point on the set $\{1,\ldots,2n+2\}$ such that the flat torus $E_{\tau,z}$ is obtained from the flat closed disk $\overline{\mathbb D}\simeq 
\overline{Q}_{\tau,z}
$ by gluing isometrically the `flat arcs' $I_i\simeq [\zeta_i,\zeta_{i+1}]$ and 
$I_{\theta i}\simeq [\zeta_{\theta i},\zeta_{\theta (i+1)}]$.   Let $J$ be a subset of $\{1,\ldots,2n+2\}$  such that $J\cap \theta J=\emptyset$. Then $J$ has cardinality $n+1$ and if one sets 
$$\xi_j=\xi_j(\tau,z)= \zeta_{j+1}-\zeta_j$$ for every $j\in J$, then these  complex numbers satisfy a linear relation which  depends only on $\mathcal T, \theta$ and on the linear holonomy $\rho_a$ ({\it cf.}\;\cite[\S11]{Veech}).
 
 Consequently the $\xi_j$'s are the components of  a map 
\begin{equation}
\label{E:VeechLocal}
U_\mathcal T\rightarrow \mathbb P^{n-1} : (\tau,z)\mapsto \big[\xi_j(\tau,z)
\big]_{j\in J}
\end{equation}
which  it is nothing else but a local holomorphic expression of $V_{\! a}^\alpha$ on $U_\mathcal T$ (see \cite[\S10]{Veech}). \sk

It it then easy to relate \eqref{E:VeechLocal} to \eqref{E:VeechMapDMExplicit}. Indeed $T^\alpha_{\tau,z}$ admits a global determination on  the complement of $\Gamma_{\!\tau,z}$ since the latter is simply connected. The crucial but easy point is  that the developing map $D_{\tau,z}$ considered above is a primitive of the global holomorphic 1-form $T^\alpha_{\tau,z}(u)du$ on $Q_{\tau,z}$.  Once one is aware of this, it comes at once  that 
for every  $j\in J$,  $\xi_j(\tau,z)$ can be written as  $\int_{e_j} T^\alpha(u,\tau,z)du$
where $e_j$ stands for  the edge of $\mathcal T_{\tau,z}$ in $E_{\tau}$ which corresponds to the `flat circular arc' $I_j$.   In other terms: $\xi_j(\tau,z)$ is equal to the integral along $e_j$ of a determination of the multivalued 1-form $T^\alpha(u,\tau,z)du$. 
 
Then for every $j\in J$, 
setting $\boldsymbol{e}_j
={\rm reg}(e_j)\in H_1(E_{\tau,z},L_{\tau,z}^\alpha)$ where ${\rm reg}$ is the regularization map considered in \S\ref{S:}, one obtains: 
$$
\xi_j(\tau,z)=\int_{\boldsymbol{e}_j} T^\alpha(u,\tau,z)d\,u
.$$ 

It is not difficult to see that $(\boldsymbol{e}_j)_{j\in J}$ is a basis of $H_1(E_{\tau,z},L^\alpha_{\tau,z})$ for every $(\tau,z)\in U_\mathcal T$. Even better, it follows from \cite[{\it Remark} ${\boldsymbol{(3.6)}}$]{DeligneMostow} that $(\int_{\boldsymbol{e}_j}\cdot )_{j\in J}$ constitutes a horizontal system of projective coordinates on $B_a^\alpha$ over $U_\mathcal T$.  Since two such systems of projective coordinates are related by a constant projective transformation when both are horizontal,  it follows that \eqref{E:VeechLocal} coincides with \eqref{E:VeechMapDMExplicit} up to a constant projective transformation.  The proposition is proved.
\end{proof}\mk

By definition, Veech's hyperbolic structure on $\mathcal F_{\! a}^\alpha$ is obtained by pull-back by $V_a^\alpha$ of the natural one on the target space $\mathbb C\mathbb H^{n-1}\subset \mathbb P^{n-1}$. 
What makes   the elliptic-hypergeometric definition of 
$V_a^\alpha$  la Deligne-Mostow   interesting is that it  allows to make everything explicit. 
 Indeed, in addition to the local explicit expression \eqref{E:VeechMapDMExplicit} obtained above, the use of twisted-(co)homology also  allows  to give an explicit expression for Veech's hyperbolic hermitian form on the target space. \sk 

\subsubsection{\bf An explicit expression for the Veech form}
\label{SS:VeechFormEXplicited}
Since we are going to  focus only on the $n=2$ case in the sequel, 
we only consider this case 
in the next result.   However, the proof given hereafter generalizes in a straightforward way to the general case when $n\geq 2$. 
 \begin{prop}
\label{P:VeechFormExplicit}
The hermitian form of signature $(1,1)$ on the target space 
 of \eqref{E:VeechMapDMExplicit} which corresponds to Veech's  form is  the one 
given by 
$$
Z\longmapsto \overline{Z} \cdot I\!\!H_{\rho_a} \cdot {}^t\! Z
$$
for $Z=(z_\infty,z_0)\in \mathbb C^{2}$,  where $I\!\!H_{\rho_a}$ stands for the matrix  defined in \eqref{E:Hrho}.
\end{prop}

Note that the arguments of \cite[Chap. IV, \S7]{Yoshida} apply to our situation. Consequently, the hermitian form  associated to $ I\!\!H_{\rho_a}$
 is invariant by the hyperbolic holonomy of the corresponding leaf 
 $\mathscr F_a^{\alpha}$ of Veech's foliation. 
  In the classical hypergeometric case ({\it i.e.}\;when $g=0$),  this is sufficient to characterize the Veech form and get the corresponding result. However, in the genus 1 case, since some leaves  of Veech's foliation $\mathscr F^\alpha$ on the moduli space ${\mathscr M}_{1,2}$ (such as the generic ones, see Corollary \ref{C:LeafFr-n=2}) are simply connected, there is no  holonomy  whatsoever to consider hence such a proof is not possible.\sk 

 The proof of Proposition \ref{P:VeechFormExplicit} which  we give below  is a direct generalization  of the one of Proposition 1.11 in \cite{Looijenga} to our case. Remark that  although elementary, this proof is long and computational. 
It would be interesting to give a more conceptual proof of this result.
 \mk 

\begin{proof} 
We continue to use the notations introduced in the proof of Proposition \ref{P:VDM=VVeech}. 
Let $\nu$ be  the hermitian form on the target space $\mathbb P^{n-1}$ of \eqref{E:VeechMapDMExplicit} which corresponds to the one considered by Veech in \cite{Veech}.  
\mk 

For $(\tau,z)\in \mathcal F_{\!a}^\alpha$, 
 the wedge-product 
$$
\eta_{\tau,z}=\omega_{\tau,z}^\alpha\wedge 
\overline{\omega^\alpha_{\tau,z}}= \big\lvert T^\alpha_{\tau,z}(u)\big\lvert^2 du\wedge d\overline{u}
$$ does not depend on the determination of $T^\alpha_{\tau,z}( u)$ and extends to an integrable  positive  2-form on $E_{\tau,z}$. Moreover,  according to \cite[\S12]{Veech}, one has (up to multiplication by $-1$):
$$
\nu\big(V_r^\alpha(\tau,z)\big)= -\frac{i}{2}\int_{E_\tau} \eta_{\tau,z}
<0.
$$

The complementary set ${Q}_{\tau,z}=E_{\tau,z}\setminus {\gamma}_{\tau,z}$ of the union of the supports of the three 1-cycles $\gamma_0,\gamma_2$ and $\gamma_\infty$  in $E_\tau$  is homeomorphic to a disk.  Its  boundary in the metric completion $\overline{Q}_{\tau,z}$ 
(defined as in the proof of Proposition \ref{P:VDM=VVeech})  is 
$$\partial \overline{Q}_{\tau,z}=\overline{\gamma}_0+\overline{\gamma}_\infty'-\overline{\gamma}_0'-\overline{\gamma}_\infty+\overline{\gamma}_2'-\overline{\gamma}_2$$ where the six elements in this sum are the boundary segments defined in  the figure below.\sk


\begin{center}
\begin{figure}[!h]
\psfrag{v1}[][][1]{$v_1 $}
\psfrag{v2}[][][1]{$v_2 $}
\psfrag{v3}[][][1]{$v_3 $}
\psfrag{v4}[][][1]{$v_4 $}
\psfrag{v5}[][][1]{$v_5 $}
\psfrag{v6}[][][1]{$v_6 $}
\psfrag{g0}[][][1]{${\overline{\gamma}_0} $}
\psfrag{gi}[][][1]{${\overline{\gamma}_\infty} $}
\psfrag{g0p}[][][1]{${\overline{\gamma}_0'} $}
\psfrag{g2}[][][1]{${\overline{\gamma}_2} $}
\psfrag{gip}[][][1]{${\overline{\gamma}_\infty'} $}
\psfrag{g2p}[][][1]{${\overline{\gamma}_2'} $}
\includegraphics[scale=0.6]{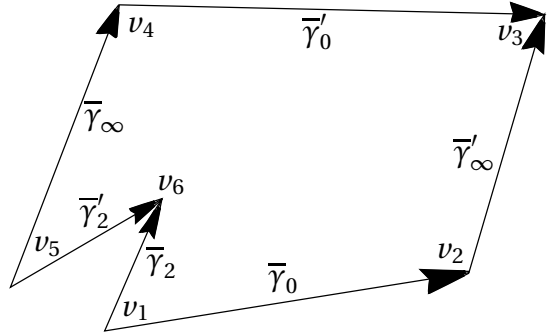}
\caption{The closed disk $\overline{Q}_{\tau,z}$ and its boundary}
\end{figure}
\end{center}

Let $\Phi=\Phi_{\tau,z}$ be a primitive of $\omega =\omega^\alpha_{\tau,z}$ on $\overline{Q}_{\tau,z}$. 
 For any symbol $\bullet\in \{0, 2,\infty\}$ we will denote by $\omega_\bullet$ and $\omega_\bullet'$ the restriction  of $\omega$ on $\overline{\gamma}_\bullet$ and $\overline{\gamma}_\bullet'$ respectively. We will use similar notations  for $\overline{\omega}$ and for $\Phi$ as well, being aware that  $\Phi_\bullet'
$ has nothing to do with a derivative but refers to the restriction of $\Phi$ on $\overline{\gamma}_\bullet'$. \sk
 
 Since $d(\Phi\cdot \overline{\omega})=\eta_{\tau,z}$, it follows from Stokes theorem  that 
  \begin{align}
  \label{E:terms}
- {2}{i}\nu\big(V_r^\alpha(\tau,z)\big)   =\int_{
\partial \overline{\gamma}_{\tau,z}^c} \Phi\cdot \overline{\omega}
= & \, 
 \int_{\gamma_0}\Phi_{0}\cdot \overline{\omega}_{0}- 
  \int_{\gamma_0'}\Phi_0' \cdot \overline{\omega}_{0}'\\
\, + &    \int_{\gamma_\infty'}\Phi_{\infty}'\cdot \overline{\omega}_{\infty}'- 
  \int_{\gamma_\infty}\Phi_{\infty}\cdot \overline{\omega}_{\infty} \nonumber\\
   + &    \int_{\gamma_2'}\Phi_{2}'\cdot \overline{\omega}_{2}'- 
\int_{\gamma_2}\Phi_{2}\cdot \overline{\omega}_{2}\, . \nonumber
\end{align}

For any $\bullet \in \{0,2,\infty\}$, both segments $\overline{\gamma}_\bullet$ and $\overline{\gamma}_\bullet'$ identify to $\gamma_\bullet\subset E_{\tau,z}$,  hence there is a natural identification between them. For  $\zeta\in \overline{\gamma}_\bullet$, we will denote by $\zeta'$ the corresponding point 
on $ \overline{\gamma}_\bullet'$.  Up to these correspondences, one has 
$$
\omega_0'=\rho_\infty\omega_0, \qquad
 \omega_2'=\rho_2 \omega_2\qquad  \mbox{ and } \qquad 
 \omega_\infty'=\rho_0\omega_\infty.
$$

  It follows that for 
for every $\zeta'$ in $ \overline{\gamma}_0'$, in $ \overline{\gamma}_2'$ and  in 
$ \overline{\gamma}_\infty'$, 
 one has respectively 
\begin{align}
\label{E:fifi}
 \Phi'_0(\zeta')=  & \, F^0 +\rho_0 F^\infty-\rho_\infty F^0+ 
 \int_{v_4}^{\zeta'} \omega_{0}' 
  =(1-\rho_\infty ) F^0 +\rho_0 F^\infty+ 
\rho_\infty  \int_{v_1}^{\zeta} \omega_{0}\, , 
\nonumber \\
\Phi'_2(\zeta')=   & \,F^2-\rho_2F^2+\int_{v_5}^{\zeta'} \omega_2'=
(1-\rho_2)F^2+\rho_2\int_{v_1}^{\zeta} \omega_2
\\
\qquad \quad \mbox{ and}\quad \Phi'_\infty(\zeta')= & \,
 F^0+ 
\int_{v_2}^{\zeta'} \omega_{\infty}'= F^0+ 
\rho_0 \int_{v_5}^{\zeta} \omega_{\infty}
\nonumber
 \end{align}
   with 
\begin{align}
\label{E:F0F2Finfinity}
F^0=& \, \Phi(v_2)- \Phi(v_1)=\int_{v_1}^{v_2} \omega
= \int_{\boldsymbol{\gamma}_0} 
T^\alpha(u,\tau,z)du\, ;  \nonumber
 \\
F^2=& \, \Phi(v_6)- \Phi(v_1)=\int_{v_1}^{v_6} \omega= \int_{\boldsymbol{\gamma}_2} 
T^\alpha(u,\tau,z)du
\\
\mbox{ and}
\quad F^\infty= & \, \Phi(v_4)-\Phi(v_5)=\int_{v_4}^{v_5} \omega= \int_{\boldsymbol{\gamma}_\infty} 
T^\alpha(u,\tau,z)du  \, .\qquad \nonumber
\end{align}

  Since $\overline{\rho_0}=\rho_0^{-1}$,  it follows from \eqref{E:fifi}
 that for every $\zeta\in \overline{\gamma}_0$, one has
 \begin{align*}
\big(\Phi_0\cdot \overline{\omega}_0\big)(\zeta)-
\big(\Phi_0'\cdot\overline{\omega}_0'\big)(\zeta')=& \, 
 \Phi(\zeta) \cdot \overline{\omega}_0
-  \Big[ (1-\rho_\infty ) F_0 +\rho_0 F_\infty+ 
\rho_\infty  \Phi({\zeta})
\Big]\cdot \rho_\infty^{-1} \overline{ \omega}_0
\\
=  & \, \left[  \frac{d_\infty}{\rho_\infty} F_0 -\frac{\rho_0}{\rho_\infty} F_\infty \right]\cdot  \overline{ \omega}_0(\zeta)\, .
\end{align*}

Similarly,   since $\overline{\rho_2}=\rho_2^{-1}$,  it follows from \eqref{E:fifi}
 that for every $\zeta\in \overline{\gamma}_2$, one has
 \begin{align*}
\big(\Phi_2\cdot \overline{\omega}_2\big)(\zeta)-
\big(\Phi_2'\cdot\overline{\omega}_2'\big)(\zeta')=& \, 
 \Phi(\zeta) \cdot \overline{\omega}_2
-  \Big[ (1-\rho_2 ) F_2 + 
\rho_2 \Phi({\zeta})
\Big]\cdot \rho_2^{-1} \overline{ \omega}_2
\\ = & \,  \left[  \frac{d_2}{\rho_2} F_2  \right]\cdot  \overline{ \omega}_2(\zeta)\, .
\end{align*}
Finally,   since $\overline{\rho_\infty}=\rho_\infty^{-1}$,  it follows from \eqref{E:fifi}
 that for every $\zeta\in \overline{\gamma}_\infty$, one has
 \begin{align*}
 \big(\Phi_\infty\cdot \overline{\omega}_\infty\big)(\zeta)-
\big(\Phi_\infty'\cdot\overline{\omega}_\infty '\big)(\zeta')=& \, 
 \Phi(\zeta) \cdot \overline{\omega}_\infty
-  \left[ F_0+\rho_0\int_{v_5}^\zeta \omega_\infty
\right]\cdot \rho_0^{-1} \overline{ \omega}_\infty
\\
=& \, 
\left[ (1-\rho_2)F_2+\int_{v_5}^\zeta \omega_\infty
\right]  \cdot \overline{\omega}_\infty
-  \left[ F_0+\rho_0\int_{v_5}^\zeta \omega_\infty
\right]\cdot \rho_0^{-1} \overline{ \omega}_\infty
\\
=  & \, \left[  (1-\rho_2) F_2 -\frac{1}{\rho_0} F_0 \right]\cdot  \overline{ \omega}_\infty(\zeta). 
\end{align*}

From the three relations above, one deduces that
\begin{equation}
\int_{\overline{\gamma}_\bullet } \Phi \cdot \overline{\omega}-
\int_{\overline{\gamma}_\bullet'} \Phi \cdot \overline{\omega}
=
\begin{cases}
  {d_\infty}{\rho_\infty^{-1}} F_0{\overline{F}_0}   -{\rho_0}{\rho_\infty^{-1}} F_\infty
  {\overline{F}_0}  
   \hspace{0.4cm}\mbox{for }\, \bullet=0; \mk \\
    {d_2}{\rho_2^{-1}} F_2\overline{F}_2   \hspace{3cm}    \mbox{for }\, \bullet=2; \mk \\ 
     
      -d_2 F_2 \overline{F}_\infty-{\rho_0^{-1}} F_0\overline{F}_\infty
              \hspace{1.2cm}           \mbox{for }\, \bullet=\infty. 
\end{cases}
\end{equation}

Injecting these computation in \eqref{E:terms} and using the relations
$$ d_2 F_2={d_\infty} F_0-{d_0}F_\infty
\qquad 
\mbox{and}
\qquad 
\frac{1}{\rho_2} \overline{F}_2=\frac{d_\infty}{\rho_\infty d_2} \overline{F}_0-\frac{d_0}{\rho_0d_2}\overline{F}_\infty
$$
one gets
 \begin{align*}
  \label{E:terms}
{2}{i}\nu\big(V_r^\alpha(\tau,z)\big)   =  & \, \frac{d_\infty}{\rho_\infty} F_0{\overline{F}_0}   -\frac{\rho_0}{\rho_\infty} F_\infty
  {\overline{F}_0}    +d_2 F_2 \overline{F}_\infty+\frac{1}{\rho_0} F_0\overline{F}_\infty
   -   \frac{d_2}{\rho_2} F_2\overline{F}_2 \\
    =  & \, \frac{d_\infty}{\rho_\infty} F_0{\overline{F}_0}   -\frac{\rho_0}{\rho_\infty} F_\infty
  {\overline{F}_0}    +\Big({d_\infty} F_0-{d_0}F_\infty
  \Big)
   \overline{F}_\infty  +\frac{1}{\rho_0} F_0\overline{F}_\infty  \\  &\,
   -   
\Big( d_\infty F_0-d_0F_\infty
   \Big)\Big(  \frac{d_\infty}{\rho_\infty d_2} \overline{F}_0-\frac{d_0}{\rho_0d_2}\overline{F}_\infty \Big) 
   \\
  = &
 2i \,   {}^t \overline{F} \cdot 
H\cdot   F\, , 
      \end{align*}
where $F$ and $H$ stand respectively for the matrices 
\begin{align*} 
F=\begin{bmatrix} 
   F_\infty  
  \\
F_0  
\end{bmatrix}\qquad 
\mbox{ and }
\qquad 
H= \, \frac{1}{2i}\begin{bmatrix}
  -d_0\Big(  1+
    \frac{d_0}{\rho_0d_2}
    \Big)
&  
  d_\infty  +\frac{1}{\rho_0}
   +\frac{d_\infty d_0}{\rho_0d_2}
 \\ 
   -\frac{\rho_0}{\rho_\infty}+\frac{d_0d_\infty}{\rho_\infty d_2} 
    &    \frac{d_\infty}{\rho_\infty}  
\Big(    1  -\frac{d_\infty}{ d_2}\Big) 
\end{bmatrix}\, .\end{align*}

Because $\rho_2=\rho_1^{-1}$, one verifies that 
\begin{align*}
H=
& \, \frac{1}{2i}\begin{bmatrix}
 \frac{d_0}{d_1}\Big(  1-
    \frac{\rho_1}{\rho_0}
    \Big) &  
\frac{1-\rho_0^{-1}-\rho_\infty+\rho_1\rho_\infty\rho_0^{-1}}{d_1}
  \\ 
  \frac{\rho_1-\rho_0\rho_1-\rho_1\rho_\infty^{-1}+\rho_0\rho_\infty^{-1}}{d_1}
    &   
    \frac{d_\infty d_{1\infty}}{\rho_\infty d_1}
\end{bmatrix}= I\!\!\!H_\rho\, .
\end{align*}

Finally,  it follows from \eqref{E:F0F2Finfinity} that $F_\infty$ and $F_0$ are nothing else but the components of the map \eqref{E:VeechMapDMExplicit} and the proposition follows. \end{proof}

\subsubsection{\bf A normalized version of Veech's map (when $\boldsymbol{n=2}$ and 
$\boldsymbol{\rho_0=1}$)}
\label{S:NormalizationOfVeech'sMap-n=2}${}^{}$

According to \S\ref{SS:normalization},  
when $n=2$ and $\rho_0=1$, 
setting 
$
X=\begin{bmatrix}
- \frac{d_{1\infty}}{d_1}   &   1 
  \\ 
  {\rho_\infty}
   &   0  \end{bmatrix}\, 
$, 
one has
$$
\overline{X}\cdot I\!\!H_\rho \cdot {}^t\! X=
\begin{bmatrix}  0  & \frac{i}{2}\\
-\frac{i}{2} & 0\end{bmatrix}\, .$$ 

Consequently, setting  $\boldsymbol{F}=(F_\infty,F_0)$ and 
$$
\boldsymbol{G}=\big( G_\infty, G_0  \big)= \big( F_\infty, F_0  \big)\cdot X^{-1}=
 \Big( F_0  , \frac{1}{\rho_\infty} F_\infty +\frac{d_{1\infty}}{\rho_\infty d_1}F_0
 \Big)\, ,  $$

one obtains that   
$$\Im{\rm m}\big( \overline{G_\infty} G_0\, \big)
=   \overline{\boldsymbol{G}}
\cdot 
\begin{bmatrix}  0  &   \frac{i}{2}\\
- \frac{i}{2} & 0\end{bmatrix}\cdot {}^t\! \boldsymbol{G}  
=  \overline{\boldsymbol{F}}\cdot
I\!\!H_\rho 
 \cdot{}^t\! \boldsymbol{F}  =\nu <0,$$ 
 which implies that the imaginary part of the ratio $G_0/G_\infty$ is negative. \sk 
 
 It follows that the map
 $$
- \frac{G_0}{G_\infty}= -   \bigg(   \frac{1}{\rho_\infty}  \frac{F_\infty}{F_0}+
 \frac{d_{1\infty}}{\rho_\infty d_1}
 \bigg)
 $$
 is a normalized version of Veech's map, with values into the  upper half-plane.


\section{\bf 
Flat tori with two conical points} 
From now on, we focus  on the first nontrivial case, namely $g=1$ and $n=2$.  We fix 
$\alpha=(\alpha_1,\alpha_2)$ with $
\alpha_1=-\alpha_2\in ]0,1[$.    Since it is fixed, we will often omit the subscript $\alpha$ or $\alpha_1$ in the notations. 
 We want to study 
the hyperbolic structure on the leaves of Veech's foliation at the level of the moduli space 
${\mathscr M}_{1,\alpha}\simeq {\mathscr M}_{1,2}$.  
We will only  consider the most interesting leaves, namely the algebraic ones.

\subsection{\bf Some  notations
}
In what follows,  $N$ stands for an integer bigger than 1.

\subsubsection{}\hspace{-0.4cm}  For any  $(a_0,a_\infty)
 \in \mathbb R^2\setminus \alpha_1\mathbb Z^2$, we set 
$$
r=\big(r_0,r_\infty\big)=\frac{1}{\alpha_1} \big(a_0,a_\infty\big )\in \mathbb R^2\setminus \mathbb Z^2
$$
and 
we  denote by $\mathcal F_r^{\alpha_1}$ (or just by $\mathcal F_r$ for short)  the leaf $(\xi^\alpha)^{-1}(a_0,a_\infty
)=\Xi^{-1}(r)$ of Veech's foliation in the Torelli space.  This  is the subset of  $ \mathcal T\!\!\!{\it or}_{1,2}$ cut out by any one of the following two (equivalent) equations: 
$$
a_0\tau+\alpha_2 z_2=a_\infty
\qquad \mbox{ or }
\qquad 
z_2=
r_0\tau-r_\infty.
$$

We remind the reader that  $\mathscr F_{[r]}^{\alpha_1}$ (or just 
$\mathscr F_{\!r}$  for short)  stands for the corresponding leaf in the moduli space of elliptic curves with two marked points: 
$$\mathscr F_{r}=\mathscr F_{[r]}=p_{1,2}\big(\mathcal F_r\big)
\subset {\mathscr M}_{1,2}.$$

\subsubsection{}\hspace{-0.4cm} From a geometric point of view,   the most interesting leaves  clearly 
are  the leaves $\mathcal F_r$ with $r\in \mathbb Q^2\setminus \mathbb Z^2$. Indeed, these are exactly the ones which are algebraic subvarieties (we should say `suborbifolds') of ${\mathscr M}_{1,2}$   and  there is such a leaf for each integer $N\geq 2$ (see Corollary \ref{C:AlgebraicLeavesg=1n=2}), which is 
$$
\mathscr F_N=\mathscr F^{\alpha_1}_N=\mathscr F_{(0,-1/N)}\, .
$$

An equation of the corresponding leaf $\mathcal F_N=\mathcal F_{(0,-1/N)}$ in $\mathcal T\!\!\!{\it or}_{1,2}$ is
\begin{equation}
\label{E:Z21N}
z_2=\frac{1}{N}\, . 
\end{equation}
It  induces a natural  identification $\mathbb H\simeq \mathcal F_{N}$ which is compatible with the action of $\Gamma_1(N)\simeq {\rm Stab}(\mathcal F_{(0,-1/N)})$ (see \eqref{E:FrF[r]}) hence induces an 
 identification 
\begin{equation}
\label{E:Y1(N)F01}
Y_1(N) \simeq  \mathscr  F_N .
\end{equation}

For computational reasons, it will be useful later to consider   other identifications between $ \mathscr F_N$ and the modular curve $ Y_1(N)$ (see \S \ref{S:} just below). However \eqref{E:Y1(N)F01} will be the privileged one. For 
this reason, we will use the (somewhat abusive) notations  
$$ Y_1(N)=\mathscr F_N  \qquad \big(\mbox{and }\;  Y_1(N)^{\alpha_1}=\mathscr F_N^{\alpha_1}\big)
$$ (the second one to emphasize the fact that $Y_1(N)$ is endowed with the $\mathbb C\mathbb H^1$-structure corresponding to Veech's one on $\mathscr F_N^{\alpha_1}$)  to distinguish  \eqref{E:Y1(N)F01} from the other identifications between $Y_1(N)$ and $\mathscr F_N$ that we will consider below. \sk 

\subsubsection{}\hspace{-0.4cm}  
For any $c \in \mathbb P^1(\mathbb Q)=\mathbb Q\cup \{ i \infty \}\subset \partial \mathbb H$, one denotes by $[c]$ the associated cusp of $Y_1(N)=\mathscr F_N$. Then the set of cusps 
$$ C_1(N)=\big\{ \; [c] \, \big\lvert \, c\in \mathbb P^1(\mathbb Q)\, \big\}$$ is finite and $$X_1(N)=Y_1(N)\sqcup  C_1(N)$$
is a compact smooth algebraic curve (see \cite[Chapter I]{DS} for instance). 
\sk 

Our goal in this section is to study the hyperbolic structure, denoted by  $\boldsymbol{{\bf hyp}_{1,N}^{\alpha_1}}$,  of the leaf $Y_1(N)^{\alpha_1}=\mathscr F_N^{\alpha_1}$ of Veech's foliation $\mathscr  F^{\alpha_1}$ on ${\mathscr M}_{1,2}$ in the vicinity of any one of its cusps. More precisely, we want  to prove  that ${\rm hyp}_{1,N}^{\alpha_1}$ extends as a conifold $\mathbb C\mathbb H^1$-structure at such a cusp $\mathfrak c$ and give a closed formula for the associated conical angle 
which will be denoted by  
\footnotetext{By convention, a {\bf complex hyperbolic conifold point of conical angle $\boldsymbol{0}$} is nothing else but 
an  usual cusp for a hyperbolic surface (see A.1.1.\;in Appendix A for more details).} 
$$\theta_N(\mathfrak c)\in \big[0,+\infty\big[\footnotemark\, .$$

\subsubsection{}\hspace{-0.4cm} With this aim in mind, it will be more convenient to deal with the ramified cover $Y(N)$ over $Y_1(N)$ associated to the principal congruence subgroup $\Gamma(N)$. 
   In order to do so, we consider the subgroup 
$$
G(N)=\Gamma(N)\rtimes \mathbb Z^2\lhd  {\rm SL}_2(\mathbb Z)\rtimes \mathbb Z^2
$$
and we denote the associated quotient map by 
\begin{equation}
\label{E:M12N}
p_{1,2}^N: \mathcal T\!\!\!{\it or}_{1,2}\rightarrow {\mathcal T\!\!\!{\it or}_{1,2}}{/ G(N)}=:{\mathscr M}_{1,2}(N)\,.
\end{equation}

Then from \eqref{E:Z21N}, one deduces an identification between the `level $N$ modular curve $Y(N)=\mathbb H/\Gamma(N)$' and 
the corresponding leaf in ${\mathscr M}_{1,2}(N)$: 
\begin{equation}
\label{E:Y(N)F01}
Y(N) \simeq  p_{1,2}^N(\mathcal F_N) =: F_N.
\end{equation}

As above, we will consider this identification as the privileged one and for this reason,  it will be indicated by means of the equality symbol. 
In other terms, we have fixed identifications 
\begin{align*}
Y_1(N)^{\alpha_1}=& \, \mathscr  F^{ \alpha_1}_N  \quad \mbox{  (or } \, Y_1(N)=\mathscr F_N\; \mbox{  for short);} \\
\mbox{ and }\; Y(N)^{\alpha_1}= & \, F^{ \alpha_1}_N \quad \mbox{  (or } \, Y(N)=  F_N\; \mbox{  for short).}
 \end{align*}

One denotes by $C(N)$ the set of cusps of $Y(N)$. Then 
$$X(N)=Y(N)\sqcup C(N)$$
is a compact smooth algebraic curve.  The complex hyperbolic structure on $Y(N)$ corresponding to Veech's one  (up to the identification \eqref{E:Y(N)F01})
will be denoted by 
$\boldsymbol{{\bf hyp}_{N}^{\alpha_1}}$. Under the assumption that it extends as a conifold structure at $\mathfrak c \in  C(N)$, we will denote by 
$ \vartheta_N(\mathfrak c)$ the associated conical angle. 

\subsubsection{}\hspace{-0.4cm} 
\label{SS:Y(N)-and-Y1(N)}
The natural quotient map $Y(N)\rightarrow Y_1(N)$ (coming from the fact that $\Gamma(N)$ is a subgroup of $\Gamma_1(N)$) induces an algebraic cover  
\begin{equation}
\label{E:X(N)->X1(N)}
X(N)\longrightarrow X_1(N)
\end{equation}
 which is ramified at the cusps of $X_1(N)$. More precisely, at a cusp $\mathfrak  c\in C(N)$, a local analytic model for this cover is $z\mapsto z^{N/w_{\mathfrak c}}$ where $w_{\mathfrak c}$ stands for the {\bf width}
of $\mathfrak c$, the latter being now seen as a cusp of $X_1(N)$.\footnote{We recall that, $w_{\mathfrak c}$ divides $N$ for any $\mathfrak c\in C(N)$ hence the map $z\mapsto z^{N/w_{\mathfrak c}}$ is holomorphic.} \sk 

It follows that, for any $\mathfrak c\in C(N)$,  ${\rm hyp}_{1,N}^{\alpha_1}$  extends as a $\mathbb C\mathbb H^1$-conifold structure at $\mathfrak c$ now considered as a cusp of $Y_1(N)^{\alpha_1}$ if and only if the same holds true,   at $\mathfrak c$, for the corresponding complex hyperbolic structure ${\rm hyp}_{N}^{\alpha_1}$ on $X(N)^{\alpha_1}$. 
 In this case, one has the following relation between the corresponding cone angles: 
\begin{equation}
\label{E:RelationAnglesX(N)-X1(N)}
\theta_N(\mathfrak c)=\frac{w_{\mathfrak c}}{N}\vartheta_N(\mathfrak c)\,.
\end{equation}

In order to get explicit results, it is necessary to have a closed explicit formula for the width of a cusp. 
\begin{lemma}
\label{L:WidthCuspX1(N)}
Assume that $\mathfrak c=[-a'/c']\in  C_1(N)$ with  $a',c'\in \mathbb Z$  are coprime. Then
$$
w_{\mathfrak c}=\frac{N}{{\rm gcd}(c',N)}.
%
$$
\end{lemma}
\begin{proof} 
The set of cusps of $X(N)$ can be identified with the set of classes $\pm{\tiny{[ \!\! \!\!  \begin{tabular}{c}
$a$\vspace{-0.15cm}\\ $c$
\end{tabular}}  \!\!  \!\!  ]}$ 
of the points ${\tiny{[\!\! \!\!  \begin{tabular}{c}
$a$\vspace{-0.15cm}\\ $c$
\end{tabular}}  \!\!  \!\!  ]}\in (\mathbb Z/N\mathbb Z)^2$ of order $N$. 
To $\mathfrak c=[-a'/c']$ is associated $\pm{\tiny{[ \!\! \!\!  \begin{tabular}{c}
$a$\vspace{-0.15cm}\\ $c$
\end{tabular}}  \!\!  \!\!  ]}$ where $a$ and $c$ stand for the residu modulo $N$ of $a'$ and $c'$ respectively  ({\it cf.}\;\cite[\S3.8]{DS}).  It follows that ${\rm gcd}(c',N)={\rm gcd}(c,N)$. 

On the other hand, according to \cite[\S1]{Ogg}, the ramification degree of the covering $X(N)\rightarrow X_1(N)$  at  $\pm{\tiny{[ \!\! \!\!  \begin{tabular}{c}
$a$\vspace{-0.15cm}\\ $c$
\end{tabular}}  \!\!  \!\!  ]}$  viewed as a cusp of $X_1(N)$ is ${\rm gcd}(c,N)$. 
Since the width of any cusp of $X(N)$ is $N$, it follows that $w_{\mathfrak c}={N}/{{\rm gcd}(c',N)}$. \end{proof}


\subsection{\bf Auxiliary leaves}
\label{S:AuxiliaryLeaves}
\sk 

For any  $(m,n)\in \mathbb Z^2\setminus N\mathbb Z^2$ with ${\rm gcd}(m,n,N)=1$, one sets
$$\mathcal F_{m,n}
 =\mathcal F_{({m}/{N},-{n}/{N})}\, .
 $$ 
 
This is leaf of Veech's foliation on  $\mathcal T\!\!\!{\it or}_{1,2}$ cut out by  
$$
z_2=
\tau {m}/{N}+{n}/{N}\, .
$$

The latter equation induces a natural  identification 
\begin{align}
\label{E:HcalFmn}
\mathbb H & \stackrel{\sim}{\longrightarrow}\;  \mathcal F_{m,n}\subset \mathcal T\!\!\!{\it or}_{1,2} \\
\tau & \longmapsto \left( \tau, 
\frac{m}{N}\tau+\frac{n}{N}
\right)\nonumber
\end{align}
 which is compatible with the action of $\Gamma(N)\lhd \Gamma_1(N)\simeq  {\rm Stab}(\mathcal F_{m,n})$ ({\it cf.}\;\eqref{E:FrF[r]}),  hence induces a well-defined identification 
\begin{equation}
\label{E:Y(N)Fmn}
 Y(N)\simeq p_{1,2}^N(\mathcal F_{m,n})=:F_{m,n} 
\subset {\mathscr M}_{1,2}(N). 
\end{equation}

For any $s\in \mathbb P^1(\mathbb Q)=\mathbb Q\cup \{ i \infty \}\subset \partial \mathbb H$, one denotes by $[s]$ the associated cusp of $Y(N)$ and by $[s]_{m,n}$ the  corresponding cusp for 
$F_{m,n}$ relatively to the  identification \eqref{E:Y(N)Fmn}.  Then if one denotes by 
$$ C_{m,n}=\Big\{ [s]_{m,n}\, \lvert \, s\in \mathbb P^1(\mathbb Q)\, \Big\}$$
 the set of cusps of  $F_{m,n}\simeq Y(N)$, one gets a compactification
$$
X(N) \simeq X_{m,n}:=F_{m,n}\sqcup C_{m,n} 
$$
where the identification with $X(N)$ is the natural extension of \eqref{E:Y(N)Fmn}.  One will denote by $\boldsymbol{{\bf hyp}_{m,n}^{\alpha_1}}$ the complex hyperbolic structure on $X(N)$ corresponding to Veech's one of $F_{m,n}$ up to the preceding identification.
\sk

Since $(0,-1/N)$ is a representative  for  the orbit of $(m/N,-n/N)$ under the action of ${\rm SL}_2(\mathbb Z)\rtimes \mathbb Z^2$ ({\it cf.}\;Proposition \ref{P:NormalFormForAinQ2}), all the leaves $ F_{m,n}$ are isomorphic to  
$$ F_{0,1}=F_N=Y(N) $$
(where the first equality comes from the very definition of $F_{0,1}$ whereas the second one  refers to the privileged identification 
\eqref{E:Y(N)F01}). 
\sk

What makes considering the whole bunch of leaves $ F_{m,n}$ interesting for us
 is that 
the natural identifications \eqref{E:Y(N)Fmn} do depend on $(m,n)$ (even if $F_{m,n}$ coincides with 
$F_{0,1}$ as a subset of ${\mathscr M}_{1,2}^N$,  as it can happen). 
 Thus,  we will see that,  for any 
cusp $[s]=[s]_{0,1}$ of $Y(N)=F_{0,1}$, there is  a leaf $F_{m,n}$ such that 
$$[s]_{0,1}=\big[i\infty\big]_{m,n}.$$
 
  Since the  hyperbolic structures of $F_{0,1}$ and of $F_{m,n}$ coincide, this implies that 
\begin{quote}
{\it the study of the hyperbolic structure ${\rm hyp}_{N}^{\alpha_1}={\rm hyp}_{0,1}^{\alpha_1}$ of $Y(N)=F_{0,1}$ in the vicinity of its cusps amounts to   the study of the hyperbolic structures  
${\rm hyp}_{m,n}^{\alpha_1}$
of  the leaves  $F_{m,n}$,  only in the vicinity of the cusp $[i\infty]_{m,n}$.}
\end{quote}\mk

We want to make  the above considerations as explicit as possible. Let
$\mathfrak c$  be a cusp of $F_N$ distinct from $[i\infty]$. There exist $a',c'\in \mathbb Z$ with $c'\neq 0$ and ${\rm gcd}(a',c')=1$ such that 
$$
{\mathfrak c}=\big[-a'/c'\big]=\big[-a'/c'\big]_{0,1}.
$$

Since $a'$ and $c'$ are coprime, there exist $d'$ and $b'$ in $\mathbb Z$ such that $a'd'-b'c'=1$. 
Then  one considers the following element of ${\rm SL}_2(\mathbb Z)\rtimes \mathbb Z^2$: 
\begin{equation}
M_{a',c'}=\left(
\begin{bmatrix}
d' &      b'\\
c' &    a'
\end{bmatrix}\, , \, 
\Big(
-\lfloor a'/N\rfloor, -\lfloor c'/N\rfloor
\Big)
\right).
\end{equation}
It induces an automorphism of the intermediary moduli space ${\mathscr M}_{1,2}(N)$ (defined above in \eqref{E:M12N}), denoted somewhat abusively 
by  the same notation $M_{a',c'}$. This automorphism leaves the corresponding intermediary Veech's foliation invariant and is compatible with  the   hyperbolic structure on the leaves. \sk

Setting $a=a'-\lfloor a'/N\rfloor$ and $c=c'-\lfloor c'/N\rfloor$
, one verifies that 
$$
M_{a',c'} \bullet \left( 0, -\frac{1}{N}\right)=\left(\frac{c}{N},-\frac{a}{N}
\right)\, , 
$$
where $\bullet$ stands for the action \eqref{E:ActionOnHolonomy}.  Thus $M_{a',c'}$ induces an isomorphism $Y(N)=F_{0,1}\stackrel{\sim}{\rightarrow} F_{c,a}$ 
which extends to an isomorphism between the compactifications  $X(N)=X_{0,1}\stackrel{\sim}{\rightarrow}X_{c,a}$,  such that 
$$M_{a',c'}\Big( \big[-a'/c'\big]\Big)=\big[i\infty\big]_{c,a}\, .$$

Moreover, $M_{a',c'}$  induces an isomorphism between the $\mathbb C\mathbb H^1$-structures 
  ${\rm hyp}_{0,1}^{\alpha_1}$ and ${\rm hyp}_{c,a}^{\alpha_1}$.  In particular, one deduces the following result:
  
  \begin{prop} 
  \label{P:FN-vs-Fmn}  
  Let $a'$ and $c'$ be two coprime integers  and denote respectively by $a$ and $c$ 
  their residues modulo $N$ in $\{0,\ldots,N-1\}$. 
  \begin{enumerate}
  \item 
There is an isomorphism of pointed 
  curves carrying a $\mathbb C\mathbb H^1$-structure
  $$
  \Big( 
  Y(N), \big[-a'/c'\big]
  \Big)\simeq   \Big( 
  F_{c,a}, \big[i\infty\big]_{c,a}
  \Big)\,.
  $$
  \item  The two following assertions are equivalent: 
  \begin{itemize}\sk 
\item   
      ${\rm hyp}^{\alpha_1}_N$ extends as a 
    conifold $\mathbb C\mathbb H^1$-structure to   $X(N)$; 
    \sk 
        \item  for every $a,c\in \{0,\ldots,N-1\}$ with ${\rm gcd}(a,c,N)=1$,  ${\rm hyp}^{\alpha_1}_{c,a}$ extends as a conifold 
      $\mathbb C\mathbb H^1$-structure in the vicinity of the cusp $[i\infty]_{c,a}$ of $F_{c,a}$. 
      \end{itemize}\sk 
\item  When the two equivalent assertions of (2) are satisfied, the conifold angle $\vartheta(-a'/c')$ 
      of ${\rm hyp}^\alpha_N$ 
      at the cusp $[-a'/c']$ of $Y(N)$ is equal to the conifold angle $\vartheta_{c,a}(i\infty)$ of\;\,${\rm hyp}^{\alpha_1}_{c,a}$ at the cusp $[i\infty]$ of $F_{c,a}$.
      \end{enumerate}
  \end{prop}

\subsection{\bf Mano's differential system for  algebraic leaves}
\label{SS:ManoDifferentialSystemAlgebraicLeaves}
We will now focus on the auxiliary algebraic  leaves  of Veech's foliation  considered in Section \ref{S:AuxiliaryLeaves}.   The arguments and results used below are taken from \cite{ManoWatanabe,Mano} (see also Appendix B). 
 
\subsubsection{}\hspace{-0.4cm} 
 We fix $m,n\in \{0,\ldots,N-1\}$ such that $(m,n)\neq 0$.  For any  $\tau\in \mathbb H$, one  sets: 
 $$
t=t_\tau= \frac{m}{N}\tau+\frac{n}{N}\, .
$$

Hence,  correspondingly, one has 
$$a_0= \frac{m}{N}\, \alpha_1\; \,,  \quad \qquad a_\infty=- \frac{n}{N}\, \alpha_1$$ 
and 
$$T(u)=T^\alpha_{}(u,\tau)=
 e^{2i\pi\frac{m}{N}\alpha_1  u }
 \left(\frac{\theta(u)}{\theta\big(u-t_\tau\big)}\right)^{\alpha_1}\, .
$$

 In order to to make the connection with the results in \cite{Mano}, we recall the following notations introduced there: 
 $$
 \theta_{m,n}(u)= \theta_{m,n}(u,  \tau)=
 e^{-i\pi \frac{m^2}{N^2} \tau-2i\pi\frac{m}{N}\big( u+\frac{n}{N} \big)}\theta\left(u+\frac{m}{N}\tau+\frac{n}{N}, \tau
\right).
 $$

 Then setting 
 $$
 T_{m,n}(u)
 = \left(    \frac{\theta(u)}{\theta_{m,n}(u)}
 \right)^{\alpha_1}\, , 
 $$ 
 one verifies that, when the   determinations of $T(u)$ and of $T_{m,n}(u)$ are fixed, then up to the change of variable $u\rightarrow -u$, these two  functions coincide up   to multiplication by a non-vanishing complex function of $\tau$.   
 This can be written a little abusively
\begin{equation}
\label{E:T(u)Tmn(u)}
 T(u)=\lambda(\tau) T_{m,n}(-u)
 \end{equation}
 where $\lambda$ stands for the aforementioned  holomorphic function which
 depends only on $\tau$ (and on the integers $m,n$ and $N$) but not  on $u$.
 
 \subsubsection{}\hspace{-0.4cm} Since it has values in a projective space,  the Veech map stays unchanged if all its components
 are multiplied by the same non-vanishing function of $\tau$.  
From \eqref{E:T(u)Tmn(u)} and in view of  the local expression 
 \eqref{E:VeechMapDMExplicit} for the Veech map in terms of
 elliptic hypergeometric integrals, 
  it follows that the holomorphic map \begin{align*}
V=V_{m,n}\, : \; \mathcal F_{m,n}^{\alpha_1} \simeq \mathbb H  
 & \longrightarrow \mathbb P^1
 \\
\tau & \longmapsto 
\big[ 
V_0(\tau)\, :\,  
V_\infty(\tau)
\big] 
\end{align*}
whose two components are given by (for $\tau\in \mathbb H$) 
\begin{align*}
\qquad 
V_0(\tau)= \int_{\boldsymbol{\gamma}_0} T_{m,n}( u)du 
\qquad \mbox{ and }
\qquad  
V_\infty(\tau)= \int_{\boldsymbol{\gamma}_\infty} T_{m,n}(u)du\, ,  \qquad
\end{align*}
is nothing else but another expression for the Veech map on $\mathcal F_{m,n}\simeq \mathbb H$.  
\mk 

We introduce two other holomorphic functions of $\tau\in \mathbb H$  defined by 
\begin{align*}
\qquad 
W_0(\tau)= \int_{\boldsymbol{\gamma}_0} T_{m,n}( u)\rho'(u)du 
\qquad \mbox{ and }
\qquad  
W_\infty(\tau)= \int_{\boldsymbol{\gamma}_\infty} T_{m,n}(u)\rho'(u)du\,
\end{align*}
(we recall that $\rho$ denotes the logarithmic derivative of $\theta$ w.r.t.\,$u$, {\it cf.}\;\S\ref{SS:NotationsTheta}).\sk

The two maps  $ \tau\mapsto \boldsymbol{\gamma}_\bullet$  for $\bullet=0,\infty$ 
form a basis of the space of local sections   of the  local system over $\mathcal F_{m,n}$  whose fibers are the  twisted 
 homology 
groups $H_1(E_{\tau,t},  L_{\tau,t})$'s  (see B.3 in Appendix B). 
Then it follows  from \cite{ManoWatanabe,Mano} (see also B.3.5 below) that the functions $V_0,V_\infty, W_0$ and $W_\infty$ satisfy the following differential system  
\begin{align*}
\frac{d}{d\tau}
\begin{bmatrix}
V_0 &   V_\infty \\
 W_0 & W_\infty 
\end{bmatrix}=
M_{m,n} 
\begin{bmatrix}
V_0 & V_\infty  \\
W_0 & W_\infty 
\end{bmatrix}
\end{align*}
on $\mathbb H\simeq \mathcal F_{m,n}$, with 
\begin{align}
\label{E:MatrixABCDmn}
M_{m,n}=
\begin{bmatrix}
A_{m,n} &  B_{m,n}\\
C_{m,n} & D_{m,n}
\end{bmatrix}=
\begin{bmatrix}
\alpha_1 \left( \frac{\stackrel{\bullet}{\theta}_{m,n}}{\theta_{m,n}}
-\frac{\stackrel{\bullet}{\theta}{}'}{\theta'}
\right)
& \frac{\alpha_1-1}{2i\pi} \\
2i\pi\alpha_1 
\left(
\frac{\stackrel{\bullet\bullet}{\theta}_{m,n}}{\theta_{m,n}}
- \Big(\frac{\stackrel{\bullet}{\theta}_{m,n}}{\theta_{m,n}}\Big)^2
-
 \frac{\stackrel{\bullet\bullet}{\theta}{}'}{\theta'}
+ \Big(\frac{\stackrel{\bullet}{\theta}{}'}{\theta'}\Big)^2
\right)
&-\alpha_1 \left( \frac{\stackrel{\bullet}{\theta}_{m,n}}{\theta_{m,n}}
-\frac{\stackrel{\bullet}{\theta}{}'}{\theta'}
\right)
\end{bmatrix}\, .
\end{align}

The trace of this matrix vanishes and the upper-left coefficient $B_{m,n}$ is constant.
Consequently, it follows from a classical result of the theory of linear differential equations 
({\it cf.}\;Lemma 6.1.1 of \cite[\S3.6.1]{Gauss2Painlev} for instance or Lemma A.2.2 in Appendix A) that $V_0$ and $V_{\infty}$ form a basis 
of the space of solutions of 
the  associated second order linear differential equation

$$
{V}^{\bullet\bullet}+ \Big( \det\big(M_{m,n}\big)-A_{m,n}^{\bullet}\Big) \, V=0
$$
where  the superscript ${}^{\bullet}$ indicates the derivative with respect to the variable $\tau$. 
\sk 

Since 
\begin{align*}
A_{m,n}^{\bullet}= & \, \alpha_1\left[
 \frac{\stackrel{\bullet\bullet}{\theta}_{m,n}}{\theta_{m,n}}
 - \left(\frac{\stackrel{\bullet}{\theta}_{m,n}}{\theta_{m,n}}\right)^2
-\frac{\stackrel{\bullet\bullet}{\theta}{}'}{\theta'}
+\left(\frac{\stackrel{\bullet}{\theta}{}'}{\theta'}\right)^2
\right]
\end{align*} and 
\begin{align*}
\det\big(M_{m,n}\big)= & \, 
-(\alpha_1)^2\left[ \frac{\stackrel{\bullet}{\theta}_{m,n}}{\theta_{m,n}}
-\frac{\stackrel{\bullet}{\theta}{}'}{\theta'}
\right]^2-\alpha_1(\alpha_1-1)\left[
 \frac{\stackrel{\bullet\bullet}{\theta}_{m,n}}{\theta_{m,n}}
 - \left(\frac{\stackrel{\bullet}{\theta}_{m,n}}{\theta_{m,n}}\right)^2
-\frac{\stackrel{\bullet\bullet}{\theta}{}'}{\theta'}
+\left(\frac{\stackrel{\bullet}{\theta}{}'}{\theta'}\right)^2
\right]\, , 
\end{align*}
this differential equation can be written more explicitly 
\begin{equation}
\label{E:Vtautau}
{V}^{\bullet\bullet}-
(\alpha_1)^2\left[
\left( \frac{\stackrel{\bullet}{\theta}_{m,n}}{\theta_{m,n}}
-\frac{\stackrel{\bullet}{\theta}{}'}{\theta'}
\right)^2+
\frac{\stackrel{\bullet\bullet}{\theta}_{m,n}}{\theta_{m,n}}
- \Bigg(\frac{\stackrel{\bullet}{\theta}_{m,n}}{\theta_{m,n}}\Bigg)^2
-
 \frac{\stackrel{\bullet\bullet}{\theta}{}'}{\theta'}
+ \Bigg(\frac{\stackrel{\bullet}{\theta}{}'}{\theta'}\Bigg)^2
\right]\cdot  V=0\, . 
\end{equation}

\subsubsection{}\hspace{-0.4cm} By a direct computation (left to the reader), one  verifies that the matrix \eqref{E:MatrixABCDmn}
 and consequently the coefficients of the preceding second order differential equation are invariant by 
 the translation $\tau\mapsto \tau+N$.   It follows that the restriction of \eqref{E:Vtautau} to 
 the vertical band of width $N$
 $$H_N=\Big\{ \,  \tau \in \mathbb H\; \big \lvert   \; 0\leq {\rm Re}(\tau)< N \;  \Big\}$$  can be pushed-forward onto a differential equation of the same type on a punctured open  neighborhood $U^*$ of the cusp $[i\infty]$ in  $Y(N)$.  
 
 Let $x$ be  the local holomorphic coordinate on $Y(N)$ centered at $[i\infty]$ and  related to the variable $\tau$ through the formula 
 $$x=\exp\big(2i\pi \tau/N\big).$$
 
      Then  $v(x)=V(\tau(x))$ satisfies a second order differential equation 
 \begin{equation}
 \label{E:v(x)}
 v''(x)+P_{m,n}(x)\cdot v'(x)+Q_{m,n}(x)\cdot v(x)=0
 \end{equation}
 whose coefficients $P_{m,n}$ and $Q_{m,n}$ are holomorphic on $(\mathbb C^*,0)$. 
\mk 

In \cite{Mano}, Mano establishes the following limits when $\tau\in H_N$ tends to $i\infty$: 
\begin{align*}
& \frac{\stackrel{\bullet}{\theta}{}'}{\theta'}  \longrightarrow \frac{i\pi}{4}\; 
&&   \frac{\stackrel{\bullet\bullet}{\theta}{}'}{\theta'}-\left(
 \frac{\stackrel{\bullet}{\theta}{}'}{\theta'}
\right) \longrightarrow 0 \;  \\
 &\frac{\stackrel{\bullet}{\theta}_{m,n}}{\theta_{m,n}} 
 \longrightarrow  
 i\pi\left(  \frac{m}{N}-\frac{1}{2}   \right)^2\; 
  && \frac{\stackrel{\bullet\bullet}{\theta}_{m,n}}{\theta_{m,n}}
- \Bigg(\frac{\stackrel{\bullet}{\theta}_{m,n}}{\theta_{m,n}}\Bigg)^2
\longrightarrow  0\, . 
\end{align*}

It follows that the coefficient of $V$ in  \eqref{E:Vtautau} tends to 
$$- \left[ \alpha_1 \cdot  i\pi\left(  \frac{m^2}{N^2}-\frac{m}{N}
   \right)  \right]^2$$  as $\tau$ goes to $ i\infty$  in $H_N$.  This implies that 
   the functions $P_{m,n}$ and $Q_{m,n}$ in \eqref{E:v(x)} actually extend meromorphically through the origin.  \sk 
   
   Since for   $x\in \mathbb C\setminus [0, +\infty[$ sufficiently close to the origin, one has 
   \begin{align*}
v'(x)= & \,
 V^\bullet\big(\tau(x)\big)\cdot \left(   \frac{N}{2i\pi x}\right)\\
\mbox{and} \quad v''(x)= & \,
  V^{\bullet\bullet}\big(\tau(x)\big)\cdot \left(   \frac{N}{2i\pi x}\right)^2-
V^\bullet\big(\tau(x)\big)\left(   \frac{N}{2i\pi x^2}\right)\, , 
\end{align*}
one gets that the asymptotic expansion of \eqref{E:v(x)} at the origin is written 
\begin{align*}
v''(x) +\frac{1}{x}\cdot v'(x)
+\left( -
\left[  \frac{m(m-N)}{2N} \alpha_1  \right]^2\cdot \frac{1}{x^2}   +O\bigg( \frac{1}{x} \bigg) \right)\cdot v(x)=0\, .
\end{align*}

In particular, this shows that the origin is a regular singular point for \eqref{E:v(x)} and    the associated characteristic (or indicial) equation is 
$$s(s-1)+s-\left[  \frac{m(m-N)}{2N} \alpha_1  \right]^2
= s^2-\left[  \frac{m(m-N)}{2N} \alpha_1  \right]^2=0
\, .$$ 

Thus the two associated characteristic exponents are 
$$
s_+= \frac{m(N-m)}{2N} \alpha_1
\qquad \mbox{ and }
\qquad 
 s_{-}=   \frac{m(m-N)}{2N} \alpha_1=-s_+\, , $$
hence  the corresponding index is 
$$
\nu=\nu_{m,n}^N=s_+-s_-=2s_+=\frac{m(N-m)}{N} \alpha_1\, . 
$$

We now have everything in hand to get the result we were looking for. 

\begin{prop} 
\label{P:anglez2=}
Veech's $\mathbb C\mathbb H^1$-structure on   $F_{m,n}$ extends to a conifold 
complex hyperbolic structure at the cusp $[i\infty]_{m,n}$. The associated conifold angle is 
$$
\vartheta_N\big([i\infty]_{m,n}\big)=2\pi m \Big(     1-\frac{m}{N} \Big) \alpha_1\, .
$$

In particular, $[i\infty]_{m,n}$ is a cusp for Veech hyperbolic structure on $F_{m,n}$ (that is, the associated conifold angle is equal to $0$) if and only 
if $m=0$. 
\end{prop}
\begin{proof} In view of  the results and computations above, this is an immediate  
consequence of  Proposition A.2.3.\,of Appendix A.
\end{proof}\sk

Combining the preceding result with Proposition \ref{P:FN-vs-Fmn}, one gets the  
\begin{coro} 
\label{C:AnglesF(N)}
Veech's  hyperbolic structure ${\rm hyp}_N^{\alpha_1}$ on   $Y(N)^{\alpha_1}=F_N=F_{0,1}$ extends to a\;$\mathbb C\mathbb H^1$-conifold 
structure  
on the modular compactification $X(N)^{\alpha_1}$. \sk 

For any coprime $a',c'\in \mathbb Z$, the conifold angle at  $\mathfrak{c}=[-a'/c']\in C(N)$ is equal to 
$$
\vartheta_N\big({\mathfrak{c}}\big)=\vartheta_N\big(c\big)=2\pi c \Big(     1-\frac{c}{N} \Big) \alpha_1
$$
where $c$ stands for the residue of $c'$ modulo $N$: 
$
c=c'-\left\lfloor \frac{c'}{N}\right\rfloor\in \big\{0,\ldots,N-1\big\}
$.
\end{coro}

\subsubsection{}\hspace{-0.4cm} 
\label{SS:MetricCompletionOfY1(N)}
We remind the reader of the following classical description of the cusps of $F_{0,1}=Y(N)$ ({\it cf.}\,\cite[\S3.8]{DS} for instance): the  set of cusps $C(N)$ of $Y(N)$ is in bijection with the set of $N$-order points of the additive group $\big(\mathbb Z/N\mathbb Z\big)^2$, up to multiplication by $-1$, the bijection being given by 
\begin{align*}
\big(\mathbb Z/N\mathbb Z\big)^2[N]_{/\pm} & \longrightarrow C(N)\\ 
\pm \big( a ,  c
\big) & \longmapsto \big[  -a'/c'
\big]
\end{align*}
where $a'$ and $c'$ are pairwise prime and  congruent to $a$ and $c$ modulo $N$ respectively.  From the preceding corollary, it comes that 
the conifold angle associated to  the cusp corresponding to 
$\pm(a,c)$  with $c\in \{0,\ldots,N-1\}$ is 
$$
\vartheta_N\big(\pm(a,c)\big)=\vartheta_N\big(c\big)=
2\pi \frac{c(N-c)}{N}
\alpha_1\, . 
$$

Note that since $-(a,c)=(N-a,N-c)$ in $\big(\mathbb Z/N\mathbb Z)^2$ and because $\vartheta_N(c)=\vartheta_N(N-c)$, this formula makes sense.
\mk 

Using the preceding corollary, it is then easy to  describe the metric completion $\overline{Y(N)}$ of  $Y(N)= F_{0,1}$, the latter being endowed with Veech's $\mathbb C\mathbb H^1$-structure. 
From the preceding results, it comes that this metric completion 
 is the union of the  intermediary leaf $Y(N)= F_{0,1}$ with the subset of its cusps 
of the form $[-a'/c']$ with $c'$ not a multiple of $N$.  Such cusps correspond to classes $\pm  (a,c)\subset \big(\mathbb Z/N\mathbb Z\big)^2[N]$ with $c\in \{1,\ldots,N-1\}$.  
The cusps $[-a'/c']$ with $c'\in N\mathbb Z$ are cusps in the classical sense for  the $\mathbb C\mathbb H^1$-conifold $\overline{Y(N)}$. The number of these genuine  cusps is then equal to $\phi(N)$\footnote{Here $\phi$ stands for 
Euler's totient function.}. 
\sk 

At this point, it is easy to give an explicit description of the metric completion of the  leaf $Y_1(N)^{\alpha_1}=\mathscr F_N^{\alpha_1}$ of Veech's foliation on $\mathscr M_{1,2}$: 
\begin{coro} 
\label{C:AnglesX1(N)alpha}
Veech's $\mathbb C\mathbb H^1$-structure on   $Y_1(N)^{\alpha_1}$ extends to a conifold 
complex hyperbolic structure on the modular compactification $X_1(N)^{\alpha_1}$.\sk 
 
For any coprime $a',c'\in \mathbb Z$, the conifold angle at  $\mathfrak{c}=[-a'/c']\in C_1(N)$ is equal~to 
\begin{equation}
\label{E:ConifoldAnglesY1(N)}
\theta_N\big({\mathfrak{c}}\big)=\theta_N\big(c\big)=2\pi  
 \frac{c(N-c)}{N  \, {\rm gcd}(c,N)}  
   \alpha_1
\end{equation}
where $c\in \big\{0,\ldots,N-1\big\}$ stands for the residue of $c'$ modulo $N$.
\end{coro}
\begin{proof}
This follows at once from \eqref{E:RelationAnglesX(N)-X1(N)}, Lemma \ref{L:WidthCuspX1(N)} and Corollary 
\ref{C:AnglesF(N)}.
\end{proof}
\mk 

To conclude this section, we would like to add two remarks. \sk

First, it is to be understood that the preceding corollary completely characterizes  
Veech's  complex hyperbolic structure of the leaf $Y_1(N)^{\alpha_1}=\mathscr F_N^{\alpha_1}$ since the latter  is completely determined by 
the conformal structure of $\mathscr F_N^{\alpha_1}$ (which is the one of $Y_1(N)$) together with 
the cone angles at the conifold points,  as it follows from a classical result ({\it cf.}\;Picard-Heins' theorem  in Appendix A).  

Since the conifold angles 
\eqref{E:ConifoldAnglesY1(N)}  depend linearly on $\alpha_1$, the family of $Y_1(N)^{\alpha_1}$'s for $\alpha_1\in [0,1[$ is a real-analytic deformation of the usual modular curve $Y_1(N)$, if one sees it as $Y_1(N)^0$ (as it is natural to do so).
\mk

Finally, note that the preceding results refine and generalize 
 (the specialization to the case when $g=1$ and $n=2$ of) the main result of \cite{GP}. In this article, we prove that when $\alpha$ is rational,  Veech's hyperbolic structure of an algebraic leaf of Veech's foliation extends as a conifold structure of the same type to the metric completion of this leaf. 
 Not only our results above  give a more precise and explicit version of this result 
  in the case under scrutiny but they also show that the same statements  
 hold true even without assuming that $\alpha$ is rational.  This assumption appears  to be crucial  to make effective  the geometric methods  la Thurston  used in \cite{GP}.


\subsection{\bf Some explicit examples}
\label{S:ExplicitExamples}
To illustrate the results obtained above, we first treat explicitly  the cases of $Y_1(N)$ for $N=2,3,4$. 
 These are related to `classical hypergeometry' and this will be investigated further  in Section \ref{S:LinksWithClassicalHypergeometry}. Then we consider the case of $Y_1(5)$ then eventually  the one of 
 $Y_1(p)$ for $p$ an arbitrary prime number bigger than or equal to 5. 
\sk 

In what follows, the parameter $\alpha_1\in ]0,1[$ is fixed once and for all. A useful reference for  the elementary results on modular curves used below is \cite{DS}, especially sections \S3.7 and \S3.8 therein. 

\subsubsection{\bf The case of $\boldsymbol{Y_1(2)}$}
\label{S:CaseOfY1(2)}
 The congruence subgroup $\Gamma_1(2)$  has two cusps (namely $[i\infty]$ and $[0]$)  and one elliptic point of order 2. Consequently, $Y_1(2)^{\alpha_1}
$ is a genuine  orbi-leaf of Veech (orbi-)foliation on ${\mathscr M}_{1,2}$.  It is the Riemann sphere  punctured at two points, say $0$ and $\infty$, with one orbifold point of weight 2, say at 1.   The corresponding conifold angles of the associated Veech's $\mathbb C\mathbb H^1$-structure are given in Table \ref{Ta:Y1(2)angles} below in which the cusps are seen as points of $X_1(2)=\mathbb P^1$: 
\begin{table}[h!]
\begin{tabular}{|c|c|c|c|}
\hline 
Cusps of $Y_1(2)^{\alpha_1}$  & 
 0  & 1 
&  $\infty$ 
  \\ \hline
Conifold angles  &  $\pi\alpha_1$  &  $\pi$  &  0  \\
\hline 
\end{tabular}\bk
\caption{The cusps and the associated conifold angles of $Y_1(2)^{\alpha_1}$.}
\label{Ta:Y1(2)angles}
\vspace{-0.5cm}
\end{table}
\subsubsection{\bf The case of $\boldsymbol{Y_1(3)}$}  This case is very similar to the preceding one:  $\Gamma_1(3)$  has two cusps   and one elliptic point, but of order 3. Hence $Y_1(3)^{\alpha_1}
$ is an orbi-leaf of Veech's (orbi-)foliation.  It is $\mathbb P^1$  punctured at three points, say $0$, $1$ and $\infty$, with one orbifold point of weight 3, say at 1.   The corresponding conifold angles are given in the table below.
\begin{table}[h!]
\begin{tabular}{|c|c|c|c|}
\hline 
Cusps of $Y_1(3)^{\alpha_1}$  & 
 0  & 1 
&  $\infty$ 
  \\ \hline
Conifold angles  &  $({4\pi}/{3})\alpha_1 $  &  ${2\pi}/{3}$  &  0  \\
\hline
\end{tabular}\bk
\caption{The cusps and the associated conifold angles of $Y_1(3)^{\alpha_1}$.}
\end{table}

\subsubsection{\bf The case of $\boldsymbol{Y_1(4)}$}  The group  $\Gamma_1(4)$  has three cusps   and no elliptic point. Hence $Y_1(4)^{\alpha_1}
$ is the trice-punctured sphere $\mathbb P^1\setminus \{0,1,\infty\}$.  The corresponding conifold angles are given in the table below.
\begin{table}[h!]
\begin{tabular}{|c|c|c|c|}
\hline 
Cusps of $Y_1(4)^{\alpha_1}$  & 
 0  & 1 
&  $\infty$ 
  \\ \hline
Conifold angles  &  $({3\pi}/{2})\alpha_1 $  &  ${\pi}\alpha_1$  &  0  \\
\hline
\end{tabular}\bk
\caption{The cusps and the associated conifold angles of $Y_1(4)^{\alpha_1}$.}
\vspace{-0.3cm}
\end{table}

The three cases considered above are very particular since they are the only algebraic leaves 
 of Veech's foliation 
  on ${\mathscr M}_{1,2}$  which are isomorphic to the thrice-punctured sphere,  hence   whose hyperbolic structure can be uniformized by means of Gau{\ss} hypergeometric functions. 
 We will return to this later on.

\subsubsection{\bf The case of $\boldsymbol{Y_1(5)}$}

The modular curve $Y(5)$ is of genus 0 and has  twelve cusps. Some representatives  
 of these cusps are given in the first row of Table  \ref{T:Y(5)angles} below, the associated conifold angles of $Y(5)^{\alpha_1}=F_{0,1}^{\alpha_1}$ (we use here the notation of 
\S\ref{S:AuxiliaryLeaves})  are given in the second row. 
 \begin{table}[!h]
\begin{tabular}{|c
|c|c|c|c|c|c|c|c|c|c|c|c|}
\hline 
Cusp  $\mathfrak c$& 
$ i \infty $  & 0
&  1
&  -1 
&  2
& -2
& $\frac{1}{2}$
& $-\frac{1}{2}$
& $\frac{3}{2}$
& $-\frac{3}{2}$
& $-\frac{5}{2}$
& $-\frac{2}{5}$  \\ \hline
$\frac{\vartheta_5(\mathfrak c)}{2\pi\alpha_1}{}$ & 
0 &   $\frac{4}{5} $  & 
 $\frac{4}{5} $  &   $\frac{4}{5} $  &   $\frac{4}{5} $  & $\frac{4}{5} $  &
  $\frac{6}{5} $  &
   $\frac{6}{5} $  &  $\frac{6}{5} $  &  $\frac{6}{5} $  &  $\frac{6}{5} $  
& 0  \\
\hline
\end{tabular}\bk
 \caption{The cusps and the associated conifold angles of $Y(5)^{\alpha_1}$.}
 \label{T:Y(5)angles}
 \vspace{-0.3cm}
\end{table}
\sk

Since $\Gamma_1(5)$ is generated by adjoining $\tau\mapsto \tau+1$ to $\Gamma(5)$, it comes that  among those of Table \ref{T:Y(5)angles}, $0$, $1/2$, $2/5$ and $i\infty$ form a complete set of representatives of the cusps of $\Gamma_1(5)$.  Moreover, it can be verified that the quotient map 
$$
X(5)\longmapsto X_1(5)\, , 
$$
which is a ramified covering of degree $5$,  ramifies at order 5 at the two cusps $[2/5]_{X(5)}$ and $[i\infty]_{X(5)}$  and is \'etale at the ten others cusps of $X(5)$. It follows that the conifold angles of Veech's complex hyperbolic structure on $Y_1(5)^{\alpha_1}$ 
are 0,   0, $\frac{12}{5}\pi\alpha_1$  and    $\frac{8}{5}\pi \alpha_1$ at $[i\infty]_{X_1(5)}$,  $[2/5]_{X_1(5)}$, 
$[1/2]_{X_1(5)}$
 and $[0]_{X_1(5)}$ respectively. 
 
The map $
 \mathbb H^\star  \rightarrow \mathbb P^1, \, 
  \tau \mapsto q^{-1}\prod_{n\geq 1}(1-q^n)^{-5 \big( \! \frac{n}{5}\! \big) }$ 
 (with $q=e^{{2i\pi\tau}}$ and where $( \! \frac{\, }{\ }\!)$ stands for Legendre's symbol) is known to be a Hauptmodul for $\Gamma_1(5)$ which sends $i\infty$, $2/5$, $0$ and  $1/2$ onto 
 $0$, $\infty$, $(11-5\sqrt{5})/2$ and $(11+5\sqrt{5})/2$ respectively. 
 
 \begin{table}[!h]
\begin{tabular}{|l|c|c|c|c|}
\hline 
\; Cusp  in $\mathbb H^\star$    & 
$ i \infty $  & 0  &   $\frac{1}{2}$
& $\frac{2}{5}$     \\ \hline
\;   Cusp in $X_1(5)\simeq \mathbb P^1$ & 0    & $\frac{11-5\sqrt{5}}{2}$  & $\frac{11+5\sqrt{5}}{2}$  &  $\infty$
   \\ \hline
\; Conifold angle  &   0  &  $\frac{8}{5}\pi \alpha_1$ &  $\frac{12}{5}\pi \alpha_1$ & 0  \\
\hline
\end{tabular}\bk
 \caption{The cusps and the associated conifold angles of $Y_1(5)^{\alpha_1}$.}
  \label{Table:Y1(5)angles}
  \vspace{-0.4cm}
\end{table}

\subsubsection{\bf The case of $\boldsymbol{Y_1(p)}$ with $\boldsymbol{p}$ prime}
\label{SS:CaseY1(p)ForpPrime}
Now let $p$ be a prime  bigger than 3.   \sk 

It is known that $\Gamma_1(p)$  has genus $\frac{1}{24}(p-5)(p-7)$, no elliptic point and 
$p-1$ cusps, among which $(p-1)/2$ have width 1, the $(p-1)/2$ other ones  having width $p$. Combining the formalism of Section \ref{S:AuxiliaryLeaves} with Proposition \ref{P:anglez2=}, 
one easily verifies that  the two following assertions  hold true: 
\begin{itemize}
\item  the cusps of width 1 of $Y_1(p)^{\alpha_1}= \mathscr F_{0,1}$ correspond to the cusps $[i\infty]_{0,k}$ of the leaves
$\mathscr F_{0,k}$ (associated to the equation $z_2={k}/{p}$ in $\mathcal T\!\!{\it or}_{1,2}$) for $k=1,\ldots,(p-1)/2$, thus  the associated conifold angles are all 0; \mk 
\item 
the  cusps of width $p$ correspond  to the cusps $[i\infty]_{k,0}$ of the leaves
$\mathscr F_{k,0}$ (associated to the equation $z_2={k}\tau/{p}$ in the Torelli space)  for  $k=1,\ldots,$ $(p-1)/2$.  For any such $k$, the associated cusp is $[-p/k]$ and the associated conifold angle 
is $2\pi k (1-k/p)\alpha_1$. 
\end{itemize}

\section{\bf Miscellanea}

\subsection{\bf Examples of explicit degenerations towards flat spheres with three conical singularities}
\label{S:ExplicitDegen}
The main question investigated above is that of the metric completion of the closed (actually algebraic) leaves
of Veech's foliation on ${\mathscr M}_{1,2}$.

 In \cite{GP}, under the supplementary hypothesis that $\alpha$ is rational (but then in arbitrary dimension), the same question has been  investigated by geometrical methods.  In the particular case when Veech's foliation $\mathscr F^{\alpha_1}$ 
is of (complex) dimension 1, our results in \cite{GP} show that, in terms of (equivalence classes of) flat surfaces,  the metric completion of an algebraic leaf $\mathscr F_N^{\alpha_1}$ in  ${\mathscr M}_{1,2}$ is obtained by attaching to it a finite number of points which correspond to flat spheres with three conical singularities, whose associated cone angles can be determined by  geometric arguments. 
\mk 

Using some formulae of  \cite{ManoKumamoto}, one can recover the result just mentioned but in an explicit analytic form. 
 We treat succinctly  below  the case of the cusp $[i\infty]$ of the leaf $\mathscr F_{(1/N,0)}\simeq Y_1(N)$ associated to the equation $z_2={\tau}/{N}$ in $\mathcal T\!\!{\it or}_{1,2}$, for any integer $N$ bigger than or equal to $2$. 
 Details are left to the reader in this particular case  as well as in the general case ({\it i.e.}\;at any other cusp of $Y_1(N)$).\sk

Let $\alpha_1$ be fixed in $]0,1[$. We consider the flat metric $m_\tau=m_\tau^{\alpha_1}=\lvert \omega_\tau \lvert^2$  on $E_{\tau,\tau/N}=E_\tau\setminus \{[0], [\tau/N]   \}$ with conical singularities at $[0]$ and at $[\tau/N]$ where for any $\tau\in \mathbb H$, $\omega_\tau$ stands for the following (multivalued) 1-form on $E_{\tau,\tau/N}$: 
$$
\omega_\tau
(u)=e^{\frac{2i\pi\alpha_1}{N}u }\left[\frac{\theta(u)}{\theta\big(u-{\tau}/{N}\big)}
\right]^{\alpha_1} du\, .
$$

The map which associates the class of 
the flat tori $(E_{\tau,\tau/N}, m_\tau)$ 
in  ${\mathscr M}_{1,2}^\alpha$ 
 to  any $\tau$ in Poincar\'e's upper half-plane   uniformizes  the  leaf $\mathscr F_{(1/N,0)}$ of Veech's foliation. Studying the latter in the vicinity of the cusp $[i\infty]$ is equivalent to studying the $(E_{\tau,\tau/N}, m_\tau)$'s when $\tau$ goes to $i\infty$ in a vertical band of width 1 of $\mathbb H$.  
\sk

First, we perform the change of variables $u-\tau/N=-v$. Then up to a non-zero constant which does not depend on $v$, we have
$$
\omega_\tau(v)=e^{-{\frac{2i\pi\alpha_1}{N}v}}\left[\frac{\theta\big(-v+{\tau}/{N}\big)}{\theta(v)}
\right]^{\alpha_1} dv\, . 
$$

We want to look at  the degeneration of $m_\tau$ when $\tau\rightarrow i\infty\in \partial \mathbb H$.  To this end, one sets  $q=\exp(2i\pi \tau)$.  Then using the natural isomorphism $E_\tau=\mathbb C^*/q^{\mathbb Z}$ induced by   $v\mapsto x=\exp(2i\pi v)$,  one sees that  $E_\tau$ converges towards the degenerated elliptic curve 
$\mathbb C^*/0^{\mathbb Z}=\mathbb C^*$  as $\tau$ goes to $i\infty$. Moreover, for any fixed $\tau\in \mathbb H$, since $dv=(2i\pi x)^{-1}dx$, the 1-form 
$\omega_\tau$ writes as follows in the variable $x$: 
$$
\omega_\tau(x)=(2i\pi)^{-1} x^{-\frac{\alpha_1}{N}-1}
\theta\big({\tau}/{N}-v\big)^{\alpha_1} \theta(v)^{-\alpha_1}
dx\, .
$$

Then, from the classical formula 
$$
\theta(v,\tau)=2\, \sin(\pi v )\cdot q^{1/8}\prod_{n=1}^{+\infty}\big(1-q^n\big)\big(1-xq^n\big)\big(1-x^{-1}q^n\big)
$$
it follows that
\begin{align*}
 \frac{\theta(\tau/N-v)}{\theta(v)}=
 \frac{\sin\left(\pi 
  {\tau}/{N}-\pi v \right)}{\sin(\pi v)}
\cdot 
 \Theta_N(x, q)
\end{align*}
with 
$$
\Theta_N(x, q)= \prod_{n=1}^{+\infty}
\frac{\big(1-x^{-1}q^{n+1/N}\big)\big(1-xq^{n-1/N}\big)}{
\big(1-xq^{n}\big)\big(1-x^{-1}q^{n}\big)
}\,.
$$

An important fact concerning the latter function is that as a function of the variable $x$, $\Theta_N(\cdot,  q)$ tends uniformly towards 1  on any compact set as $q\rightarrow 0$, that is as $\tau$ goes to $i\infty$, for $\tau$ varying in any fixed vertical band.
\sk 

On the other hand, we have
\begin{align*}
 \frac{\sin\left(\pi 
  {\tau}/{N}-\pi v \right)}{\sin(\pi v)}
 = & \frac{e^{i\pi(\tau/N-v)}-  e^{-i\pi(\tau/N-v)}   }{e^{i\pi v}-  e^{-i\pi v}}\\
 = & 
 \frac{q^{\frac{1}{2N}}x^{-{1}/{2}}-   q^{-\frac{1}{2N}}x^{{1}/{2}}  }{x^{1/2}-  x^{-1/2}}
  =   q^{-\frac{1}{2N}} 
  \frac{q^{\frac{1}{N}}- x }{x-  1}\end{align*}
hence, up to multiplication by a nonzero constant that does not depend on $x$ and `up to multi-valuedness', one has 
\begin{align*}
 \omega_\tau(x)= &  x^{\frac{-\alpha_1}{N}-1}
\left[\frac{x-q^{1/N}}{x-1}
\right]^{\alpha_1} \Theta_N(x,q)^{\alpha_1}dx \, . 
\end{align*}

For $\tau \rightarrow i\infty $, one obtains  as limit the following multivalued 1-form  
\begin{equation}
 \omega_{i\infty}(x)= 
 x^{\alpha_1 \frac{N-1}{N}-1} \big(x-1\big)^{-\alpha_1} dx
\end{equation}
on $\mathbb P^1\setminus \{0,1,\infty\}$. 
The associated flat metric $m_{i\infty}=\lvert \omega_{i\infty} \lvert^2$ defines a singular flat structure  of bounded area on $\mathbb P^1$, with 3 conical singular  points at $0,1$ and $\infty$, the
conical angles of which are respectively 
\begin{equation}
\label{E:AnglesP1degeneration}
\theta_0=2\pi \left( 1-\frac{1}{N}\right)\alpha_1\, , \qquad \theta_1=2\pi(1-\alpha_1)
\qquad \mbox{ and }\qquad 
\theta_\infty=\frac{2\pi}{N}\alpha_1. 
\end{equation} 

We verify that there exists a positive function $\lambda(\tau)$ such that 
$$
\lim_{\tau\rightarrow i\infty}
\int_{E_\tau}  {\lambda(\tau)} \lvert \omega_\tau \lvert^2 
= \int_{\mathbb P^1} \lvert \omega_{i\infty} \lvert^2>0. $$

This shows that the $\mathbb C\mathbb H^1$-structure of $\mathscr  F_{(1/N,0)}$ is  not complete at the cusp $[i\infty]$ and that  the (equivalence class of the) flat sphere with three conical singularities of  angles as in \eqref{E:AnglesP1degeneration} belongs to the metric completion of the considered leaf in the vicinity of this cusp.  When $\alpha_1$ is assumed to be in $\mathbb Q$, the metric completion at this cusp is obtained by adding this flat sphere and nothing else, as it is proved in \cite{GP}.  
The previous analytical considerations show in an explicit manner  that this  still holds true even  without assuming $\alpha_1$ to be rational. 
\mk 

%
More generally, let $\mathfrak c=[-a/c]\in X_1(N)$ be a cusp which must be added to $Y_1(N)^{\alpha_1}$ in order to get the metric completion.  From \cite{GP}, we know that when $\alpha_1$ is rational, this cusp can be interpreted geometrically as a moduli space $\mathscr M_{0,3}(\theta)$ of flat structures on $\mathbb P^1$ with three cone points. The associated angle datum $\theta
 \in ]0,2\pi[^3$  depends on $\alpha_1$, $N$  and on $\mathfrak c$. It would be interesting to get a general explicit  formula for  $\theta$ in function of $N,a,c$ and $\alpha_1$ and to verify that this geometric interpretation still holds true without assuming that $\alpha_1$ is rational.  
 
 If we have answered to a particular case of this question above,  
 the latter is still open in full generality.



\subsection{\bf The cases when $N$ is small:  relations with classical special  functions}
\label{S:LinksWithClassicalHypergeometry}
Exactly three of the  leaves of Veech's foliation on ${\mathscr M}_{1,2}$
give rise to a hyperbolic conifold of genus 0 with 3 conifold points: the leaves $\mathscr F_N=Y_1(N)^{\alpha_1}$ for $N=2,3,4$, see 
\S\ref{S:ExplicitExamples} above.  Since any such hyperbolic conifold can be uniformized by a classical hypergeometric differential equation 
(as it is known from  the fundamental work of Schwarz recalled in the Introduction), there must be some formulae 
expressing the Veech map of any one of these three leaves in terms of classical hypergeometric functions.  We consider only the case when $N=2$ in \S\ref{S:TheCaseN=2LinkedWithClassicalHypergeometry}. 
\sk

Since both 
$Y_1(5)$  and $Y_1(6)$  are $\mathbb P^1$ with four cusps, the leaves 
$\mathscr F_5$ and $\mathscr F_6$ correspond to $\mathbb C\mathbb H^1$-conifold structures with four cone points on the sphere.  Since any such structure can be uniformized by means of a Heun's differential equation,  one deduces that the Veech map of these two leaves can be expressed in terms of Heun functions.  We just say a few words about this in \S\ref{SS:Heun}.

\subsubsection{\bf The leaf $\boldsymbol{Y_1(2)^{\alpha_1}}$ and classical hypergeometric functions}
\label{S:TheCaseN=2LinkedWithClassicalHypergeometry}
It turns out that the case when $N=2$ can be handled very  explicitly by specializing a classical result of Wirtinger. 
The modular lambda function 
$\lambda: \tau\mapsto 
\theta_1(\tau)^4/\theta_3(\tau)^4$ 
 is a  Hauptmodul for $\Gamma(2)$: it corresponds to the quotient $\mathbb H \rightarrow Y(2)=\mathbb P^1\setminus\{0,1,\infty\}$. Furthermore, it induces a correspondence between the cusps $[i\infty]$, $[0]$, $[1]$ and the points 
$0,1$ and $\infty$ of $X(2)=\mathbb P^1$ respectively ({\it cf.}\;\cite[Chap.VII.\S8]{Chandrasekharan} if needed).
\sk 

\paragraph{} \hspace{-0.3cm} Specializing a formula obtained by Wirtinger in  \cite{Wirtinger1902}  (see also  (1.3) in 
 \cite{Watanabe2007} with $\alpha={(\alpha_1+1)}/{2}, \beta={1}/{2}$ and $\gamma=1$), one obtains the following formula: 
\begin{equation}
\label{E:N=2Wirtinger}
F\left(\frac{\alpha_1+1}{2} ,  \frac{1}{2}, 1; \lambda(\tau)\right)
= 
\frac{2
\cos\left(\frac{\pi  \alpha_1 
}{2}\right) \theta_3(\tau)^2 \big(1-\lambda(\tau)\big)^{-\frac{\alpha_1}{4}}}{(1-e^{2i\pi\alpha_1})({1-e^{-2i\pi\alpha_1}})}
 \int_{\mathcal P(0,\frac{1}{2})}
\frac{
\theta(u,\tau)^{\alpha_1}}{
\theta_1(u,\tau)^{\alpha_1}}
 du
\end{equation}
where for any $\tau\in \mathbb H$, the integration in the left hand-side is performed along the Pochammer cycle 
$\mathcal P(0,{1}/{2})$   constructed from the segment $[0,1/2]$ in $E_\tau$ (see Figure \ref{F:Pochammer} below).

\begin{center}
\begin{figure}[h]
\psfrag{0}[][][1]{$ 0$}
\psfrag{t}[][][1]{$\tau$}
\psfrag{1}[][][1]{$1$}
\psfrag{Etau}[][][1]{$E_\tau $}
\psfrag{12}[][][1]{$1/2$}
\includegraphics[scale=0.9]{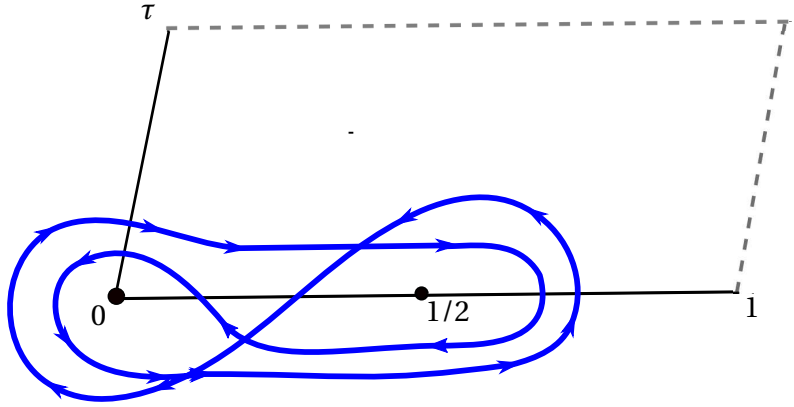}
\caption{The Pochammer contour $\mathcal P(0,{1}/{2})$ in $E_\tau$ (in blue)}
\label{F:Pochammer}
\end{figure}
\end{center} 

For every $\tau$, the flat metric on $E_\tau$ with conical points at $[0]$ and $[1/2]$ associated to the corresponding points of the leaf of equation $z_2=1/2$ in $\mathcal T\!\!\!{\it or}_{1,2}$ is given by $\lvert T^{\alpha_1}(u, \tau)du\lvert^2$ up to normalization,  with 
$$T^{\alpha_1}(u,\tau)=\big( -\theta(u)/\theta(u-1/2)\big)^{\alpha_1}= 
\big( \theta(u)/\theta_1(u)\big)^{\alpha_1}\, .$$

 Equation  \eqref{E:N=2Wirtinger}
can be written out 
\begin{equation}
F\left(\frac{\alpha_1+1}{2} ,  \frac{1}{2}, 1; \lambda(\tau)\right)
=  
\Lambda^\alpha_1( \tau )  \int_{\boldsymbol{\gamma}_2}
T^{\alpha_1}(u, \tau)du
\end{equation}
where $\Lambda^\alpha_1(\tau )$ is a function of $\tau$ and $\alpha_1$ (easy to make explicit with the help of 
 \eqref{E:N=2Wirtinger}) and where 
 $\boldsymbol{\gamma}_2$ stands for the element of the twisted homology group 
 $H_1(E_{\tau,1/2},L_{\tau,1/2})$ obtained after regularizing the twisted 1-simplex  $\boldsymbol{\ell}_2$ defined in \S\ref{SS:ellbullet}.\sk

More generally, for any $\tau$ and any twisted cycle $\boldsymbol{\gamma}$, there is a formula 
\begin{equation}
\label{E:HGF-VF}
F_{\boldsymbol{\gamma}}\left(\frac{\alpha_1+1}{2} ,  \frac{1}{2}, 1; \lambda(\tau)\right)
=  
\Lambda^{\alpha_1}( \tau ) \int_{\boldsymbol{\gamma}}
T^{\alpha_1}(u,\tau)du
\end{equation}
where $F_{\boldsymbol{\gamma}}({(\alpha_1+1)}/{2} ,  {1}/{2}, 1;\cdot )$ is a solution of the hypergeometric differential equation \eqref{HGE} for the corresponding parameters. An important point is that the function 
$\Lambda^\alpha_1( \tau )$ in such a formula is independent of the considered twisted cycle. It follows that 
the map 
\begin{align*}
\tau & \longmapsto \Bigg[ F\left(\frac{\alpha_1+1}{2} ,  \frac{1}{2}, 1; \lambda(\tau)\right) 
: \frac{d}{d\varepsilon} x^{\varepsilon}F\left(\frac{\alpha_1+1}{2}+\varepsilon ,  \frac{1}{2}+\varepsilon, 1+\varepsilon; \lambda(\tau)\right) \bigg\lvert_{\varepsilon=0}
\Bigg]\end{align*}
(whose components form a basis of the associated hypergeometric differential equation, see  \cite[Chap.III\S3]{Yoshida})  is nothing else than an expression of the Veech map $V^{\alpha_1}_{0,1/2}:\mathbb H   \simeq \mathcal F_2^{\alpha_1}\rightarrow  \mathbb P^1$ in terms of classical hypergeometric functions.   
 As an immediate consequence,  one gets  that the corresponding conifold structure on $\mathbb P^1$ is given by the `classical hypergeometric Schwarz's map' $S({(\alpha_1+1)}/{2},{1}/{2},1; \cdot)$. It follows that   the conical angles at the  cusps $0, 1$ and $\infty$ of $X(2)=
\mathbb P^1$ are respectively 
$0$, $\pi\alpha_1$ and $\pi\alpha_1$. 
\sk

Note that this is consistent with our results in 
 \S\ref{S:CaseOfY1(2)}: the lambda modular function 
satisfies $\lambda(\tau+1)={\lambda(\tau)}/({\lambda(\tau)-1})$ for every $\tau \in \mathbb H$ ({\it cf.}\;\cite{Chandrasekharan}). Thus 
$\mu=\mu(\lambda)= 4 ({\lambda-1})/{\lambda^2}$ is  invariant by $\Gamma(2)$ and by $\tau\mapsto \tau+1$,  hence is a Hauptmodul for $\Gamma_1(2)$.  Veech's hyperbolic conifold structure on $X_1(2)$ is the push-forward by $\mu$ of the one just considered on $X(2)$.  Moreover, $\mu$ is \'etale at $0, 1$ and $\infty$, ramifies at the order 2 at $\lambda=2$ and one has $
\mu(0)= \infty$, $\mu(1)=\mu(\infty)=0$ and $\mu(2)=1$.  It follows  that the conifold angles at the cusps $0$, $1$ and $\infty$ of $\mathscr F_{2}^{\alpha_1}=X_1(2)^{\alpha_1}\simeq \mathbb P^1$ are $\pi\alpha_1$,   $\pi$   and $0$ in perfect accordance with the results given in Table \ref{Ta:Y1(2)angles}.   \mk 

 \paragraph{} \hspace{-0.3cm}
 Actually, there is a slightly less explicit but much more geometric approach of the $N=2$ case. 
Indeed, for every $\tau\in \mathbb H$, the flat metric $m^{\alpha_1}_\tau$ on $E_\tau$ with conical points at $[0]$ and $[1/2]$ of respective angles $2\pi(\alpha_1+1)$ and $2\pi(\alpha_1-1)$ is invariant by the elliptic involution (the metric 
$\lvert T^{\alpha_1}(u,\tau)du\lvert^2$ is easily seen to be  invariant by  $u\mapsto -u$). Consequently, $m^{\alpha_1}_\tau$  can be pushed-forward by the quotient map $\nu: E_\tau\rightarrow E_{\tau}/\iota\simeq \mathbb P^1$ and gives rise 
 to a flat metric on the Riemann sphere.   In the variable $u$, for the map $\nu$, it is convenient to take the map induced by 
\begin{align*}
u & \longmapsto \frac{\wp(1/2)-\wp(\tau/2)}{\wp(u)-\wp(\tau/2)}\, .
\end{align*}

Since $\nu$ ramifies at the second order exactly at the 2-torsion points of $E_\tau$, it follows that the push-forward metric $\nu_*(m^{\alpha_1}_\tau)$ is `the' flat metric on 
$\mathbb P^1$ with four conical points at  $0, 1, \infty$ and $\lambda(\tau)^{-1}=\nu((1+\tau)/2)$ whose associated cone angles are respectively $\pi(1+\alpha_1)$, $\pi(1-\alpha_1)$, $\pi$ and $\pi$. 

Consequently, 
in the usual affine coordinate $x$ on $\mathbb P^1$, one has 
\begin{equation}
\label{E:popu}
\nu_*\big(m^{\alpha_1}_\tau\big)= \epsilon^{\alpha_1}(\tau) \,
\Big\lvert  
 x ^{\frac{\alpha_1-1}{2}}
(1-x)^{-\frac{\alpha_1+1}{2}} \big(1-\lambda(\tau)x\big)^{-\frac{1}{2}}
dx
\Big\lvert^2
\end{equation}
for some positive  function $\epsilon^{\alpha_1}$ which does not depend on $x$ but only on $\tau$. 
From \eqref{E:popu}, one deduces immediately that a  formula such as \eqref{E:HGF-VF} holds true 
for any twisted cycle $\boldsymbol{\gamma}$ on $E_\tau$.  Then one can conclude in the same way as at the end of the preceding paragraph.  

\subsubsection{\bf About the  $\boldsymbol{N=3}$ case}
The hyperbolic conifold  $Y_1(3)^{\alpha_1}$ is $\mathbb P^1$ with  three cone points 
 whose  conifold angles are $2\pi (2\alpha_1/3), 2\pi/3$ and $0$.  This $\mathbb C\mathbb H^1$-structure is induced by the hypergeometric equation  \eqref{HGE} with $a=({1+\alpha_1})/{3}$, $b=({1-\alpha_1})/{3} $  and $ c=1$. 
 On the other hand, the Veech map of $\mathcal F_{3}^{\alpha_1}\simeq \mathbb H$ admits as its  components the 
   hypergeometric integrals $
\int_{\boldsymbol{\gamma}_{\bullet}} {\theta(u)}^{\alpha_1}{\theta(u-{1}/{3})}^{-\alpha_1}du
$ with $\bullet=0,\infty$. 

Since  $\delta	(\tau)=(\eta(\tau)/\eta(3\tau))^{12}$ is a Hauptmodul for $\Gamma_1(3)$ (see case 3B in Table 3 of \cite
 {ConwayNorton}), it comes that there exists a function $\Delta^{\alpha_1}( \tau )$ depending only on $\alpha_1$ and on $\tau$, as well as a  twisted cycle $\boldsymbol{\beta}$ on $E_\tau$  such that a formula of the form 
\begin{equation}
F\left(\frac{1+\alpha_1}{3} ,  \frac{1-\alpha_1}{3}, 1; \delta(\tau)\right)
=  
\Delta^{\alpha_1}( \tau ) \int_{\boldsymbol{\beta}}
  \frac{\theta(u,\tau)^{\alpha_1}}{\theta\big(u-\frac{1}{3},\tau\big)^{\alpha_1}}
du
\end{equation}
holds true for every $\tau\in \mathbb H$ and every $\alpha_1\in ]0,1[$ (compare with \eqref{E:N=2Wirtinger}).  \sk 

It would be nice to give explicit formulae for $\Delta^{\alpha_1}$ and $\boldsymbol{\beta}$. Note that a similar question  can be asked in the case when $N=4$.

\subsubsection{\bf  A few words about the case when $\boldsymbol{N=5}$}  
\label{SS:Heun}
  Since $Y_1(5)^{\alpha_1}$ is a four punctured sphere, its
$\mathbb C\mathbb H^1$-structure  can be recovered by means of the famous Heun equation ${\rm Heun}\,(c, \theta_1,\theta_2,\theta_3,\theta_4,p)$. 
As  is well-known, it is a Fuchsian second-order linear differential equation with four simple poles on $\mathbb P^1$. It depends on 6 parameters: the first,  $c$,  is the cross-ratio of the four singularities; the next 4 parameters $\theta_1,\ldots,\theta_4$   
are the angles corresponding to the exponents of the considered equation at the singular points; 
finally, $p$ is the so-called {\it `accessory parameter'} which is the most mysterious one. \sk 

In the case of $Y_1(5)^{\alpha_1}$, $c$ is equal to $\omega=
({11-5\sqrt{5}})/({11+5\sqrt{5}})$,  hence only depends on  the conformal type of $Y_1(5)$. The angles $\theta_i$'s are precisely the conifold angles $\theta_i^{\alpha_1}$ 
of Veech's hyperbolic structure on $Y_1(5)$ ({\it cf.}\;Table \ref{Table:Y1(5)angles}). 

It would be interesting to find an expression for the accessory parameter $p^{\alpha_1}$ of the Heun equation associated to $Y_1(5)^{\alpha_1}$ in terms of $\alpha_1$.  Indeed, in this case it might be possible to express the Schwarz map associated to any Heun equation  of the form $$
{\rm Heun}\,\big(\omega, \theta_1^{\alpha_1},\theta_2^{\alpha_1},\theta_3^{\alpha_1},\theta_4^{\alpha_1},p^{\alpha_1}\big)$$ 
as the ratio of two elliptic hypergeometric integrals. 
 Since the monodromy of such integrals can be explicitly determined (cf.\;\S\ref{S:HyperbolicHolonomy} below), this could be a way to determine explicitly the monodromy of a new class of 
Heun equations. 
\mk

To conclude, note  that this approach should also work 
when $N=6$  since 
 $Y_1(6)$  is also of genus 0 with four cusps.


\subsection{\bf Holonomy of the algebraic leaves}
\label{S:HyperbolicHolonomy}
We fix an integer  $N\geq 2$. \sk 

\subsubsection{}\!\!\!\! By definition, for $\alpha_1\in ]0,1[$, the {\bf holonomy} of the leaf $Y_1(N)^{\alpha_1}$ 
is the holonomy of the  complex hyperbolic structure it carries.  It is a  morphism of groups (well defined up to conjugation) which will be denoted by 
\begin{equation}
\label{E:HolonomyRepresentation}
{\rm H}_N^{\alpha_1} \, : \; 
\Gamma_1(N) \simeq 
\pi_1\big(Y_1(N)^{\alpha_1}\big) \longrightarrow  {\rm PSL}_2(\mathbb R)
\simeq {\rm PU}(1,1)
\, .\footnotemark
\end{equation}
\footnotetext{
\label{ulm}
Since $Y_1(N)$ has orbifold points when $N=2,3$, it is necessary to consider instead the orbifold fundamental group $\pi_1(Y_1(N)^{\alpha_1})^{\rm orb}$ in these two cases. We will keep this in mind in what follows and will commit the abuse  to speak always of the usual fundamental group.}
\hspace{0.4cm}Its image 
  will be denoted by 
$$
\boldsymbol{\Gamma}_{\!1}(N)^{\alpha_1}
=
{\rm Im}\big( {\rm H}_N^{\alpha_1}   \big) \subset 
 {\rm PSL}_2(\mathbb R)\, .\sk
$$ and will be called the {\bf holonomy group of} $\boldsymbol{Y_1(N)^{\alpha_1}}$. 
\sk 

It follows from  some results in \cite{ManoKyushu} that the $\mathbb C\mathbb H^1$-structure of $Y_1(N)^{\alpha_1}$ is induced by a Fuchsian second-order differential equation\footnote{To be precise, the differential equation considered in Theorem 3.1 of \cite{ManoKyushu}  is defined on the cover $Y(N)$ of $Y_1(N)$ but it is easily seen that it can be pushed forward onto $Y_1(N)$.}. 
      This  directly links our work 
to very classical ones about the monodromy of Fuchsian differential equations. In our situation, the general problem 
considered by Poincar at the very beginning of \cite{Poincare1884} is twofold and can be stated as follows: 
\begin{enumerate}
\item[({\bf P1})] for $\alpha_1$ and $N$ given, determine the holonomy group $\boldsymbol{\Gamma}_1(N)^{\alpha_1}$;
\sk
\item[({\bf P2})]  find all the parameters $\alpha_1$ and $N$ such that $\boldsymbol{\Gamma}_1(N)^{\alpha_1}$ 
is  Fuchsian.\mk 
\end{enumerate}  

To these two problems, we would like to add a third one, namely 
 \begin{enumerate}
\item[({\bf P3})]  among  the parameters $\alpha_1$ and $N$ such that $\boldsymbol{\Gamma}_1(N)^{\alpha_1}$ 
is Fuchsian, determine the ones for which this group is arithmetic. 
\end{enumerate}
\mk

\subsubsection{} We say a few words about the reason why 
 we believe that the three problems (P1), (P2) and (P3) are important. 
 For this purpose, we recall briefly below  the general strategy followed by Deligne and Mostow in \cite{DeligneMostow,Mostow} and by Thurston in \cite{Thurston} to find new non-arithmetic lattices in ${\rm PU}(1,n-1)$, which is nowadays one of the main problems in complex hyperbolic geometry.
 \sk 
 
  For any manifold $M$ carrying a $\mathbb C\mathbb H^n$-structure, one denotes by $\boldsymbol{\Gamma}(M)$ its holonomy group, namely the image of the associated holonomy representation $\pi_1(M)\rightarrow {\rm PU}(1,n-1)$. 
  \mk 
 
\paragraph{} For  $n\geq 4$, let  $\theta=(\theta_i)_{i=1}^n$ be a $n$-uplet of angles  $\theta_i\in ]0,2\pi[$. 
The moduli space ${\mathscr M}_{0,\theta}$ of flat spheres with $n$ conical points of angles $\theta_1,\ldots,\theta_n$  identifies to ${\mathscr M}_{0,n}$ and moreover carries
 a natural complex hyperbolic structure ({\it cf.}\;\S\ref{S:MultidimContext}).
 
   For some $\theta$'s, which   have been completely determined, the associated holonomy group $\boldsymbol{\Gamma}_\theta=\boldsymbol{\Gamma} ({\mathscr M}_{0,\theta})$  is a lattice $ {\rm PU}(1,n-3)$ and some of them are not arithmetic. 
 This has been obtained via the following approach: 
 the metric completion of ${\mathscr M}_{0,\theta}$ is obtained by adding to it several strata that are themselves 
 moduli spaces ${\mathscr M}_{0,\theta'}$ for some angle-data $\theta'$ deduced explicitly from  $\theta$. Furthermore,  
 the discreteness of the holonomy is also  hereditary: if $\boldsymbol{\Gamma}_\theta$ is discrete,  then all the $\boldsymbol{\Gamma}_{\theta'}$'s corresponding to the  $\theta'$'s associated to the strata  appearing in the metric completion of ${\mathscr M}_{0,\theta}$ must be discrete as well.  
  One ends up with  the case when $n=4$: in this situation, 
  the corresponding holonomy groups $\boldsymbol{\Gamma}_{\theta''}$'s are triangle subgroups of ${\rm PSL}_2(\mathbb R)$  and the $\theta''$'s corresponding to discrete subgroups are known. This allows to find an explicit finite list of  original $\theta$'s for which $\boldsymbol{\Gamma}_{\theta}$  may be discrete.    At this point, there is still some work to do  to verify that these angle-data give indeed complex hyperbolic lattices and to determine the 
  arithmetic ones but this was achieved in \cite{Mostow}. 
 \mk 
 
 \paragraph{} 
Our results proven in \cite{GP} show that  a very similar picture occurs for the metric completion 
 of an algebraic leaf of Veech's foliation in  $\mathscr M_{1,n}$ for any $n\geq 2$.   Hence 
 a strategy similar to  the one outlined above is possible when looking at algebraic leaves with discrete holonomy group in ${\rm PU}(1,n-1)$.
  \sk

Let ${\mathscr F}_{\! N}^\alpha$ be an algebraic leaf of Veech's foliation 
on $\mathscr M_{1,n}$.  
  As already mentioned in  \S\ref{SS:Towards},   it is proven in  \cite{GP} that the metric completion of $\mathscr F_{\!\!N}^\alpha$ can be inductively constructed by adjoining strata $\mathscr S'$ 
  which are covers of moduli spaces of flat surfaces $\mathscr S$. These moduli spaces 
     can be of two different kinds:  either  $\mathscr S$ is  an algebraic leaf of Veech's foliation $\mathscr F^{\alpha'}$ of $\mathscr M_{1,n'}$ with $n'<n$,    for some particular $n'$-uplet $\alpha'$ constructed from $N$ and $\alpha$;  or $\mathscr S$ is a moduli space of flat spheres $\mathscr M_{0,\tilde \theta}$ for a  $\tilde n$-uplet $\tilde \theta$ (with $\tilde n\leq n+1$) which also depends only on $N$ and $\alpha$. 
\sk

In this situation, the property of having a discrete holonomy group happens to be hereditary as well. Consequently,  a necessary condition for   $\boldsymbol{\Gamma}({\mathscr F}_{\! N}^\alpha)$ to be discrete is 
  that $\boldsymbol{\Gamma}(\mathscr S')=\boldsymbol{\Gamma}(\mathscr S)$ be discrete as well, this for any stratum $\mathscr S'$ appearing in the metric completion of ${\mathscr F}_{\! N}^\alpha$.    Since all the genus 0 holonomy groups $\boldsymbol{\Gamma}_\theta$ which are discrete are known, 
we only have to   
  consider the  groups 
  $\boldsymbol{\Gamma}(\mathscr S)$ for the 
   genus 1 strata $\mathscr S$.    Arguing inductively, one  ends up,  as in the genus 0 case,  to considering 
  the holonomy groups of the 1-dimensional strata of $
  \overline{{\mathscr F}_{\! N}^\alpha}$ which are of genus 1.  \sk 
  
  The preceding discussion shows that in order that 
   a strategy similar to the one used in the genus 0 case 
    succeeds in the genus 1 case that we are considering in this paper, it is crucial to have a perfect understanding of the holonomy groups  
  of the algebraic leaves of Veech's foliation on $\mathscr M_{1,2}$. 
  From this perspective, the  questions ({\bf P1}),({\bf P2}) and ({\bf P3})  appear  to be    particularly relevant. \sk 
    
\subsubsection{} It is easy to deduce from the results obtained above a vast class of parameters $\alpha_1$ for which $\boldsymbol{\Gamma}_1(N)^{\alpha_1}$ is a discrete subgroup of  ${\rm PSL}_2(\mathbb R)$. \sk 

\paragraph{} It follows  from  Poincar's uniformization theorem that 
 the  holonomy group $\boldsymbol{\Gamma}_1(N)^{\alpha_1}$ is Fuchsian as soon as $Y_1(N)^{\alpha_1}$ is an orbifold, that is as soon as 
for any cusp $\mathfrak c$ of $Y_1(N)^{\alpha_1}$, the associated conifold angle $\theta_N^{\alpha_1}(\mathfrak c)$ is an integral part of $2\pi$. 

Now it has been shown above (see \eqref{E:ConifoldAnglesY1(N)}) that  for such a cusp $\mathfrak c$, one has 
$$\theta_N^{\alpha_1}(\mathfrak c)= 2\pi  \frac{c(c-N)}{N\gcd(c,N)}    \alpha_1$$
 for a certain integer $c\in \{0,\ldots,N-1\}$ depending on ${\mathfrak c}$. 
  Then, setting $N'$ as the least common multiple of 
 the integers $\frac{c(N-c)}{\gcd(c,N)}$ for $c=1,\ldots,N-1$, we get the 
  \begin{coro} 
  \label{P:alpha1orbifold}
 If $\alpha_1=\frac{N}{N' \ell}$ with $\ell\in \mathbb N_{>0}$, then $\boldsymbol{\Gamma}_1(N)^{\alpha_1}$ is Fuchsian.
  \end{coro}

\paragraph{} 
It is more than likely that the preceding result only gives a partial answer to 
({P2}). 
Indeed,  it is well-known that there exist  triangle subgroups of ${\rm PSL}_2(\mathbb R)$ which are not of orbifold type ({\it i.e.}\;not all the angles of `the' corresponding hyperbolic triangle
  are integral parts of  $\pi$) but such that the $\mathbb C\mathbb H^1$-holonomy of the associated $\mathbb P^1$ with three conifold points is Fuchsian, see \cite[p. 572]{Greenberg}, \cite[Theorem 2.3]{Knapp} or \cite[Theorem 3.7]{Mostow}\footnote{Beware that  two cases have been forgotten in reference \cite{Mostow}.}. 
 
 The situation is  certainly similar for the holonomy groups $\boldsymbol{\Gamma}_1(N)^{\alpha_1}$: for $N$ fixed,  there are certainly more parameters $\alpha_1$ whose associated holonomy group is discrete than the ones given by Corollary \ref{P:alpha1orbifold} which correspond to 
 the cases when Veech's $\mathbb C\mathbb H^1$-structure on  $X_1(N)^{\alpha_1}$ actually is of orbifold type. 
\sk

A complete  answer to (P2) would be very interesting but,  for the moment, we do not see how this problem can be attacked in full generality. A  difficulty inherent to this problem is that there is no known explicit finite type representation 
of $\Gamma_1(N)$ as a group  for $N$ arbitrary, except when $N=p$ for a  prime number $p$ and,  even in this case, the known set of generators  
of $\Gamma_1(p)$ is quite complicated, see \cite{Frasch}\footnote{Actually the results contained in \cite{Frasch}  concern the $\Gamma(p)$'s for $p$ prime but they straightforwardly apply to the $\Gamma_1(p)$'s since $\Gamma_1(N)=\langle \Gamma(N)\, , \, 
\tau\mapsto \tau+1  \rangle$ for any $N$).}. Note that this is in sharp contrast with the corresponding situation in the genus 0 case, where the ambient space  is always $\mathscr M_{0,4}\simeq \mathbb P^1\setminus \{0,1,\infty\}$ whose topology, if not trivial, is particularly simple. 
\mk 

\subsubsection{} According to a well-known result of Takeuchi \cite[Theorem 3]{Takeuchi}, there only exist a finite number of triangle subgroups of ${\rm PU}(1,1)$ which are arithmetic.  It is natural to expect that a similar situation does occur among the groups $\boldsymbol{\Gamma}_1(N)^{\alpha_1}$  which are discrete. 
 However,  we are not aware of any conceptual approaches to tackles a question such as (P3) for the moment.   For instance, determining the holonomy groups of  Corollary \ref{P:alpha1orbifold} which are arithmetic when $N\geq 5$ seems out of reach for now.\footnote{Note that since $\boldsymbol{\Gamma}_1(N)^{\alpha_1}$ are triangle subgroups of ${\rm PSL}_2(\mathbb R)$ for $N=2,3,4$ (see \S\ref{S:LinksWithClassicalHypergeometry}), the three problems (P1),(P2) and (P3) can 
 be completely solved in these cases.}
Here again, the main reason being the  inherent complexity of the congruences subgroups $\Gamma_1(N)$ 
 in their whole.  \sk 

The situation is not as bad for (P1), at least if one considers the problem for a fixed $N$ and from a computational perspective.  Indeed, we have now at disposal integral representions for the components of the developing map of $Y_1(N)^{\alpha_1}$ and this can be used, as in the classical hypergeometric case (see  for instance \cite[Chap.IV.\,\S5]{Yoshida}), to determine explicitly the corresponding holonomy group $\boldsymbol{\Gamma}_1(N)^{\alpha_1}$.   More precisely, it follows from our results in \S\ref{S:AnalyticExpressionVeechMapg=1} that,  setting $T_N(u,\tau)=\theta(u,\tau)^{\alpha_1}/\theta(u-1/N,\tau)^{\alpha_1}$,   the map 
\begin{equation}
\label{E:FNlalala}
F_N : \tau \longmapsto 
\begin{bmatrix}
F_N^\infty(\tau)\vspace{0.1cm}
\\ F_N^0(\tau)
\end{bmatrix}=
\begin{bmatrix}
 \int_{\boldsymbol{\gamma}_\infty}   T_N(u,\tau)du\vspace{0.1cm}
\\ \int_{\boldsymbol{\gamma}_0}   T_N(u,\tau)du
\end{bmatrix} \nonumber
\end{equation}
is the developing map of the lift of Veech's hyperbolic structure on $Y_1(N)^{\alpha_1}$ to its universal covering $\mathcal F_N\simeq \mathbb H$. 
 Consequently, to any projective transform $\widehat \tau=(a\tau+b)/(c\tau+d)$ corresponding to an element $g=\tiny{\big[\!\!
\begin{tabular}{cc}
$a$ \!\!&\!\! \!\!\!\!$b$\vspace{-0.1cm}\\
$c$ \!\!&\!\! \!\!\!\! $d$
\end{tabular}\!\!
\big]}\in \Gamma_1(N)$ will correspond a matrix  $M(g)$ such that 
$F_N(\widehat \tau)=M(g)\cdot F_N(\tau)$ for every $\tau\in \mathbb H$. 
Since $\Gamma_1(N)\simeq \pi_1(Y_1(N)^{\alpha_1})$, 
 it comes that the map 
$g\mapsto M(g)$ induces the holonomy representation 
\eqref{E:HolonomyRepresentation} (up to a suitable conjugation which can be determined explicitly from \S\ref{SS:normalization}, see also \S\ref{S:NormalizationOfVeech'sMap-n=2}).\sk 

Using Mano's connexions formulae presented in \S\ref{SS:ConnectionFormulae},  it is essentially  a computational task to determine explicitly $M(g)$ from $g$ if the latter is given. This is what we explain in \S\ref{SS:ExplicitHolonomyyy} just below. 
Then, once a finite set of explicit generators $g_1,\ldots,g_\ell$ of $\Gamma_1(N)$ is known, one can compute  the matrices $M(g_1),\ldots,M(g_\ell)$ which generate   $\boldsymbol{\Gamma}_1(N)^{\alpha_1}$   (modulo conjugation) and then 
study this group, for instance  by means of algebraic methods.
\mk

\subsubsection{\bf Some explicit connection formulae}
\label{SS:ExplicitHolonomyyy}
Let  $\alpha_1\in ]0, 1[$ be fixed.   
Below, we use the formulae of \S\ref{SS:ConnectionFormulae}
 to obtain some lemmata which can be used to  compute explicitly  the image of 
 a given element of $ \Gamma_1(N)$ in $\boldsymbol{\Gamma}_1(N)^{\alpha_1}$. We end by illustrating our method with  two explicit computations in \S\ref{Parag:2ExplicitComputations}
\mk 

\paragraph{}
For $a=(a_0,a_\infty)\in \mathbb R^2 \setminus \alpha_1\mathbb Z^2$, one sets $r=\alpha_1^{-1}(a_0,a_\infty)\in \mathbb R^2 \setminus \mathbb Z^2$
and 
\begin{equation}
\label{E:mrutau}
 \omega_a(u, \tau)=\exp\big(2i\pi a_0 u \big)\,  \theta(u,\tau)^{\alpha_1}\, \theta\big(u-(r_0\tau-r_\infty), \tau\big)^{-\alpha_1} du. 
\end{equation}

As seen before, the map 
\begin{align*}
\label{E:Fa}
F_a \; : \;  \mathbb H & \longrightarrow \mathbb C^2
\vspace{-0.1cm}\\
\tau & \longmapsto F_a(\tau)=
\begin{bmatrix}
F_a^\infty(\tau) \\
F_a^0(\tau)
\end{bmatrix}= 
\begin{bmatrix}
 \int_{\boldsymbol{\gamma}_\infty}    \omega_a(u,\tau)\vspace{0.1cm}
\\
\int_{\boldsymbol{\gamma}_0}    \omega_a(u, \tau)
\end{bmatrix} \nonumber
\end{align*}
can be seen as an affine lift of the Veech map 
$V_a^{\alpha_1}: \mathcal F_a\rightarrow \mathbb C\mathbb H^1$ of the leaf 
$$\mathcal F_a
=\big\{
(\tau,z_2)\in \mathcal T\!\!{\it or}_{1,2}\, \big\lvert \, a_0\tau-\alpha_1 z_2=a_\infty
\big\}\simeq \mathbb H
$$ of Veech's foliation on the Torelli space $\mathcal T\!\!{\it or}_{1,2}$. \sk

In order to determine the hyperbolic holonomy of an algebraic leaf $Y_1(N)^{\alpha_1}$ it is necessary to establish some connection formulae for the function $F_a$. By this, we mean two 
very slightly distinct things: 
\begin{itemize}
\item first, given a modular transformation $\widehat \tau=(m\tau+n)/(p\tau+q)$, we want to express $F_a( \widehat \tau)$ in terms of $F_{\widehat a}(\tau)$ for a certain $\widehat a$ (which is easy to determine explicitly);  
\sk 
\item second, we want to relate $F_a(\tau)$ and $F_{a''}(\tau)$ for any $\tau$, when  $a$ and $a''$ are congruent modulo $\alpha_1\mathbb Z^2$.  
\end{itemize}

Connection formulae of the first type will be said of {\bf modular type}
whereas those of the second type will be said of {\bf translation type}. 
 \mk


\paragraph{\bf Connection formulae of modular type}
Consider the following two elements of ${\rm SL}_2(\mathbb Z)$
whose classes generate the  modular group: 
$$
T=\begin{bmatrix}
1  &   1  \\
0  &   1
\end{bmatrix}
\qquad \mbox{ and }
\qquad 
S=\begin{bmatrix}
0  &  1  \\
-1  &   0
\end{bmatrix}\, . $$

For $a=(a_0,a_\infty)\in \mathbb R^2 \setminus \alpha_1\mathbb Z^2$, one sets 
 $$a'=\big(a_0,a_\infty-a_0\big)\qquad \mbox{ and }\qquad  
 \widetilde a=\big(a_\infty, -a_0\big)\, . $$  
 Then 
 the restriction of $T$ (resp.\;of  $S$) to $\mathcal F_{a'}$ (resp.\;to 
  $\mathcal F_{\widetilde a}$)
  induces an isomorphism from this leaf onto $\mathcal F_{a}$.    Moreover, it follows from \cite[Theorem\,0.7]{Veech} that both isomorphisms are compatible with the $\mathbb C\mathbb H^1$-structures of these leaves.      The point is that in order  to determine the holonomy of any leaf $Y_1(N)^{\alpha_1}$, we need to make this completely explicit.  \sk

Each of the   two matrices $T$ and $S$ induces an automorphism of the Torelli space ${\mathcal T}\!\!{\it or}_{1,n}$ that will be designated slightly  abusively by the same letter.  
In the natural affine coordinates $(\tau,z_2)$ on the Torelli space (see \S\ref{S:TorelliNag} above), these two automorphisms are written 
\begin{equation*}
T(\tau,z_2)=\big(\tau+1,z_2\big) \qquad \mbox{ and }
\qquad 
S(\tau,z_2)=\big(-{1}/{\tau},-{z_2}/{\tau}\big)\, .
\end{equation*}

   We recall that $\rho=\rho(a)$ stands for 
   $$(\rho_0,\rho_\infty)=\big(\rho_0(a),\rho_\infty(a)\big)=\big(\exp({2i\pi a_0}), \exp({2i\pi a_\infty}\big)$$ with corresponding notations for $\rho'$ and $\widetilde \rho$, that is   
  \begin{align*}
  \rho'= & \rho(a')=\big(\rho_0',\rho_\infty'\big)=\big(\rho_0, \rho_\infty\rho_0^{-1}\big) \\  
\mbox{    and} \quad 
\widetilde \rho= & \rho(\widetilde a)= \big(\widetilde \rho_0,\widetilde \rho_\infty\big)=\big(\rho_\infty, \rho_0^{-1}\big) .
\end{align*}
 
 To save space, we will denote   by $I\!\!H_a$ the matrix $I\!\!H_{\rho(a)}$ 
 ({\it cf.}\;\eqref{E:Hrho}) below. We recall that it is the matrix of Veech's form on the target space of the map $F_a$.  \mk

  To state our result concerning 
  connection formulae of modular type, we will assume that 
  \begin{equation}
  \label{E:rNormal}
    \big(a_0,-a_\infty\big) \in \alpha_1 \, \big[ 0,1\big[^2\, . 
   \end{equation}
   
   This condition can be interpreted geometrically as follows: \eqref{E:rNormal} is equivalent to the fact that, for any $\tau\in \mathbb H$,  the point $(a_0\tau-a_\infty)/\alpha_1$,  which  is a singular point of the multivalued holomorphic 1-form $m_r(u,\tau)$, see   \eqref{E:mrutau}, belongs to the standard fundamental domain $[0,1[_\tau=[0,1[ +[0,1[\tau\subset \mathbb C$ of $E_\tau$.  
   
   Remark that this condition is not really restrictive. Indeed, considering the action \eqref{E:ActionOnHolonomy}, it comes easily that   for any $a\in \mathbb R^2\setminus \alpha_1\mathbb Z^2$, there exists $a^*$ in the same $({\rm SL}_2(\mathbb Z)\ltimes \mathbb Z^2)$-orbit (hence such that 
   $\mathcal F_a\simeq \mathcal F_{a^*}$) 
   which satisfies 
    \eqref{E:rNormal}.
     \begin{lemma}    
\label{L:modularTandS}
Assume that condition \eqref{E:rNormal} holds true.\sk

\begin{enumerate}
\item 
 For  every $\tau\in \mathbb H$, one has $$F_{a}(\tau+1)=  T_{a}\cdot F_{a'}(\tau) $$ with 
 \begin{equation}
\label{E:Tr}
T_{a}= \, 
\begin{bmatrix}
1  &    \rho_\infty/\rho_0 \\
0  &   1
\end{bmatrix}\, .
\end{equation}\mk 
\item  There exists a function $\tau\mapsto \sigma_{a}(\tau)$ such that for every $\tau\in \mathbb H$, one has 
\begin{equation*}
\label{E:Fab}
F_{a}(-1/\tau)=  \sigma_{a}(\tau) \,  \big( S_{a}\cdot F_{\widetilde a}(\tau) \big)
\end{equation*}
with \begin{equation}
\label{E:Sr}
S_{a}=  \, 
   \begin{bmatrix}
1-\rho_\infty &  \rho_0^{-1}
 \\ 
-\rho_0 &  0
 \end{bmatrix} .
\end{equation}

\end{enumerate}
\end{lemma}
    The following notations will be convenient in the proof below: for $\bullet=0,\infty$ and $\tau\in \mathbb H$, we denote by $\boldsymbol{\gamma}_{\!\bullet}(\tau)$ the twisted 1-cycle constructed in \S3 with   $\tau$ seen as a point of $\mathcal F_{a}$:  the ambient torus is $E_{\tau}$  which is punctured  at $[0]$ and 
$[z_2]$ with $z_2=\alpha_1^{-1}(a_0\tau-a_\infty)$.  And for any superscript ${}^\star$, 
we will write $\boldsymbol{\gamma}^\star_{\!\bullet}(\tau)$ for the corresponding cycles but when $a$ has been replaced by $a^\star$.  
For instance,  $\boldsymbol{\gamma}_{0}'(\tau)$ is the twisted cycle on the torus $E_\tau$,  punctured at $[0]$ and $[z_2']$ with $z_2'=(a_0'\tau-a_\infty')/\alpha_1=r_0\tau+(r_0-r_\infty)$,  obtained by the regularization of  $]0,1[$. 
 \sk

\begin{proof}
%
We first treat the  case of the transformation $\tau\mapsto \tau+1$ that will be used to deal with the second one after.\sk

On the one hand, one has 
\begin{equation}
\label{E:ploc}
\omega_{a}(u,\tau+1)=\omega_{a'}(u, \tau).
\end{equation}

  On the other hand, the map associated to 
$E_{\tau}\rightarrow E_{\tau+1}$ lifts to the identity in the variable $u$. Consequently, one has (see Figure \ref{F:TransfoT} below)
\begin{equation}
\label{E:gogogo}
\boldsymbol{\gamma}_\infty(\tau+1)= \boldsymbol{\gamma}_\infty'(\tau)+\rho'_{\infty} \boldsymbol{\gamma}_0'(\tau)
\qquad \mbox{ and }\qquad 
\boldsymbol{\gamma}_0(\tau+1)=\boldsymbol{\gamma}_0'(\tau)\, .
\end{equation}
\begin{center}
\begin{figure}[h]
\psfrag{z2}[][][0.8]{$ z_2$}
\psfrag{z2'}[][][0.8]{$ \textcolor{Griz}{z_2'}$}
\psfrag{0}[][][1]{$ 0$}
\psfrag{t}[][][1]{$\tau$}
\psfrag{1+t}[][][1]{$1+\tau$}
\psfrag{L0}[][][1]{$\qquad \qquad  \boldsymbol{\ell}_0(\tau+1)=\textcolor{Griz}{\boldsymbol{\ell}_0'(\tau)}$}
\psfrag{pL0}[][][1]{$\qquad \quad  \textcolor{Griz}{ \rho_\infty' \boldsymbol{\ell}_0'(\tau)}$}
\psfrag{Li}[][][1][28]{$\qquad   \boldsymbol{\ell}_\infty(\tau+1)$}
\psfrag{pLi}[][][1][77]{$\textcolor{Griz}{ \boldsymbol{\ell}_\infty'(\tau)}$}
\includegraphics[scale=0.9]{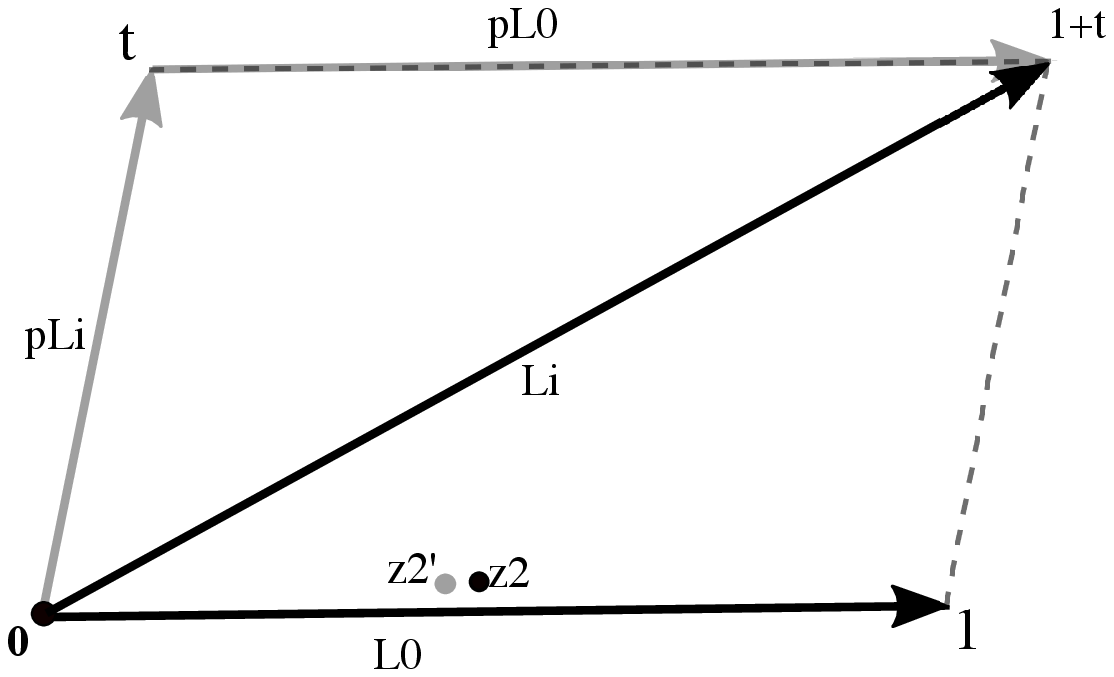}
\caption{Relations between the locally-finite twisted 1-simplices 
$\boldsymbol{\ell}_0(\tau+1), \boldsymbol{\ell}_\infty(\tau+1), { \boldsymbol{\ell}_0'(\tau)}$ and ${ \boldsymbol{\ell}_\infty'(\tau)}$ 
which give  \eqref{E:gogogo} after regularization  ({\it cf.}\;\S\ref{SS:NiceBasisTwistedHomologyClasses}).  The point $z_2$ has been assumed to be of the form $\epsilon\tau+1/N$ with $N=2$ and $\epsilon>0$ small ({\it i.e.}\;the pictured case corresponds to $a=(\epsilon, -1/2)$).}
\label{F:TransfoT}
\end{figure}
\end{center} 

The two looked forward relations 
$ 
F^\infty_{a}(\tau+1)=F^\infty_{a'}(\tau)+ 
(\rho_\infty/\rho_0)\cdot F^0_{a'}(\tau)
$ and 
$F^0_{a}(\tau+1)=F^0_{a'}(\tau)
$  then  follow immediately from \eqref{E:ploc} and    \eqref{E:gogogo}.

 Note that the following relation holds true
$$
I\!\!H_{a'}= 
{}^t\overline{T_{a}}\cdot     I\!\!H_{a}\cdot   T_{a}\, ,   
 $$
which is coherent with the fact that $T$ induces an isomorphism between the 
$\mathbb C\mathbb H^1$-structure of the leaves $\mathcal F_{a'}$ and $\mathcal F_{a}$ of Veech's foliation on the Torelli space.\sk 

We now turn to the  case of $S$. 
Considering the principal determination of the square root, one has 
$$
\theta\Big(-\frac{u}{\tau},-\frac{1}{\tau}\Big)
=
i\cdot  \sqrt{\frac{\tau}{i}} \cdot 
e^{\frac{i\pi u^2}{\tau}}\cdot 
\theta(u,\tau) 
$$ 
for every $\tau\in \mathbb H$ and every $u\in \mathbb C$ (see \cite[(8.5)]{Chandrasekharan}). 
\sk 

By a direct computation (see also \cite[Prop. 6.1]{Mano}), one gets 
\begin{align}
\label{E:mu(-1/tau)}
\omega_{a}\big(-u/\tau, -1/\tau\big)=  & \, 
\sigma_{a}(\tau)\cdot 
\omega_{\widetilde a}(u,  \tau)
\end{align}
with $\sigma_{a}(\tau)=-{\tau}^{-1}
\exp\big(      -{i\pi }(a_0+a_\infty \tau)^2/(\alpha_1\tau)\big)$. 
\sk 

Now one needs to
determine the action of $S_*$   on the twisted 1-cycles $\widetilde{ \boldsymbol{\gamma}}_\infty$ and 
$\widetilde {\boldsymbol{\gamma}}_0$. 
 To this end, one remarks that 
the following decomposition holds true
$$
\begin{bmatrix}
0 & 1\\
-1 & 0
\end{bmatrix}= 
\begin{bmatrix}
1 & 1\\
 0& 1
\end{bmatrix}
\begin{bmatrix}
1 & 0\\
-1 & 1
\end{bmatrix}
\begin{bmatrix}
1 & 1\\
0 & 1
\end{bmatrix}\, .
$$
This  relation will allow us to express the action of $S_*$ as a composition of three maps of the type of the one previously determined. More explicitly, 
from the preceding matricial decomposition we get that
 the isomorphism 
 between ${\mathcal F_{\widetilde a}}$ and ${\mathcal F_{a}}$
  induced 
  by the restriction of $S$ factorizes as follows
 \begin{equation}
 \label{E:factoS}
    \xymatrix@R=1cm@C=1.7cm{  
      \mathcal F_{\widetilde a}       
      \ar@/_2pc/[rrr]_{
        \tiny{
\big[\!\!\!\!
 \begin{tabular}{cc}
0 \!\!\!\!& 1
\vspace{-0.1cm}
\\
-1 \!\!\!\!& 0
\end{tabular}\!\!\!\!
\big]
} 
      }
          \ar@{->}[r]^{
      \tiny{
\big[\!\!\!\!
 \begin{tabular}{cc}
1 \!\!\!\!& 1
\vspace{-0.1cm}
\\
0 \!\!\!\!& 1
\end{tabular}\!\!\!\!
\big]
} 
       }  
  & 
   \mathcal F_{ a^\star}           \ar@{->}[r]^{
       \tiny{
\big[\!\!\!\!
 \begin{tabular}{cc}
1 \!\!\!\!& 0
\vspace{-0.1cm}
\\
-1 \!\!\!\!& 1
\end{tabular}\!\!\!\!
\big]
} 
}  
  &  
    \mathcal F_{a'}           \ar@{->}[r]^{
    \tiny{
\big[\!\!\!\!
 \begin{tabular}{cc}
1 \!\!\!\!& 1
\vspace{-0.1cm}
\\
0 \!\!\!\!& 1
\end{tabular}\!\!\!\!
\big]
} 
}  
  &   \mathcal F_{a}           }\, , 
        \end{equation}
where  $a^\star=(a_\infty, a_\infty-a_0)$.   \sk

From what has been done before, we get that, in the suitable corresponding basis, the matrices of the linear transformations on the twisted 1-cycles associated to the first and the last isomorphisms in  \eqref{E:factoS} are $T_a$ and $T_{a^\star}$ respectively.

By an analysis  completely similar and symmetric to the one we did to treat the first case, one gets that the isomorphism between 
$ \mathcal F_{a^\star}$ and $ \mathcal F_{a'}$ in the middle of \eqref{E:factoS}  induces the following transformation for the corresponding twisted 1-cycles: 
$$
\begin{bmatrix}
\boldsymbol{\gamma}'_\infty\big(  \frac{\tau}{1-\tau} \big)
\vspace{0.1cm}
 \\ 
 \boldsymbol{\gamma}'_0\big(  \frac{\tau}{1-\tau} \big)
 \end{bmatrix}= 
  \begin{bmatrix}
1 &  0
 \\  
-\rho_0' &  1
 \end{bmatrix}
 \begin{bmatrix}
\boldsymbol{\gamma}^*_\infty(\tau)
 \\ 
 \boldsymbol{\gamma}^*_0(\tau)
 \end{bmatrix}\, .
$$

From \eqref{E:factoS} and what follows, it comes that for every $\tau \in \mathbb H\simeq \mathcal F_{\widetilde a}$, the action of $S$ on twisted 1-cycles is given by 
$$ \begin{bmatrix}
\boldsymbol{\gamma}_\infty(-1/\tau)
 \\ 
 \boldsymbol{\gamma}_0(-1/\tau)
 \end{bmatrix}
= 
   \begin{bmatrix}
1 &  \frac{\rho_\infty}{\rho_0}
 \\ 
0 & 1
 \end{bmatrix}
  \begin{bmatrix}
1 & 0
 \\  
-\rho_0 &  1
 \end{bmatrix}
  \begin{bmatrix}
1 & \frac{\rho_\infty^*}{\rho_0^*}
 \\  
0 & 1
 \end{bmatrix}
 \begin{bmatrix}
\widetilde{\boldsymbol{\gamma}}_\infty(\tau)
 \\ 
\widetilde{\boldsymbol{\gamma}}_0(\tau)
 \end{bmatrix}$$
that is by
\begin{align}
\label{E:explicitS}
 \begin{bmatrix}
\boldsymbol{\gamma}_\infty(-1/\tau)
 \\ 
 \boldsymbol{\gamma}_0(-1/\tau)
 \end{bmatrix}
=
   \begin{bmatrix}
1-\rho_\infty &  \rho_0^{-1}
 \\ 
-\rho_0 &  0
 \end{bmatrix}
 \begin{bmatrix}
\widetilde{\boldsymbol{\gamma}}_\infty(\tau)
 \\ 
\widetilde{\boldsymbol{\gamma}}_0(\tau)
 \end{bmatrix}.
\end{align}

The second part of the lemma thus follows  by combining 
\eqref{E:mu(-1/tau)} with \eqref{E:explicitS}. 
\end{proof}
\mk


\paragraph{\bf Connection formulae of translation type}

Since our main interest is in the algebraic leaves of $\mathcal F^{\alpha_1}$, we will
establish connection formulae of translation type only for the maps $F_a$'s associated to such leaves. The general case 
does not present more difficulty but is not of interest
  to  us.
\mk 

Let $N$ be a fixed integer strictly bigger than 1. One sets 
$$
\mu=e^{  \frac{2i\pi}{N} \alpha_1  }. 
$$

For $m,n\in \mathbb Z^2$ with $(m,n,N)=1$, 
remember the notation  $\mathcal F_{m,n}$ of Section \ref{S:AuxiliaryLeaves}: 
$$
\mathcal F_{m,n}=\mathcal F_{(m/N,-n/N)}=\bigg\{
\big(\tau,z_2\big)\in {\mathcal{ T}}\!\!\!{\it or}_{1,2}
\; \big\lvert \; 
z_2=(m/N)\tau+n/N
\bigg\}\, .$$
The  lifted holonomy  associated to this leaf of Veech's foliation is 
$$
	a^{m,n}=\left(  \frac{ m }{N}\alpha_1 , -\frac{ n }{N}\alpha_1      \right)
$$
whose  associated linear  holonomy is given by 
$$\rho^{m,n}=\big(\rho_0^{m,n}, \rho_1^{m,n}, \rho_\infty^{m,n}\big)=
\left(\mu^m,\mu^N,\mu^{-n}\right)=
\left(e^{  \frac{2i\pi m}{N} \alpha_1  },
e^{ {2i\pi} \alpha_1  } , 
e^{  -\frac{2i\pi n}{N} \alpha_1  }
\right)
 \, .
$$

From now on, we use the notations $\omega_{m,n}$ and $F_{m,n}$ for $\omega_{a^{m,n}}$ and $F_{a^{m,n}}$ respectively: for $\tau\in \mathbb H$, one has 
\begin{align}
\label{E:VeechFab}
F_{m,n}(\tau)=\begin{bmatrix}
 \int_{\boldsymbol{\gamma}_\infty}  
 \omega_{m,n}(u,\tau)\\ 
 \int_{\boldsymbol{\gamma}_0}  
 \omega_{m,n}(u,\tau)
 \end{bmatrix}
 \quad 
 \mbox{ with }
 \quad 
  \omega_{m,n}(u, \tau)=    
  \frac{e^{ \frac{2i\pi  m\alpha_1}{N}} \theta(u,\tau)^{\alpha_1}}{\theta\left(u-
\big(\frac{m\tau+n}{N}\big), \tau\right)^{\alpha_1}}
du\,.
\end{align}
%

 Then, using the notations of Section \ref{S:IntersectionProduct}, one sets (see  \eqref{E:Hrho}): 
$$
I\!\!H_{m,n}=I\!\!H_{\rho^{m,n}}= \frac{1}{2i}
\begin{bmatrix} 
  \frac{
    (\mu^m-1)(1-\mu^{N-m})}{\mu^{N}-1}
&
   \frac{1-\mu^{-m}-\mu^{-n}+\mu^{N-m-n}}{\mu^N-1}  \\
    \frac{\mu^{N}-\mu^{N+m}-\mu^{N+n}     + \mu^{m+n}}{\mu^N-1}  &
  \frac{
    (\mu^n-1)(1-\mu^{N-n})}{\mu^{N}-1}
       \end{bmatrix} .$$

It  is the matrix of Veech's hermitian form in the basis $(F_{m,n}^\infty, F_{m,n}^0)$.

\sk

In the lemma below, we use the notations of \S\ref{SS:ConnectionFormulae}.
\begin{lemma} 
\label{L:ConnectionFormulaeTranslation}
\begin{enumerate}
\item 
For any $\tau\in \mathbb H$, one has 
\begin{equation*}
\label{E:Fab}
F_{m,n-N}(\tau)=  B_{m,n}\cdot F_{m,n}(\tau)
 \end{equation*}
with $B_{m,n}=\mu^{-\frac{N}{2}}\cdot  {\rm HT2}_{\rho^{m,n}}$; \mk 
 \item 
 There exists a function $\tau\mapsto \eta_{m,n}(\tau)$ such that  for
 every $\tau\in \mathbb H$, one has: 
 \begin{equation*}
 F_{m-N,n}(\tau)= 
 \eta_{m,n}(\tau)\, 
A_{m,n}  \cdot 
 F_{m,n}(\tau)
 \end{equation*}
with $A_{m,n}=  \mu^{-\frac{N}{2}-n} \cdot 
{\rm VT2}_{\rho^{m,n}}$.
\end{enumerate}
\end{lemma}
\begin{proof}
The relation $ \omega_{m,n-N}=\mu^{-N/2} \omega_{m,n}$ follows easily from  the quasi-periodicity property \eqref{E:ThetaQuasiPeriodicity} of  $\theta$.
  On the other hand, one can 
 write 
 $F_{m,n-N}=\langle \omega_{m,n-N}, \boldsymbol{\gamma}_{m,n-N}\rangle$ with 
${}^t\! \boldsymbol{\gamma}_{m,n-N}=(\boldsymbol{\gamma}^\infty_{m,n-N}, \boldsymbol{\gamma}^0_{m,n-N})$. From \S \ref{SS:ConnectionFormulae}, it comes that 
 $ \boldsymbol{\gamma}_{m,n-N}=  
  {\rm HT2}_{\rho^{m,n}} \cdot   
  \boldsymbol{\gamma}_{m,n}$ 
  where ${\rm HT2}_{\rho^{m,n}}$ stands for the $2\times 2$ matrix ${\rm HT2}_{\rho}$  
 defined in 
 \eqref{E:HTrans2rho22}
 with $\rho=\rho^{m,n}$.  The first connection formula follows immediately. \sk

 From   \eqref{E:ThetaQuasiPeriodicity} again, one deduces that the following relation holds true:   $ \omega_{m-N,n}=\mu^{-\frac{N}{2}-n} \exp({i\pi \tau \alpha_1(1-2m/N)})
  \omega_{m,n}$. 
On the other hand, one has $ \boldsymbol{\gamma}_{m-N,n}= {\rm VT2}_{\rho^{m,n}}  \cdot  \boldsymbol{\gamma}_{m,n}$.  Setting 
$\eta_{m,n}(\tau)=e^{i\pi \tau \alpha_1(1-2m/N)})$, 
the second formula follows. \mk
  \end{proof}

Note that what is actually interesting in the preceding lemma is that the matrices $B_{m,n}$ and $A_{m,n}$ can be explicited.

 Indeed, one has 
$B_{m,n}= (\mu^{-\frac{N}{2}} \beta_{m,n}^{i,j})_{i,j=1}^2$ with 
 \begin{align*}
 \beta_{m,n}^{1,1}=&
   {\mu}^{2\,N-n}-{\mu}^{2\,N+m-n}+{\mu}^{N+m}+{\mu}^{N-m-n}-{\mu}^{N-n} \, ,  
 \\
\beta_{m,n}^{1,2}=&
 {\mu}^{2\,N-m-n}+{\mu}^{2\,N-2\,n}-{\mu}^{2\,N-n}-2\,{\mu}^{N-m-n}+{\mu}^{N}+{\mu}^{N-m-2\,n}-{\mu}^{N-n}\, , 
  \\
 \beta_{m,n}^{2,1}=&
  -{\mu}^{N+2\,m}-{\mu}^{m}+2\,{\mu}^{N+m}+2-{\mu}^{N}-{\mu}^{-m}
  \\
 \mbox{and }\;  \beta_{m,n}^{2,2}= &
  -{\mu}^{N+m}-{\mu}^{N-m}+{\mu}^{N+m-n}-{\mu}^{N-n}+2\,{\mu}^{N}+2\,{\mu}^{-m}+{\mu}^{-n}-{\mu}^{-m-n}-1   
 \, . 
 \end{align*}
 
(Verification: the following relation holds true:  $
 {}^t\!\overline{B_{m,n}}\cdot I\!\!H_{m,n-N}\cdot  {B_{m,n}}= I\!\!H_{m,n}$). 

The matrix $A_{m,n}$ is considerably simpler. One has: 
$$A_{m,n}=\mu^{-\frac{N}{2}-n}
\begin{bmatrix}
1 &   0
  \\
  \mu^n(\mu^{m-N}-1)  & \mu^{n-N}
\end{bmatrix}.$$
(Verification: the following relation holds true:  $
  {}^t\! \overline{A_{m,n}}\cdot I\!\!H_{m-N,n}\cdot {A_{m,n}}= I\!\!H_{m,n}$). 

\subsubsection{\bf Effective computation of the holonomy of $Y_1(N)^{\alpha_1}$}
\label{S:HolonomyY1(N)alpha1}
We now explain how the connection formulae 
that we have just established can be used to compute  the  holonomy group 
$\boldsymbol{\Gamma}_1(N)^{\alpha_1}$ of  $Y_1(N)^{\alpha_1}$ in an effective way. 
\sk

\paragraph{} As a concrete model for this  `hyperbolic conicurve', we choose the quotient of the leaf 
$\mathcal F_{0,1}$ (cut out by $z_2=1/N$  in the Torelli space) by its stabilizer. 
 Actually, we will use the natural isomorphism $\mathbb H\simeq \mathcal F_{0,1}$ to 
 see  $Y_1(N)^{\alpha_1}$ as the standard modular curve $\mathbb H/\Gamma_1(N)$.  From this point of view,  the developing map of 
 the associated $\mathbb C\mathbb H^1$-conifold structure 
is nothing else but  the map $F_{0,1}$ considered above ({\it cf.}\;\eqref{E:VeechFab} with $m=0$ and $n=1$). 
 It follows that the $\mathbb C\mathbb H^1$-holonomy of $Y_1(N)^{\alpha_1}$ 
 can be determined through the connection formulae satisfied by $F_{0,1}$.  
 Note that the corresponding hermitian matrix $I\!\!H_{0,1}$ simplifies and has a relatively simple expression ({\it cf.}\;also \S\ref{SS:normalization}):
 $$
I\!\!H_{0,1}= \frac{1}{2i}
\begin{bmatrix} 
0
&
   \mu^{-1}  \\
  -\mu  &
  \frac{
    (\mu-1)(1-\mu^{N-1})}{\mu^{N}-1}
       \end{bmatrix} .$$

\paragraph{}  Let $g\cdot \tau=(p\tau+q)/(r\tau+s)$ be the image of $\tau\in \mathbb H$ by an element 
 $g=\big[\!\!\!\!\tiny{\begin{tabular}{cc}
 $p$\!\!\!\! &$q$\vspace{-0.1cm}\\
 $r$\!\!\!\!& $s$ \end{tabular} 
}\!\!\!\!\big]$ of $ \Gamma_1(N)$.  From the two lemmata proved above, it comes that there exists  
a matrix $\Lambda'(g)\in {\rm Aut}(I\!\!H_{0,1}) $   
as well as a non-vanishing function $\lambda_g(\tau)$ such that 
\begin{equation}
\label{E:F01(gtau)}
F_{0,1}(g\cdot \tau)=\lambda_g(\tau)\, \Lambda'(g) \cdot F_{0,1}(\tau)\, . 
\end{equation}

Moreover, one can ask  $\Lambda'(g)$  to have coefficients in  $\mathbb R(\mu)$. The map $ \Lambda': g\mapsto \Lambda'(g)$ is a representation of $
\Gamma_1(N)$  in  ${\rm Aut}(I\!\!H_{0,1}) \cap {\rm PSL}_2(\mathbb R(\mu))$. 

Then,  conjugating this representation by the matrix 
$$
Z=\sqrt{2}
\begin{bmatrix}
 \mu^{-1} & - \frac{\mu^{N-1}-1}{\mu^N-1}      \\ 
 0 &   1  
\end{bmatrix} 
$$
({\it cf.}\;\S\ref{SS:normalization}), one gets a normalized representation of $\Gamma_1(N)$ in ${\rm PSL}_2(\mathbb R)$ 
\begin{align}
\label{E:holonomyF01}
\Lambda=\Lambda_N^{\alpha_1} : 
\Gamma_1(N)    & \longrightarrow     {\rm PSL}_2(\mathbb R) \\
g & \longmapsto \Lambda(g)=Z^{-1}\cdot \Lambda'(g)\cdot Z
 \nonumber
\end{align}    
for the considered $\mathbb C\mathbb H^1$-holonomy. It is  a deformation of the standard  inclusion of the projectivization $\boldsymbol{\Gamma}_1(N)$ of $\Gamma_1(N)$ as a subgroup of ${\rm PSL}_2(\mathbb Z)\subset {\rm PSL}_2(\mathbb R)$ 
which is  analytic with respect to the parameter $\alpha_1\in ]0,1[$. \mk 

\paragraph{} We now explain how to compute $\Lambda(g)$ explicitly for 
$g=\big[\!\!\!\!\tiny{\begin{tabular}{cc}
 $p$\!\!\!\! &$q$\vspace{-0.1cm}\\
 $m$\!\!\!\!& $n$ \end{tabular}} \!\!\!\!\big]\in \Gamma_1(N) $. 
 
 Writing $g$ as a word in $T=\big[\!\!\!\!\tiny{\begin{tabular}{cc}
 $1$\!\!\!\! &$1$\vspace{-0.1cm}\\
 $0$\!\!\!\!& $1$ \end{tabular}} \!\!\!\!\big]$ 
 and $S=\big[\!\!\!\!\tiny{\begin{tabular}{cc}
 $0$\!\!\!\! &$1$\vspace{-0.1cm}\\
 $-1$\!\!\!\!& $0$ \end{tabular} 
}\!\!\!\!\big]$, one can use Lemma \ref{L:modularTandS} to get that 
\begin{equation}
\label{E:F01(gtau)Fcd}
F_{0,1}\left(\frac{a \tau+b}{c\tau+d}\right)= M(g)\cdot F_{m,n}(\tau)
\end{equation}
 where $M(g)$ is a product (which can be made explicit) of the matrices $T_{a'}$ and $S_{a''}$ 
 (see formulae \eqref{E:Tr} and \eqref{E:Sr} respectively) 
 for some $a'$ and $a''$ easy to determine. 
 
 Next, the fact that $g$ belongs to $\Gamma_1(N) $ implies in particular that $m=m'N$ and $n=1+n'N$ for some integers $m',n'$.  One can then use Lemma \ref{L:ConnectionFormulaeTranslation}  and construct a function $\lambda_g(\tau)$ and  a matrix $N(g)$ which is a product of $m'$ (resp.\;$n'$)  matrices of the type ${\rm VT2}_{\hat m,\hat n}$ (resp.\;${\rm HT2}_{\tilde m,\tilde n}$), for some $(\hat m,\hat n)$'s and 
 $(\tilde m,\tilde n)$'s,  which are easy to make explicit, 
 such that 
 \begin{equation}
 \label{E:Fcd}
 F_{m,n}(\tau)=\lambda_g(\tau) N(g)\cdot F_{0,1}(\tau)\, .  
 \end{equation}
 
 Then setting $\Lambda'(g)=M(g)\cdot N(g)$,  one gets \eqref{E:F01(gtau)}   from  
\eqref{E:F01(gtau)Fcd} and \eqref{E:Fcd}. 
\sk 

 Below, we illustrate the method just described by computing explicitly 
  the image by $\Lambda$ of two simple elements of $\Gamma_1(N)$.
\mk 

\begin{rem}
\label{Rem:remrem}
{\rm We have described above  an  algorithmic  method to compute $\Lambda (g)$ when $g$ is given. It would be interesting to have a closed formula for $\Lambda(g)$ in terms of the coefficients of $g$. Such  formulae have been obtained by Graf \cite{Graf1908a,Graf1908b} and more recently (and independently) by Watanabe \cite{Watanabe2007,Watanabe2014} in the very similar case of the {\it `complete elliptic hypergeometric integrals'} which are the 
 hypergeometric integrals associated to $\Gamma(2)$ of the following form
$$
\int_{\boldsymbol{\gamma}} \theta(u,\tau)^{\beta_0}
\theta_1(u,\tau)^{\beta_1}
\theta_2(u,\tau)^{\beta_2}
\theta_3(u,\tau)^{\beta_3} du\, , 
$$
where $\boldsymbol{\gamma}$ stands for a twisted cycle supported in $E_\tau\setminus E_\tau[2]$ (the $\beta_i$'s being fixed real parameters summing up to 0). }
\end{rem}
\mk

\paragraph{\bf Two explicit computations ($N$ arbitrary)} 
\label{Parag:2ExplicitComputations}
We consider the two following elements of $\Gamma_1(N)$: 
$$
T=\begin{bmatrix}
1 & 1\\
0 & 1
\end{bmatrix} 
\qquad 
\mbox{ and } \qquad 
U_N=\begin{bmatrix}
1 & 0\\
-N & 1
\end{bmatrix} 
\, . 
$$
We want to compute  their respective images by $\Lambda$ in 
 ${\rm SL}_2(\mathbb R)$.
\sk 

The case of $T$ is very easy to deal with.  From the first point of Lemma \ref{L:modularTandS}, one has 
$$\Lambda'(T)=\begin{bmatrix}
1 & \rho_\infty^{0,1}/\rho_{0}^{0,1}\\
0 & 1
\end{bmatrix}= \begin{bmatrix}
1 & \mu^{-1}\\
0 & 1
\end{bmatrix}\,.
$$
After conjugation by $Z$, one gets 
$$
\Lambda(T)= 
Z^{-1} \cdot \Lambda'(T) \cdot Z=
\begin{bmatrix}
1 & 1\\
 0 & 1
\end{bmatrix}
\in {\rm PSL}_2(\mathbb R)\,.$$

To deal with $U_N$, one begins by writing
$$
U_N= S\cdot T^N\cdot S^{-1}.
$$

In what follows, we write $=_\tau$  to designate an equality which holds true up to multiplication by a function depending on $\tau$. 

 Then, for any $\tau\in \mathbb H\simeq \mathcal F_{0,1}$,  setting $\tau'=-1/\tau$, one has 
\begin{align*}
F_{0,1}\left( \frac{\tau}{1-N\tau}
\right)=
F_{0,1}\left( \frac{-1}{\tau'+N}
\right)
=_\tau  &\,    S_{0,1}\cdot F_{-1,0}(\tau'+N) \\
=_\tau &  S_{0,1}\cdot T_{-1,0 }\cdot F_{-1,-1}(\tau'+N-1)\\
=_\tau &  S_{0,1}\cdot T_{-1,0 }\cdots T_{-1,-N+1}\cdot 
F_{-1,-N}(\tau')\\
=_\tau &  S_{0,1}\cdot T_{-1,0 }\cdots T_{-1,N-1}\cdot 
B_{-1,0}\cdot 
F_{-1,0}(\tau')\\
=_\tau &  S_{0,1}\cdot T_{-1,0 }\cdots T_{-1,N-1}\cdot 
B_{-1,0}\cdot  \big(S_{1,0} \big)^{-1}   \cdot 
F_{0,1}(\tau)\, .
\end{align*}

It follows that 
\begin{align*}
\Lambda'(U_N)=   & \, S_{0,1}\cdot T_{-1,0 }\cdots T_{1,-N+1}\cdot 
B_{-1,0}\cdot  \big(S_{0,1}\big)^{-1} \\
=   & \, S_{0,1} \cdot  
\begin{bmatrix}
1 &  \mu\\
0 & 1
\end{bmatrix}
\begin{bmatrix}
1 &  \mu^{2}\\
0 & 1
\end{bmatrix}\cdots 
\begin{bmatrix}
1 &  \mu^{N}\\
0 & 1
\end{bmatrix}\cdot B_{-1,0}\cdot  \big(S_{0,1}\big)^{-1}  \\ 
=   & \, 
S_{0,1}
\cdot  
\begin{bmatrix}
1 & \frac{ \mu (1-\mu^{N})}{(1-\mu)}
 \\
0 & 1
\end{bmatrix}\cdot 
B_{-1,0}\cdot  \big(S_{0,1}\big)^{-1}.
 \end{align*}

Since 
\begin{align*}
S_{0,1}= & \, \begin{bmatrix}
1-\mu^{-1} & 1\\
-1 & 0
\end{bmatrix}\\
\mbox{and }\quad 
B_{-1,0}=& \, \mu^{1-N/2}
\begin{bmatrix}
\mu^{2N-1}-\mu^{2N-2}+\mu^{N-2}+\mu^N-\mu^{N-1}
 &  \mu^{2N}- \mu^N
\\
-\frac{(\mu-1)^2(\mu^{N-1}+1)}{\mu^2}
 & \mu^{N-1}-\mu^N+1
\end{bmatrix}, 
\end{align*}
an explicit computation gives 
\begin{align*}
\Lambda'(U_N)=
 & \,
 \mu^{\frac{N}{2}} \cdot 
 \begin{bmatrix}
1  &   \frac{(1-\mu)^2}{\mu^2}  \\
  \frac{\mu^2(1-\mu^{-N})}{1-\mu}
 & 1-\mu-\mu^{-N}+2\mu^{1-N}
 \end{bmatrix}\, .\mk 
\end{align*}
\big(Remark: for $\alpha_1\rightarrow 0$, one has $\mu\rightarrow 1$ hence  $
\Lambda'(U_N)\rightarrow \tiny{\big[ \!\! \!\!\begin{tabular}{cc}
1 & 0 \vspace{-0.1cm}
\\
$-N$  \!\! \!\!  \!\! \!\! \!\!& 1
\end{tabular} \!\! \!\!\big]}$,  as expected\big). 

One deduces the following explicit expression for 
$\Lambda(U_N)=Z^{-1} \Lambda'(U_N)  Z$: 
\begin{align*}
\Lambda(U_N)
= 
\mu^{\frac{N}{2}} 
\begin{bmatrix} {\frac {1+ {\mu}^{2}
-{\mu}^{N}-{\mu}^{2-N}}{ \left( \mu-1 \right)  \left( {\mu}^{N}-1 \right) }}&
\frac{ \lambda(U_N)}{\left( \mu-1 \right)  \left( {\mu}^{N}-1 \right) ^{2}\mu}
\vspace{0.15cm}
\\
{\frac {\mu\, \left( {\mu}^{-N}-1 \right) }{\mu-1}}&-{\frac {-3\,{\mu}^{N+1}+6\,\mu+{\mu}^{2-N}+{\mu}^{2+N}+{\mu}^{N}-2-3\,{\mu}^{1-N}+{\mu}^{-N}-2\,{\mu}^{2}}{ \left( \mu-1 \right)  \left( {\mu}^{N}-1 \right) }}
\end{bmatrix}
\end{align*}
with 
\begin{align*}
\lambda(U_N)= 
  -  1 -&5\,{\mu}^{N+1}-2  \,{\mu}^{4}-{\mu}^{2+2N}-  {\mu}^{2N}+2\,{\mu}^{N}-2\,{\mu}^{3+N} +{\mu}^{2+N}  \\ +&{\mu}^{4-N}
  -{\mu}^{2}+{\mu}^{N+4}+{\mu}^{2-N}-3\,{\mu}^{3-N}+2\,\mu+3\,{\mu}^{1+2N}+5\,{\mu}^{3}\, .
 \end{align*}
 \big(Remark :  one has $
 \Lambda^*(U_N)=
     \mu^{-\frac{1}{2}}\Lambda(U_N)\in {\rm SL}_2(\mathbb R)$\big). 
 \sk
 
 A necessary condition for  the $\mathbb C\mathbb H^1$-holonomy of $Y_1(N)^{\alpha_1}$ to be   discrete is that  $ \Lambda^*(U_N)$ together with the fixed parabolic element 
 $\Lambda(T)= \tiny{\big[ \!\! \!\!\begin{tabular}{cc}
1  \!\! \!\!  \!\! \!\! \!\! & \!\! \!\!  \!\! \!\! \!\!  1 \vspace{-0.1cm}
\\
0   \!\! \!\!  \!\! \!\! \!\!& \!\! \!\!  \!\! \!\! \!\!\!  1
\end{tabular} \!\! \!\!\big]}$ generates a discrete subgroup of ${\rm PSL}_2(\mathbb R)$. 
There are many papers dealing  with this problem. For instance, in \cite{GilmanMaskit}, Gilman and Maskit give an explicit algorithm to answer this question. However, if this algorithm  can be used quite effectively to solve any given explicit case, the complexity of $ \Lambda^*(U_N)$ 
seems to make its use too involved to describe
precisely the set of parameters $\alpha_1\in ]0,1[$ and $N\in \mathbb N_{\geq 2}$  so  that $\langle  \Lambda^*(U_N) , \tiny{\big[ \!\! \!\!\begin{tabular}{cc}
1  \!\! \!\!  \!\! \!\! \!\! & \!\! \!\!  \!\! \!\! \!\!  1 \vspace{-0.1cm}
\\
0   \!\! \!\!  \!\! \!\! \!\!& \!\! \!\!  \!\! \!\! \!\!\!  1
\end{tabular} \!\! \!\!\big]}  \rangle$ be a lattice in ${\rm SL}_2(\mathbb R)$. 


\subsection{\bf Volumes}
\label{S:Volumes}
We recall that $Y_1(N)^{\alpha_1}$ stands for the modular curve $Y_1(N)$ endowed with the pull-back by Veech's map 
of the standard hyperbolic structure of $\mathbb C\mathbb H^1$. In particular, the curvature of 
the metric which is considered on $Y_1(N)^{\alpha_1}$ is constant and equal to -1. 
We denote by ${\rm Vol}(Y_1(N)^{\alpha_1})$ the corresponding volume (the `area' would be more accurate) of $Y_1(N)^{\alpha_1}$. 
\sk 

\subsubsection{}
According to  the version for compact hyperbolic surfaces with conical singularities  of Gau{\ss}-Bonnet's Theorem  ({\it cf.}\;Theorem A.1 in Appendix A) and in view of 
our results in \S \ref{SS:MetricCompletionOfY1(N)}, one has 
\begin{equation}
\label{E=Vol1NAlpha1}
{\rm Vol}\big( Y_1(N)^{\alpha_1}\big)=2\pi \Bigg[ 2g_1(N)-2
+ \sum_{\mathfrak c\in C_1(N)}  \bigg(  1-\frac{\theta_N(\mathfrak c)}{2\pi}   \bigg)    \Bigg]\; \footnotemark
\end{equation}
\footnotetext{Actually, this formula is only valid when $N\geq 4$. Indeed,  
$Y_1(N)$  has an orbifold point when $N=2,3$ and it   has  to be taken into account when computing ${\rm Vol}(Y_1(N)^{\alpha_1})$ in these two cases.}
 where
\begin{itemize}
\item   $g_1(N)$ stands for the genus of the compactified modular curve $X_1(N)$;  
\sk
\item  for any $\mathfrak c\in C_1(N)$,  $\theta_N(\mathfrak c)$ denotes  the conifold angle of $X_1(N)^{\alpha_1}$ 
at $\mathfrak c$.\sk
\end{itemize}

  Since $\theta_N(\mathfrak c)$ depends linearly on $\alpha_1$ for every $\mathfrak c$ ({\it cf.}\;\eqref{E:ConifoldAnglesY1(N)}), it follows that 
$$
{\rm Vol}\big( Y_1(N)^{\alpha_1}\big)=A(N)+B(N)\, \alpha_1
$$ 
for two arithmetic constants $A(N),B(N)$ depending only on $N$. 
\sk 

We recall that the following  closed formula 
$$
g_1(N)= 
g\big(X_1(N)\big)=1+\frac{N^2}{24}\prod_{p\lvert N}\big(1-p^{-2}\big)-\frac{1}{4}\sum_{0<d\lvert N}\phi(d)\phi(N/d)
$$ holds true 
for 
any $N\geq 5$, with $g_1(M)=0$ for $M=1,\ldots,4$ (see \cite{KimKoo}).   

On the other hand, 
we are not aware of any general closed formula, in terms of $N$,  for a set of representatives $[-a_i/c_i]$ with $i=1,\ldots,\lvert C_1(N)\lvert$ of the set of cusps $C_1(N)$ of $Y_1(N)$.  Consequently, obtaining  closed formulae
for $A(N)$ and $B(N)$ in terms of $N$ does not seem easy in general. 
However,  there are algorithmic methods  determining explicitly such a set of representatives. 
 Then determining  ${\rm Vol}(Y_1(N)^{\alpha_1})$ reduces to a computational task  once  $N$ has been given.

\subsubsection{} Since the two values $A(N)$ and $B(N)$ depend heavily  on the arithmetic properties of $N$, one can expect to be able to say more about them when $N$ is simple from this point of view, for instance when $N$ is prime. 
\sk 

Let $p$ be a prime number bigger than or equal to 5.  Then 
$$g_1(p)=\frac{1}{24}(p-5)(p-7)$$ and there is an explicit description 
of the conifold points and of the associated conifold angles of $X_1(p)^{\alpha_1}$ (see \S\ref{SS:CaseY1(p)ForpPrime}). 

In the case under scrutiny, formula \eqref{E=Vol1NAlpha1} specializes into 
\begin{align*}
{\rm Vol}\big( Y_1(p)^{\alpha_1}\big)
= & \, 2\pi\Big(2g_1(p)-2+(p-1)   \Big)-
2\pi   \alpha_1  \sum_{k=1}^{(p-1)/2} k\Big(1-\frac{k}{p}\Big) 
\end{align*}
and after a simple computation, one obtains the nice formula
\begin{equation}
\label{E:HyperbolicVolumeY1(p)Alpha1}
{\rm Vol}\big( Y_1(p)^{\alpha_1}\big)= \frac{ \pi  }{6}\big( p^2-1\big) \left(
1-\alpha_1 \right) \, . 
\end{equation}

\subsubsection{}   Even if it only concerns  the  algebraic leaves $Y_1(p)^{\alpha_1}$ 
of Veech's foliation  associated to prime numbers, the preceding formula can be used 
to determine 
$$
{\rm Vol}^{\alpha_1}\big(   \mathscr M_{1,2} \big) =\int_{\mathscr M_{1,2}}\Omega^{\alpha_1}
$$
  ({\it cf.}  \S\ref{SS:InAddition} and \S\ref{SS:Intro-Volume} for a few words about Veech's volume form $\Omega^{\alpha_1}$).  \sk 

We first review more carefully than in the Introduction the construction of Veech's volume form 
$\Omega^{\alpha_1}$  then present an easy proof of  Theorem \ref{T:VolumeIsFinite}. The latter  relies  on the following quite intuitive fact that  the euclidean volume ({\it i.e.} the area) of a nice open set $U \subset \mathbb{R}^{2}$ can be well approximated by counting the number of elements of $U \cap (1/p) \mathbb{Z}^2$, as soon as $p$ is a sufficiently big prime number. 
 For instance,   for any  parallelogram
  $P$ of the euclidean plane $\mathbb R^2$, 
  if $d \sigma=dr_0\wedge dr_{\infty}$ stands for the standard euclidean volume form on $\mathbb R^2$, then 
\begin{equation}
\label{F:EuclideanArea}
 \mathrm{Area}(P) =\int_P d \sigma
 =  \lim_{p \rightarrow +\infty}{{p^{-2}} \# \big( U \cap (1/p) \mathbb{Z}^2 \big )}. 
\end{equation}

 \paragraph{} 
Since we assume that $\alpha_1\in ]0,1[$ is fixed,  when we refer to the volume of a subset of $\mathscr{M}_{1,2}$ below,  it is always relatively to Veech's volume form $\Omega^{\alpha_1}$.
\sk 

 Let $\mathscr M_{1,2}^*$ stand for the set of non-orbifold points of $\mathscr M_{1,2}$. It is a dense open-subset whose complementary set has zero measure.  Consequently one can restrict to  consider only the volume of $\mathscr M_{1,2}^*$. This is what we do from now on.\mk

 Let $[m^*]$ be a point in $\mathscr M_{1,2}^*$ and consider a lift $m^*=(\tau^*,z^*)$ in $\mathcal T\!\!\!{or}_{1,2}$ over it.  One defines the germ of  real-analytic map $V^{\alpha_1}: ( \mathcal T\!\!\!{or}_{1,2},m^*) \rightarrow\mathbb H$ at $m$ by setting 
 $$
 V^{\alpha_1}(\tau,z) = V_{\xi^{\alpha_1}(\tau,z)}^{\alpha_1}(\tau)
 $$
  for any $(\tau,z)$ sufficiently close to $m^*$, where $\xi^{\alpha_1}$ is the map considered in Proposition  \ref{P:Main} and   $V_{\xi^{\alpha_1}(\tau,z)}^{\alpha_1}$ stands for the map \eqref{E:V-alpha1-a} with $a=\xi^{\alpha_1}(\tau,z)$. (Note that such a  $V^{\alpha_1}$ is not canonically defined but this will not cause any problem hence we will not dwell on this in what follows). 
  \sk

 The pull-back  by $V^{\alpha_1}$ of the  volume form
 $d\zeta\wedge d\overline{\zeta}/(2i \lvert\Im{\rm m}(\zeta)\lvert ^2)$ 
  inducing the standard hyperbolic structure on ${\mathbb H}$,  is a (germ of)  smooth 
  $(1,1)$-form at $m^*$, denoted by $\omega^{\alpha_1}$,  whose restriction along any leaf of Veech's foliation $\mathcal F^{\alpha_1}$ close to $m^*$ locally   induces Veech's hyperbolic structure of this leaf (as it follows from the second point of Proposition \ref{P:InH}). 
  \sk 
  
  Setting 
  $v^*=V^{\alpha_1}(m^*)$ and $r^*=\Xi(m^*)$, 
     one obtains a germ of real analytic diffeomorphism 
  \begin{align*}
 \varphi_{m^*}^{\alpha_1} =  V^{\alpha_1}\times \Xi : \big( \mathcal T\!\!\!{or}_{1,2},
  m^*
 \big) 
   \longrightarrow \Big( \mathbb H\times \big(\mathbb R^2\setminus \mathbb Z^2\big), \big(v^*,r^*\big)\Big)
    \end{align*}
  such that the pull-back by it of the horizontal foliation  
   on $\mathbb H\times \mathbb R^2\setminus \mathbb Z^2$, with all the horizontal upper-half planes  are endowed with the standard hyperbolic structure, is exactly (the germ at $m^*$ of) Veech's foliation $\mathcal F^{\alpha_1}$ with Veech's hyperbolic structures on the leaves. (Beware that  the germ $\varphi_{m^*}^{\alpha_1}$ is quite distinct from  the one of the global 
diffeomorphism    \eqref{E:Pi-Xi} at $m^*$).   
  \sk 
  
For $\epsilon>0$,  let  $D^\epsilon$ stand for the hyperbolic disk of radius $\epsilon $ centered at $v^*$ in $\mathbb H$ and denote by $S^\epsilon$ the euclidean square $r^*+]-\epsilon,\epsilon[^2\subset \mathbb R^2$. 
 If $\epsilon$ is chosen sufficiently small, $U_{m^*}^\epsilon= (\varphi_{m^*}^{\alpha_1})^{-1}(D^\epsilon\times S^\epsilon)\subset  
 \mathcal T\!\!\!{or}_{1,2} $ is well defined and $(U_{m^*}^\epsilon, \varphi_{m^*}^{\alpha_1})$ is a foliated chart  for Veech's foliation $\mathcal F^{\alpha_1}$ at $m^*$.  Since $[m^*]$ is not an orbifold point, it induces a foliated chart at $[m^*]$ for Veech's foliation 
 $\mathscr F^{\alpha_1}$ 
on the moduli space $\mathscr M_{1,2}$, which will be denoted by the same notation.
\sk

 Up to $\varphi_{m^*}^{\alpha_1}$, for any prime $p$, we have  
 $  U_{m^*}^\epsilon \cap  Y_1(p)^{\alpha^1}  \simeq  D^\epsilon \times   (S^\epsilon  \cap ({1}/{p}) \mathbb{Z}^2)$. Therefore, from \eqref{F:EuclideanArea}, it comes that the volume of $U_{m^*}^\epsilon$ is given by
$$ \mathrm{Vol}^{\alpha_1}\big(U_{m^*}^\epsilon\big) = \int_{U_{m^*}^\epsilon} \Omega^{\alpha_1}
=
\lim_{
\substack{p\rightarrow +\infty\\ 
 \tiny{p \; \mbox{\it prime }}}}\; 
 {\frac{1}{p^2}{\nu_p^{\alpha_1}\Big( U_{m^*}^\epsilon \cap 
Y_1(p)^{\alpha_1}\Big)}}$$
where $\nu^{\alpha_1}_p$ stands for the volume ({\it i.e.} the area) on $Y_1(p)^{\alpha_1}$, the latter being endowed with Veech's hyperbolic structure. 
  \mk

 \paragraph{} 
  Now let $(U_i)_{i\in \mathbb{N}}$ be a family of open  subsets  of $\mathscr{M}_{1,2}^*$ such that 
\begin{itemize}
\item each $U_i$ is the domain $U_{m_i^*}^{\epsilon_i}$ of a foliated chart  as above;


\item one has  $\bigcup_{i \in \mathbb{N}}{\overline{U_i}}=\mathscr{M}_{1,2}$;
\sk 

\item the $U_i$ are pairwise disjoint.

\end{itemize}
(We let the reader verify that such a family of open subsets indeed exists).
\sk 

From these assumptions and considering what has been established in the  preceding paragraph, we get that 
\begin{equation}
\label{E:lmo}
 \mathrm{Vol}^{\alpha_1}\big(\mathscr{M}_{1,2}\big) 
=\sum_{i\in \mathbb{N}} \mathrm{Vol}^{\alpha_1}\big(U_i\big)  
= \sum_{i\in \mathbb{N}}  \; 
\lim_{p\rightarrow + \infty}\; 
\frac{1}{p^2} \; \nu_p^{\alpha_1}\big(  U_i\cap Y_1(p)^{\alpha_1} \big)
 \;,     
\end{equation}
where (here and as  in any formula below) $p$ ranges amongst prime numbers.\sk 

Our main concern now consists in inverting the sum and the limit in the last term of the preceding equality. 

To this end, for any $I\in \mathbb N$, one considers  the partial sum 
$$
 \sum_{i=0}^I \mathrm{Vol}^{\alpha_1}\big(U_i\big)
 =
 \sum_{i=0}^I\; 
\lim_{p\rightarrow + \infty}
\; 
 \frac{1}{p^2} \, \nu^{\alpha_1}_p\big( U_i  \cap      Y_1(p)^{\alpha_1}    \big).$$

The summation being finite, 
we can invert the sum and the limit to obtain 
$$
 \sum_{i=1}^I \mathrm{Vol}^{\alpha_1}\big(U_i\big)
 =
\lim_{p\rightarrow + \infty}\; 
\frac{1}{p^2} \nu_p^{\alpha_1}\Big( \big(\cup _{i=1}^I{U_i}\big)  \cap 
Y_1(p)^{\alpha_1}\Big)\, . $$

From \eqref{E:HyperbolicVolumeY1(p)Alpha1}, it comes that  for any prime $p$, one has 
$$\nu_p^{\alpha_1}\Big( \big(\cup _{i=1}^I{U_i}\big)  \cap 
Y_1(p)^{\alpha_1}\Big)\leq \nu_p^{\alpha_1}\big(Y_1(p)^{\alpha_1}\big)=
\frac{\pi}{6}(1-\alpha_1)(p^2-1)$$ 
from which it follows that 
$$ 
 \sum_{i=1}^I \mathrm{Vol}^{\alpha_1}\big(U_i\big)
 \leq  
\lim_{p\rightarrow + \infty}
\;  \frac{1}{p^2}\cdot 
\frac{\pi}{6}(1-\alpha_1)(p^2-1)  =\frac{\pi}{6}(1-\alpha_1).$$

The integer $I$ being arbitrary, we conclude that 
$$  \mathrm{Vol}^{\alpha_1}\big(\mathscr{M}_{1,2}\big) \leq \frac{\pi}{6}(1-\alpha_1).$$

Since the volume  of $\mathscr{M}_{1,2}$ is finite, and because the 
 quantities  $p^{-2} \nu_{p}^{\alpha_1}(Y_1(p)^{\alpha_1} )$ 
 are uniformly bounded (precisely by the RHS of the preceding inequality), it is straightforward  that one can invert the sum and the limit in \eqref{E:lmo} and get that 
$$
{\rm Vol}^{\alpha_1}\big(   \mathscr M_{1,2} \big) =\frac{\pi}{6}\big(1-\alpha_1\big).
$$

\newpage

\section*{\bf Appendix A :   1-dimensional complex hyperbolic conifolds}
We define  and state a few basic results concerning $\mathbb C\mathbb H^1$-conifolds below.  The general notion of conifolds is rather abstract (see \cite{Thurston,McMullen} or \cite[Appendix\,B]{GP}) but  greatly simplifies  in the case under scrutiny.   
\sk 

We denote by $\mathbb D
$ the unit disk in the complex plane. As the upper half-plane $\mathbb H
 $, it is a model of the complex hyperbolic space $\mathbb C\mathbb H^1$.

\subsection*{\bf A.1. Basics} 
The map $f: \mathbb H\rightarrow \mathbb D^*,\, w\mapsto e^{iw}$ is (a model  of)  the universal cover of the punctured disk $\mathbb D^*=\mathbb D\setminus \{0\}$.  
We denote by $\widetilde {\mathbb D}^*$ the upper-half plane endowed with the pull-back by $f$ of the standard hyperbolic structure on $\mathbb D$. 

\subsection*{\bf A.1.1. $\boldsymbol{\mathbb C\mathbb H^1}$-cones} For any 
$\theta\in ]0,+\infty[$, the translation $t_\theta: w\mapsto w+\theta$ leaves  invariant the complex hyperbolic structure of $\widetilde {\mathbb D}^*$ (since it is a lift of the rotation $z\mapsto e^{i\theta}z$ which is an automorhism of $\mathbb D$  fixing the origin).  It follows that the complex hyperbolic structure of  $\widetilde {\mathbb D}^*$ factors through the action of $t_\theta$.  The quotient  $\mathfrak C_\theta^*= \widetilde {\mathbb D}^* /\langle t_\theta\rangle$ carries an hyperbolic structure which is not metrically complete.  Its metric completion, denoted by $\mathfrak C_\theta$,  is obtained by adjoining  only one point to $\mathfrak C_\theta^*$, called the {\bf apex} and denoted by 0.     
By definition, 
$\mathfrak C_\theta$ (resp.\;$\mathfrak C_\theta^*$) is the  {\bf (punctured) $\mathbb C\mathbb H^1$-cone of angle $\theta$}. 

 It will be convenient to also consider the case when $\theta=0$.  By convention, we define $\mathfrak C_0^*$ as $\mathbb H/{\langle \tau\mapsto \tau+1\rangle}$ when $\mathbb H$ is endowed with its standard hyperbolic structure.  It is nothing else but $\mathbb D^*$ but now endowed with the hyperbolic structure given by the uniformization (and by restriction  from the standard one of $\mathbb D$).  
   Note that $\mathfrak C_0^*$  is nothing else than a neighborhood of what is classically  called a {\bf cusp} in the   theory of Riemann surfaces. 
\sk

As is well-known, a $\mathbb C\mathbb H^1$-structure on an orientable smooth surface $\Sigma$  can be seen geometrically  as a (class for a certain equivalence relation of a) pair $(D,\mu)$ where $\mu: \pi_1(\Sigma)\rightarrow {\rm Aut}(\mathbb C\mathbb H^1)$ is a representation (the {\bf holonomy representation}) and $D:\widetilde \Sigma \rightarrow  \mathbb C\mathbb H^1$ a $\mu$-equivariant tale map (the {\bf developing map}).  
With this formalism, it is easy to give concrete models of the $\mathbb C\mathbb H^1$-cones defined 
 above.  

For any $\theta>0$,  one defines $D_\theta(z)=z^\theta$, 
and $\mu_\theta$ stands for the character associating $e^{i\theta}$ to the class of a small positively oriented circle  around the origin in $\mathbb D$. 
We see $D_\theta$ as a multivalued map from $\mathbb D$ to itself. Its monodromy $\mu_\theta$ leaves  the standard hyperbolic structure of $\mathbb D$ invariant. Consequently, the pair  $(D_\theta,\mu_\theta)$ defines a 
$\mathbb C\mathbb H^1$-structure
 on $\mathbb D^*$ and one verifies promptly that it identifies with the 
 one of the punctured  $ \mathbb C\mathbb H^1$-cone $\mathfrak C_\theta^*$.   To define $\mathfrak C_0^*$ this way, one can take $D_0(z)=\log(z)/(2i\pi)$ as a developing map and as  holonomy representation, we take the parabolic element $\mu_0: x\mapsto x+1$ of 
 the automorphism group of ${\rm Im}(D_0)=\mathbb H$ ($\mu_0$ is nothing else but the monodromy of $D_0$).

By computing the pull-backs of the standard hyperbolic metric on their target space by the elementary developing maps considered just above,  one gets the following  
 characterization of the  $\mathbb C\mathbb H^1$-cones in terms of the corresponding hyperbolic metrics: $\mathfrak C_0^*$ and $\mathfrak C_\theta^*$ for any $\theta>0$  can respectively be defined as the hyperbolic structure on $\mathbb D^*$ associated to the metrics 
$$
d\!s_0=\frac{\lvert dz\lvert}{\lvert z\lvert
\log\lvert z\lvert} 
\qquad 
\mbox{ and }
\qquad 
d\!s_\theta= 2\theta
 \frac{ \lvert z\lvert^{\theta-1}  \lvert dz\lvert}{1-\lvert z\lvert^{2\theta}}  \, . 
$$

Note that for any positive integer  $k$,   $\mathfrak C_{2\pi/k}$ is the orbifold quotient of $\mathbb D$ by $z\mapsto e^{2i \pi/k}z$. In particular, $\mathfrak C_{2\pi}$ and $\mathfrak C_{2\pi}^*$ are nothing else than $\mathbb D$ and $\mathbb D^*$ respectively,  hence most of the time it will be assumed that $\theta\neq 2\pi$.

One verifies that among all the $\boldsymbol{\mathbb C\mathbb H^1}$-cones, the one of angle 0 is characterized geometrically by the fact that the associated holonomy is parabolic, or metrically, by the fact that $\mathfrak C_0^*$ is complete.  
 Finally, we mention that the area of the $\mathfrak C_\theta^*$ is locally finite at the apex 0 for any $\theta\geq 0$. 
\sk


\subsection*{\bf A.1.2.   $\boldsymbol{\mathbb C\mathbb H^1}$- conifold structures} Let $S$ be a smooth oriented  surface and let $(s_i)_{i=1}^n$ be a $n$-uplet of pairwise distinct points on it. One sets $S^*=S\setminus \{s_i\}$.  A $\mathbb C\mathbb H^1$-structure on $S^*$    naturally induces a  conformal structure  or, equivalently,  a structure of Riemann surface 
  on $S^*$.  When endowed with this structure, we will denote $S^*$ by $X^*$ and $s_i$ by $x_i$ for every $i=1,\ldots,n$.
  \sk 
    
  We will say that the 
hyperbolic structure  on $X^*$   {\bf extends as} (or just {\bf  is} for short) {\bf  a   $\mathbb C\mathbb H^1$- conifold (structure)} on $X$ if, for every puncture  $x_i$,  
 there exists $\theta_i\geq 0$ and a germ of pointed biholomorphism 
 $(X^*,x_i)\simeq (\mathfrak C_\theta^*,0)$ which is compatible with the $\mathbb C\mathbb H^1$-structures on the source and on the target.  In this case, each puncture $x_i$ will be called a {\bf conifold point}  and  $\theta_i$  will be the associated  {\bf conifold} (or {\bf cone}) {\bf angle}.      Remark  that our definition differs from the classical one since we allow  some 
 cone angles $\theta_i$ to vanish.  The punctures with conifold angle  0 are just  cusps of $X$. \sk 

Note that when the considered hyperbolic structure on $X^*$ is conifold then  its metric completion (for the distance induced by the $\mathbb C\mathbb H^1$-structure) is obtained by adding to $X^*$ the set of conifold points of positive cone angles. \sk 

An important question is the existence (and possibly the unicity) of such conifold structures when $X$ is assumed to be compact.  
 In this case, as soon as the genus $g$ of $X$ and the number $n$ of punctures verify $2g-2+n>0$, it follows from {\bf Poincar-Koebe uniformization theorem}  
that there exists a Fuchsian group $\Gamma$ 
 such that $\mathbb H/\Gamma\simeq X^*$ as Riemann surfaces with cusps (and $\Gamma$ is essentially unique).  Actually, this theorem generalizes to any  $\mathbb C\mathbb H^1$-orbifold structures on $X$ (see {\it e.g.}\;Theorem IV.9.12 in \cite{FarkasKra} for a precise statement). It implies in particular that such a structure, when it exists, is uniquely characterized by the conformal type of $X^*$ and by the cone angles at the orbifold points.  \sk

It turns out  that the preceding corollary of Poincar's uniformization theorem generalizes to the class of $\mathbb C\mathbb H^1$-conifolds.  Indeed, long before  Troyanov proved his theorem  (recalled in \S\ref{S:VeechIntro}) concerning the existence and the unicity 
of a flat structure with conical singularities on a surface (we could call such a structure a `{\it $\mathbb E^2$-conifold structure}'), Picard had established the corresponding result for compact complex hyperbolic conifolds of dimension 1: 
\sk 

\noindent{\bf Theorem A.1.2.}
{\it 
Assume that $2g-2+n>0$ and let $(X, (x_i)_{i=1}^n)$ be a compact $n$-marked Riemann surface of genus $g$.  Let  $(\theta_i)_{i=1}^n\in [0, \infty[^n$ be an angle datum. 
\begin{enumerate}
\item The following two assertions are equivalent:\sk  
\begin{itemize}
\item there exists a hyperbolic conifold structure on $X$ with a conical singularity of  angle $\theta_i$  at $x_i$, for $i=1,\ldots,n$;\sk 
\item the $\theta_i$'s are such that the  following inequality is satisfied: 
\begin{equation}
\label{E:AreaCH1structure}
2\pi\big(  2g-2+n\big) - \sum_{i=1}^n \theta_i>0\, . 
\end{equation}
\end{itemize}
\sk
\item  When the two preceding conditions are satisfied, then the corresponding conifold  hyperbolic 
metric on $X$ is unique (if normalized in such a way that its curvature be -1) and 
 the area of $X$ is equal to the LHS 
 of 
      \eqref{E:AreaCH1structure}.
\end{enumerate}}

Actually, the preceding theorem 
has been obtained by Picard at the end of the 19th century under the assumption that $\theta_i>0$ for every $i$ (see \cite{Picard1905}  and the references therein). For the extension to the case when some hyperbolic cusps are allowed ({\it i.e}.\;when some of the angles $\theta_i$ vanish), we refer to \cite[Chap.II]{Heins} 
although it is quite likely that this generalization  was already  known to Poincar.

\subsection*{\bf A.2. Second order differential equations and $\boldsymbol{\mathbb C\mathbb H^1}$-conifold structures} 
Given a $\mathbb C\mathbb H^1$-structure on a punctured Riemann surface $X^*$, 
the question is to verify whether it actually extends as a conifold structure at the punctures.   This can be achieved by looking at the associated Schwarzian differential equation.  
\sk 

We detail below some aspects of the theory of second order differential equations which are needed for this. Most of the material presented below is very classical and well-known (the reader can consult \cite{YoshidaFuchsian,StGervais} among the huge am\-ount of references  which address the issue).

\subsubsection*{\bf A.2.1}  Since we are concerned by a local phenomenon, we will work locally and assume that $X^*=\mathbb D^*$.  In this case, the considered $\mathbb C\mathbb H^1$-structure on $X^*$, which we will denote by $\boldsymbol{\mathscr  X^*}$ for convenience,  is characterized by the data of its developing map $D: X^*\rightarrow \mathbb C\mathbb H^1$  alone.  Let $x$ be the usual coordinates on $\mathbb D$.  Although $D$ is a multivalued function of $x$, its 
monodromy lies in ${\rm Aut}( \mathbb C\mathbb H^1)$ hence is projective. 
It 
follows that   the {\bf Schwarzian derivative of $\boldsymbol{D}$ with respect to $
\boldsymbol{x}$},  defined as
$$
\big\{ D,x\big\}=\left(\frac{D''(x)}{D'(x)}\right)'-\frac{1}{2}
\left(\frac{D''(x)}{D'(x)}\right)^2=
\frac{D'''(x)}{D'(x)}-\frac{3}{2}
\left(\frac{D''(x)}{D'(x)}\right)^2\, , 
$$
is non-longer multivalued. In other words,  there exists a holomorphic function $Q$ on $X^*$ such that the 
following 
{\bf Schwarzian differential equation} holds true: 
\begin{equation*}
\big(  \mathscr S  \boldsymbol{\mathscr  X^*}\big) 
\qquad \qquad \qquad \qquad \qquad \qquad 
\big\{ D,x\big\}=Q(x). 
\qquad \qquad \qquad \qquad \qquad \qquad \qquad \qquad 
\end{equation*}

It turns out that the property for $\boldsymbol{\mathscr  X^*}$ to extend as a $\mathbb C\mathbb H^1$-conifold  at the origin can be deduced from this differential equation 
as we will  
explain 
 below. \sk 

Note that, since any function of the form $(aD+b)/(cD+d)$ with $\big[\!\!\!\!\tiny{\begin{tabular}{cc}
 $a$\!\!\!\! &$b$\vspace{-0.1cm}\\
 $c$\!\!\!\!& $d$ \end{tabular}} \!\!\!\!\big]\in {\rm SL}_2(\mathbb C)$ satisfies 
 $( \mathscr S  \boldsymbol{\mathscr  X^*})$, this differential equation (or, in other terms, the function $Q$) alone does not characterizes  $\boldsymbol{\mathscr  X^*}$.  This $\mathbb C\mathbb H^1$-structure is characterized by 
 the data of an explicit model $U$ 
 of $\mathbb C\mathbb H^1$ as a domain  in $\mathbb P^1$ (for instance $U=\mathbb D$ or $U=\mathbb H$) and by a ${\rm Aut}(U)$-orbit 
  of $U$-(multi)valued solutions of $( \mathscr S  \boldsymbol{\mathscr  X^*})$.

 \subsubsection*{\bf A.2.2} 
We now recall some  very classical material about Fuchsian differential equations  (see \cite{YoshidaFuchsian,Gauss2Painlev,StGervais} among many  references). 
\sk 

As is well-known, given $R\in \mathcal O(X^*)$, the Schwarzian differential equation 
 \begin{equation*}
\label{E:SchwarzianDifferentialEquationR}
\qquad \quad 
\big(\mathscr S_R\big)\;  \qquad \qquad \qquad \quad\quad  \qquad
\big\{ S,x\big\}=R(x)
\qquad \quad \qquad \quad \qquad \quad \qquad \quad\qquad \qquad 
\end{equation*}
is associated to the class of second-order differential equations of the form 
 \begin{equation*}
\label{E:SchwarzianDifferentialEquationF}
\qquad \quad 
\big(\mathscr E_{p,q}\big)\;  \qquad \quad \qquad \quad\quad  \quad
F''+p F' + q  F=0
\qquad \quad \qquad \quad \qquad \quad \qquad \quad\qquad \quad 
\end{equation*}
for any  function $p,q\in \mathcal O(X^*)$ such that 
$
R=2(   q-p'/2-  p/4)
$. Given two such functions $p$ and $q$, the solutions of $(\mathscr S_R)$ 
are the functions of the form $F_1/F_2$ for any basis  $(F_1,F_2)$  of the space of solutions of $(\mathscr E_{p,q})$.

In what follows, we fix such an equation $(\mathscr E_{p,q})$ and will work with it. The reason for doing so is twofold: first, it is easier to deal with such a linear equation than with $(\mathscr S_R)$ which involves a  non-linear Schwarzian derivative. Secondly, it is through  some second-order linear differential equations that we are studying Vecch's $\mathbb C\mathbb H^1$-structures on the algebraic leaves  of Veech's foliation on $\mathscr M_{1,2}$  in this text 
(see  \S\ref{SS:ManoDifferentialSystemAlgebraicLeaves} for more details). 
\sk 

We recall that $(\mathscr E_{p,q})$ (resp.\;$(\mathscr  S_R)$) is {\bf Fuchsian} (at the origin)  if $p,q$ are (resp.\;$R$ is) meromorphic at this point  with $p(x)=O(x^{-1})$  and $q(x)=O(x^{-2})$ (resp.\;$R(x)=O(x^{-2})$ for $x$ close to $0$ in $\mathbb D^*$.  
 In this case, defining $p_0$ and $q_0$ as the complex numbers such that 
 $p(x)=p_0x^{-1}+O_0(1)$ and $q(x)=q_0x^{-2}+O_0(x^{-1})$, one can construct
 the quadratic equation 
 $$ s(s-1)+sp_0+q_0=0
 $$
which is called the  {\bf characteristic equation} of $( \mathscr E_{p,q})$.  Its two roots $\nu_+$ and $\nu_-$ (we assume that  $\Re{\rm e}(\nu_+)\geq \Re{\rm e}(\nu_-)$)  are the two {\bf characteristic exponents} of this equation and their difference $\nu=\nu_+-\nu_-$ is the associated {\bf projective index}\footnote{Note that $\nu$ is actually only defined up to sign in full generality.}. The latter can also be defined as  the complex number
such that $R(x)=\frac{1-\nu^2}{2 }x^{-2}+O(x^{-1})$ in the vicinity of the origin,  which  shows that it is actually associated to the Schwarzian equation $(\mathscr  S_R)$
rather than to $( \mathscr E_{p,q})$.\sk 

It is known that one can give a normal form of a solution of $(\mathscr  S_R)$ in terms of $\nu$
 : generically (and this will be referred to as the {\bf standard case}), there 
is a local invertible  change of coordinate $x\mapsto y=y(x)$ at the origin
so that $y^\nu$ provides a solution of  $(\mathscr  S_R)$ on a punctured neighborhood of $0$.  Another  case is possible only when $\nu=n\in \mathbb N$. 
In this case, known as the {\bf logarithmic case},  a solution of $(\mathscr  S_R)$ could be of the form $y^{-n}+\log(y)$. These results (which are simple consequences  of Frobenius theorem 
for  Fuchsian second-order differential equations, see  for instance \cite[\S2.5]{YoshidaFuchsian})\footnote{See also \cite[{{Thorme\,IX.1.1}}]{StGervais} for the sketch of a more direct proof.} are summarized in Table \ref{T:yn}.

\begin{table}[!h]
\begin{center}
\begin{tabular}{|C|C|C|C|}
\hline
\multicolumn{1}{|c|}{\backslashbox{\hspace{1.5cm} {\bf Case}}{\vrule width 0pt height 1.25em  \hspace{-0.5cm} $
{\bf Index}\; \boldsymbol{\nu}$}} & $\boldsymbol{\nu\not \in \mathbb N}$    &
$\boldsymbol{\nu=n \in \mathbb N^*}$
&$\boldsymbol{\nu=0}$     \\
\hline
{\bf Standard} & $y^\nu$  &        $y^n$   & $-\!\!-$\\
\hline
{\bf Logarithmic} & $-\!\!-$   & $y^{-n}+\log y $     &  $\log y$ \\
\hline 
\end{tabular}\mk
\caption{}
\label{T:yn}
\end{center}
\end{table}

We will use this result  to determine  when the $\mathbb C\mathbb H^1$-structure $\boldsymbol{\mathscr X}^*$ extends as a conifold structure at the origin  by means of some analytical considerations about the associated Schwarzian differential equation 
 $( \mathscr S  \boldsymbol{\mathscr  X^*})$. 
  
  Before turning to this, we would like to state another very classical (and elementary) result about Fuchsian differential systems and equations that we use several times in this text (for instance in \S\ref{SS:ManoDifferentialSystemAlgebraicLeaves} above or in B.3.5 below). 
 \sk 
 
Let 
$$ \qquad \quad 
(\mathcal S)\;  \qquad \qquad \qquad \qquad \qquad
Z'=
M \cdot 
Z\qquad \quad \qquad 
\qquad \quad \qquad \quad \quad \quad \quad \quad 
 $$
be a meromorphic linear $2\times 2$ differential system on $(\mathbb C,0)$: 
$M=(M_{i,j})_{i,j=1}^2$ is a matrix of (germs of) meromorphic functions at the origin and the unknown $Z={}^t( F,G)$ is a $2\times 1$ matrix
whose coefficients $F$ and $G$ are  (germs  of) holomorphic  functions at a point $x_0\in (\mathbb C,0)$ distinct from 0.
\newpage

\noindent{\bf Lemma A.2.2.}
{\it 
\label{L:LemmaFromFG2P}
 Assume that $M_{1, 2}$ does not vanish identically. Then: ${}^{}$
\begin{enumerate}
\item   the space of first components $F$ 
 of solutions $Z={}^t( F,G)$ of 
$(\mathcal S)$ coincides with the space of solutions of the second-order differential equation
$$
\qquad \quad 
(\mathcal E_{\mathcal S})\;  \qquad \quad \qquad \quad
F''+p\, F'+q \, F=0\qquad \quad 
\qquad \quad \qquad \quad \qquad \quad \qquad \quad 
$$
with 
\begin{equation*}
p=-{\rm Tr}(M)-\frac{M_{12}'}{M_{12}} \qquad \mbox{ and }\qquad 
q=\det(M)-M_{11}'+M_{11}\frac{M_{12}'}{M_{12}}\, ; 
\end{equation*}
\item the differential equation 
 $(\mathcal E_{\mathcal S})$ is Fuchsian  if and only if 
  $M$ has a pole of order at most 1 at the origin.  In this case, the characteristic exponents of $(\mathcal E_{\mathcal S})$ coincide with the eigenvalues of the residue matrix of $M$ at $0$. 
\end{enumerate}}
\begin{proof} 
This is a classical result which can be proved by straightforward computations
 (see {\it e.g.}\;\cite[Lemma 6.1.1, \S3.6.1]{Gauss2Painlev} for the first part).
\end{proof}
 
 \subsubsection*{\bf A.2.3} 
 We now return to the problematic mentioned  in {\bf A.2.1.} above. 
\sk 

We first consider the models of $\mathbb C\mathbb H^1$-cones considered in {\bf A.1}.  By some easy computations, one gets that 
$$\big\{  D_s(x),x \big\}= \frac{1-s^2}{2x^2}$$ 
for any $s\geq 0$ (we recall that $D_0(x)=\log(x)$ and $D_s(x)=x^s$ for $s>0$). 

It follows that a necessary condition for  the origin to be a conifold 
point for the $\mathbb C\mathbb H^1$-structure $\boldsymbol{\mathscr X^*}$ is that 
the  Schwarzian differential equation $(\mathscr S{\boldsymbol{\mathscr X}^*})$ 
 must be  Fuchsian at this point, {\it i.e.}\;that $Q(x)=O(x^{-2})$ in the vicinity of $0$. 

A natural guess at this point would be that the preceding condition is also sufficient. 
 It turns out that it  is  the case indeed: \sk

\noindent{\bf Proposition A.2.3.}
{\it 
 The two following assertions are equivalent: 
\begin{enumerate}
\item the $\mathbb C\mathbb H^1$-structure $\boldsymbol{\mathscr X^*}$ extends as a conifold structure  at the origin; 
\sk 
\item the Schwarzian differential equation $(\mathscr S{\boldsymbol{\mathscr X^*}})$ is Fuchsian.
\end{enumerate}
}

Proving this result is not difficult. We provide a proof below for the sake of completeness.  We will denote  the monodromy operator acting 
on (germs at the origin of)  multivalued holomorphic functions 
on $(\mathbb D^*,0)$ by $M_0$. 
\sk 

We will need the following\sk

\noindent{\bf Lemma A.2.3.} 
{\it  For any positive integer $n$ 
 and any Moebius transformation 
 $g\in {\rm PGL}_2(\mathbb C)$, the multivalued map $D(x)=g(x^{-n}+\log(x))$ is not the developing map of a  $\mathbb C\mathbb H^1$-structure 
 on a punctured open neighborhood of the origin in $\mathbb C$.}
\sk 

\begin{proof} The monodromy around the origin of such a (multivalued) function $D$ is projective. 
  Let $T_0$ stand for  
the matrix associated to it.  On the one hand,  $T_0$ is parabolic with $g(0)\in \mathbb P^1$ as its unique fixed point. On the other hand, the image of any punctured small open neighborhood of the origin by $D$ is a punctured open neighbourhood of $g(0)$. These two facts imply that there does not exist a model $U$ of $\mathbb C\mathbb H^1$ in $\mathbb P^1$ (as an open domain) such that $D$ has values in $U$ and $T_0\in {\rm Aut}(U)$. This shows in particular that $D$ can not be the developing map of a $\mathbb C\mathbb H^1$-structure on any punctured open neighborhood of $0\in \mathbb C$.
\end{proof}
\sk

\noindent{\bf Proof of Proposition A.2.3.}
According to the discussion which precedes  the Proposition,  (1) implies (2), hence  the only thing remaining to be proven is the converse implication.  We assume that $(\mathscr S{\boldsymbol{\mathscr X^*}})$ is Fuchsian and let $\nu$ be its index. 
We will consider the different cases of Table \ref{T:yn} separately.    
\sk

We assume first that $\nu$ is not an integer. Then there exists a local change of coordinates $x\mapsto y=y(x)$ fixing the origin such that $y^\nu$ is a solution of $(\mathscr S{\boldsymbol{\mathscr X^*}})$. Consequently, the developing map $D : X^*\rightarrow \mathbb D$ of 
${\boldsymbol{\mathscr X}^*}$  can be written 
  $D=(a y^\nu+b)/({cy^\nu+d})$  for 
some complex numbers $a,b,c,d$ such that $ad-bc=1$. 

Clearly, $b/d\in \mathbb D$,  hence up to post-composition by an element of ${\rm Aut}(\mathbb D)={\rm PU}(1,1)$ sending 
$b/d$ onto $0$, one can assume that $b=0$.  By assumption, the monodromy of $D$ belongs to ${\rm PU}(1,1)$. Since it has necessarily the origin as a fixed point, it follows that this monodromy is given by $$M_0(D)= e^{2i\pi\mu} D$$
 for a certain real number $\mu$.  On the other hand, one has 
$$M_0(D)=\frac{a M_0(y^\nu)}{c M_0(y^\nu)+d}=\frac{a e^{2i\pi\nu} y^\nu}{c e^{2i\pi\nu} y^\nu+d}\, .$$

From the two preceding expressions for $M_0(D)$ and since $a\neq 0$, one deduces that  the relations $$\begin{tabular}{rcl}
$e^{2i\pi\mu} \frac{aY}{cY+d}=\frac{a e^{2i\pi\nu} Y}{c e^{2i\pi\nu} Y+d}$ 
& $\Longleftrightarrow$ &  $e^{2i\pi\mu}\big(  c e^{2i\pi\nu} Y+d \big)= 
e^{2i\pi\nu}\big(  c  Y+d \big)$ 
\end{tabular}
$$
hold true as rational/polynomial identities in $Y$. 
Because $e^{2i\pi\nu}\neq 1$ by assumption, it follows that $c=0$ and $\nu\in \mathbb R^+\setminus \mathbb N$.   
Consequently, one has $D(x)=\widetilde y(x)^\nu$ for a certain multiple $\widetilde y$ of $y$. It is a local biholomorphism  which induces an isomorphism of   $\mathbb C\mathbb H^1$-structures $\boldsymbol{\mathscr X}^*\simeq \mathfrak C_{2\pi\nu}^*$.  This proves (1) in this case.\sk

 Assume now that $\nu=0$. Then $\log(y)/(2i\pi)$ is a solution of $(\mathscr S{\boldsymbol{\mathscr X^*}})$ for a certain local coordinate $y$ fixing the origin. In this case, it is more convenient to take $\mathbb H$ as the target space of the developing map $D$ of $\boldsymbol{\mathscr X^*}  $.  Since the monodromy of $D$ is parabolic, one can assume that its fixed point is $i\infty$,  which implies that $D$ can be written $D=a\log(y)/(2i\pi)+b$ with $a,b\in \mathbb C$ and $a\neq 0$.  Setting $\beta=\exp(2i\pi b/a)\neq 0$ and replacing $y$ by $\beta y$, one can assume that $b=0$. 

Moreover, since $D$ has monodromy in ${\rm PSL}_2(\mathbb R)$,  $a$ must be real and positive. Computing the pull-back by $D$ of the hyperbolic metric of $\mathbb H$, one gets 
$$
D^*\left( \frac{\lvert dz\lvert}{  \big\lvert\Im{\rm m}(z)\big\lvert }
\right)= \frac{\lvert dD\lvert}{  \big\lvert\Im{\rm m}(D)\big\lvert }
=\frac{ \frac{a}{2\pi}\frac{\lvert dy\lvert}{\lvert y\lvert} }{   \frac{a}{2\pi}   \big\lvert  \Re{\rm e}\big(   \log(y)  \big)   \big\lvert    }
= \frac{\lvert dy\lvert}{ \lvert y\lvert    \log\lvert y\lvert}$$
which shows that $y$ induces an isomorphism of
$\mathbb C\mathbb H^1$-structures 
 $\boldsymbol{\mathscr X^*}\simeq  \mathfrak C_{0 }^*$. 
\sk 

 We now consider the case when $\nu=n\in \mathbb N^*$ and $y^n$ is a solution of $(\mathscr S{\boldsymbol{\mathscr X^*}})$ for a certain local coordinate $y=y(x)$ fixing the origin.  As above, one can assume that the developing map $D$ of $\boldsymbol{\mathscr X^*}$ is written $D=a y^n/(cy^n+d)$. In this case, the monodromy argument used previously does not apply but one can conclude directly by remarking that since $n$ is an integer, there exists another local coordinate $\widetilde y$ at the origin such that the relation $ay^n/(cy^n+d)=\widetilde y^n$    holds true identically.  This shows that $\boldsymbol{\mathscr X^*}$ is isomorphic to 
$ \mathfrak C_{2\pi n }^*$. \sk 

Finally, the last case of Table 6, namely the logarithmic case with $\nu\in \mathbb N^*$,     does not occur according to Lemma A.2.3., hence we are done.
\qed



\section*[The Gau{\ss}-Manin connection]{\bf Appendix B\,: 
the  Gau{\ss}-Manin connection associated to Veech's map}

Many  properties of the hypergeometric function $F(a,b,c; 
\boldsymbol{\cdot})$ hence of the associated $\mathbb C\mathbb H^1$-valued multivalued Schwarz map $S(a,b,c; \boldsymbol{\cdot})$ can be deduced from the hypergeometric differential equation \eqref{HGE}.  
\sk 

Let $\mathcal F_a^\alpha$ be a leaf of Veech's foliation  in the Torelli space $\mathcal T\!\!\!{\it or}_{1,n}$. 
 As shown in \S\ref{SS:ComparisonVeechDM}, Veech's map $V_a^\alpha: \mathcal F_a^\alpha\rightarrow \mathbb C\mathbb H^{n-1}$ has an expression $V_a^\alpha=[v_\bullet]$ whose components $v_\bullet=\int_{\boldsymbol{\gamma}_{\!\!\bullet}}  T^\alpha_{\!\!a}(u)du$, with $\bullet=\infty,0,3,\ldots,n$ are elliptic hypergeometric integrals. A very  natural approach to the study of $V_a^\alpha$ is by first constructing the differential system satisfied by these. \sk 

Something very similar has been done in the papers \cite{Mano} and \cite{ManoWatanabe} but in the   more general context of isomonodromic deformations of linear differential systems on punctured elliptic curves.  The results of these two papers can be specialized to the case we are interested in, but this requires a little work in order to be made completely explicit.  This is what we do in this appendix. 
\mk 

We first introduce the Gau{\ss}-Manin connection in a general context and  then specialize and make everything explicit in the case of punctured elliptic curves.


\subsection*{\bf B.1. Basics on Gau{\ss}-Manin}
In this subsection, we present general facts relative to the construction of the 
Gau{\ss}-Manin connection $\nabla^{GM}$. We first define it analytically in B.1.2. Then we explain how it can be computed  by means of relative differential forms, see B.1.3. We conclude in B.1.4 by stating the comparison theorem which asserts that, under reasonable hypotheses, 
one can construct$\nabla^{GM}$ by considering only algebraic relative differential forms. 
\sk 

The material presented below is well-known hence no proof is given. 
Classical references are the paper \cite{KatzOda} by Katz and Oda and the book \cite{Deligne} by Deligne. 

Another  more recent and useful reference is  the book \cite{AndrBaldassarri} by Andr and Baldassarri, in particular the third chapter. Note that the general strategy followed in this book goes by `dvissage' and reduces the proofs of most of the main results  to a particular ideal case, called an {\it `elementary coordinatized fibration}' by the authors ({\it cf.}\;\cite[Chap.\,3, Definition 1.3]{AndrBaldassarri}).  We think it is worth mentioning   that the specific case of punctured elliptic curves we are interested in is precisely of this kind, see Remark B.2.4 below. 


\subsubsection*{\bf B.1.1} 
Let  $\pi: \mathcal X\longrightarrow S$ be a family of  Riemann surfaces  over
a complex manifold $S$.  This means that $\pi$ is a holomorphic morphism whose fibers  
$X_s=\pi^{-1}(s)$, $s\in S$,  all are (possibly non-compact) Riemann surfaces.  We assume that $\pi$  is smooth and as  nice as needed 
to make everything we say below work well.
\sk

Let $\Omega$ be a holomorphic 1-form on $\mathcal X$ 
and for any $s\in S$, denote by $\omega_s$ its restriction to the fiber $X_s$: 
$\omega_s=\Omega\lvert_{X_s}$.  
Then one defines differential covariant operators by setting 
$$
\nabla(\eta)=d\eta+\Omega\wedge \eta \qquad \Big(\mbox{resp. }\; \nabla_{s}(\eta)=d\eta+\omega_s\wedge \eta\Big)
$$
for any (germ of) differential form $\eta$ on $\mathcal X$ (resp.\;on $X_s$, for any $s\in S$).

The associated kernels 
$$L={\rm Ker}\big( \nabla: \mathcal O_{\mathcal X}\rightarrow 
\Omega^1_{\mathcal X}\big) \qquad \mbox{ and }\qquad L_s={\rm Ker}\big( \nabla_{s}: \mathcal O_{X_s}\rightarrow 
\Omega^1_{X_s}\big)\,  $$  are local systems  on $\mathcal X$ and $X_s$ respectively, such that $L\lvert _{X_s}=L_s$  for any $s\in S$. 
\mk


\subsubsection*{\bf B.1.2} 
  Let $B$ 
 be 
the first derived direct image 
of $L$ by $\pi$: 
$$B=R^1\pi_*(L).$$ 

It is the sheaf on $S$  the stalk of which at $s\in S$ is the first group of twisted cohomology $H^1(X_s, L_s)$. 
 We assume that $\pi: \mathcal X\rightarrow S$ and $\Omega$ are such that $B$ is a local system on $S$, of finite rank denoted by $r$. 
 Then,  tensoring by the structural sheaf of $S$, one obtains 
$$\mathcal B=B\otimes _{\mathbb C} \mathcal O_S \, .$$  
It is a locally free sheaf of rank $r$ on $S$. Moreover, there exists a unique connection  on $\mathcal B$ whose kernel  is $B$. The latter is known as the {\bf Gau{\ss}-Manin connection}  and will be denoted by 
$$\nabla^{GM}: \mathcal B\rightarrow \mathcal B\otimes \Omega_S^1.$$

We have thus given an analytic  definition of the Gau{\ss}-Manin connection in the relative twisted context. Note that this definition, although rather direct,  is not constructive at all. We will remedy to this below. 


\subsubsection*{\bf B.1.3} We recall that  the sheaves $\Omega^\bullet_{\mathcal X/S}$ of {\bf relative differential forms} on $\mathcal X$ are 
the ones characterized by requiring that the following short sequences of $\mathcal O_{\mathcal X}$-sheaves are exact: 
$$  \xymatrix@R=0.5cm@C=0.5cm{0     \ar@{->}[r]  &    \pi^* \Omega_S^\bullet    \ar@{->}[rr] &&
 \Omega_{\mathcal X}^\bullet    \ar@{->}[rr]  &&    \Omega_{\mathcal X/S}^\bullet    \ar@{->}[r]  &    0   }.  $$

More concretely, let $s_1,\ldots,s_n$ stand for local holomorphic coordinates on a small open subset $U\subset S$ and let $z$ be a vertical local coordinate on an open subset $\widetilde U\subset \pi^{-1}(U)$ tale over $U$. 
 Then there are natural isomorphisms 
\begin{equation}
\label{E:LocalIsomSheaves}
\Omega_{\mathcal X/S}^0\big\lvert_{\widetilde U}\simeq \mathcal O_{\widetilde U} 
\qquad \mbox{ and }\qquad 
\Omega_{\mathcal X/S}^1\big\lvert_{\widetilde U}\simeq \mathcal O_{\widetilde U}\cdot dz \, .  
\end{equation}

For any local section $\eta$   of $\Omega_{\mathcal X}^\bullet$, we denote by $\eta_{\mathcal X/S}$ the section of $\Omega_{\mathcal X/S}^\bullet$ it induces. With the above notation, assuming that $\eta$ is a holomorphic 1-form on $\widetilde U$, then 
the local isomorphism \eqref{E:LocalIsomSheaves} identifies 
$\eta_{\mathcal X/S}$ with  $\eta- \eta(\partial_z)dz$. \sk

Since the exterior derivative $d$ commutes with the pull-back by $\pi$, one obtains 
the {\bf relative de Rham complex} $(\Omega_{\mathcal X/S}^\bullet, d)$. One verifies easily that the connexion $\nabla_\Omega$ on $\Omega_{\mathcal X}^\bullet$ induces a connexion  $\nabla _{\mathcal X/S}$ on the  relative de Rham complex, so that any square of 
$\mathcal O_{\mathcal X}$-sheaves as 
below is commutative: 
 $$  \xymatrix@R=1.4cm@C=1cm{
 \Omega_{\mathcal X}^\bullet    \ar@{->}[r]  
 \ar@{->}[d]^{\nabla}
 &    \Omega_{\mathcal X/S}^\bullet    \ar@{->}[r]   \ar@{.>}[d]^{\nabla_{\mathcal X/S}}
 &    0 \; \\
 \Omega_{\mathcal X}^{\bullet+1}    \ar@{->}[r]  &    \Omega_{\mathcal X/S}^{\bullet+1}    \ar@{->}[r]  &    0\; .
 } $$
 
 In the local coordinates $(s_1,\ldots,s_n,z)$ considered above, writing  $
 \Omega=\omega+\varphi dz$ for a holomorphic function $\varphi$ and a 1-form $\omega$ such that $\omega(\partial_z)=0$   ({\it i.e.}\;$\omega=\Omega _{\mathcal X/S}$ with the notation introduced above), it comes that  $\nabla_{\mathcal X/S}$ satisfies 
\begin{equation}
\label{E:NablaX/Sf}
 \nabla_{\mathcal X/S}(f) =\sum_{i=1}^n 
 \big({\partial f}/{\partial s_i}\big)ds_i+ \omega\, f 
 \end{equation}
 for any holomorphic function $f$ on $\widetilde U$
  and is characterized by this property.\sk 
  
  By definition, $( \Omega_{\mathcal X/S}^\bullet, \nabla_{\mathcal X/S})$ is the {\bf relative twisted de Rham complex} associated to $\pi$ and $\Omega$.  Under some natural assomptions, the direct images $\pi_*\Omega^\bullet_{\mathcal X/S}$ are coherent sheaves of $\mathcal O_S$-modules and $\nabla_{\mathcal X/S}$ gives rise to a connection on $S$ 
  $$
  \pi_*\big(\nabla_{\mathcal X/S}\big) \, :\, \pi_*({\mathcal O}_{\mathcal X})\longrightarrow \pi_*\big(\Omega_{\mathcal X/S}^1\big)\, .
  $$
  
  Note that  $\pi_*({\mathcal O}_{\mathcal X})$ is nothing else but $\mathcal O_{\mathcal X}$   seen as a $\mathcal O_S$-module by means of $\pi$.  For this reason, we will abusively write down
   the preceding connection as 
 \begin{equation}
 \label{E:tytu}
\nabla_{\mathcal X/S} \, :\, {\mathcal O}_{\mathcal X}\longrightarrow \pi_*\big(\Omega_{\mathcal X/S}^1\big)\, .
  \end{equation}
 
   \subsubsection*{\bf B.1.4}
 On the other hand,  the map 
 $$
 U\longmapsto {\bf H}^1\Big(\pi^{-1}(U), \big( \Omega_{\mathcal X/S}^\bullet, \nabla_{\mathcal X/S}\big)\big\lvert_U\Big) 
 $$
defines a presheaf (of hypercohomology groups) on $S$.   The associated sheaf  is denoted by $R^1 \pi_*(\Omega_{\mathcal X/S}^\bullet, \nabla)$. Its stalk at any  $s\in S$ coincides with ${\bf H}^1(X_s,(\Omega_{X_s}^\bullet,\omega_s))$ hence is naturally isomorphic to $H^1(X_s,L_s)$ (see \S\ref{S:TwistedDeRham}).  

It follows that one has a natural isomorphism 
$$
\mathcal B\simeq R^1 \pi_*\big(\Omega_{\mathcal X/S}^\bullet, \nabla\big)\, . $$

 We make the supplementary assumption that $\pi$ is affine (this implies in particular that the fibers $X_s$ can no more  be assumed to be compact).
 Then it follows (see \cite[Chapt.III,\S2.7]{AndrBaldassarri}) that $R^1 \pi_*\big(\Omega_{\mathcal X/S}^\bullet, \nabla\big)$ hence $\mathcal B$ identifies with the cokernel of the connection  $\pi_*\big(\nabla_{\mathcal X/S}\big)$, denoted by $\nabla_{\mathcal X/S}$ for short, see \eqref{E:tytu}. 
 
 In other terms, one has a natural isomorphism of $\mathcal O_S$-sheaves
 \begin{equation}
 \label{E:IsomB-CokerPushForwardNablaX/S}
 \mathcal B\simeq \frac{\pi_*\Omega^1_{\mathcal X/S}}{
 \nabla_{\mathcal X/S}\big(
 \mathcal O_{\mathcal X} \big)}\, . 
 \end{equation}

 For a local  section $\eta_{\mathcal X/S}$ of $\Omega^1_{\mathcal X/S}$, we denote  by $[\eta_{\mathcal X/S}]$ its class in $\mathcal B$, or equivalently, its class modulo $ \nabla_{\mathcal X/S}(
 \mathcal O_{\mathcal X})$.\sk

 By means of the  latter isomorphism, one can give an effective description of the action of the Gau{\ss}-Manin connection.  Let $\nu$ be   a vector field over the open subset $U\subset S$ of $T_S$ ({\it i.e.}\;an element of $\Gamma(U,T_S)$). 
Then 
$$\nabla^{GM}_\nu=\big\langle \nabla^{GM}(\cdot)
 \, , \nu\big\rangle
 $$ 
 is a derivation of the $\mathcal O_U$-module $\Gamma(U,\mathcal B)$. An element of this space of sections is represented by the class $[\eta_{\mathcal X/S}]$  (that is 
$\eta_{\mathcal X/S}$ modulo 
$\nabla_{\mathcal X/S}\mathcal O(\widetilde U) $)
of a relative 1-form $\eta_{\mathcal X/S} \in \Gamma(\tilde U,\Omega^1_{\mathcal X/S})$.  Let 
$ \tilde \eta$ be a section of $\Omega^1_{\mathcal X}$ over $\tilde U$ 
such that   $ \tilde \eta_{\mathcal X/S}=\eta_{\mathcal X/S}$.  Then, for any lift $\tilde \nu$ of $\nu$ over $\tilde U$, one has
$$
\nabla_\nu^{GM}\big(   [\eta_{\mathcal X/S}  ]   \big) =   \big[  \nabla_{\tilde \nu}(   \tilde \eta)_{\mathcal X/S}\big]\, . 
$$

Finally, we mention that when not only $\pi$ but also $S$ is supposed to be affine (as  will hold true in the case we will interested in, {\it cf.}\;B.3 below), then there is a more elementary description of the RHS of the isomorphism \eqref{E:IsomB-CokerPushForwardNablaX/S}. Indeed, in  this case, according to  \cite[p.117]{AndrBaldassarri}, $\mathcal B$ identifies with the $\mathcal O_S$-module attached to 
the first cohomology group of the complex of global sections 
$$
\mathcal O({\mathcal X})\rightarrow \Omega^1_{\mathcal X/S}(\mathcal X)
\rightarrow \Omega^2_{\mathcal X/S}(\mathcal X)\rightarrow \cdots .
$$

If additionally $S$ is assumed to be of dimension 1, then  $\Omega^2_{\mathcal X/S}$ is trivial,  hence one obtains 
the following  generalization  of  \eqref{E:popol} in the relative case: 
 \begin{equation}
 \label{E:IsomB-CokerSAffineDim1}
 \mathcal B\simeq \mathcal O_S\otimes_{\mathbb C}\frac{
H^0\big(\mathcal X, 
 \Omega^1_{\mathcal X/S}\big)}{
 \nabla_{\mathcal X/S} \big( H^0\big(
 \mathcal X, 
 \mathcal O_{\mathcal X}\big)\big)}\, .
 \end{equation}

  \subsubsection*{\bf B.1.5}
Assume that the fibers $X_s$'s  are punctured Riemann surfaces 
and that $\mathcal X$ can be compactified in the vertical direction into a family $\overline \pi: \overline{\mathcal X}\rightarrow S$ of compact Riemann  surfaces.  
The original map $\pi$ is the restriction of $\overline{\pi}$ to $\mathcal X$ which is nothing else but the complement of a divisor  $\mathcal Z$ in $\overline{\mathcal X}$. 

Instead of considering holomorphic (relative) differential forms on $\mathcal X$ as above, one can make the same constructions using   rational (relative) differential forms 
on $\overline{\mathcal X}$ with poles on $\mathcal Z$.   
More concretely, one makes all the constructions sketched above starting from the sheaves of $\mathcal O_{\overline{\mathcal X}}(*\mathcal Z)$-modules 
$\Omega_{\overline{\mathcal X}}^\bullet(*\mathcal Z)$ on $\mathcal X$. 
\sk 

A fundamental result of the field,  due to Deligne in its full generality,   is that the twisted comparison theorem mentioned in \S\ref{S:AlgebraicdeRhamComparisonTHM} can be generalized to the relative case, at least  when $\mathcal Z$ is a relative divisor 
with normal crossing over $S$ 
 (see \cite[Thorme 6.13]{Deligne} or \cite[Chap.4,\,Theorem 3.1]{AndrBaldassarri} for precise  statements). \sk
 
In the particular case of relative dimension 1, this gives the following version of the isomorphism  \eqref{E:IsomB-CokerPushForwardNablaX/S}: 
 \begin{equation*}
 \mathcal B\simeq \frac{\pi_*\Omega^1_{\mathcal X/S}(*\mathcal Z)
 }{
 \nabla_{\mathcal X/S}\big(
 \mathcal O_{\mathcal X}(*\mathcal Z) \big)}\; . 
 \end{equation*}

When $S$ is also assumed affine, one gets the following 
generalization of \eqref{E:IsomTutu}:
 \begin{equation}
 \label{E:IsomB-CokerSAffineDim1}
 \mathcal B\simeq \mathcal O_S\otimes_{\mathbb C}\frac{
H^0\big(\mathcal X, 
 \Omega^1_{\mathcal X/S}(*\mathcal Z)\big)}{
 \nabla_{\mathcal X/S} \big( H^0\big(
 \mathcal X, 
 \mathcal O_{\mathcal X}(*\mathcal Z)\big)\big)}\; .
 \end{equation}

  \subsubsection*{\bf B.1.6}
  We now explain how the material introduced above can be used to construct differential systems satisfied by generalized hypergeometric integrals.\sk
  
 Let $\check{B}$ be the dual of $B$. It is the  local system on $S$ whose fiber $\check{B}_s$ at $s$ is the twisted homology group $H_1(X_s,{L}_s^\vee)$.    Let $\check{\nabla}^{GM}$ be the {\bf dual Gau{\ss}-Manin connection} on the associated sheaf $\check{\mathcal B}=\mathcal O_S\otimes \check{B}$. We recall that, by definition, it is the connection  the solutions of which form the local system $\check{B}$. 
  It can also be characterized by the following property: 
    for any local sections with the same definition domain $b$ and $\check{\beta}$ of $\mathcal B$ and $\check{\mathcal B}$ respectively, one has
  \begin{equation}
  \label{E:CheckNablaGM}
  d\big\langle b, \check{\beta} \big\rangle 
  =\Big\langle
\nabla^{GM}(b), \check{\beta}
\Big\rangle+\Big\langle
b , \check{\nabla}^{GM}\big(\check{\beta}\big)\Big\rangle\, . 
    \end{equation}

   \subsubsection*{\bf B.1.7}
We use again  the notations of B.1.1. Let $T$ be a global (but multivalued) function on $\mathcal X$ satisfying $\check{\nabla}(T)=dT-\Omega T=0$. 
For any $s\in S$, one denotes  its restriction to $X_s$ by $T_s$.  Let $I$ be the local holomorphic function  defined on a small open subset $U\subset S$ as the following generalized hypergeometric integral depending holomorphically on $s$:  
\begin{equation}
\label{E:I(s)}
I(s)=\int_{\boldsymbol{\gamma}_{\!\!s}} T_s\cdot \eta^s\, .
\end{equation}
  
  Here the $ \boldsymbol{\gamma}_{\!\!s}$'s stand for  ${L}_s^\vee$-twisted 1-cycles which depend analytically on $s\in U$ and $s\mapsto \eta^s$ is a holomorphic `section of $\Omega_{\mathcal X}^1$ over $U$', {\it i.e.}\;$\eta^s\in \Omega^1(X_s)$ for every $s\in U$ and the dependency with respect to $s$ is holomorphic. 
   From what has been said above, the value $I(s)$ actually depends only on the twisted homology classes $[\boldsymbol{\gamma}_s]$ and on the class $[\eta^s_{\mathcal X/S}]$ of $\eta^s$ in $H^0(X_s, \Omega^1_{X_s})/\nabla_{\!\!s}(\mathcal O(X_s))$. 
  
In other terms, for every $s\in U$, one has 
\begin{equation}
\label{E:(Co)homDefForI(s)}
I(s)=
 \Big\langle 
\big[\boldsymbol{\gamma}_{\!\!s}\big]\, , \, \big[\eta^s_{\mathcal X/S}\big]
\Big\rangle \, . 
\end{equation}

To simplify the discussion, assume now that $S$ is affine and of dimension 1 (as it will be the case in B.3 below).  Now $s$ has to be understood as a global holomorphic coordinate on $U=S$. Setting $\sigma=\partial/\partial s$,  one denotes 
 the associated derivation by $\nabla_{\!\!\sigma}^{GM}(\cdot)=\big\langle \nabla^{GM}(\cdot ) \, , \, \sigma \big\rangle$. We define $\check{\nabla}_{\!\!\sigma}^{GM}$ similarly. 
\sk

In most of the cases (if not all), the twisted 1-cycles $\boldsymbol{\gamma}_{\!\!s}$'s appearing in such an integral are locally obtained by topological deformations. In this case, it is well known ({\it cf.}\;\cite[Remark ({\bf 3.6})]{DeligneMostow}) that $s\mapsto [\boldsymbol{\gamma}_{\!\!s}]$ is a  section of $\check{ B}$ hence belongs to the kernel of $\check{\nabla}^{GM}$,  {\it i.e.}\;$\check{\nabla}^{GM}(\boldsymbol{\gamma}_{\!\!s})\equiv 0$. \sk 

Let $\widetilde \sigma$ be a fixed lift of $\sigma$ over $U$. Then from \eqref{E:CheckNablaGM} and \eqref{E:(Co)homDefForI(s)}, it follows  that 
\begin{align*}
I'(s)
 =  \frac{d\;}{ds}\int_{\boldsymbol{\gamma}_{\!\!s}} T_s\cdot \eta^s
=   & \, \frac{d\;}{ds}\Big\langle 
\big[\boldsymbol{\gamma}_{\!\!s}\big]\, , \, \big[\eta^s_{\mathcal X/S}\big]
\Big\rangle\\
=   & \, \Big\langle 
[\boldsymbol{\gamma}_{\!\!s}]\, , \, 
\nabla_{\!\!\sigma}^{GM}
\big[\eta^s_{\mathcal X/S}\big]
\Big\rangle
=  \int_{\boldsymbol{\gamma}_{\!\!s}} T_s \cdot 
\nabla_{\!\!\widetilde \sigma} \big(\eta^s\big)
\end{align*}
for every $s\in U$.  
More generally, for any integer $n$, one has 
\begin{equation}
\label{E:I(n)(s)}
I^{(n)}(s)=
\Big\langle 
[\boldsymbol{\gamma}_{\!\!\sigma}]\, , \, 
\left(\nabla_{\!\!\sigma}^{GM}\right)^n 
\big[\eta^s_{\mathcal X/S}\big]
\Big\rangle
=\int_{\boldsymbol{\gamma}_{\!\!s}} T_s \cdot 
\nabla_{\!\!\widetilde \sigma}^n \big(\eta^s\big)
\end{equation}
 where $\nabla_{\!\!\sigma}^n$ stands for the 
$n$-th iterate of $\nabla_{\widetilde \sigma}$ acting on the sheaf of 1-forms 
on $\mathcal X$.\sk 

To make the writing simpler, if 
 $\mu$ is a section of $\Omega_{\mathcal X}^1$, 
 we will denote the section $[\mu_{\mathcal X/S}]$ of $ H^0(\mathcal X, 
 \Omega^1_{\mathcal X/S})/
 \nabla_{\mathcal X/S} ( H^0(
 \mathcal X, 
 \mathcal O_{\mathcal X}))$  that it induces just by $[\mu]$ below.

By hypothesis, the twisted cohomology groups $H^1(X_s,L_s)$  are all of the same finite dimension $N>0$. It follows that there is a non-trivial $\mathcal O(U)$-linear relation  between the 
classes of  the $
\nabla_{\!\!\sigma}^k \big(\eta^s\big)$'s for $k=0,\ldots,N$, {\it i.e.}\;there exists 
 $(A_0,\ldots,A_N)\in \mathcal O(U)^{N+1}$ non-trivial 
 and such that  the following relation 
$$
A_0(s)\cdot \big[\eta^s\big]+A_1(s) \big[
\nabla_{\!\!\sigma}\big(\eta^s\big)
\big]+\cdots+A_N(s) \big[
\nabla_{\!\!\sigma}^N\big(\eta^s\big)\big] =0  
 $$
  holds true for every $s\in U$.
Since the value  of the $k$-th derivative $I^{(k)}$ at $s$ actually  depends only on the class 
of $\nabla_{\!\!\sigma}^k(\eta^s)$ (see \eqref{E:I(n)(s)}), one obtains that 
the function $I$ satisfies the following linear differential equation on $U$: 
$$
A_0\cdot I+A_1\cdot I'+\cdots+A_N\cdot I^{(N)} =0 \, . $$

Note that the function  $I$ defined in \eqref{E:I(s)} is not the only solution of this differential equation. Indeed, it is quite clear that this equation is also 
satisfied by any function of the form $
s\mapsto \int_{\boldsymbol{\beta}_s} T_s\cdot \eta^s$ as soon as   $s\mapsto \boldsymbol{\beta}_s$  is a section of $\check{B}$.

\subsection*{\bf B.2. The  Gau{\ss}-Manin connection on a leaf of Veech's foliation}
We are now going to specialize the material presented in the preceding subsections to the case of punctured elliptic curves we are interested in. 
\mk 

In what follows, as before, $\alpha_1,\ldots,\alpha_n$ stand for fixed real numbers bigger than $-1$ that sum up to $0$: one has $\alpha_i\in ]-1,\infty[$ for $i=1,\ldots,n$ and $\sum_{i=1}^n \alpha_i=0$. 

\subsubsection*{\bf B.2.1.} 
Forgetting the last variable $z_{n+1}$ induces a projection
  from $ \mathcal T\!\!\!{\it or}_{1,n+1}
$ onto $ \mathcal T\!\!\!{\it or}_{1,n}$. For our purpose, it will be convenient  to 
see 
this  space  rather as a kind of  covering space of the `universal curve' over the target Torelli space. For this reason, we will write $u$ instead of $z_{n+1}$ and take this  variable as the first one.

 In other terms, we consider 
$$\mathcal C\!\mathcal T\!\!\!{\it or}_{1,n}=\Big\{ \big(u,\tau,z\big)\in \mathbb C\times \mathcal T\!\!\!{\it or}_{1,n}\; \big\lvert \; u\in \mathbb C\setminus \cup_{i=1}^n \big(z_i+\mathbb Z_\tau\big)
\Big\} \simeq \mathcal T\!\!\!{\it or}_{1,n+1}
$$
and the corresponding projection $\mathcal C\!\mathcal T\!\!\!{\it or}_{1,n}\rightarrow \mathcal T\!\!\!{\it or}_{1,n}: \, (u,\tau,z)\rightarrow (\tau,z)$. \sk

We define two automorphisms of $\mathcal C\!\mathcal T\!\!\!{\it or}_{1,n}$ by setting \begin{align*}
T_{1} (u,\tau,z)= \big(u+1,\tau,z\big) \qquad \mbox{ and }
\qquad 
T_{\tau} (u,\tau,z)= \big(u+\tau,\tau,z\big)
 \end{align*}
 for 
any element  $(u,\tau,z)$  of this space.  \sk

The group spanned by $T_1$  and $T_\tau$  is isomorphic to $\mathbb Z^2$ and acts discontinuously without fixed points on $\mathcal C\!\mathcal T\!\!\!{\it or}_{1,n}$.  The associated quotient, denoted by $\mathcal E_{1,n}$, 
is nothing but the  {\bf universal $\boldsymbol{n}$-punctured elliptic curve} over $\mathcal T\!\!\!{\it or}_{1,n}$. This terminology is justified by the fact that 
the projection onto $\mathcal T\!\!\!{\it or}_{1,n}$ factorizes and gives rise to  a fibration 
$$\pi: \mathcal E_{1,n}\longrightarrow \mathcal T\!\!\!{\it or}_{1,n} $$
 the fiber of which at  $(\tau,z)\in  \mathcal T\!\!\!{\it or}_{1,n} $ is the $n$-punctured elliptic curve $E_{\tau,z}$. 

There is a partial vertical compactification 
$$\overline{\pi}: \overline{\mathcal E}_{1,n}\longrightarrow \mathcal T\!\!\!{\it or}_{1,n}$$ whose fiber at $(\tau,z)$ is the unpunctured elliptic curve $E_\tau$. The latter extends $\pi$,  is smooth and proper and comes with $n$ canonical  sections 
(for $k=1,\ldots,n$): 
\begin{align*}
[k]_{1,n}: \mathcal T\!\!\!{\it or}_{1,n}& \longrightarrow \overline{\mathcal E}_{1,n}\\
\big(\tau,z) & \longmapsto [z_k] \in E_\tau\, .
\end{align*}
    In particular, because of the normalization $z_1=0$,  $[1]_{1,n}$ is nothing 
else but 
 the zero section $[0]_{1,n}$ which associates $[0]\in E_\tau$ to $(\tau,z)$: one has $[1]_{1,n}=[0]_{1,n}$.


\subsubsection*{\bf B.2.2.}
Recall the expression 
\eqref{E:functionT} for  the function $T$ considered 
in Section  \ref{S:Twisted(Co)Homology}: 
$$T(u,\tau,z)=\exp\big(2i\pi\alpha_0u\big)\prod_{k=1}^n \theta\big(u-z_k,\tau\big)^{\alpha_k}
\, .
$$
 
  Contrary to \S \ref{S:Twisted(Co)Homology} where $\tau$ and $z$ were assumed to be fixed and only $u$ was allowed to vary, we want now all the variables $u,\tau$ and $z$ to be free. In other terms, we  now see $T$ as a multivalued holomorphic function on $\mathcal C\!\mathcal T\!\!\!{\it or}_{1,n}$. 
 \sk

 Let $\Omega$ stand for the total logarithmic derivative 
 of $T$ on 
 $\mathcal C\!\mathcal T\!\!\!{\it or}_{1,n}$:
 $$
 \Omega= {d}\!\log T=(\partial \log T/\partial u) d\!u+(\partial \log T/\partial \tau)d\!\tau+\sum_{j=2}^n (\partial \log T/\partial z_j) d\!z_j\, .
 $$
 
 A straightforward computation shows that 
 \begin{equation}
 \label{E:OMEGA}
 \Omega=  
 \omega 
+\sum_{k=1}^n
\alpha_k \bigg[ \eta(u-z_k) d\tau
- \rho(u-z_k)dz_k
\bigg]
\end{equation}
 where 
 \begin{itemize}
 \item $\omega=(2i\pi\alpha_0+\delta )du$ stands for  the logarithmic total derivative of $T$ with respect to the single variable $u$ (thus $\delta=\sum_{k=1}^n\alpha_k \rho(u-z_k)$ 
 see \eqref{E:dLogT} in Section \ref{S:Twisted(Co)Homology}) but now considered as a  holomorphic 1-form on $\mathcal C\!\mathcal T\!\!\!{\it or}_{1,n}$;
\sk  
  \item we have set for any $(u,\tau,z)\in \mathcal C\!\mathcal T\!\!\!{\it or}_{1,n}$: 
 $$
 \eta(u)=\eta(u,\tau)=
 {\partial \log\theta(u,\tau)}/{\partial \tau}
 = \frac{1}{4i\pi}\frac{\theta''(u, \tau)}{\theta(u,\tau)}\, .
 $$
 \end{itemize}

After easy computations, one deduces from the functional equations \eqref{E:ThetaQuasiPeriodicity} that for every $\tau \in \mathbb H$ and every $u\in \mathbb C\setminus \mathbb Z_\tau$, one has 
\begin{align}
\label{E:mumu}
\rho(u+1)=& \rho(u)      && \eta(u+1)= \eta(u)  \\
\rho(u+\tau)= & \rho(u)-2i\pi && \eta(u+\tau)=
\eta(u)-\rho(u)+i\pi\,.
\label{E:mumuu}
\end{align}


In Section \ref{S:Twisted(Co)Homology}, we have shown that when $\tau$ is assumed to be fixed, $\omega$ is $\mathbb Z_\tau$-invariant.  It follows that,  on $\mathcal C\!\mathcal T\!\!\!{\it or}_{1,n}$,   one has: 
\begin{equation}
\label{E:T1omega}
T_1^*(\omega)=\omega \qquad \mbox{ and }
\qquad  T_\tau^*(\omega)=\omega+\big(2i\pi\alpha_0+\delta\big) d\tau\, .
\end{equation}

We  set  
$$\widetilde\omega=\Omega- \omega=
\sum_{k=1}^n\alpha_k\,  \eta\big(u-z_k\big) 
d\tau
- \sum_{k=1}^n\alpha_k \, \rho\big(u-z_k\big)dz_k\, .$$ 
 It follows immediately from \eqref{E:mumu} that $T_1^*(\widetilde\omega)=\widetilde\omega$. 
 With \eqref{E:T1omega}, this gives us
 \begin{equation}
 \label{E:T1Omega}
 T_1^*\big(  \Omega  \big)=\Omega\, . 
 \end{equation}
 
  On the other hand, using \eqref{E:mumuu} and the fact that $\sum_{k=1}^n\alpha_k=0$, one has 
  \begin{align*}
\label{E:TtauTildeomega}
T_\tau^*\big(\widetilde\omega\big)= & \widetilde\omega
 -\delta  d\tau +
 2i\pi\sum_{k=1}^n\alpha_k dz_k\, . 
\end{align*}
Combining the latter equation  with \eqref{E:T1omega}, 
 one eventually  obtains 
\begin{equation}
 \label{E:TtauOmega}
T_\tau^*(  \Omega) = \Omega +  2i\pi \Big(\alpha_0d\tau +\sum_{k=2}^n \alpha_k dz_k \Big)\, .
\end{equation}


\subsubsection*{\bf B.2.3.}
The fact that $\Omega$ is not $T_\tau$-invariant prevents this 1-form 
from descending onto $\mathcal E_{1,n}$. However, 
viewed the obstruction  $T_\tau^*(\Omega)-\Omega$ explicited just above, this will not be the case  over a leaf of Veech's foliation on the Torelli space. \sk

More precisely, let $a=(a_0,a_\infty)\in \mathbb R^2$ be such that the leaf 
$\mathcal F_a=\mathcal F_{(a_0,a_\infty)}$ of Veech's foliation on $\mathcal T\!\!\!{\it or}_{1,n}$ 
is not empty. Remember that this leaf is 
cut out by the equation
\begin{equation}
\label{E:Falpha0alphaInfty}
a_0\tau+\sum_{j=2}^n \alpha_j z_j=a_\infty\, .
\end{equation}

Let $\mathcal E_a$ and $\mathcal C\!\mathcal T\!\!\!{\it or}_a$
stand for the restrictions of $\mathcal E_{1,n}$ and  $\mathcal C\!\mathcal T\!\!\!{\it or}_{1,n}$ over $\mathcal F_a$ respectively.  Clearly, $\mathcal C\!\mathcal T\!\!\!{\it or}_a$ is invariant by $T_1$ and $T_\tau$. Moreover,  from \eqref{E:Falpha0alphaInfty}, it comes that  $a_0d \tau+\sum_{j=2}^n \alpha_j dz_j=0$ when restricting to $\mathcal F_a$.  \sk

Thus, denoting by $\Omega_a$   the restriction of $\Omega$ to $\mathcal C\!\mathcal T\!\!\!{\it or}_a$, 
it follows from \eqref{E:T1Omega} and 
\eqref{E:TtauOmega} that
$$
T_1^*\big(\Omega_a\big)=\Omega_a\qquad \mbox{and}\qquad 
T_\tau^*\big(\Omega_{a}\big)=\Omega_{a}\, .
$$

This means that $\Omega_{a}$ descends to $\mathcal E_{a}$ as a holomorphic 1-form. We denote again its push-forward onto $\mathcal E_{a}$ by $\Omega_{a}$.  

Looking at  \eqref{E:OMEGA}, it is quite clear that for any $(\tau,z)\in \mathcal F_{a}$, one has
\begin{equation}
\label{E:Omega-omega}
\Omega_{a}\lvert_{E_{\tau,z}}= \omega_a(\cdot\, ,  \tau,z)
\end{equation}
where the right-hand side  is the rational 1-form \eqref{E:dLogT} on 
$E_{\tau,z}=\pi^{-1}(\tau,z)$.\sk 

With the help of $
\Omega_{a}$ we are going to make the same constructions as in  Section \ref{S:Twisted(Co)Homology} but relatively over the leaf $\mathcal F_{a}$. 


\subsubsection*{\bf B.2.4.} \hspace{-0.2cm}  
We now specialize the constructions and results of  B.1 by taking 
$$\mathcal X=\mathcal E_a\, , \qquad
S=\mathcal F_a  \qquad \mbox{ and }
\qquad \Omega=\Omega_a\, .$$

The covariant operator $\nabla_{\Omega_{a}}: \eta  \mapsto d\eta+\Omega_{a}\wedge \eta$  induces an integrable connexion on $\Omega^\bullet_{\mathcal E_a}$. Its kernel $L_a$  is a local system of rank 1 on $\mathcal E_{a}$. Moreover, it follows 
immediately from \eqref{E:Omega-omega} that given $(\tau,z)$ in  $\mathcal F_{a}$, its restriction to 
 $E_{\tau,z}=\pi^{-1}(\tau,z)$ coincides with the local system 
 $L_{\omega(\cdot\, ,  \tau,z)}$  associated to the 1-form $\omega(\cdot \, , \tau,z)$ on $E_{\tau,z}$ constructed in \S\ref{SS:OnTori},  denoted here by $ L_{\tau,z}$ for simplicity. 
 
On the leaf $\mathcal F_a\subset \mathcal T\!\!\!{\it or}_{1,n}$, one considers the 
local system $B_a=R^1 \pi_*  (L_a)$ whose stalk  at $(\tau,z)$ is nothing else but $H^1(E_{\tau,z}, L_{\tau,z})$.  The associated  sheaf $\mathcal B_a=\mathcal O_{\mathcal F_a}\otimes_{\mathbb C} B_a$ is  locally free and of rank $n$ according to Theorem \ref{T:DescriptionTwistedH1}. 

We are interested in the Gau{\ss}-Manin connection 
$$\nabla_{\!\!a}^{GM}: \mathcal B_a\rightarrow \mathcal B_a\otimes \Omega^1_{\mathcal F_a}$$ which we would like to make explicit. \mk 

Let $
 \overline{\mathcal E}_{a}$ and $[k]_a$ (for $k=1,\ldots,n$) stand for the restrictions of $\overline{\mathcal E}_{1,n}$ and of $[k]_{1,n}$ over $\mathcal F_a$.  For any $k=1,\ldots,n$, the image   of $\mathcal F_a$ by  $[k]_a$  
 is a divisor in $\overline{\mathcal E}_{a}$, denoted by $Z[k]_a$. 
 Consider their union 
  $$\mathcal Z_a=\bigcup_{k=1}^n Z[k]_a\, .$$
It is a relative divisor in  $\overline{\mathcal E}_a$ with simple normal crossing (the $Z[k]_a$'s are smooth and pairwise disjoint!),  hence Deligne's comparison theorem of B.1.5 applies: 
 there is an isomorphism of $\mathcal O_{\mathcal F_a}$-sheaves 
 \begin{equation}
 \label{E:IsomB-CokerSAffineDim1Ba}
 \mathcal B_a\simeq \mathcal O_{\mathcal F_a}\otimes_{\mathbb C}\frac{
H^0\big({\mathcal E}_a, 
 \Omega^1_{{\mathcal E}_a/{\mathcal F_a}}(*\mathcal Z_a)\big)}{
 \nabla_{{\mathcal E}_a/{\mathcal F_a}} \big( H^0\big(
{\mathcal E}_a, 
 \mathcal O_{{\mathcal E}_a}(*\mathcal Z_a)\big)\big)}\; .
 \end{equation}

 \noindent
 {\bf Remark B.2.4.}    
{\rm Actually, the geometrical picture we have can be summarized by the 
following commutative diagram}
  $$  \xymatrix@R=1.3cm@C=2cm{
 {\mathcal E}_a   \,   \ar@{^{(}->}[r]
   \ar@{->}[dr]_{\pi_a }
 &    \overline{\mathcal E}_a      \ar@{->}[d]^{\overline{\pi}_a}
 &  \ar@{_{(}->}[l]  \;  \mathcal Z_a
 \ar@{->}[dl]
  \\
 &    \mathcal F_a  \,&     } $$
 {\rm   where the two horizontal arrows are complementary inclusions. Since the restriction of $\overline{\pi}_a$ to $\mathcal Z_a$ is obviously an tale covering, this means that $\pi_a: {\mathcal E}_a \rightarrow \mathcal F_a$ 
 is precisely what  is called an {\it `elementary fibration'}
  in \cite{AndrBaldassarri}.  Even better, quotienting by the elliptic involution 
over $\mathcal F_a$ (which exists since the latter is affine), one sees that   
  $\overline{\pi}_a$ factorizes through the relative projective line $\mathbb P^1_{\!\!\mathcal F_a}\rightarrow \mathcal F_a$. In the terminology of \cite{AndrBaldassarri}, this means that the elementary fibration $\pi_a$ can be {\it `coordinatized'}. } 


\subsubsection*{\bf B.2.5.} \hspace{-0.2cm} 
  At this point, we use Theorem \ref{T:DescriptionTwistedH1} to obtain a relative version of it.
  
   We consider the horizontal non-reduced divisor supported on $\mathcal Z_a$: 
  $$
  \mathcal Z_a'=\mathcal Z_a+Z[0]_a=2Z[0]_a+\sum_{k=2}^n Z[k]_a\,. 
  $$
  
For dimensional reasons, it follows immediately from Theorem \ref{T:DescriptionTwistedH1} that 
  \begin{equation*}
 \mathcal B_a\simeq \mathcal O_{\mathcal F_a}\otimes_{\mathbb C}\frac{
H^0\big({\mathcal E}_a, 
 \Omega^1_{{\mathcal E}_a/{\mathcal F_a}}(\mathcal Z'_a)\big)}{
 \nabla_{{\mathcal E}_a/{\mathcal F_a}} \big( H^0\big(
{\mathcal E}_a, 
 \mathcal O_{{\mathcal E}_a}(\mathcal Z'_a)\big)\big)}\; .
 \end{equation*}
  
  Recall the 1-forms 
  $$\varphi_0=du\, , \qquad  \varphi_1=\rho'(u, \tau,z)du\quad  \mbox{ and } \quad \varphi_k=\big(\rho(u-z_k, \tau)-\rho(u,  \tau)\big)du$$ 
   (with $k=2,\ldots,n$)   considered in \S\ref{SS:TwistedCohomologyH1}.  We now consider them with $(\tau,z)$ varying in $\mathcal F_a$. Then these appear as elements of $H^0({\mathcal E}_a, 
 \Omega^1_{{\mathcal E}_a/{\mathcal F_a}}(\mathcal Z'_a))$.  Moreover, they span this space 
and if $
[\varphi_0],\ldots, [\varphi_n]$ stand for 
their associated classes up to   the image of $ H^0(
{\mathcal E}_a, 
 \mathcal O_{{\mathcal E}_a}(\mathcal Z'_a))$  by $\nabla_{{\mathcal E}_a/{\mathcal F_a}}$, it follows from Theorem \ref{T:DescriptionTwistedH1} that 
 $[\varphi_0],\ldots, [\varphi_{n-1}]$ form a basis of $\mathcal B_a$ over $\mathcal O_{\mathcal F_a}$. In other terms, one has  
 $$\mathcal B_a\simeq \mathcal O_{\mathcal F_a} \otimes \left( \oplus_{i=0}^{n-1}\mathbb C [\varphi_i]
 \right)
 .$$
 
 From the preceding trivialization, one deduces that
  $$
 \nabla^{GM} \begin{pmatrix}  [\varphi_0] \\  
 \vdots \\
 [\varphi_{n-1}]  
 \end{pmatrix}=  M\, \begin{pmatrix}  [\varphi_0] \\  
 \vdots \\
 [\varphi_{n-1}]  \end{pmatrix}
 $$
for a certain matrix $M\in GL_{n}(\Omega^1_{\mathcal F_a})$ which completely characterizes 
  the Gau{\ss}-Manin connection.
  We explain below how $M$ can be explicitly computed.  
  
\subsubsection*{\bf B.2.6.} \hspace{-0.2cm}    Knowing  $\nabla^{GM}$
 is equivalent to knowing the action of any  $\mathcal O_{\mathcal F_a}$-
  derivation 
  $$\nabla^{GM}_{\sigma}=\big\langle 
  \nabla^{GM}\, , \, {\sigma}
  \big\rangle\, 
   : \, \mathcal B_a\longrightarrow \mathcal B_a$$ 
  for any  vector field $\sigma$ on $\mathcal F_a$.  
 Since $\tau$ and $z_2,\ldots,z_{n-1}$ are global affine coordinates on $\mathcal F_a$, $T\mathcal F_a$ is a locally free $\mathcal O_{\mathcal F_a}$-module with $(\partial/\partial \tau,\partial/\partial z_2,\ldots,\partial/\partial z_{n-2})$ as a basis.  It follows that the Gau{\ss}-Manin connection 
 we are interested in is completely determined by the $n$ `Gau{\ss}-Manin derivations'
 \begin{equation}
\label{E:GMderivations} 
  \nabla^{GM}_\tau:= \nabla^{GM}_{\partial/\partial \tau}
 \qquad \mbox{ and }
 \qquad  \nabla^{GM}_{z_i}:= \nabla^{GM}_{\partial/\partial z_i}\quad \mbox{for }\, i=2,\ldots,n-1.
 \end{equation}
  
  \sk 
  
  Let $U$ be a non-empty open sub-domain of $\mathcal F_a$ and 
 set 
  $ \widetilde U=\pi^{-1}(U)$.
     \\ For $\widetilde \eta\in \Gamma(   \widetilde U   , \Omega^1_{\mathcal E_a})$,   
we recall the following notations: 
  \begin{itemize}
\item   $\widetilde \eta_{\mathcal E_a/\mathcal F_a}$ stands for the class of $\eta$ in 
  $\Gamma(   \widetilde U   , \Omega^1_{\mathcal E_a/\mathcal F_a})$. 
%
  \sk
  \item  $[\widetilde \eta_{\mathcal E_a/\mathcal F_a}]$ stands for the class of $\eta_{\mathcal E_a/\mathcal F_a}$ modulo the image of 
  $ \nabla_{\mathcal E_a/\mathcal F_a} $.\sk
\end{itemize}
 
   Let $\mu$ be a section  of $\pi_*\Omega^1_{\mathcal E_a/\mathcal F_a}$ over $U$.  To compute $\nabla^{GM}_\xi(\mu)$ with $\xi=\tau$ or $\xi=z_i$ with $i\in \{2,\ldots,n-1\}$, 
 we first consider a relative 1-form $\eta_{\mathcal E_a/\mathcal F_a}$  over $\widetilde U$ such that $[ \eta_{\mathcal E_a/\mathcal F_a}]=\mu$ (here we use the isomorphism \eqref{E:IsomB-CokerSAffineDim1Ba}). 
 
  In the coordinates $u,\tau,z=(z_2,\ldots,z_{n-1})$ on $\mathcal E_a$, one can write explicitly 
 $$
 \eta_{\mathcal E_a/\mathcal F_a}=N(u,  \tau,z)du
 $$
  for a holomorphic function $N$ such that for any $(\tau,z)\in U$, the map  $u\mapsto N(u,  \tau,z)$ is a  rational function on $E_\tau$, with poles at $[0]$ and $[z_2],\ldots,[z_n]$ exactly, where 
$$
z_n=\frac{1}{\alpha_n}\Big(a_\infty-a_0\tau-\sum_{k=2}^{n-1}\alpha_k z_k
\Big)
\,. 
$$

Consider the following 1-form 
 $$  \Xi=du+\frac{\rho(u,\tau)}{2i\pi} d\tau
  $$ 
which is easily seen to be invariant by $T_1$ and $T_\tau$. 

Then one defines 
\begin{equation}
\label{E:tildeeta}
 \widetilde  \eta= N\cdot \Xi=
N(u,\tau,z)\Big(du+ \frac{\rho(u,\tau)}{2i\pi}d\tau\Big). 
\end{equation}
Using the fact that  $N(u,\tau,z)$ is $\mathbb Z_\tau$-invariant with respect to $u$ when $(\tau,z)\in U$ is fixed, one verifies easily 
that the 1-form $ \widetilde  \eta$ defined just above is invariant by $T_1$ and $T_\tau$ hence descends to a section of 
$\pi_* \Omega^1_{{\mathcal E}_a}$ 
over $U$, again denoted by $\widetilde{\eta}$\footnote{More conceptually, the map $N(u,\tau,z)du\mapsto N(u,\tau,z)(du+ (2i\pi)^{-1}{\rho(u, \tau)}d\tau)$ can be seen as a splitting of the epimorphism of sheaves $\Omega^1_{\mathcal E_a}\rightarrow \Omega^1_{\mathcal E_a/\mathcal F_a}$.}. 
\sk 

The vector fields 
\begin{equation}
\label{E:ZetaStar}
  \zeta_\tau
  =\frac{\partial}{\partial \tau}-\frac{\rho}{2i\pi}   \frac{\partial}{\partial u}
  \qquad \mbox{ and }\qquad 
  \zeta_i=\frac{\partial}{\partial z_i}\quad \mbox{for }\, i=2,\ldots,n-1
  \end{equation}
all are 
 invariant by $T_1$ and by $T_\tau$ hence descend  to  rational vector fields on 
  $\overline{\mathcal E}_a$ with poles along $\mathcal Z_a$, all denoted by the same notation.   Clearly, one has $\pi_*( \zeta_\tau)=\partial/\partial\tau$ 
and   $\pi_*(\zeta_i)=\partial/\partial z_i$ for $i=2,\ldots,n-1$. 
\sk 

We now have at our disposal everything we need to compute the actions of the derivations 
\eqref{E:GMderivations} on $\mu\in \Gamma(U, \pi_*\Omega^1_{\mathcal E_a/\mathcal F_a})$: for $\star \in \{\tau,z_2,\ldots,z_{n-1}\}$, one has 
$$
\nabla_\star^{GM}\mu=     \left[   \big\langle \nabla\widetilde \eta,\zeta_\star \big\rangle_{\mathcal E_a/\mathcal F_a}\right]=
\left[   \big\langle d\widetilde \eta
+\Omega_a\wedge \widetilde \eta
,\zeta_\star \big\rangle_{\mathcal E_a/\mathcal F_a}\right]
$$
and the right hand side can be explicitly computed with the help of the explicit formulae 
\eqref{E:OMEGA}, \eqref{E:tildeeta} and \eqref{E:ZetaStar}.  

We will not make  the computations of the $\nabla_\star^{GM}[\varphi_k]$ explicit in the general case but only in the case when $n=2$ just below.

  \subsection*{\bf B.3. The  Gau{\ss}-Manin connection for elliptic curves with two conical points} One specializes now in the case when $n=2$. 
 Then   the leaf $\mathcal F_a$ is isomorphic to $\mathbb H$,  hence the $\mathcal O_{\mathcal F_a}$-module of derivations on $\mathcal F_a$  is $\mathcal O_{\mathcal F_a}\cdot ({\partial}/{\partial \tau})$. Thus in this case,  the Gau{\ss}-Manin connection is completely determined  by $
  \nabla^{GM}_\tau$. 
  \sk


   We will use below the following convention about the partial derivatives of 
  a function $N$ holomorphic in the variables $u$ and $\tau$: we will denote by $N_u$ or $N'$\vspace{-0.1cm}\\ (resp.\;$N_\tau$ or $\stackrel{\bullet}{N}$) the partial derivative of $N$ with respect to $u$ (resp.\;to $\tau$).  The notation $N'$ will be used to mean that we consider $N$ as a function of $u$ with $\tau$ fixed (and vice versa for $\stackrel{\bullet}{N}$). 
  
  
 \subsection*{\bf B.3.1.}\hspace{-0.2cm}  As in B.2.6,  let  $\eta$ be a section of $\pi_* \Omega^1_{{\mathcal E_a/\mathcal F_a}}$ over a small open subset $U \subset  \mathcal F_a\simeq \mathbb H$. It is written 
\begin{equation*}
\label{E:nota}
\eta=N(u, \tau)du
\end{equation*}
for a holomorphic function $N$ which, for any $\tau\in U$,  is  rational on $E_\tau$, with poles at $[0]$ and $[t]$ exactly, with
$$
t=t_\tau=\frac{a_0}{\alpha_1}\tau-\frac{a_\infty}{\alpha_1}\,. 
$$

Then one has (with 
$ \widetilde  \eta= N\cdot \Xi=
N(u,\tau) (du+ (2i\pi)^{-1}{\rho(u,\tau)}d\tau) 
$): 
\begin{align*}
\nabla_{\!a} \widetilde \eta= 
\nabla_{\!a}\big(  N\cdot \Xi   \big)
= dN\wedge\Xi+
N\cdot  \nabla_{\!a}\Xi
\end{align*}
and since $\langle \Xi\, , \,  \zeta_\tau    \rangle =0$, it follows that 
\begin{align}
\label{E:mopo}
\big\langle \nabla_{\!a} \widetilde \eta, 
\zeta_\tau \big\rangle
=  &  \, \big \langle dN , 
\zeta_\tau \big \rangle
\cdot \Xi
+
N\cdot  
\big\langle  d\Xi  \, , \, 
\zeta_\tau
\big\rangle
+
N\cdot  
\big\langle  \Omega_a\wedge \Xi \, , \, 
\zeta_\tau
\big\rangle\, .
\end{align}

Easy computations give 
\begin{align*}
\big \langle dN , 
\zeta_\tau  \big  \rangle= & \, N_\tau-(2i\pi)^{-1} \rho\cdot N_u\, ,\\
\big \langle d\Xi , 
\zeta_\tau \big  \rangle= & \, -(2i\pi)^{-1} \rho_u \cdot \Xi \\
\mbox{and}\quad 
\big\langle  \Omega_a\wedge \Xi \, , \, 
\zeta_\tau  \big\rangle= & \, 
\big( 
\Omega_\tau-(2i\pi)^{-1} \rho\cdot \Omega_u
\big)\cdot \Xi\,. 
\end{align*}
Injecting these into \eqref{E:mopo}
and  since $\Xi_{{\mathcal E_a/\mathcal F_a}}=du$, one  finally  gets 
\begin{equation}
\label{E:tildeNablaTau}
  \big\langle   \nabla_{\!a} \widetilde \eta, \zeta_\tau  \big\rangle
  _{{\mathcal E_a/\mathcal F_a}}
  = N_\tau du 
  + \Omega_\tau 
Ndu  
  -(2i\pi)^{-1}
  \nabla_{\mathcal E_a/\mathcal F_a}\big(  \rho\, N   \big)
  \end{equation}
  where $ \nabla_{\mathcal E_a/\mathcal F_a}=d_u(\cdot )+\Omega_u du\wedge\cdot $ stands for the 
vertical covariant derivation  
  \begin{align*}
 \nabla_{\mathcal E_a/\mathcal F_a} \; : \; \mathcal O_{\mathcal E_a/\mathcal F_a} & \longrightarrow \Omega^1_{\mathcal E_a/\mathcal F_a}\\
 F=F(u, \tau) & \longmapsto  F_u du+F\,\Omega_u\,  du\, . 
\end{align*}

It follows essentially from \eqref{E:tildeNablaTau} that the 
 differential operator  
\begin{align}
\widetilde \nabla_\tau : \Omega^1_{\mathcal E_a/\mathcal F_a}& \longrightarrow \Omega^1_{\mathcal E_a/\mathcal F_a}     \nonumber \\
N du & \longmapsto N_\tau du 
  + \Omega_\tau  N du  
  -\frac{1}{2i\pi}
  \nabla_{\mathcal E_a/\mathcal F_a}\big(  \rho\, N   \big)
  \label{E:TildeNablaTau}
\end{align}
 is a $\pi^{-1} \mathcal O_{\mathcal F_a}$-derivation which  is nothing else but a lift of the Gau{\ss}-Manin deri\-vation 
$\nabla^{GM}_\tau$ we are interested in.  The fact that $\widetilde \nabla_\tau$ is explicit will allow us to determine explicitly the action of $\nabla^{GM}_\tau$ below. 
\mk 

\noindent{\bf Remark B.3.1.} 
{\rm  It is interesting to compare our formula \eqref{E:TildeNablaTau} for $\widetilde \nabla_\tau$ to the corresponding one in \cite{ManoWatanabe}, namely the 
specialization when $\lambda=0$ of the one for the  differential operator 
$\nabla_\tau$ given 
 just before Proposition 4.1 page 3878 in  \cite{ManoWatanabe}.   
 The latter is not completely explicit since  in order to compute $\nabla_\tau N\,du$ with $N$ as above it is necessary to introduce a deformation $N(u,\tau,\lambda)$ of $N=N(u, \tau)$ which 
 is meromorphic with respect to $\lambda$.  However such deformations $\varphi_i(u,\tau,\lambda)du$ are explicitly given for the $N_i=\varphi_i(u,\tau,0) du$'s ({\it cf.}\;\cite[p.\,3875]{ManoWatanabe}),  hence Mano and Watanabe's formula can be used  to effectively determine the Gau{\ss}-Manin connection. 
 Note that our arguments above show that  $\widetilde \nabla_\tau$ is a lift of 
the Gau{\ss}-Manin derivation $\nabla_{\tau}^{GM}$  indeed. The corresponding statement is not justified  in 
\cite{ManoWatanabe}  and is 
implicitly left to  the reader.\sk 

Finally, it is fair to mention a  notable feature of Mano-Watanabe's operator $\nabla_\tau$ that our $\widetilde \nabla_\tau$ does not share: for $i\in \{0,1,2\}$, $ \nabla_\tau N_i$ is a rational 1-form on $E_\tau$, with polar divisor $\geq 2[0]+[t_\tau]$,  hence can be written as a linear combination in $N_0,N_1$ and $N_2$.  This is not the case for the $\widetilde \nabla_\tau N_i$'s. For instance, $\widetilde \nabla_\tau N_1$ has a pole of order four at $[0]$ (see also B.3.3 below).} 


 \subsection*{\bf B.3.2. Some explicit formulae}\hspace{-0.0cm}
In the case under study, we have
$$
T(u,\tau)=
 e^{2i\pi a_0 u}\theta(u)^{\alpha_1}\cdot \theta(u-t)^{-\alpha_1}
$$
(with $t=(a_0/\alpha_1)\tau-(a_\infty/\alpha_1)$) hence 
$$\Omega=d\log T=\Omega_u du+\Omega_\tau d\tau$$
 with 
  \begin{align}
  \label{E:OmegaU&OmegaTau}
   \Omega_u= & \, {\partial \log T}/{\partial u}=   2i\pi {a_0}+ \alpha_1 \big( \rho(u)-\rho(u-t)
  \big)\\
  \mbox{ and }\quad 
\Omega_\tau= & \, {\partial \log T}/{\partial \tau}= 
   \frac{\alpha_1}{4i\pi} \left(
  \frac{\theta''(u)}{\theta(u)}- \frac{\theta''(u-t)}{\theta(u-t)}
  \right)+ a_0 \rho(u-t)\, . \nonumber
 \end{align}
  
 For $i=0,1,2$, one writes $ \varphi_i=N_i(u) du$ with 
  \begin{equation*}
N_0(u)=1\, , \qquad 
N_1(u)=\rho'(u)\qquad \mbox{ and }
\qquad  
N_2(u)=\rho(u-t)-\rho(u)\, .
\end{equation*}

The following functions will appear in our computations below: 
    \begin{align*}
      P(u)=P(u,\tau)=
    &  \frac{\theta''(u)}{\theta(u)}- \frac{\theta''(u-t)}{\theta(u-t)}
  - 2 \big( \rho(u)-\rho(u-t)
  \big)\cdot\rho(u) 
  \\ 
  \mbox{ and }\quad   
  \mu(u)=\mu(u,\tau)
  =  & \, -\frac{1}{2} 
  \left(
  \frac{\theta'''(u)}{\theta(u)} -
  \frac{\theta''(u)\theta'(u)}{\theta(u)^2}
  \right)   
 \, .    \end{align*}
  \sk 
  
  \noindent{\bf Lemma B.3.2.1.} 
{\it For any fixed $\tau\in \mathbb H$, $P(u)$ is $\mathbb Z_\tau$-invariant  and one has 
\begin{equation}
\label{E:DecompositionP}
P=
 \left[\rho'(t)+\rho(t)^2-\frac{\theta'''}{\theta'}
  \right]  \cdot N_0
 +2\cdot N_1+ 2\rho(t)  \cdot N_2
 \end{equation}
 as an elliptic function of $u$.
}\sk 

\begin{proof} Using \eqref{E:ThetaQuasiPeriodicity} and \eqref{E:mumu}, one verifies easily that for $\tau$ fixed, $P(\cdot,\tau)$ is $\mathbb Z_\tau$-invariant and, viewed as a rational function on $E_\tau$,  its polar divisor is $2[0]+[t]$.  By straightforward computations, one verifies that $P(\cdot)$ has the same polar part  as  the right-hand-side of \eqref{E:DecompositionP}.  By evaluating at one point (for instance at $u=0$), the lemma follows.
%
\end{proof}
\mk

By straightforward computations, one verifies that the following holds true: 
\sk 

  \noindent{\bf Lemma B.3.2.2.} 
{\it For $\tau\in \mathbb H$ fixed, the meromorphic function 
$$  u \longmapsto \mu(u)+\rho(u)\rho'(u)$$ 
is an elliptic function, {\it i.e.}\;is $\mathbb Z_\tau$-invariant in the variable $u$. 
}


        \subsection*{\bf B.3.3. Computation of $\boldsymbol{\nabla_\tau^{GM} [{\varphi}}_0]$}
         Since $N_0$ is constant, the partial derivatives 
          ${\partial N_0}/{\partial u}$ and ${\partial N_0}/{\partial \tau}$ both vanish. Then
          from \eqref{E:tildeNablaTau}, it comes 
      \begin{align*}
         \widetilde  \nabla_\tau \varphi_0=  & \, 
          \bigg[
           \Omega_\tau
          -\frac{1}{2i\pi} \Big( \rho_u
  +  \Omega_u \cdot \rho  \Big) 
   \bigg] du \\
   =  & \, 
          \Bigg[
           \frac{\alpha_1}{4i\pi} \left(
  \frac{\theta''(u)}{\theta(u)}- \frac{\theta''(u-t)}{\theta(u-t)}
  \right) +  a_0 \rho(u-t)  \\
   &              \hspace{2.2cm}-   \frac{1}{2i\pi} \Big( \rho'(u)
  +  
\Big(2i\pi  \frac{a_0}{\alpha_1}+ \alpha_1 \big( \rho(u)-\rho(u-t)
  \big)\Big)  
   \cdot \rho(u)  \Big) 
   \Bigg] du
   \\
  =   & \,  \frac{a_0}{\alpha_1} du -\frac{1}{2i\pi}\rho'(u)du+
     \frac{\alpha_1}{4i\pi} P(u) du\, . 
 \end{align*}       
 
 It follows from Lemma B.3.2.1. 
 that 
  $$
\widetilde \nabla_\tau {\varphi}_0=
\frac{\alpha_1}{4i\pi} \left(\rho'(t)+\rho(t)^2-\frac{\theta'''}{\theta'}\right)\cdot  {\varphi}_0
+
\frac{\alpha_1-1}{2i\pi}\cdot  {\varphi}_1
+\left(a_0+\frac{\alpha_1}{2i\pi}\rho(t)\right)\cdot  {\varphi}_2
$$
  thus  in (twisted) cohomology, because  $2i\pi a_0[\varphi_0]=\alpha_1[\varphi_2]$ ({\it cf.}\;\eqref{E:RELATION}), one deduces that the following relation holds true: 
   \begin{equation}      
\label{E:NablaTauGLPHI0}
           \nabla_\tau^{GM} [{\varphi}_0]
       =  \left(2i\pi\frac{a_0^2}{\alpha_1} +a_0\rho(t)+
\frac{\alpha_1}{4i\pi} \left(\rho'(t)+\rho(t)^2-\frac{\theta'''}{\theta'}\right)
\right)
 [{\varphi}_0]
+
\frac{\alpha_1-1}{2i\pi}\,   [{\varphi}_1]\,.
       \end{equation}


  \subsection*{\bf B.3.4. Computation of $\boldsymbol{ \nabla_\tau^{GM}[ {\varphi}_1]}$}
\sk 

From  \eqref{E:tildeNablaTau}, it comes
\begin{align*}
\widetilde \nabla_\tau\varphi_2= 
\widetilde \nabla_\tau \big(\rho' du\big)= &  \bigg[ 
\stackrel{\bullet}{\rho}\! {}'+ 
 \Omega_\tau \, \rho'
 -\frac{1}{2i\pi} \Big( \rho\cdot \rho''+(\rho')^2
   + \Omega_u\cdot \rho \rho'  \Big) \bigg] du \, .
   \end{align*}

By construction, for any $\tau\in \mathbb H$ fixed, the right-hand-side is a rational 1-form on $E_\tau$.   It follows from \cite{ManoWatanabe} that there exist three `constants depending on $\tau$', $A_i(\tau)$ with $i=0,1,2$ and a rational function $\Phi(\cdot)=\Phi(\cdot,\tau)$ depending on $\tau$, all to be determined,  such that 
\begin{align*}
\widetilde \nabla_\tau\varphi_2=       A_0(\tau)\cdot \varphi_0+  A_1(\tau) \cdot \varphi_1+
  A_2(\tau) \cdot \varphi_2 
  -\frac{1}{2i\pi}
 \nabla_{\mathcal E_a/\mathcal F_a}\Phi. 
\end{align*}

Using \eqref{E:OmegaU&OmegaTau} and the following formulae
\begin{align*}
\rho(u)= & \, {\theta'(u)}/{\theta(u)}\\
\rho'(u)=& \, {\theta''(u)}/{\theta(u)}-\big({\theta'(u)}/{\theta(u)}\big)^2\\
\rho''(u)=& \, {\theta'''(u)}/{\theta(u)}-3{\theta''(u)\theta'(u)}/{\theta(u)^2}
+2\big( {\theta'(u)}/{\theta(u)}  \big)^3
\\
\mbox{and}
\;    \stackrel{\bullet}{\rho}\! {}'(u)
= & \, \frac{1}{4i\pi}
\left[\frac{
{\theta}^{(4)}(u)}{\theta(u)}-
\left(\frac{\theta''(u)}{\theta(u)}\right)^2
-2 \frac{{\theta}'''(u)\theta'(u)}{\theta(u)^2}
+2 \frac{\theta''(u)\theta'(u)^2}{\theta(u)^3}
\right]
\end{align*}
one verifies by lengthy but straightforward computations that one has 
\begin{align*}
A_0(\tau) = & \, -a_0 \mu(t)
-\frac{\alpha_1}{4i\pi}\Big(
\mu'(t)+2\rho(t)\mu(t)-3\, \mu'(0)\Big);
\\
A_1(\tau) = & \, 
-a_0\rho(t)
-\frac{\alpha_1}{4i\pi}\left(\rho'(t)+\rho(t)^2-\frac{\theta'''}{\theta'}
\right);
\\
A_2(\tau) = & \, a_0 \rho'(t)-\frac{\alpha_1}{2i\pi}\mu(t) \\
\mbox{ and }\quad   
\Phi(u) = & \,  \mu(u)+\rho(u)\rho'(u) 
 \, .
\end{align*}

Since $\Phi(u)$ is rational according to  Lemma B.3.2.2., 
one has $ [\nabla_{\mathcal E_a/\mathcal F_a}\Phi]=0$ in (twisted) cohomology and because  $2i\pi a_0[\varphi_0]=\alpha_1[\varphi_2]$,  
one obtains that \begin{align*}
\nabla_\tau^{GM} [{\varphi}_1]= & \, 
\left(
A_0(\tau)+2i\pi\frac{a_0}{\alpha_1}A_2(\tau)
\right)\cdot [{\varphi}_0]+
A_1(\tau) \cdot [{\varphi}_1]
\end{align*}
uniformly with respect to $\tau\in \mathbb H$, that is, more explicitly
\begin{align}
\label{E:NablaTauGLPHI1}
\nonumber
\nabla_\tau^{GM} [{\varphi}_1]= & \,
\left(
2i\pi\frac{a_0^2}{\alpha_1}\rho'(t)
-2 a_0\mu(t)
-\frac{\alpha_1}{4i\pi}\Big(
\mu'(t)+2\rho(t)\mu(t)-3\, \mu'(0)\Big)
\right)\cdot [ {\varphi}_0]
\\ & \hspace{2cm} \, -\left(a_0\rho(t)
+\frac{\alpha_1}{4i\pi}\left(\rho'(t)+\rho(t)^2-\frac{\theta'''}{\theta'}
\right)
\right)
 \cdot  [{\varphi}_1]\, . 
\end{align}



 \subsection*{\bf B.3.5. 
The Gau{\ss}-Manin connection $\nabla^{GM}$ and the differential equation 
satisfied by the components of Veech's map}

From \eqref{E:NablaTauGLPHI0} and \eqref{E:NablaTauGLPHI1}, one deduces the  
\sk 

\noindent{\bf Thorem B.3.5.} 
{\it The action of the Gau{\ss}-Manin derivation $\nabla_{\!\tau}^{GM}$ in the basis 
formed by $[{\varphi}_0]$ and $ [{\varphi}_1]$ is given by 
\begin{equation}
\label{T:GM}
\nabla_{\!\tau}^{GM} \begin{pmatrix}
[{\varphi}_0]\\
[{\varphi}_1]
\end{pmatrix}=
\begin{pmatrix}
 M_{00}&  M_{01} \\ 
 M_{10}&   M_{11}
\end{pmatrix}
\cdot 
\begin{pmatrix}
[{\varphi}_0]\\
[{\varphi}_1]
\end{pmatrix}
\end{equation}
with
\begin{align*}
M_{00}=  &  2i\pi \frac{ a_0^2}{\alpha_1}+ a_0 \rho(t)+
     \frac{\alpha_1}{4i\pi} \left(      \rho'(t)+\rho(t)^2-\frac{\theta'''}{\theta'}
      \right) \, ;    \\ 
M_{01}=  &   \frac{\alpha_1-1}{2i\pi}\, ;   \\ 
M_{10}=  & 2i\pi\frac{ a_0^2}{\alpha_1}\rho'(t)
-2 a_0\mu(t)
-\frac{\alpha_1}{4i\pi}\Big(
\mu'(t)+2\rho(t)\mu(t)-3\, \mu'(0)\Big)
\\ 
\mbox{and}\quad M_{11}= &  -a_0\rho(t)
-\frac{\alpha_1}{4i\pi}\left(\rho'(t)+\rho(t)^2-\frac{\theta'''}{\theta'}
\right)\, . 
\end{align*}
}

         Consequently, according to B.1.5, for any horizontal family of  twisted 1-cycles $\tau\mapsto \boldsymbol{\gamma}(\tau)$, if one sets 
    \begin{align*} 
    F_0(\tau)=      \int_{\boldsymbol{\gamma}(\tau)} T(u,\tau) du
\qquad  \mbox{and}\qquad 
         F_1(\tau) = 
           \int_{\boldsymbol{\gamma}(\tau)}  T(u,\tau) \rho'(u,\tau) du
\end{align*}
then $F={}^t\!(F_0,F_1)$       
          satisfies the differential system 
          \begin{equation}
          \label{E:tyto}
          \stackrel{\bullet}{F}=dF/d\tau=M\, F
          \end{equation}
           where $M=M(\tau)$  is the $2\times 2$  matrix appearing in \eqref{T:GM}.  \sk
          
 At this point, we recall the definition of Veech's map: it is the map          
 \begin{align}
\label{E:DernireEquation}
V
 \, : \; \mathcal F_a\simeq \mathbb H 
 & \longrightarrow \mathbb P^1\; , \quad
\tau \longmapsto 
V(\tau)=\begin{bmatrix}
v_0(\tau) \\
v_\infty (\tau)
\end{bmatrix}
\end{align}
 with for every $\tau\in \mathbb H$: 
$$ v_0(\tau)= \int_{\boldsymbol{\gamma}_0} T( u, \tau)du 
\qquad \mbox{ and }
\qquad  
v_\infty(\tau)=
 \int_{\boldsymbol{\gamma}_\infty} T(u,  \tau)du\, .\sk
$$

Then applying Lemma 6.1.1 of   \cite[\S3.6.1]{Gauss2Painlev} (see also Lemma 
 A.2.2. 
above) 
  to the differential system  \eqref{E:tyto}, one obtains the
\mk

\noindent{\bf Corollary B.3.5.} 
{\it 
The components $v_0$ and $v_\infty$ of Veech's map of the leaf $\mathcal F_a$ 
form a basis of the space of solutions of the  following  linear differential equation 
\begin{equation}
\label{E:2dOrderEDOVeech}
\stackrel{\bullet\bullet}{v} -\big(2i\pi{ a_0^2}/{\alpha_1}\big) \,    \stackrel{\bullet}{v}+
\Big(   \det M(\tau)+  \stackrel{\bullet}{M_{11}} \Big)
 \, v = 0 \, . 
\end{equation}}

The coefficient of $\stackrel{\bullet}{v}$ in \eqref{E:2dOrderEDOVeech} being constant, 
the functions  
$$
\qquad \qquad 
\widetilde v_\star(\tau)=\exp\big(-i\pi ({a_0^2}/{\alpha_1})\cdot \tau\big)\, v_\star(\tau)\qquad 
\mbox{with }\; \star=0,\infty
$$
 satisfy a  linear second order differential equation in reduced form and can be taken as  the components of Veech's map \eqref{E:DernireEquation}.\sk 
 
 From our point of view, the second-order Fuchsian differential equation \eqref{E:2dOrderEDOVeech} is for elliptic curves with two punctures what {Gau{\ss}}  
      hypergeometric differential equation \eqref{HGE} is for $\mathbb P^1$ with four punctures.\mk 
  
Finally, in the case when $a=(a_0,a_\infty)=\alpha_1(m/N,-n/N)$ with $N\geq 2$ and $(m,n)\in \{0,\ldots,N-1\}^2\setminus \{(0,0)\}$, we have $t=(m/N)\tau+(n/N)$,  thus 
$$
T(u)=e^{\frac{2i\pi m}{N}\alpha_1} \theta(u)^{\alpha_1}\theta\big(u-(m/N)\tau-n/N\big)^{-\alpha_1}\, .
$$

Specializing Theorem B.3.5.\;and Corollary B.3.5.\;to this case, we let 
the readers verify that one recovers (the special case of) Mano's differential system 
considered in \S\ref{SS:ManoDifferentialSystemAlgebraicLeaves}.


\end{document}